\documentclass[11pt]{amsart}
\pdfoutput=1

\usepackage{amsmath}
\usepackage{amsfonts}
\usepackage{amssymb}
\usepackage{graphicx}
\usepackage{mathrsfs}
\usepackage{pb-diagram}
\usepackage{epstopdf}
\usepackage{amsmath,amsfonts,amsthm,enumerate,amscd,latexsym,curves}
\usepackage{bbm}
\usepackage[mathscr]{eucal}
\usepackage{epsfig,epsf,epic}
\usepackage{caption}
\usepackage[font=small]{subcaption}
\usepackage{multirow}
\usepackage{color}
\usepackage[all]{xy}
\usepackage{fourier}
\usepackage{tikz}
\usepackage{tikz-cd}
\usetikzlibrary{fit,matrix}
\usetikzlibrary{matrix,decorations.pathreplacing,calc,positioning}

\usepackage{calc}
\usepackage{float}
\usepackage{comment}
\usepackage{longtable}
\usepackage{adjustbox}
\usepackage{enumitem}
\usepackage{systeme}
\usepackage{mathtools}
\usepackage{array,booktabs,colortbl} 
\usepackage{collcell}
\usepackage{xcolor}
\usepackage{cancel}
\usepackage{lipsum}
\usepackage{arydshln}
\usepackage[linktocpage]{hyperref}
\hypersetup{
    colorlinks=true,
    linkcolor=blue,
    filecolor=magenta,      
    urlcolor=cyan,
}
\usepackage{rotating}
\usepackage{bm}
\usepackage{stmaryrd} 
\usepackage{blkarray} 
\usepackage{stackrel}
\usetikzlibrary{decorations.pathmorphing}
\usepackage{kotex}
\usepackage[percent]{overpic}
\usetikzlibrary{decorations.markings}
\usepackage[font=scriptsize,justification=centering]{caption}

\usetikzlibrary{angles,quotes,calc}
\pgfdeclarelayer{bg}
\pgfsetlayers{bg,main}
\usepackage{setspace}

\graphicspath{ {images/} }

\definecolor{colorG}{rgb}{0, 0.5, 0.0} 
\definecolor{colorH}{rgb}{0.7, 0.7, 0.0} 
\definecolor{colorF}{rgb}{1.0, 0.5, 0.0} 
\definecolor{colorR}{rgb}{0.4, 0.5, 0.13} 
\definecolor{colorT}{rgb}{0.7, 0.5, 0.13} 
\definecolor{colorQ}{rgb}{0.6, 0.3, 0} 
\definecolor{rosepink}{rgb}{1.0, 0.4, 0.8}

\definecolor{darkgreen}{rgb}{0.0, 0.4, 0.0}
\definecolor{Lightgray}{rgb}{0.7, 0.7, 0.7}
\definecolor{LLightgray}{rgb}{0.9, 0.9, 0.9}

\definecolor{yellow2}{rgb}{0.9, 0.9, 0.0}
\definecolor{yellow3}{rgb}{0.8, 0.8, 0.0}
\definecolor{yellow4}{rgb}{0.7, 0.7, 0.0}

\newtheorem{thm}{Theorem}[section]
\newtheorem{lemma}[thm]{Lemma}

\newtheorem{cor}[thm]{Corollary}
\newtheorem{prop}[thm]{Proposition}
\newtheorem{defn}[thm]{Definition}
\newtheorem{ex}[thm]{Example}
\newtheorem{remark}[thm]{Remark}

\newtheorem{notation}[thm]{Notation}

\numberwithin{equation}{section}

\newcommand{\AI}{A_\infty}

\newcommand{\Hom}{{\rm Hom}}

\newcommand{\Fuk}{\operatorname{Fuk}}
\newcommand{\sing}{\operatorname{sing}}

\newcommand{\BD}{\operatorname{BD}}

\newcommand{\sign}{\operatorname{sign}}
\newcommand{\hol}{\operatorname{hol}}

\newcommand{\x}{\operatorname{x}}

\newcommand{\XY}{\operatorname{xy}}
\newcommand{\YZ}{\operatorname{yz}}
\newcommand{\ZX}{\operatorname{zx}}

\newcommand{\OL}[1]{\overline{#1}}

\newcommand{\Z}{\mathbb{Z}}

\newcommand{\Tri}{\operatorname{Tri}}

\newcommand{\CM}{\operatorname{CM}}

\newcommand{\im}{\operatorname{im}}
\newcommand{\cok}{\operatorname{coker}}

\newcommand{\id}{\operatorname{id}}

\newcommand{\MF}{\mathrm{MF}}

\newcommand{\Spec}{\mathrm{Spec}}

\newcommand{\bL}{\mathbb{L}}

\newcommand{\field}{\mathbb{C}}
\newcommand{\Span}{\operatorname{Span}_{\field}}
\newcommand{\POP}{\Sigma}
\newcommand{\hash}{}

\begingroup\lccode`~=`& \lowercase{\endgroup
\newcommand{\smat}[1]{%
  \let~=&
  \begin{smallmatrix}#1\end{smallmatrix}
}}

\begingroup\lccode`~=`& \lowercase{\endgroup
\newcommand{\spmat}[1]{%
  \left(
  \let~=&
  \begin{smallmatrix}#1\end{smallmatrix}
  \right)
}}

\makeatletter
\def\l@subsection{\@tocline{2}{0pt}{3pc}{5pc}{}}
\def\l@subsubsection{\@tocline{3}{0pt}{5pc}{7.5pc}{}}
\makeatother

\newcommand{\primemu}{\rho}
\newcommand{\primelambda}{\eta}
\newcommand{\primeLambda}{H}
\newcommand{\operatorm}{\mathfrak{m}}
\newcommand{\LocalF}{\mathcal{F}^\mathbb{L}} 
\newcommand{\LocalPhi}{\varPhi^\mathbb{L}} 
\newcommand{\LocalPsi}{\varPsi^\mathbb{L}} 

\makeatletter
\def\mathunderbar#1{\underline{\sbox\tw@{$#1$}\dp\tw@\z@\box\tw@}}
\makeatother
\usepackage[normalem]{ulem}

\newenvironment{xsmallmatrix}[1] 
  {\renewcommand\thickspace{\kern#1}\smallmatrix}
  {\endsmallmatrix}

\begin{document}

\title[Canonical form of matrix factorizations from Fukaya category of surface]{Canonical form of matrix factorizations from Fukaya category of surface}
\author[Cho]{Cheol-Hyun Cho}
\address{Department of Mathematical Sciences, Research Institute in Mathematics\\ Seoul National University\\ Gwanak-gu\\Seoul \\ South Korea}
\email{chocheol@snu.ac.kr}
\author[Rho]{Kyungmin Rho}
\address{Institut f\"ur Mathematik\\ Universit\"at Paderborn\\ Warburger Str. 100\\Paderborn \\ Germany}
\email{rho@math.uni-paderborn.de}


\begin{abstract}

This paper concerns homological mirror symmetry for the pair-of-pants surface (A-side) and the non-isolated surface singularity $xyz=0$ (B-side). Burban-Drozd classified indecomposable maximal Cohen-Macaulay modules on the B-side. We prove that higher-multiplicity band-type modules correspond to higher-rank local systems over closed geodesics on the A-side, generalizing our previous work for the multiplicity one case. This provides a geometric interpretation of the representation tameness of the band-type maximal Cohen-Macaulay modules, as every indecomposable object is realized as a geometric object.

We also present an explicit canonical form of matrix factorizations of $xyz$ corresponding to Burban-Drozd's canonical form of band-type maximal Cohen-Macaulay modules. As applications, we give a geometric interpretation of algebraic operations such as AR translation and duality of maximal Cohen-Macaulay modules as well as certain mapping cone operations.

\end{abstract}

\maketitle

\vspace{-5mm}

\begin{spacing}{0}
\tableofcontents
\end{spacing}
\vspace{-0mm}

\section{Introduction}

A version of Homological mirror symmetry (HMS) conjecture of Kontsevich
\cite{kontsevich94,kontsevich98}
says that
the derived wrapped Fukaya category of a symplectic manifold $\Sigma$ (A-side) and 
 the singularity category of its mirror
Landau-Ginzburg (LG) model
$
\left(X,W:X\rightarrow\mathbb{C}\right)
$
(B-side) are equivalent:
$$
D^{\pi}\left(W\Fuk\left(\Sigma\right)\right) \simeq D_{\sing}\left(W^{-1}\left(0\right)\right).
$$
%
%
%
%
%
HMS
between punctured Riemann spheres
$\Sigma = S^2\setminus\left\{n\text{ points}\right\}$
($n\ge3$)
and the corresponding LG models
was  established
in \cite{AAEKO}.

In the case of the $3$-punctured sphere,
the corresponding LG model is given by $\left(\field^3,xyz\right)$.
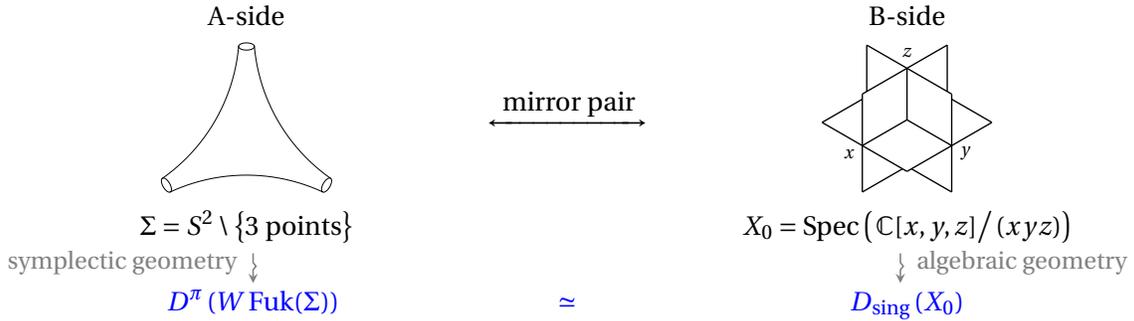
$$

\end{matrix}
$$
%
%
%
%
%
This equivalence has been shown on the level of generators, which in this case consist of
any two of three
non-closed
curves
$L_{\XY}$, $L_{\YZ}$ and $L_{\ZX}$ on the A-side
(Figure \ref{fig:GeneratingArcs}),
and the corresponding objects on the B-side
(Remark \ref{rmk:StringCorrespondence}).
This work of Abouzaid-Auroux-Efimov-Katzarkov-Orlov \cite{AAEKO} inspired a lot of further developments in homological mirror symmetry.
However,
it is hard to compare more complicated objects in both sides
directly from this equivalence.

On the other hand, it is known that the following three categories are equivalent \cite{E80,Buc87,Orlov03}:
$$
\setlength\arraycolsep{2pt}
\color{blue}
\begin{matrix}
  \underline{\operatorname{MF}}(xyz)
&
  \color{black}
  \overset{
  \small
  \begin{matrix}
    \text{\tiny Eisenbud}
  \\[-1mm]
    \simeq
   \\[-1mm]
  \end{matrix}
  }{\longrightarrow}
&
  \underline{\operatorname{CM}}(A)
&
  \color{black}
  \overset{
  \small
  \begin{matrix}
    \text{\tiny Buchweitz}
  \\[-1mm]
    \simeq
   \\[-1mm]
  \end{matrix}
  }{\longrightarrow}
&
  D_{\sing}\left(\hat{X_0}\right)
\end{matrix}
$$
\hspace{-1.9mm}
Here,
$
\underline{\MF}(xyz)
$
is the homotopy category of \textbf{matrix factorizations} of $xyz$,
$
\underline{\CM}(A)
$
is the stable category of \textbf{maximal Cohen-Macaulay modules} over $A:=\left.\field[[x,y,z]]\right/(xyz)$,
and
$D_{\sing}\left(\hat{X_0}\right)$
is the singularity category of
$
\hat{X_0}:=\Spec(A).
$
In this paper,
we work with power series rings instead of polynomial rings (Remark \ref{rmk:Completion}).

In a recent work \cite{BD17},
Burban-Drozd developed a new representation-theoretic method
to deal with maximal Cohen-Macaulay modules over certain non-isolated surface singularities including
$
A
$.
As a consequence, they classified all indecomposable
classes of such
modules,
which fall into \textbf{band-type} (continuous series)
and \textbf{string-type} (discrete series).
This proves
that those
singularities have \textbf{tame} Cohen-Macaulay representation type. 

Thus, a natural question is which objects of the Fukaya category correspond to the indecomposable maximal Cohen-Macaulay modules over $A$
(Remark 9.8.8 in \cite{BD17}).
Especially, it is of great interest whether their symplectic counterparts are realized as \textbf{geometric objects} in the Fukaya category.
This question will be answered in the present paper
by giving an explicit correspondence: 
\newpage
\begin{thm}\label{thm:ModuleGeodesic}
Under homological mirror symmetry,
there is a one-to-one correspondence
$$
\left\{
\textnormal{\textnormal{\textbf{closed geodesics}}}
\right.
\hspace{-0.5mm}
\footnote{
In our paper,
a closed geodesic is always oriented, non-periodic,
immersed
(i.e., not necessarily simple)
and considered up to orientation-preserving reparameterization.
}
\hspace{-0.5mm}
\left.\left.
\textnormal{\textnormal{\textbf{
in $\POP$ with an indecomposable local system}}}
\right\}
\right/\sim_\textnormal{gauge equivalence}
\quad\overset{1:1}{\leftrightarrow}
$$
$$
\left.\left\{\textnormal{\textnormal{\textbf{band-type indecomposable objects in $\underline{\CM}(A)$}}}
\right\}
\right/
\sim_\textnormal{isomorphism},
$$
where $\POP$ is given a hyperbolic metric with three cusps.

\end{thm}

In our setting of the Fukaya category of $\Sigma$,
objects are oriented immersed curves in $\POP$ with a local system.
We call them \textbf{loop-type} or \textbf{arc-type}
according to whether the curve is a loop (closed curve)
or an arc (starting and ending at $\partial \POP$).
Closed geodesics are representatives in certain (but not all) free homotopy classes of oriented loops
in $\POP$.
Thus,
Theorem \ref{thm:ModuleGeodesic}
describes a correspondence between indecomposable objects of loop-type in the Fukaya category
and band-type in the category of maximal Cohen-Macaulay modules.
A similar correspondence between arc-type and string-type objects
can be also made,
but we do not cover them in the present paper (see Remark \ref{rmk:StringCorrespondence}).

In our previous work \cite{CJKR},
we already established a correspondence between
loop-type objects of rank $\primemu=1$ and
band-type objects of multiplicity $\mu=1$,
and found a canonical form of matrix factorizations for that case.
The main purpose of the present paper is to extend it to a correspondence between
loop-type objects of arbitrary \textbf{geometric rank $\primemu$}
and
band-type objects of arbitrary \textbf{algebraic multiplicity $\mu$}.

To convert objects of Fukaya category into matrix factorizations,
we use the \textbf{localized mirror functor} 
$$
\setlength\arraycolsep{2pt}
\color{blue}
\begin{matrix}
\LocalF:
&
  D^{\pi}\left(W\Fuk\left(\POP\right)\right)
&
  \color{black}
  \overset{
  }{\longrightarrow}
&
  \underline{\MF}\left(xyz\right)
\end{matrix}
$$
\hspace{-1.9mm}
developed by the first author with Hong-Lau in \cite{CHL}.
In the present work,
we elaborate its computational aspect to apply it to higher-rank local systems.
In particular,
we give an explicit formula (\ref{eqn:MatrixOfComponentOfDeformedDifferential})
for resulting matrix factorizations, and use it to deduce matrix factorizations of higher-rank local systems
directly from the result on rank $1$ cases
(Proposition \ref{prop:HigherRankLMF}).
It presents us a
\textbf{canonical form of matrix factorizations}
for higher-rank objects
in terms of \textbf{loop data},
extending the previous version for rank $1$ case in \cite{CJKR}.

Burban-Drozd also provided a canonical form of maximal Cohen-Macaulay modules
in terms of \textbf{band data}.
But the corresponding matrix factorizations under Eisenbud's equivalence were not known,
due to the complexity of
\emph{Macaulayfication} process.
With the help of homological mirror symmetry, we now have a candidate.
Indeed, we demonstrate that our canonical form of matrix factorizations fits perfectly into this framework under
an explicit \textbf{conversion formula} between loop data and band data.
The presence of this conversion formula suggests that it would have been hardly attainable otherwise.


For
the proof,
we define the notion of \textbf{$\left(\lambda,\Lambda\right)$-substitution pair}
and use its homological property.
It enables us to extend the Macaulayfication result as well as conversion formula obtained in \cite{CJKR}
to higher-multiplicity cases
(Theorem \ref{thm:MFCMCorrespondence}).

After all,
it turns out that 
the geometric rank $\primemu$
and
the algebraic multiplicity $\mu$
coincide in a majority of cases.
But surprisingly,
there are a few (countably many) exceptions called \textbf{degenerate cases},
where
two parameters differ by $1$ as $\primemu=\mu-1$.
This can be interpreted as
an inevitable phenomenon
following from
the elimination of the regular module $A$
that occurs
when we take the stable category
$\underline{\CM}(A)=\CM(A)\setminus\left\{A\right\}$.

Analyzing
the correspondence
of objects in those cases
is quite tricky both on geometric and algebraic sides:
The geometric loop is freely homotopic to the reference loop (Seidel Lagrangian),
so we perturb it to prevent an immersed cylinder (\S \ref{sec:ExceptionalCase}).
On algebraic side, the corresponding module has one exceptional Macaulayfying element,
which does not appear as a family of $\lambda$
and cannot be obtained from the above $\left(\lambda,\Lambda\right)$-substitution process.
So we perform an additional computation for this case and find that the existence of such an element causes the degeneration $\primemu=\mu-1$
(\S \ref{sec:DegenerateCase}).
\newpage

\begin{spacing}{0.98}
\noindent
\textbf{Applications.}
The
correspondence obtained in Theorem \ref{thm:ModuleGeodesic}
can be used to relate natural geometric symmetries to algebraic operations.
Here we present
some of them,
while expecting that there will be further interesting translations between two
languages.
The first two are from geometry to algebra,
and the last two are the other way around.
(These applications
were not presented and
have been postponed from \cite{CJKR} to include general relations between higher rank/multiplicity objects.)

First, taking the \textbf{duality functor} $\Hom_A\left(-,A\right)$ of modules in $\underline{\rm CM}\left(A\right)$ corresponds to
\textbf{flipping} loops in
$\Fuk\left(\POP\right)$
(see figures in Example \ref{ex:FlipDual}).
We show the commutativity of the following diagram of functors in ($\AI$-)categorical level
(Proposition \ref{prop:CommutativityFlipTranspose} + Proposition \ref{prop:CommutativityTransposeDual}).
Then
we give a clear description of these operations in terms of loop/band data
(\S \ref{sec:DualFlipCanonicalForms}).
\begin{equation}\label{eqn:FlipDualDiagramIntro}
\begin{tikzcd}[arrow style=tikz,>=stealth,row sep=2.5em,column sep=3em] 
&
H^0\Fuk\left(\POP\right)
  \arrow[r,"\LocalF"]
  \arrow[d,swap,"\textnormal{flip}"]
  \arrow[d,"\imath"]
&
\underline{\MF}(xyz)
  \arrow[r,"\simeq"]
  \arrow[d,swap,"\textnormal{transpose}"]
  \arrow[d,"-\operatorname{Tr}"]
&
\underline{\CM}(A)
  \arrow[d,swap,"\textnormal{dual}"]
  \arrow[d,"{\Hom_A\left(-,A\right)}"]
&
\\
&
H^0\Fuk\left(\POP\right)
  \arrow[r,"\LocalF"]
&
\underline{\MF}(xyz)
  \arrow[r,"\simeq"]
&
\underline{\CM}(A)
&
\end{tikzcd}
\end{equation}

Second,
we consider the \textbf{AR translation},
which is given by the \textbf{shift functor} of the
triangulated category $\underline{\rm CM}\left(A\right)$.
It is not easy to compute in terms of band data,
but it is equivalent to \textbf{reversing the orientation} of underlying loops in
$\Fuk\left(\POP\right)$
(Proposition \ref{prop:CommutativityReversingSwitching} + Proposition \ref{prop:CommutativitySwitchingShift}),
which we can compute in a geometric way.
We will give an algorithm to compute them
using conversion to the loop data
(Proposition \ref{prop:ShiftAlgorithm}).
\begin{equation}\label{eqn:ReversingTranslateDiagramIntro}
\begin{tikzcd}[arrow style=tikz,>=stealth,row sep=2.5em,column sep=3em] 
&
H^0\Fuk\left(\POP\right)
  \arrow[r,"\LocalF"]
  \arrow[d,swap,"\smat{\textnormal{orientation}\\ \textnormal{reversing}}"]
  \arrow[d,"\jmath"]
&
\underline{\MF}(xyz)
  \arrow[r,"\simeq"]
  \arrow[d,swap,"\smat{\textnormal{switching}\\ \textnormal{two factors}}"]
  \arrow[d,"{\left[1\right]}"]
&
\underline{\CM}(A)
  \arrow[d,swap,"\textnormal{shift}"]
  \arrow[d,"{\left[1\right]}"]
&
\\
&
H^0\Fuk\left(\POP\right)
  \arrow[r,"\LocalF"]
&
\underline{\MF}(xyz)
  \arrow[r,"\simeq"]
&
\underline{\CM}(A)
&
\end{tikzcd}
\end{equation}

We remark here that
an indecomposable object in degenerate cases
and its image under operations in
(\ref{eqn:ReversingTranslateDiagramIntro})
are invariant under operations in
(\ref{eqn:FlipDualDiagramIntro}).
Conversely, if an indecomposable object
is invariant under operations in
(\ref{eqn:FlipDualDiagramIntro}),
either it or its image under operations in
(\ref{eqn:ReversingTranslateDiagramIntro})
is of degenerate case.


Third, categories involved in HMS typically possess natural \textbf{triangulated structures}. Along with the parameterization by band or loop data,
\textbf{higher multiplicity/rank} objects are given by some
iterated \textbf{mapping cones}
(or \textbf{twisted complexes} in $\AI$-categories) involving lower multiplicity/rank objects.
Proposition \ref{prop:LMFTwistedComplexComputation}
gives an explicit way to understand
higher rank local systems in Fukaya category
as
twisted complexes of lower rank objects.

Finally,
\textbf{periodic objects} in both sides are \textbf{decomposed} into as many pieces as the number of repetitions
\footnote{
We place this discussion in the middle of the main text (\S \ref{sec:PeriodicCase})
since they are also needed to analyze degenerate cases.
}.
We give an explicit formula for this decomposition
(Theorem \ref{thm:PeriodicDecomposable}).
It shows that non-primitive loops in the Fukaya category are decomposable
(Corollary \ref{cor:PeriodicDecomposableFuk}),
which is not obvious on the A-side.

\noindent
\textbf{Geometric interpretation of tameness.}
In Theorem \ref{thm:ModuleGeodesic},
we showed that all band-type indecomposable maximal Cohen-Macaulay modules over $A$
correspond to explicit geometric objects
(rather than abstract twisted complexes)
in the Fukaya category.
It gives a geometric interpretation
\footnote{
But we are not giving a new proof,
as classification of objects in the Fukaya category here relies on the B-side result via HMS.
}
of their representation-tameness,
which was already proven algebraically in \cite{BD17}.
It is geometrically intuitive that
there are only countably many
closed geodesics
(or free homotopy classes of loops)
in $\POP$.
Then an indecomposable local system lying on a fixed loop
of rank $\primemu\in\mathbb{Z}_{\ge1}$
is determined by
its holonomy
up to gauge equivalence,
which can be represented (up to basis change) by the
$\primemu\times\primemu$
Jordan block
$J_{\primemu}\left(\primelambda\right)
\in\operatorname{GL}_{\primemu}\left(\field\right)$
with some eigenvalue $\primelambda\in\field^{\times}$.
As a result, the elements of sets in Theorem \ref{thm:ModuleGeodesic}
are parameterized by closed geodesics in $\POP$, a rank $\primemu$ (discrete parameters),
and an eigenvalue of holonomy $\primelambda$ (continuous parameter).
That is, they consist of countably many one-parameter families.
\end{spacing}

\noindent
\textbf{Relation with other mirror symmetries of surfaces.}
Recently, there have been many studies on homological mirror symmetry
between \textbf{$\mathbb{Z}$-graded} partially wrapped Fukaya categories
(or topological Fukaya categories) of \textbf{graded marked surfaces} (A-side),
derived categories of modules over \textbf{gentle algebras},
and derived categories of coherent sheaves on certain \textbf{non-commutative curves} (B-side).
Gentle algebras have long been an intriguing topic in representation theory since they are derived tame,
closed under derived equivalences,
and have well-understood indecomposable objects in their derived categories
(see \cite{SZ03}, \cite{BD-gentle} and references therein).
Their connection with the Fukaya categories was first established in \cite{HKK},
graded marked surfaces corresponding to them were constructed in \cite{LP20},
and independently,
a closely related algebraic model was constructed in \cite{OPS18}.
Their relation with certain non-commutative curves was first found in \cite{BD11},
extended to nodal stacky curves in \cite{LP18},
and again generalized to much broader class of non-commutative nodal curves in \cite{BD18}.

It is especially remarkable that
\textbf{indecomposable objects} on each side
are classified,
have a concrete one-to-one correspondence
and therefore
the derived tameness
of gentle algebras can be understood in a (symplectic) geometric way.
Also, many purely representation-theoretic problems
concerning derived equivalence of finite-dimensional algebras
have been attacked and solved
using geometric insights and techniques
(e.g. \cite{PPP19,APS23,Opp19,KS22,CJS22,CHS23,CK24,AP24}).

The triangulated categories which are the focus of the present work
are related to \textbf{$\mathbb{Z}_2$-graded} Fukaya categories
(using oriented Lagrangian curves as objects) of \textbf{surfaces} (A-side).
Their mirrors are usually given by certain categories of \textbf{matrix factorizations}
or equivalently, singularity categories of \textbf{Landau-Ginzburg models}
(B-side).
There
have
been many well-studied homological mirror symmetries.
For example,
the mirrors of genus two and higher genus closed surfaces were constructed in
\cite{S08} and \cite{Ef}, respectively.
In 
\cite{AS21},
spherical objects in $\mathbb{Z}_2$-graded Fukaya categories of closed surfaces were related with
simple closed curves with a rank $1$ local system.
Mirrors of punctured spheres and their cyclic covers were established in \cite{AAEKO}.
A non-commutative mirror model of punctured surfaces were also discovered in \cite{Boc16},
and a related functor was constructed in \cite{CHLnc}.
Going in a different direction,
\cite{AEK21} considers Fukaya categories of singular surfaces
and show the reverse direction (switching A- and B-sides) of homological mirror symmetry.

On the level of all \textbf{indecomposable objects}
in $\mathbb{Z}_2$-graded Fukaya categories,
nevertheless,
their classification and correspondence under mirror symmetry
are not known
in full generality.
Compared to the situation of $\mathbb{Z}$-graded mirror symmetry,
however,
it is apparent that there will be
much
utility of establishing such a strong bridge
between curves in Fukaya categories and matrix factorizations
(in the global sense of \cite{Orlov12}).
It will provide more applications of homological mirror symmetry,
relating
new tame triangulated categories arising from representation theory
(other than finite-dimensional algebras)
with Fukaya categories of surfaces.

\noindent
\textbf{Towards global $\mathbb{Z}_2$-graded mirror symmetry.}
Our previous work \cite{CJKR} and the present work
aim
to
ignite this new direction of development in the homological mirror symmetry program.
There are many good reasons to start with the {\textbf{pair-of-pants surface}}
and its mirror $xyz=0$,
which has been of great interest:


Most importantly,
the
pair serves as a \textbf{building block} to construct more complicated mirror pairs.
For example,
the idea of constructing the mirror of general Riemann surfaces using their pair-of-pants decompositions
appeared in 
many places in the literature including
\cite{Lee,Nad16,PS19,PS21,PS22}.
(The last four uses a sheaf-theoretic version of Fukaya categories,
which is different from (but equivalent to) the Floer-theoretic version used in this paper.)
See
also the well-written survey in \cite[\S 9.4]{Boc21} and references therein.
A common approach,
often referred to as a \textbf{local-to-global principle},
involves proving the compatibility of categorical gluing on both sides,
based on the mirror symmetry of the local pair.
Moreover,
as explained in \cite{CHL-glue},
copies of the localized mirror functor employed in this paper
(as its name implies)
can be also glued together
in order to obtain
a mirror functor in the global setting.

Independently of the above gluing formalism,
mirror symmetry of the pair-of-pants surface also played a central role
in \cite{AAEKO} and \cite{HJS24}.
The authors consider its cyclic and abelian covers, respectively,
and construct their mirror LG model using the symmetry given by the deck transformation groups.

%
%


On the other hand,
representation theory and classification of objects on B-side have been explicitly developed
only for the local model $\left.\field[[x,y,z]]\right/(xyz)$ in \cite{BD17}.
So we will need to work out the corresponding theory for
more general (non-affine) normal crossing surface singularities
appearing as mirrors of other Riemann surfaces.
After
establishing it,
we hope to generalize our present results to mirror symmetry of more general Riemann (orbi-)surfaces,
which will enhance our understanding of geometric and algebraic tame categories
and give many fruitful applications.

\subsection{Proof of main theorem}
In this subsection,
we deduce Theorem \ref{thm:ModuleGeodesic} from several results summarized from the body of the paper.
The approach to prove it
will be completed through the following
two
steps:

\begin{enumerate}[label=\Roman*.]
%

\item
Compute matrix factorizations corresponding to canonical forms of local systems over loops in $\POP$.

\item
Convert them into Burban-Drozd's canonical form of maximal Cohen-Macaulay modules.

\end{enumerate}

\noindent
\textbf{I. Canonical forms of loops with a local system and corresponding matrix factorizations.}
We take the following specific representatives of free homotopy classes of loops in $\POP$:
Given a \textbf{loop word}
$$
w' = \left(l_1',m_1',n_1',l_2',m_2',n_2',\dots,l_{\tau}',m_{\tau}',n_{\tau}'\right)
\in\mathbb{Z}^{3\tau}
$$
($\tau\in\mathbb{Z}_{\ge1}$),
consider the loop $L\left(w'\right)$ described in Figure \ref{fig:CanonicalForm}.
We
restrict to
\textbf{normal loop words}
(Definition \ref{defn:normal2})
so that
they
(up to shifting) produce only one loop in each
hyperbolic
free homotopy class.
Then they are also in one-to-one correspondence with closed geodesics in $\POP$
(Proposition \ref{prop:GeodesicNormalLoopWordCorrespondence}).

We introduce
a \textnormal{\textbf{loop datum}}
$
\left(w',\primelambda,\primemu\right)
$
to parameterize loops with a local system,
which consists of a \textnormal{\textbf{normal loop word}} $w'\in\mathbb{Z}^{3\tau}$ ($\tau\in\mathbb{Z}_{\ge1}$),
a \textnormal{\textbf{holonomy parameter}} $\primelambda\in\field^\times$,
and a \textnormal{\textbf{(geometric) rank}} $\primemu\in\mathbb{Z}_{\ge1}$.
%
The associated \textbf{canonical form of a loop with a local system},
denoted by $\mathcal{L}\left(w',\primelambda,\primemu\right)$,
is given by the loop $L\left(w'\right)$
together with a rank $\primemu$ local system whose holonomy is represented by $J_\primemu\left(\primelambda\right)$
\footnote{
We denote by 
$
J_{{\primemu}}\left(\primelambda\right)
\in
\operatorname{GL}_{{\primemu}}\left(\field\right)
$
the ${\primemu}\times{\primemu}$ Jordan block
with eigenvalue $\primelambda\in\field^\times$.
}
(up to basis change).

\begin{prop}[Corollary \ref{cor:GeodesicLocalSystemLoopData}]
\label{prop:GeodesicLocalSystemLoopDataIntro}
There is a one-to-one correspondence
$$

$}
\\[19mm]
\text{\color{red}\small where $x^a, y^a, z^a$ are regarded as $0$ if $a<0$}
\end{matrix}
\\
\end{matrix}
\\
\mathcal{L}\left({\color{blue}{{\color{black}w',
\primelambda,{\primemu}}}}\right)
&
&
\varphi\left({\color{blue}{{\color{black}w',
\lambda,{\primemu}}}}\right)
\end{matrix}
$
\captionsetup{width=1\linewidth}
\caption{Canonical form of a loop-type object in $\Fuk\left(\POP\right)$ and $\underline{\MF}(xyz)$}
\label{fig:CanonicalForm}
\end{figure}



%

The \textbf{localized mirror functor
}
converts each loop with a local system
$
\mathcal{L}\left(w',\primelambda,{\primemu}\right)
$
into a matrix factorization
$
\LocalF\left(\mathcal{L}\left(w',\primelambda,\primemu\right)\right).
$
The corresponding \textbf{canonical form of
matrix factorization}
of $xyz$
is
\begin{equation}\label{eqn:CanonicalFormMF}
\left(\varphi\left(w',\lambda,\primemu\right),\psi\left(w',\lambda,\primemu\right)\right)
\footnote{
It is
also
parameterized by loop data $\left(w',\lambda,\primemu\right)$,
but it is more convenient to distinguish two continuous parameters
$\primelambda$ (for loops with a local system)
and
$\lambda$ (for matrix factorizations).
They are related under the conversion formula in \S \ref{sec:ConversionFormula}.
}
\end{equation}
where $\lambda$ is either $\primelambda$ or $-\primelambda$ depending on $w'$
(see \textnormal{Definition \ref{def:ConversionFromLooptoBand}}).
Its
first component is shown in Figure \ref{fig:CanonicalForm}
and the second one is determined by the first (Definition \ref{defn:CanonicalFormMF}).

There are some exceptions called \textbf{degenerate cases} (i.e., $w'=(2,2,2)^{\hash\tau}$,
$\primelambda=(-1)^{\tau}$),
where the second factor $\psi\left(w',\lambda,\primemu\right)$ is not defined,
and
we use an alternative form
(\textnormal{Definition \ref{def:DegenerateCanonicalFormMF}})
\begin{equation}\label{eqn:DegenerateCanonicalFormMF}
\left(\varphi_{\deg}\left((2,2,2)^{\hash\tau},1,\primemu\right),\psi_{\deg}\left((2,2,2)^{\hash\tau},1,\primemu\right)\right).
\end{equation}


To integrate (\ref{eqn:CanonicalFormMF}) and (\ref{eqn:DegenerateCanonicalFormMF})
into a unified notation,
we denote them as
$\left(\varphi_{(\deg)}\left(w',\lambda,\primemu\right),\psi_{(\deg)}\left(w',\lambda,\primemu\right)\right)$.
Namely,
it defaults to (\ref{eqn:CanonicalFormMF}) in the general case
but adopts (\ref{eqn:DegenerateCanonicalFormMF}) only in the degenerate cases.
Then 
the relation between the canonical form of loops with a local system
and the canonical form of
matrix factorizations is summarized as:


\begin{thm}[Theorem \ref{thm:LagMFCorrespondenceNondegenerate} + Theorem \ref{thm:LagMFCorrespondenceDegenerate} + Proposition \ref{prop:DegenerateOriginalCanonicalForm}]
\label{thm:LagMFCorrespondence}
For 
a
loop datum $\left(w',\primelambda,\primemu\right)$
and
$
\lambda
=\pm\primelambda
$
determined by the \emph{conversion formula} Definition \ref{def:ConversionFromLooptoBand},
there is an isomorphism
in $\underline{\MF}\left(xyz\right)$
\begin{align*}
\LocalF\left(\mathcal{L}\left(w',\primelambda,\primemu\right)\right)
\ \ &\cong\ \
\left(\varphi_{(\deg)}\left(w',\lambda,\primemu\right),\psi_{(\deg)}\left(w',\lambda,\primemu\right)\right).
\end{align*}


\end{thm}




\noindent
\textbf{II. Canonical forms of maximal Cohen-Macaulay modules and corresponding matrix factorizations.}
Burban-Drozd's classification and canonical form of maximal Cohen-Macaulay modules
are best handled in
the \textbf{category of triples} $\Tri(A)$,
which was introduced and shown to be equivalent to $\CM(A)$ in the same work.
The
\textbf{canonical form
$
\Theta\left(w,\lambda,\mu\right)
$
of a {band-type} indecomposable object} in $\Tri(A)$ is
described in Figure \ref{fig:CanonicalForm2}.
It is
parameterized by
a \textbf{band datum} $\left(w,\lambda,\mu\right)$,
which consists of
a \textbf{band word}
$$
w = \left(l_1,m_1,n_1,l_2,m_2,n_2,\dots,l_{\tau},m_{\tau},n_{\tau}\right)
\in\mathbb{Z}^{3\tau}
$$
($\tau\in\mathbb{Z}_{\ge1}$), an \textbf{eigenvalue} $\lambda\in\field^\times$,
and an \textbf{(algebraic) multiplicity} $\mu\in\mathbb{Z}_{\ge1}$.
\begin{figure}[H]
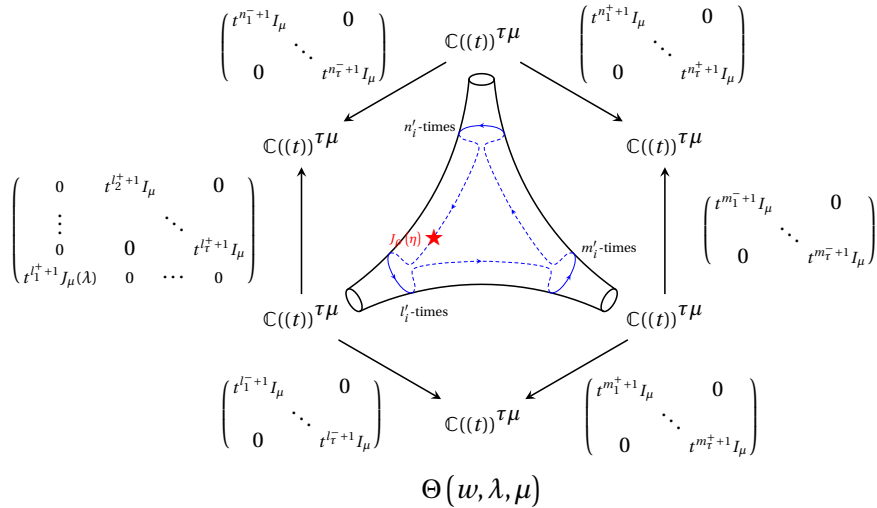

$

$
\captionsetup{width=1\linewidth}
\caption{
Canonical form of a band-type indecomposable object in $\Tri(A)$
\\
and the corresponding loop with a local system
}
\end{figure}
\vspace{-4mm}
\noindent
We denote by
$
M\left(w,\lambda,\mu\right)
$
the corresponding
object
in $\CM\left(A\right)$,
and refer to it as the
\textbf{canonical form of a band-type indecomposable maximal Cohen-Macaulay module}
over $A$.

When we take the stable category
$\underline{\CM}(A) = \left.\CM(A)\right/\left\{A\right\}$,
we lose exactly one isomorphism class $[A]$ of indecomposable objects containing the regular module $A$
(Definition \ref{def:StableCategories}).
In $\CM(A)$,
it is written in the canonical form as $A=M\left((0,0,0),1,1\right)$.
This implies that
for the band datum $\left((0,0,0),1,1\right)$,
there is no corresponding loop datum.



In \S \ref{sec:ConversionFormula},
we define a \textbf{conversion formula} between loop data and band data. 
It induces a bijection
\begin{align*}
\left.\left\{\textnormal{non-periodic loop data}\right\}\right/\sim_\textnormal{shifting}
\ \ \ &\overset{1:1}{\leftrightarrow}\ \ \
\left.\left(\left.\left\{\text{non-periodic band data}
\right\}
\right\backslash
\left\{\left((0,0,0),1,1\right)\right\}
\right)
\right/
\sim_{\text{shifting}}.
\\
\left(w',\primelambda,\primemu\right)\ \ \
\ \ &\leftrightarrow \ \
\ \ \ \left(w,\lambda,\mu\right)
\end{align*}
Note that
the set on the left side already appeared in
Proposition \ref{prop:GeodesicLocalSystemLoopDataIntro}.
The set on the right side is in bijection with
$$
\left.
\left\{\text{\text{band-type} indecomposable objects in }
\underline{\operatorname{CM}}(A)
\right\}
\right/
\sim_\text{isomorphism}
$$
by Burban-Drozd's classification (Theorem \ref{thm:BDClassification}).
In most cases,
we have $\primemu=\mu\in\mathbb{Z}_{\ge1}$.
However, for non-periodic degenerate cases,
the correspondence is given by
$$
\left(w'=(2,2,2),\primelambda=-1,\primemu\right)\ \ \
\ \ \leftrightarrow \ \
\ \ \ \left(w=(0,0,0),\lambda=1,\mu\right)
$$
with $\primemu=\mu-1\in\mathbb{Z}_{\ge1}$.

The conversion formula indeed relates
band-type indecomposable objects in $\underline{\MF}(xyz)$ and $\underline{\CM}(A)$
in their canonical forms under Eisenbud's equivalence:

\begin{thm}[Theorem \ref{thm:MFCMCorrespondence} + Theorem \ref{thm:MFCMCorrespondenceDegenerate}]
\label{thm:MFCMCorrespondenceIntro}
For a non-periodic loop datum $\left(w,\primelambda,\primemu\right)$
and band datum $\left(w,\lambda,\mu\right)$
related under the conversion formula,
there is an isomorphism
in
$\underline{\CM}\left(A\right)$
$$
\cok\mathunderbar{\varphi_{(\deg)}}\left(w',\lambda,\primemu\right)
\footnote{
We regard a matrix factorization $\left(\varphi,\psi\right)$ of $xyz$ as
a pair of homomorphisms between two free modules over
$S:=\field[[x,y,z]]$
(Definition \ref{defn:MatrixFactorization}),
then define $\mathunderbar{\varphi}:=\varphi\otimes \id_A$ as an $A$-module homomorphism
(See Theorem \ref{thm:EisenbudTheorem}).
}
\ \  \cong\ \
M\left(w,\lambda,\mu\right).
$$
%

\end{thm}


\noindent
\textbf{Mirror symmetry correspondence.}
Summing up, we have equivalence of categories and correspondence between loop/band-type indecomposable objects as follows
(where $w'$ and $w$ are non-periodic):
\begin{equation}\label{eqn:MainCorrespondenceDiagram}
\setlength\arraycolsep{2pt}
\color{blue}
\begin{matrix}
  D^{\pi}\left(W\operatorname{Fuk}(\POP)\right)
&
  \color{black}
  \overset{
  \small
  \begin{matrix}
    \text{\tiny localized}
  \\[-1mm]
    \text{\tiny mirror functor}
  \\[0.5mm]
    \LocalPhi
  \\[-1mm]
   \\[-1mm]
  \end{matrix}
  }{\longrightarrow}
&
  \underline{\operatorname{MF}}(xyz)
&
  \color{black}
  \overset{
  \small
  \begin{matrix}
    \text{\tiny Eisenbud}
  \\[2mm]
    \cok
  \\[-1mm]
    \simeq
   \\[-1mm]
  \end{matrix}
  }{\longrightarrow}
&
  \underline{\operatorname{CM}}(A)
&
  \color{black}
  \overset{
  \small
  \begin{matrix}
    \text{\tiny Burban}
  \\[-1mm]
    \text{\tiny -Drozd}
  \\[0.5mm]
    \mathbb{F}_{\operatorname{BD}}
  \\[-1mm]
    \simeq
   \\[-1mm]
  \end{matrix}
  }{\longrightarrow}
&
  \underline{\operatorname{Tri}}\left(A\right)
\\[2mm]
\color{black}
\text{\color{black}\small$
\mathcal{L}
\text{$
\left(w',\primelambda,{\primemu}\right)
$}
$}
&
\color{black}\overset{\cong}{\mapsto}
&
\text{\color{black}\small$
\varphi_{\operatorname{(deg)}}
\text{$
\left({{{\color{black}w',\lambda,{\primemu}}}}\right)
$}
$}
&
\color{black}
\xleftrightarrow{\textbf{}}
&
\text{\color{black}\small$
M
\text{$
\left({{{\color{black}w,\lambda,\mu}}}\right)
$}
$}
&
\color{black}\mapsto
&
\text{\color{black}\small$
\Theta
\text{$
\left({{{\color{black}w,\lambda,\mu}}}\right)
$}
$}
\end{matrix}
\end{equation}
Note that objects in each category are parameterized by loop data or band data.
The conversion formula between loop data and band data realizes the one-to-one correspondence between loop/band-type indecomposable objects
(up to isomorphism)
in each category as
$$
\begin{matrix}
\left.
\left\{
\textnormal{\textnormal{{closed geodesics in $\POP$ with an indecomposable local system}}}
\right\}
\right/\sim_\textnormal{gauge equivalence}
\\[2mm]
\hspace{-55.5mm}
\overset{1:1}{\leftrightarrow}\quad
\left.\left\{\text{non-periodic loop data}
\right\}
\right/
\sim_{\text{shifting}}
\\[2mm]
\hspace{-27mm}
\overset{1:1}{\leftrightarrow}\quad
\left.\left(\left.\left\{\text{non-periodic band data}
\right\}
\right\backslash
\left\{\left((0,0,0),1,1\right)\right\}
\right)
\right/
\sim_{\text{shifting}}
\\[2mm]
\hspace{-12mm}
\overset{1:1}{\leftrightarrow}\quad
\left.
\left\{\text{\text{band-type} indecomposable objects in }
\underline{\operatorname{CM}}(A)
\right\}
\right/
\sim_\text{isomorphism}.
\end{matrix}
$$
It
proves our main Theorem \ref{thm:ModuleGeodesic}.

\begin{spacing}{0.953}
\noindent
\textbf{Periodic case.}
In each category,
objects corresponding to \textbf{periodic} loop/band data are \textbf{decomposable}:
Consider
a loop datum $\left(w',\primelambda,\primemu\right)$
with a periodic normal loop word
$w' = \left(\tilde{w}'\right)^{\hash N}\in \mathbb{Z}^{3\tau}$
($N\in\mathbb{Z}_{\ge2}$).
Its corresponding band datum $\left(w,\lambda,\mu\right)$
has also a periodic band word
$w=\tilde{w}^N \in\mathbb{Z}^{3\tau}$.
Denote by
$\lambda_{0},\dots,\lambda_{N-1}\in\field^\times$
and
$\primelambda_{0},\dots,\primelambda_{N-1}\in\field^\times$
the $N$-th roots of $\lambda$ and $\primelambda$,
respectively.
Then we have decompositions
$$
\mathcal{L}\left(w',\primelambda,\primemu\right)
\cong
\bigoplus_{k=0}^{N-1} \mathcal{L}\left(\tilde{w}',\primelambda_{k},\primemu\right),
\quad
\varphi_{(\deg)}\left(w',\lambda,\mu\right)
 \cong
 \bigoplus_{k=0}^{N-1} \varphi_{(\deg)}\left(\tilde{w}',\lambda_{k},\mu\right)
\footnote{
Note that if
the left side is in a degenerate case,
exactly one of direct summands in the right side is in a degenerate case.
},
\quad
M\left(w,\lambda,\mu\right)
\cong
\bigoplus_{k=0}^{N-1} M\left(\tilde{w},\lambda_k,\mu\right)
$$
in $\Fuk\left(\POP\right)$, $\underline{\MF}(xyz)$, and $\underline{\CM}(A)$,
respectively.
(See Corollary \ref{cor:PeriodicDecomposableFuk} + (\ref{eqn:PeriodicDegenerateCaseDecompositionFuk}),
Theorem \ref{thm:PeriodicDecomposable} + (\ref{eqn:PeriodicDegenerateCaseDecompositionMF}),
and (\ref{eqn:CMDecomposition}).)

The loop datum $\left(\tilde{w}',\primelambda_k,\primemu\right)$
and the band datum $\left(\tilde{w},\lambda_k,\mu\right)$
also correspond to each other
for each $k$
under the conversion formula.
Therefore,
by Theorem \ref{thm:LagMFCorrespondence}
and Theorem \ref{thm:MFCMCorrespondenceIntro},
the above three decompositions are compatible with each other in non-degenerate cases.
In degenerate cases,
the first and the second are still compatible with each other,
while the second and the third are not,
due to the shifting of rank/multiplicity $\primemu=\mu-1$
(see (\ref{eqn:PeriodicDegenerateCaseMappedToMFFuk}) and (\ref{eqn:PeriodicDegenerateCaseMappedToCM})).

\begin{remark}
\label{rmk:Completion}
In this paper,
we consider matrix factorizations over the power series ring
$S:=\field[[x,y,z]]$
and maximal Cohen-Macaulay modules over its quotient
$A:=\left.\field[[x,y,z]]\right/(xyz)$,
instead of the polynomial ring
$\tilde{S}:=\field[x,y,z]$
and its quotient
$\tilde{A}:=\left.\field[x,y,z]\right/(xyz)$.
To distinguish the latter,
we denote by
$\MF[xyz]$
the category of matrix of matrix factorizations of $xyz$ over $\tilde{S}$,
and by
$\CM\left(\tilde{A}\right)$
the category of maximal Cohen-Macaulay modules over $\tilde{A}$
(see \cite{HB93}).
Their stable categories and other variations are defined in the same way as in
\S \ref{sec:MFCategories} and \S \ref{sec:EisenbudEquivalence}.

There are obvious faithful functors
$
\underline{\MF}[xyz]
\rightarrow
\underline{\MF}(xyz)
$
and
$
\underline{\CM}\left(\tilde{A}\right)
\rightarrow
\underline{\CM}\left(A\right)
$
(see \cite[Theorem 2.1.3, Corollary 2.1.8]{HB93}),
which are essentially surjective
(indeed, every indecomposable object in $\underline{\CM}(A)$
is obtained as an image of the functor).
Still
there are some objects in
$\underline{\MF}[xyz]$
that are not isomorphic to each other
but become isomorphic in
$\underline{\MF}(xyz)$.
For example,
the below diagram shows
a family of matrix factorizations
($\lambda\in\field^{\times}$)
that are not zero objects in $\underline{\MF}[xyz]$
but becomes zero objects in $\underline{\MF}(xyz)$.
(Note that the vertical isomorphisms exist only in the latter category.)
$$
\begin{tikzcd}[arrow style=tikz,>=stealth,row sep=3em,column sep=6em] 
S^2
  \arrow[r,"\spmat{z & \lambda-y \\ 0 & xy}"]
  \arrow[d,swap,"\smat{\begin{psmallmatrix}1 & 0 \\ z(\lambda-y)^{-1} & 1\end{psmallmatrix}\\[3mm]}"]
&
S^2
  \arrow[r,"\spmat{xy & -\lambda+y \\ 0 & z}"]
  \arrow[d,swap,"\smat{\begin{psmallmatrix}xy & -\lambda+y \\ (\lambda-y)^{-1} & 0\end{psmallmatrix}\\[3mm]}"]
&
S^2
  \arrow[d,swap,"\smat{\begin{psmallmatrix}1 & 0 \\ z(\lambda-y)^{-1} & 1\end{psmallmatrix}\\[3mm]}"]
\\
S^2
  \arrow[r,"\spmat{xyz & 0 \\ 0 & 1}"]
&
S^2
  \arrow[r,"\spmat{1 & 0 \\ 0 & xyz}"]
&
S^2
\end{tikzcd}
$$

They come from a loop with holonomy $\lambda$ that is homotopic to one of the boundary circles in $\POP$,
which was excluded in our correspondence in Theorem \ref{thm:ModuleGeodesic}.
Thus,
$\underline{\MF}[xyz]$ has more objects than $\underline{\MF}(xyz)$.
It would be interesting to know
how many objects in the gap,
and whether they are also realized in a geometric way.

\end{remark}

\noindent
\textbf{Acknowledgement.}
This work started as an attempt to understand the mirror A-side of Igor Burban and Yuriy Drozd's representation theory of maximal Cohen-Macaulay modules over non-isolated surface singularities in \cite{BD17},
and we are grateful to Igor Burban for explanations, discussions
and many great advice on this work.
We thank
Yong-Geun Oh,
Henning Krause,
Juan Omar Gomez
and Sibylle Schroll
for their interest on this project as well as helpful comments,
Wassilij Gnedin for
carefully reading and revising the draft
from a representation-theoretic perspective,
Severin Barmeier, Dongwook Choa, Wonbo Jeong, Sangjin Lee and Sangwook Lee for many useful discussions in symplectic geometry,
and Hanwool Bae for explaining to us his thesis.
The second author is also grateful to
Jongil Park, Jae-Hoon Kwon, Hansol Hong, Philsang Yoo and Igor Burban
for many great comments in his thesis defense \cite{Rho23},
where most parts of this work were discussed. 
We also thank Kyoungmo Kim for sharing his idea and
helpful discussions on this topic
and other joint works.
The second author was supported by
Basic Science Research Program through the National Research Foundation of Korea(NRF) funded by the Ministry of Education(RS-2023-00248895),
and
the German
Research Foundation SFB-TRR 358/1 2023 — 491392403.
\end{spacing}

\newpage
\begin{spacing}{0.94}
\section{Localized Mirror Functor and Its Computation}\label{sec:LMFComputation}

This section is devoted to
an elaboration of the computational aspect
of the localized mirror functor
\cite{CHL}
applied to the \emph{pair-of-pants surface}.
Especially,
we develop a formula for finding each component of the matrix factorization
corresponding to a higher-rank local system $(E,\nabla)$ over a loop $L$.
Every convention and notation for such an object
$\mathcal{L}=\left(L,E,\nabla\right)$
is based on our geometric setting of the compact Fukaya category
(with immersed loops with a local system)
explained in Section \ref{sec:CptFukSurface}.

\subsection{Localized mirror functor for pair-of-pants surface}
We call a smooth surface with boundary $\POP$ diffeomorphic to the complement in $S^2$ of three
distinct points
a \textbf{pair-of-pants surface}.
Consider a marked loop $\mathbb{L}=\left(\mathbb{L},e_{\mathbb{L}},o_{\mathbb{L}}\right)$ in $\POP$ described in Figure \ref{fig:SeidelLagrangian},
which is called the \textbf{Seidel Lagrangian}.
(Note that its self-intersections are transversal.
See also \cite{Se}.)
Assume that the areas of two triangles bounded by $\bL$ in $\POP$ are the same
\footnote{
This condition is necessary when we consider the Novikov field $\Lambda$ instead of our field $\field$.
}.
We put on its domain a trivial line bundle $E_{\mathbb{L}} = S^1 \times \field$ equipped with a flat connection
$
\nabla_{\mathbb{L}}
$
whose holonomy
is $-1$ at the point {\small$\color{gray}\bigstar$} marked in Figure \ref{fig:SeidelLagrangian}.
We assume that the triple
$
\left(\mathbb{L}, E_{\mathbb{L}}, \nabla_{\mathbb{L}}\right)
$
is an object of $\Fuk\left(\POP\right)$
and still denote it as $\mathbb{L}$ for simplicity.

\begin{figure}[H]
   \begin{minipage}{0.5\textwidth}
     \centering
          \adjustbox{height=35mm, center}{

          }
\centering
\captionsetup{width=1\linewidth}
\caption{Seidel Lagrangian $\mathbb{L}$ in $\POP$}
\label{fig:SeidelLagrangian}
   \end{minipage}\hfill
   \begin{minipage}{0.5\textwidth}
     \centering
          \adjustbox{height=35mm}{
          %
          }
\centering
\captionsetup{width=1\linewidth}
\caption{Generating objects of $D^\pi W\Fuk\left(\POP\right)$}
\label{fig:GeneratingArcs}
   \end{minipage}
\end{figure}

\vspace{-3mm}

Since $\bL$ has $3$ self-intersections,
we have
$$
\chi^0\left(\mathbb{L},\mathbb{L}\right)
=
\left\{e_{\mathbb{L}}, \OL{X}, \OL{Y}, \OL{Z}\right\},
\quad
\chi^1\left(\mathbb{L},\mathbb{L}\right)
=
\left\{o_{\mathbb{L}}, X, Y, Z\right\}.
$$
We trivialize $E_{\mathbb{L}}$ over $S^1\setminus{\small\color{gray}\bigstar}$
as in Proposition \ref{rmk:ChooseTrivialization},
which yields identifications
$
\left.E_{\mathbb{L}}\right|_{\bullet} \cong \field
$
for each
$\bullet
\in
\chi\left(\mathbb{L},\mathbb{L}\right).
$
Then we have
$
\Hom_{\field}\left(\left.E_{\mathbb{L}}\right|_{\bullet}, \left.E_{\mathbb{L}}\right|_{\bullet}\right)
\footnote{
Note that although two $\left.E_{\mathbb{L}}\right|_{\bullet}$'s have the same notation, they are actually
fibers of $E$ over different preimages (in $S^1$)
of the point $\bullet\in\POP$ under $L$,
as noted in the definition in (\ref{eqn:DefinitionOfHom}).
}
\cong
\Span\left\{\bullet\right\}
$
as noted in Remark \ref{rmk:TrivializingHom},
and hence
$$
\hom^0(\bL,\bL)
=
\Span\left\{e_{\mathbb{L}} \OL{X}, \OL{Y}, \OL{Z}\right\}
\quad
\hom^1(\bL,\bL)
=
\Span\left\{o_{\mathbb{L}}, X, Y, Z\right\}.
$$

Seidel Lagrangian $\mathbb{L}$ is very special and useful
because it
has a \emph{weak bounding cochain}
(or it is \emph{weakly unobstructed})
in the sense of \cite{FOOO}. 
It enables $\mathbb{L}$ to serve as a reference of the localized mirror functor.

\begin{prop}\cite{CHL}\label{prop:WeakBoundingCochain}
A linear combination
$
b = xX + yY + zZ
\in
\hom^1\left(\mathbb{L},\mathbb{L}\right)
$
is
a \textnormal{\textbf{weak bounding cochain}}
for any $x$, $y$, $z\in\field$.
That is,
\begin{equation} \label{eqn:wu}
\operatorm_0^b\left(\mathbb{L}\right)
:=
\sum_{i=1}^{\infty} \operatorm_{i}({\color{black}\underbrace{\color{black}b, \ldots, b}_{i}})
=
xyz \cdot e_{\mathbb{L}}
\end{equation}
and hence $\left(\mathbb{L},b\right)$ has the \emph{disk potential}
$
W^{\mathbb{L}} = xyz.
$
\end{prop}

\begin{proof}
All $\operatorm_1$-terms in the left side vanish because there are no bigons.
There are several non-zero $\operatorm_2$-terms
such as
$
\operatorm_2\left(X,Y\right)
=\OL{Z}=-
\operatorm_2\left(Y,X\right)
$
coming from the front and back triangles,
but they cancle each other because the holonomy $-1$ contributes only to the back triangle.
The only non-zero higher $\operatorm_{\ge3}$-term is $\operatorm_3(X,Y,Z)=xyz\cdot e_{\mathbb{L}}$,
which comes from the front triangle bounded by $\bL$
passing through $e_{\mathbb{L}}$
(Remark \ref{rmk:AIoperationwithInfinitesimalGenerators}.(3)).
This is the only surviving term in $\operatorm_0^b\left(\mathbb{L}\right)$,
which gives the disk potential $W^{\mathbb{L}}=xyz$.
\end{proof}
\end{spacing}

\begin{spacing}{0.96}
Such a pair $\left(\mathbb{L},b\right)$ defines an $\AI$-functor from Fukaya category
to the $\AI$-category $\MF_{\AI}\left(xyz\right)$ of matrix factorizations of $xyz$,
called the \textbf{localized mirror functor} \cite{CHL}.
(Here we follow the convention in \cite{CHLnc}.)
It is based on the deformation theory of $\AI$-operations,
as we will see below:

For any object $\mathcal{L}=\left(L,E,\nabla\right)$ in
$
W\Fuk\left(\POP\right)
\footnote{
For
definition of wrapped Fukaya category $W\Fuk\left(\POP\right)$,
see \cite{AS} and \cite{Ab12} (also \cite{auroux2014beginner} for surfaces).
Its objects involve not only loops but also arcs between boundaries,
and it contains the compact Fukaya category $\Fuk\left(\POP\right)$
as a full subcategory.},
$
note that
two sets
$
\chi^\bullet\left(L,\mathbb{L}\right)
$
$\left(\bullet\in\mathbb{Z}_2\right)$
are finite sets of the same cardinality
$
\tau:=\frac{1}{2}\left|L\cap \mathbb{L}\right|
$.
Therefore, two $\field$-vector spaces
\begin{equation}\label{eqn:DecompositionOfHom}
\hom^\bullet\left(\mathcal{L},\mathbb{L}\right)
=
\bigoplus_{p\in \chi^\bullet\left(L,\mathbb{L}\right)}
\Hom_{\field}\left(\left.E\right|_p,\field\right)
=
\bigoplus_{p\in \chi^\bullet\left(L,\mathbb{L}\right)}
\left(\left.E\right|_p\right)^*
\quad
\left(\bullet\in\mathbb{Z}_2\right)
\end{equation}
have the same finite dimension
$
\tau\primemu,
$
where
$
\primemu
$
is the rank of $E$.

We define the \textbf{deformed differential}
\begin{align*}\label{eqn:DeformedDifferential}
\begin{split}
\operatorm_1^{0,b}
:
\hom\left(\mathcal{L},\mathbb{L}\right)
\rightarrow
\hom\left(\mathcal{L},\mathbb{L}\right),
\quad
f
\mapsto
\operatorm_1^{0,b}\left(f\right)
:=
\sum_{i=0}^{\infty}
\operatorm_{1+i}\big(f,\underbrace{b, \ldots, b}_{i}\big),
\end{split}
\end{align*}
which is a $\field$-linear map of degree $1$ and satisfies
$
\left(\operatorm_1^{0,b}\right)^2
=
xyz\cdot\id_{\hom\left(\mathcal{L},\mathbb{L}\right)}
$
by the following lemma,
which we recall for the reader's convenience:

\begin{lemma}\cite{CHL}\label{lem:MFIdentity}
For a weak bounding cochain $\left(\mathbb{L},b\right)$ with disk potential $W^{\mathbb{L}}$,
we have
$$\left(\operatorm_1^{0,b}\right)^2 = W^{\mathbb{L}} \cdot \id_{\hom\left(\mathcal{L},\mathbb{L}\right)}.$$
\end{lemma}
\begin{proof}
For any
$f\in\hom^{\bullet}\left(\mathcal{L},\mathbb{L}\right)$
($\bullet\in\mathbb{Z}_2$),
we can write
\begin{align*}
\operatorm_1^{0,b}\left(\operatorm_1^{0,b}\left(f\right)\right)
&=
\sum_{j=0}^{\infty}
\sum_{i=0}^{\infty}
\operatorm_{1+j}
\big(
\operatorm_{1+i}
\big(f,\underbrace{b,\dots,b}_{i}\big),\underbrace{b,\dots,b}_{j}\big)
\\[-2mm]
&=
-
\sum_{l_0,l_1\ge0}
\sum_{k=1}^{\infty}
(-1)^{\left|f\right|-1}
\operatorm_{2+l_0+l_1}
\big(f,\underbrace{b,\dots,b}_{l_0},\operatorm_k(\underbrace{b,\dots,b}_{k}),\underbrace{b,\dots,b}_{l_1}\big),
\end{align*}
where the second identity follows from the $\AI$-relations (\ref{eqn:AI-relation}).
Using the identity
$
\sum_{k=1}^{\infty} \operatorm_{k}({\color{black}\color{black}b, \ldots, b})
=
W^{\mathbb{L}} \cdot e_{\mathbb{L}}
$
and the fact that $e_\mathbb{L}$ is a unit (\ref{eqn:Unit}),
we can rewrite it as
$$
\sum_{l_0,l_1\ge0}
(-1)^{\left|f\right|}
\operatorm_{2+l_0+l_1}
\big(f,\underbrace{b,\dots,b}_{l_0},W^{\mathbb{L}}\cdot e_{\mathbb{L}},\underbrace{b,\dots,b}_{l_1}\big)
=
(-1)^{\left|f\right|}
\operatorm_2\left(f,W^{\mathbb{L}} \cdot e_{\mathbb{L}}\right)
=
W^{\mathbb{L}}\cdot f,
$$
which proves the claim.
\end{proof}

Restricting
the domain of
$\operatorm_1^{0,b}$
to each degree summand yields two maps
\begin{equation}\label{eqn:LMFImage}
 \begin{tikzcd}[arrow style=tikz,>=stealth, sep=65pt, every arrow/.append style={shift left}]
   \hom^0\left(\mathcal{L},\mathbb{L}\right)
   =
   \displaystyle\bigoplus_{p\in\chi^0\left(L,\mathbb{L}\right)}\left(\left.E\right|_p\right)^*
     \arrow{r}{\LocalPhi\left(\mathcal{L}\right):=\operatorm_1^{0,b}} 
   &
   \displaystyle\bigoplus_{s\in\chi^1\left(L,\mathbb{L}\right)}\left(\left.E\right|_s\right)^*
   =
   \hom^1\left(\mathcal{L},\mathbb{L}\right)
     \arrow{l}{\LocalPsi\left(\mathcal{L}\right):=\operatorm_1^{0,b}} 
 \end{tikzcd}
\end{equation}
satisfying
\begin{equation}\label{eqn:MatrixFactorizationOfLMF}
\LocalPsi\left(\mathcal{L}\right)
\LocalPhi\left(\mathcal{L}\right)
=
xyz\cdot\id_{\hom^0\left(\mathcal{L},\mathbb{L}\right)}
\quad\text{and}\quad
\LocalPhi\left(\mathcal{L}\right)
\LocalPsi\left(\mathcal{L}\right)
=
xyz\cdot\id_{\hom^1\left(\mathcal{L},\mathbb{L}\right)}.
\end{equation}

As
$
\hom^\bullet\left(\mathcal{L},\mathbb{L}\right)
$
($\bullet\in\mathbb{Z}_2$)
are $\field$-vector spaces of dimension $\tau\primemu$,
extension of scalar to the
ring
$S:=\field[[x,y,z]]$
yields
two free $S$-modules 
$
S
\otimes_{\field}
\hom^\bullet\left(\mathcal{L},\mathbb{L}\right)
$
of rank $\tau\primemu$.
Now we view each of $\LocalPhi\left(\mathcal{L}\right)$ and $\LocalPsi\left(\mathcal{L}\right)$
as a map between those $S$-modules,
regarding $x$, $y$ and $z$ as variables in the ring $S$.
Then the relation (\ref{eqn:MatrixFactorizationOfLMF}) still holds
just by replacing
$
\id_{\hom^\bullet\left(\mathcal{L},\mathbb{L}\right)}
$
with
$
\id_{S\otimes
\hom^\bullet\left(\mathcal{L},\mathbb{L}\right)},
$
so the pair
$$
\LocalF\left(\mathcal{L}\right)
:=
\left(\LocalPhi\left(\mathcal{L}\right),\LocalPsi\left(\mathcal{L}\right)\right)
$$
produces a $\tau\primemu\times\tau\primemu$ matrix factorization of $xyz$ in $S$.

\end{spacing}


\begin{thm}\cite{CHL}\label{thm:lmf}
The localized mirror functor $\LocalF : W\Fuk\left(\POP\right) \to \MF_{\AI}(xyz)$ is defined as follows:
\begin{itemize}
\item
For an object
$\mathcal{L}=\left(L,E,\nabla\right)$
in
$
W\Fuk\left(\POP\right),
$
its mirror object in
$
\MF_{\AI}(xyz)
$
is given by 
$$
\LocalF(\mathcal{L})
=
\left(\LocalPhi\left(\mathcal{L}\right),\LocalPsi\left(\mathcal{L}\right)\right).
$$

\item
Higher components
$
\left\{\LocalF_k\right\}_{k\ge1}
$
are given by
$$
\begin{matrix}
\LocalF_k
:
\hom\left(\mathcal{L}_0,\mathcal{L}_1\right) \otimes \cdots \otimes \hom\left(\mathcal{L}_{k-1},\mathcal{L}_k\right)
\to
\hom_{\MF_{\AI}\left(xyz\right)}
\left(\LocalF\left(\mathcal{L}_0\right),\LocalF\left(\mathcal{L}_k\right)\right)
\\[2mm]
\hspace{55mm}
\displaystyle
\left(f_{1}, \ldots,f_{k}\right)
\mapsto
\operatorm_{k+1}^{0,\dots,0,b}\left(f_1,\dots,f_k,-\right)
:= \sum_{i=0}^{\infty} \operatorm_{k+1+i}(f_{1},\ldots,f_{k},-,\underbrace{b, \ldots, b}_{i}),
\end{matrix}
$$
whose images are $\field$-linear maps
$
\hom\left(\mathcal{L}_k,\mathbb{L}\right)
\rightarrow
\hom\left(\mathcal{L}_0,\mathbb{L}\right),
$
also viewed as module homomorphisms
$
S\otimes
\hom\left(\mathcal{L}_k,\mathbb{L}\right)
\rightarrow
S\otimes
\hom\left(\mathcal{L}_0,\mathbb{L}\right)
$
over $S=\field[[x,y,z]]$.
\end{itemize}
Then, $\LocalF$ is an $A_{\infty}$-functor.
Moreover, it induces an equivalence from the derived wrapped Fukaya category 
$D^{\pi}W\Fuk\left(\POP\right)$
\footnote{
It is obtained from $W\Fuk\left(\POP\right)$ by taking twisted completion (Definition \ref{defn:TwistedCompletion}),
idempotent completion and then
the cohomological category (Definition \ref{defn:CohomologicalCategory}).
For more details,
we refer to \cite{S08}.
}
to the homotopy category
$\underline{\MF}\left[xyz\right]$
\footnote{
We can take the target of the functor as
$\MF_{\AI}[xyz]$
(see Remark \ref{rmk:Completion})
since every object in $W\Fuk\left(\POP\right)$
is quasi-isomorphic to an object whose image lies in 
$\MF_{\AI}[xyz]$.
Then $\underline{\MF}[xyz]$ is the same as its cohomological category
$H^0\left(\MF_{\AI}[xyz]\right)$.
}
of matrix factorizations of $xyz$. 
\end{thm}
The last statement recovers the result of \cite{AAEKO} for $\POP$.
It follows from the fact that $\LocalF$ sends the three generating arcs
$L_{\XY}$, $L_{\YZ}$, and $L_{\ZX}$
of $D^{\pi}W\Fuk\left(\POP\right)$
in Figure \ref{fig:GeneratingArcs}
to the three generating matrix factorizations $z\cdot xy$, $x\cdot yz$, and $y\cdot zx$
of $\underline{\MF}\left[xyz\right]$,
respectively.
The functor induces an isomorphism on cohomology of hom spaces between those objects
and hence we can apply Theorem 4.2 in \cite{CHL}.





\subsection{Computation of localized mirror functor}\label{sec:LMFComputationSubsection}

Note in (\ref{eqn:DecompositionOfHom}) that
$
\hom\left(\mathcal{L},\mathbb{L}\right)
$
allows a decomposition 
into
$\primemu$-dimensional vector spaces
$
\left(\left.E\right|_p\right)^*
$
for
$
p\in\chi\left(L,\mathbb{L}\right).
$
Correspondingly,
the operation
$\operatorm_1^{0,b}$
(and hence
$\LocalPhi\left(\mathcal{L}\right)$
and
$\LocalPsi\left(\mathcal{L}\right)$)
is decomposed into several
$
\left(
\left(\left.E\right|_s\right)^*,\left(\left.E\right|_p\right)^*
\right)
$-components
\footnote{
Following the standard natation in matrices,
$\left(\left.E\right|_s\right)^*$
and
$\left(\left.E\right|_p\right)^*$
are subspaces of the codomain and domain,
respectively.
}
for
$
p,s\in\chi\left(L,\mathbb{L}\right).
$

For
$
p\in\chi\left(L,\mathbb{L}\right)
$
and
$
f=\left.f\right|_p\in
\left(\left.E\right|_p\right)^*
\subseteq
\hom\left(\mathcal{L},\mathbb{L}\right),
$
we have
\begin{align*}
\operatorm_1^{0,b}\left(f\right)
&=
\sum_{i=0}^{\infty}
\operatorm_{1+i}\big(f,\underbrace{b, \ldots, b}_{i}\big)
\qquad
\left(
b=xX+yY+zX
\in
\hom^1\left(\mathbb{L},\mathbb{L}\right)
\right)
\\
&=
\sum_{
\begin{matrix}
\scriptscriptstyle
\left(x_1,X_1\right),\dots,\left(x_i,X_i\right)
\\
\scriptscriptstyle
\in\left\{(x,X),(y,Y),(z,Z)\right\}
\end{matrix}
}x_1\dots x_i\
\operatorm_{1+i}\left(f,X_1,\dots,X_i\right),
\end{align*}
and
$$
\operatorm_{1+i}\left(f,X_1,\dots,X_i\right)
=
\sum_{s\in\chi\left(L,\mathbb{L}\right)}
\sum_{u\in\mathcal{M}\left(p,X_1,\dots,X_i,\overline{s}\right)}
\sign(u)
\operatorname{hol}_s\left(\partial u\right)
\left(f,X_1,\dots,X_i\right).
$$
Therefore,
the
$
\left(\left.E\right|_s\right)^*
$-component
of
$
\operatorm_1^{0,b}\left(f\right)
$
is
$$
\sum_{
\begin{matrix}
\scriptscriptstyle
\left(x_1,X_1\right),\dots,\left(x_i,X_i\right)
\\
\scriptscriptstyle
\in\left\{(x,X),(y,Y),(z,Z)\right\}
\end{matrix}
}
x_1\cdots x_i\
\sum_{u\in\mathcal{M}\left(p,X_1,\dots,X_i,\overline{s}\right)}
\sign(u)
\operatorname{hol}_s\left(\partial u\right)
\left(f,X_1,\dots,X_i\right).
$$

So it
comes from
immersed polygons
(also called \emph{deformed strips})
$u$ bounded by $L$ and $\mathbb{L}$,
whose angles consist of $p$,
$X_1,\dots,X_i$
and $\overline{s}$
in a counterclockwise order,
for some $X_1,\dots,X_i\in\left\{X,Y,Z\right\}$.
(See Figure \ref{fig:strip}.)
Such a deformed strip contributes a monomial
$x_1\cdots x_i$,
where $x_j$ is $x$, $y$ or $z$ depending on whether $X_j$ is $X$, $Y$ or $Z$.

%

\vspace{-8mm}

\begin{figure}[H]
\adjustbox{height=35mm}{

          }
\vspace{-7mm}
\centering
\captionsetup{width=1\linewidth}
\caption{A deformed strip which contributes a monomial $x_1\cdots x_i$ to the $\left(\left.E\right|_s\right)^*
$-component
of
$
\operatorm_1^{0,b}\left(\left.f\right|_p\right)
$
}
\label{fig:strip}
\end{figure}

\vspace{-5mm}

The sign of the deformed strip $u$ is given by
$$
\sign\left(u\right)
=
\begin{cases}
1 & \text{if orientation of $\mathbb{L}$ $=$ orientation of $\partial u$,}
\\
(-1)^{i+1} & \text{otherwise}
\end{cases}
=
(-1)^
{(i+1)\mathbb{1}_{\operatorname{o}\left(\mathbb{L}\right)\ne\operatorname{o}\left(\partial u\right)}}.
$$
It follows from formula (\ref{eqn:sign(u)}) and the fact that the orientation of $\mathbb{L}$ along any angle $X_j$
is preserved
(because all $X_j$'s have odd-degrees)
and the degrees of $p$ and $s$ are always different.

The holonomy operation of $\partial u$ at $s$ is by definition given as
\begin{multline*}
\hol_s\left(\partial u\right)
:
\Hom_{\field}\left(\left.E\right|_{p},\left.E_{\mathbb{L}}\right|_{p}\right)
\otimes
\Hom_{\field}\left(\left.E_{\mathbb{L}}\right|_{X_1},\left.E_{\mathbb{L}}\right|_{X_1}\right)
\otimes \cdots \otimes
\Hom_{\field}\left(\left.E_{\mathbb{L}}\right|_{X_i},\left.E_{\mathbb{L}}\right|_{X_i}\right)
\rightarrow
\Hom_{\field}\left(\left.E\right|_{s},\left.E_{\mathbb{L}}\right|_{s}\right),
\\
\left(f,X_1,\dots,X_i\right)
\mapsto
P\left(\left(\partial u\right)
_{i+1}
\right)
\ \circ \
X_i
\ \circ \
P\left(\left(\partial u\right)
_{i}
\right)
\ \circ \
X_{i-1}
\ \circ \
\cdots
\ \circ \
X_1
\ \circ \
P\left(\left(\partial u\right)
_1
\right)
\ \circ \
f
\ \circ \
P\left(\left(\partial u\right)
_0
\right).
\end{multline*}
Under the identification
$
\Hom_{\field}\left(\left.E_{\mathbb{L}}\right|_{X_j},\left.E_{\mathbb{L}}\right|_{X_j}\right)
=
\Span\left\{X_j\right\}
\cong
\field
\quad
(j\in\left\{1,\dots,i\right\}),
$
the generator $X_j$ corresponds to $1$.
Note that
$\left(\partial u\right)_j$ lies on $\mathbb{L}$
for $j\in\left\{1,\dots,i+1\right\}$,
and
under the identification
$
\Hom_{\field}\left(\left.E_{\mathbb{L}}\right|_{X_{j-1}},\left.E_{\mathbb{L}}\right|_{X_{j}}\right)
\cong
\field
\quad
(j\in\left\{1,\dots,i+1\right\},\
X_0:=p,\
X_{i+1}:=\overline{s}),
$
in view of Proposition \ref{rmk:HolonomyComputing},
$
P\left(\left(\partial u\right)_j\right)
$
corresponds
$
(-1)^{\#\left(\left(\partial u\right)_j \cap {\small\color{gray}\bigstar}_{\mathbb{L}}\right)},
$
where
$
\#\left(\left(\partial u\right)_j \cap {\small\color{gray}\bigstar}_{\mathbb{L}}\right)
$
is the number of times $\left(\partial u\right)_j$
passes through the point ${\small\color{gray}\bigstar}$ on $\mathbb{L}$.
Therefore,
we can replace
$
\operatorname{hol}_s\left(\partial u\right)
\left(f,X_1,\dots,X_i\right)
$
with
$$
(-1)^{\#\left(\partial u \cap {\small\color{gray}\bigstar}_{\mathbb{L}}\right)}
f
\ \circ \
P\left(\left(\partial u\right)_0
\right),
$$
where
$
\#\left(\partial u \cap {\small\color{gray}\bigstar}_{\mathbb{L}}\right)
$
is the total number of times $\partial u$
passes through the point ${\small\color{gray}\bigstar}$ on $\mathbb{L}$,
and
$
P\left(\left(\partial u\right)_0\right)\in
\Hom_{\field}\left(\left.E\right|_{s},\left.E\right|_{p}\right)
$
is the parallel transport
from $\left.E\right|_{s}$ to $\left.E\right|_{p}$
along the side of $u$
lying in $L$.

We summarize the above discussion into the following formula:

\begin{prop}\label{prop:ComponentOfDeformedDifferential}
The
$
\left(
\left(\left.E\right|_s\right)^*,\left(\left.E\right|_p\right)^*
\right)
$-component of
$
\operatorm_1^{0,b}
:
\hom\left(\mathcal{L},\mathbb{L}\right)
\rightarrow
\hom\left(\mathcal{L},\mathbb{L}\right)
$
(and hence of
$
\LocalPhi\left(\mathcal{L}\right)
$
or
$
\LocalPsi\left(\mathcal{L}\right)
$)
is the $\field$-linear map
$
\operatorm_1^{0,b}
:
\left(\left.E\right|_p\right)^*
\rightarrow
\left(\left.E\right|_s\right)^*
$
that maps $f\in\left(\left.E\right|_p\right)^*$ to
\begin{equation}\label{eqn:ComponentOfDeformedDifferential}
\sum_{
\begin{matrix}
\scriptscriptstyle
\left(x_1,X_1\right),\dots,\left(x_i,X_i\right)
\\
\scriptscriptstyle
\in\left\{(x,X),(y,Y),(z,Z)\right\}
\end{matrix}
}x_1\cdots x_i\
\sum_{u\in\mathcal{M}\left(p,X_1,\dots,X_i,\overline{s}\right)}
(-1)^
{
(i+1)\mathbb{1}_{\operatorname{o}\left(\mathbb{L}\right)\ne\operatorname{o}\left(\partial u\right)}
+
\#\left(\partial u \cap {\small\color{gray}\bigstar}_{\mathbb{L}}\right)
}
f
\ \circ \
P\left(\left(\partial u\right)_0
\right).
\end{equation}
It is also
an $S$-module map
$
S\otimes
\left(\left.E\right|_p\right)^*
\rightarrow
S\otimes
\left(\left.E\right|_s\right)^*,
$
by considering $x$, $y$ and $z$ as variables in $S=\field[[x,y,z]]$.

If we
choose a point ${\small\color{red}\bigstar}$ on $L$ so that 
$
\hol_{\small\color{red}\bigstar}(E)
$
is represented by a matrix
$
\primeLambda
\in
\operatorname{GL}_{\primemu}\left(\field\right),
$
trivialize
$
\left.E\right|_{S^1\setminus{\small\color{red}\bigstar}}
\cong
\left(S^1\setminus{\small\color{red}\bigstar}\right)
\times
\field^{\primemu}
$
as in Proposition \ref{rmk:ChooseTrivialization}
and thus
$
\left(\left.E\right|_p\right)^*
\cong
\left(\left.E\right|_s\right)^*
\cong
\field^{\primemu}
$,
then
it
is represented by
the
$\primemu\times\primemu$ matrix
\begin{equation}\label{eqn:MatrixOfComponentOfDeformedDifferential}
\sum_{
\begin{matrix}
\scriptscriptstyle
\left(x_1,X_1\right),\dots,\left(x_i,X_i\right)
\\
\scriptscriptstyle
\in\left\{(x,X),(y,Y),(z,Z)\right\}
\end{matrix}
}x_1\cdots x_i\
\sum_{u\in\mathcal{M}\left(p,X_1,\dots,X_i,\overline{s}\right)}
(-1)^
{
(i+1)\mathbb{1}_{\operatorname{o}\left(\mathbb{L}\right)\ne\operatorname{o}\left(\partial u\right)}
+
\#\left(\partial u \cap {\small\color{gray}\bigstar}_{\mathbb{L}}\right)
}
\left(\primeLambda^T\right)^
{
\#
\left(\partial u \cap {\small\color{red}\bigstar}_L
\right)
\left(
\mathbb{1}_{
\operatorname{o}\left(L\right)=\operatorname{o}\left(\partial u\right)
}
-
\mathbb{1}_{
\operatorname{o}\left(L\right)\ne\operatorname{o}\left(\partial u\right)
}
\right)
},
\end{equation}
where
$
\#
\left(\partial u \cap {\small\color{red}\bigstar}_L
\right)
$
is the number of times $\partial u$ passes through the point ${\small\color{red}\bigstar}$ on $L$.
Note that the entries can be also viewed as
elements in $\field[[x,y,z]]$.

\end{prop}

\subsection{Illustration with an example}\label{sec:LMFIllustration}


In this section,
we illustrate the computation of the localized mirror functor
with a concrete example.
Namely, we will find the mirror matrix factorization of the loop with a local system
$
\mathcal{L}
:=
\mathcal{L}
\left(\left(3,-2,2\right),\primelambda,1\right)
$
\footnote{
This is the canonical form of loops with a local system corresponding to a loop datum $\left((3,-2,2),\primelambda,1\right)$.
See Definition \ref{defn:LoopData}.
}
for any $\primelambda\in\field^{\times}$.
It is an object $\left(L,E,\nabla\right)$ of $\Fuk\left(\POP\right)$
that consists of the underlying loop
$
L:=L\left(3,-2,2\right)
$
described in Figure \ref{fig:LocalizedMirrorFunctorExample},
a trivial line bundle $E$ over the domain $S^1$ of $L$,
and a flat connection $\nabla$ on $E$ whose holonomy is $\primelambda$
at the point {\small$\color{red}\bigstar$} on $L$ marked in Figure \ref{fig:LocalizedMirrorFunctorExample}.


Note that $L$ and $\mathbb{L}$ have $6$ intersections,
say $p$, $q$, $r$, and $s$, $t$, $u$.
According to their degrees,
we have
$$
\chi^0\left(L,\mathbb{L}\right)
=
\left\{p,q,r\right\},
\quad
\chi^1\left(L,\mathbb{L}\right)
=
\left\{s,t,u\right\}.
$$
We trivialize $E$ over $S^1\setminus{\small\color{red}\bigstar}$
as in Proposition \ref{rmk:ChooseTrivialization},
which yields identifications
$
\left.E\right|_{\bullet} \cong \field
$
for each
$\bullet
\in
\chi\left(L,\mathbb{L}\right).
$
Then we have
$
\Hom_{\field}\left(\left.E\right|_{\bullet}, \left.E_{\mathbb{L}}\right|_{\bullet}\right)
\cong
\Span\left\{\bullet\right\}
$
as noted in Remark \ref{rmk:TrivializingHom},
and hence
$$
\hom^0(\mathcal{L},\bL)
=
\Span\left\{p,q,r\right\},
\quad
\hom^1(\mathcal{L},\bL)
=
\Span\left\{s,t,u\right\}.
$$

Then two restricted operations
$
\operatorm_1^{0,b}
:
\Hom^0\left(\mathcal{L},\mathbb{L}\right)
\rightarrow
\Hom^1\left(\mathcal{L},\mathbb{L}\right)
$
and
$
\operatorm_1^{0,b}
:
\Hom^1\left(\mathcal{L},\mathbb{L}\right)
\rightarrow
\Hom^0\left(\mathcal{L},\mathbb{L}\right)
$
with respect to ordered bases
$
\left\{p,q,r\right\}
$
and
$
\left\{s,t,u\right\}
$
yield two $3\times3$ matrices
$\LocalPhi\left(\mathcal{L}\right)$
and
$\LocalPsi\left(\mathcal{L}\right)$,
respectively.

\begin{figure}[h]
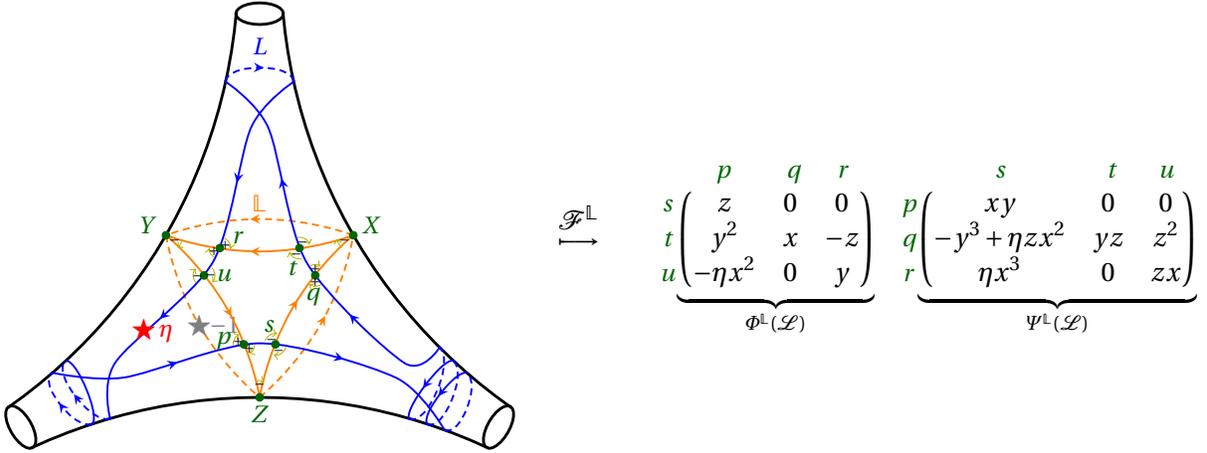

\setlength\arraycolsep{0pt}
\centering
$

}
\right)
}_{\LocalPsi(\mathcal{L})}
}
\\[-2mm]
\end{matrix}
\\[14mm]
\end{matrix}
\end{matrix}
$
\captionsetup{width=1\linewidth}
\caption{Matrix factorization of $xyz$ corresponding to $\mathcal{L}=\mathcal{L}\left(\left(3,-2,2\right),\primelambda,1\right)$}
\label{fig:LocalizedMirrorFunctorExample}
\end{figure}

Entries come from the following relations:
$$
\setlength\arraycolsep{0pt}
\setcounter{MaxMatrixCols}{100}
\left\{
\begin{matrix}
\operatorm_1^{0,b}(p) & = & z & s & + & y^2 & t & - & \primelambda & x^2 & u \\[1mm]
\operatorm_1^{0,b}(q) & = &   &   &   &   x & t &   &          &     &   \\[1mm]
\operatorm_1^{0,b}(r) & = &   &   & - &   z & t & + &          &   y & u
\end{matrix}
\right.
\quad\text{and}\quad
\left\{
\begin{matrix}
\operatorm_1^{0,b}(s) & = & xy & p & + & (-y^3 & + & \primelambda & zx^2) & q & + & \primelambda & x^3 & r \\[1mm]
\operatorm_1^{0,b}(t) & = &    &   &   &       &   &          &    yz & q &   &          &     &   \\[1mm]
\operatorm_1^{0,b}(u) & = &    &   &   &       &   &          &   z^2 & q & + &          &  zx & r
\end{matrix}
\right.
$$
Note that
the $\left({\scriptstyle\blacksquare},\bullet\right)$-entry of
$\LocalPhi\left(\mathcal{L}\right)$ or $\LocalPsi\left(\mathcal{L}\right)$
is the coefficient of ${\scriptstyle\blacksquare}$ in $\operatorm_1^{0,b}\left(\bullet\right)$
for
$\bullet,
{\scriptstyle\blacksquare}\in\left\{p,q,r,s,t,u\right\}$.
So it can be computed from the formula
(\ref{eqn:MatrixOfComponentOfDeformedDifferential}),
which in this case is just a single entry
(in $S=\field[[x,y,z]]$)
\begin{equation}\label{eqn:EntryOfDeformedDifferential}
\displaystyle
\sum_{
\begin{matrix}
\scriptscriptstyle
\left(x_1,X_1\right),\dots,\left(x_i,X_i\right)
\\
\scriptscriptstyle
\in\left\{(x,X),(y,Y),(z,Z)\right\}
\end{matrix}
}x_1\cdots x_i\
\sum_{u\in\mathcal{M}\left(\bullet,X_1,\dots,X_i,\overline{\scriptscriptstyle\blacksquare}\right)}
(-1)^
{
(i+1)\mathbb{1}_{\operatorname{o}\left(\mathbb{L}\right)\ne\operatorname{o}\left(\partial u\right)}
+
\#\left(\partial u \cap {\small\color{gray}\bigstar}_{\mathbb{L}}\right)
}
\primelambda^
{
\#
\left(\partial u \cap {\small\color{red}\bigstar}_L
\right)
\left(
\mathbb{1}_{
\operatorname{o}\left(L\right)=\operatorname{o}\left(\partial u\right)
}
-
\mathbb{1}_{
\operatorname{o}\left(L\right)\ne\operatorname{o}\left(\partial u\right)
}
\right)
}.
\end{equation}

That is,
each deformed strip
$u$ bounded by $L$ and $\mathbb{L}$,
whose angles consist of $\bullet$,
$X_1,\dots,X_i$
and $\overline{\scriptstyle\blacksquare}$
in a counterclockwise order,
contributes a monomial
$x_1\cdots x_i$.
Its coefficient is determined by coincidence of orientations of $L$, $\mathbb{L}$
with boundary orientation of $\partial u$
and the number of times $\partial u$ passes through the
points
${\small\color{red}\bigstar}_L$
or
${\small\color{gray}\bigstar}_{\mathbb{L}}$.
Figure \ref{fig:LocalizedMirrorFunctorEntryExamples} shows some deformed strips
that contribute
some of the entries above:

\newpage
\begin{figure}[H]
\vspace{-25mm}
\setlength\arraycolsep{-5pt}
\centering
$$
\hspace*{-22mm}

\\
&&&\hspace{-120mm}\text{(g) $(q,s)$-entry of $\LocalPsi\left(L\right)$: $-y^3+\primelambda zx^2$}
\end{matrix}
$$
\captionsetup{width=1\linewidth}
\caption{Some entries of $\LocalPhi(L)$ and $\LocalPsi(L)$}
\label{fig:LocalizedMirrorFunctorEntryExamples}
\end{figure}
\thispagestyle{empty}

\subsection{Higher rank computation}
We can compute the matrix factorizations corresponding to
a higher rank local system on a loop
easily once we know the result for rank $1$ local systems on the same loop.

Consider two local systems on the same loop $L$ in $\POP$
that have rank $1$ and $\primemu$, respectively:
\begin{itemize}
\item
$\mathcal{L}_1\left(\primelambda\right)
:=\left(L,E_1,\nabla_1\right)$
: loop with a local system of rank $1$
of holonomy $\primelambda\in\field^\times$.
\item
$\mathcal{L}_{\primemu}\left(\primeLambda\right)
:=\left(L,E_{\primemu},\nabla_{\primemu}\right)$
: loop with a local system of rank $\primemu$
of holonomy $\primeLambda\in\operatorname{GL}_{\primemu}\left(\field\right)$,
\end{itemize}
Namely,
there is
a point ${\small\color{red}\bigstar}$ on $L$
and
identification
$\left.E_1\right|_{\small\color{red}\bigstar}
\cong
\field$
and
$\left.E_{\primemu}\right|_{\small\color{red}\bigstar}
\cong
\field^{\primemu}$,
so that
$
\hol_{\small\color{red}\bigstar}\left(E_1\right)
$
is represented by a scalar
$
\primelambda
\in
\field^\times,
$
while 
$
\hol_{\small\color{red}\bigstar}\left(E_{\primemu}\right)
$
is represented by a matrix
$
\primeLambda
\in
\operatorname{GL}_{\primemu}\left(\field\right)
$.

For any $p,s\in\chi\left(L,\mathbb{L}\right)$,
the formula
(\ref{eqn:EntryOfDeformedDifferential})
and
(\ref{eqn:MatrixOfComponentOfDeformedDifferential})
imply that
the $(s,p)$-entry of
$
\left(\LocalPhi\left(\mathcal{L}_1\left(\primelambda\right)\right),\LocalPsi\left(\mathcal{L}_1\left(\primelambda\right)\right)\right)
$
and
the
$
\left(
\left(\left.E_{\primemu}\right|_s\right)^*,\left(\left.E_{\primemu}\right|_p\right)^*
\right)
$-component of
$
\left(\LocalPhi\left(\mathcal{L}_{\primemu}\left(\primeLambda\right)\right),\LocalPsi\left(\mathcal{L}_{\primemu}\left(\primeLambda\right)\right)\right)
$
are given in the form
\begin{equation}\label{eqn:EntryComponentExpansion}
\sum_{k=-\infty}^{\infty}a_k \primelambda^k
\in \field[[x,y,z]]
\quad\text{and}\quad
\sum_{k=-\infty}^{\infty}a_k \left(\primeLambda^T\right)^k
\in \field[[x,y,z]]^{\primemu\times \primemu},
\end{equation}
respectively,
for some $a_k\in \field[[x,y,z]]$.
(See Theorem \ref{thm:CylinderFreeMFFinite} for the convergence issue).

\begin{prop}[$\left(\primelambda,\primeLambda^T\right)$-substitution]\label{prop:HigherRankLMF}
\label{prop:SubstitutionGeometry}
The matrix factorizations
corresponding to
$\mathcal{L}_1\left(\primelambda\right)$
and
$\mathcal{L}_{\primemu}\left(\primeLambda\right)$
are
given in the form
$$
\LocalF\left(\mathcal{L}_1\left(\primelambda\right)\right)
=
\left(
\sum_{k=-\infty}^{\infty}\varphi_k\primelambda^k,
\sum_{k=-\infty}^{\infty}\psi_k\primelambda^k
\right)
\quad\text{and}\quad
\LocalF\left(\mathcal{L}_{\primemu}\left(\primeLambda\right)\right)
=
\left(
\sum_{k=-\infty}^{\infty}\varphi_k\otimes\left(\primeLambda^T\right)^k,
\sum_{k=-\infty}^{\infty}\psi_k\otimes\left(\primeLambda^T\right)^k
\right)
\footnote{
Here $\otimes$ for matrices refers to the \emph{Kronecker product}.
},
$$
respectively,
for some
$
\varphi_k, \psi_k\in \field[[x,y,z]]^{\tau\times\tau}
$
$\left(\tau=\frac{1}{2}\left|L\cap\mathbb{L}\right|\right)$.

\vspace{-1mm}

\end{prop}

\begin{proof}
Let $\varphi_k$ be a $\tau\times\tau$ matrix over
$S=\field[[x,y,z]]$
whose $(s,p)$-entry is
$a_k$
given in (\ref{eqn:EntryComponentExpansion})
(viewed as
a map of free $S$-modules
from
$\operatorname{Span}_S\left(\chi^0\left(L,\mathbb{L}\right)\right)$
to
$\operatorname{Span}_S\left(\chi^1\left(L,\mathbb{L}\right)\right)$
with respect to an appropriate order in each set $\chi^{\bullet}\left(L,\mathbb{L}\right)$).
Then
$
\sum_{k=-\infty}^{\infty}\varphi_k\primelambda^k
$
coincides with
$\LocalPhi\left(\mathcal{L}_1\left(\primelambda\right)\right)$.

Now $\varphi_k\otimes\left(\primeLambda^T\right)^k$ is a $\tau\primemu\times\tau\primemu$ matrix over $S$,
which can be also viewed as
a map of free $S$-modules
$$
\bigoplus_{p\in\chi^0\left(L,\mathbb{L}\right)}
\operatorname{Span}_S\left\{p\right\}
\otimes_{\field}
\left(\left.E_{\primemu}\right|_p\right)^*
\rightarrow
\bigoplus_{s\in\chi^1\left(L,\mathbb{L}\right)}
\operatorname{Span}_S\left\{s\right\}
\otimes_{\field}
\left(\left.E_{\primemu}\right|_s\right)^*
$$
under
an order in $\chi^{\bullet}\left(L,\mathbb{L}\right)$
and
identification
$
\left.E_{\primemu}\right|_p
\cong
\left.E_{\primemu}\right|_s
\cong
\field^{\primemu}.
$
Its
$
\left(
\left(\left.E_{\primemu}\right|_s\right)^*,\left(\left.E_{\primemu}\right|_p\right)^*
\right)
$-component
is
$
a_k \left(\primeLambda^T\right)^k,
$
where $a_k$ denotes the $(s,p)$-entry of $\varphi_k$.
Therefore,
the corresponding
component of
$
\sum_{k=-\infty}^{\infty}
\varphi_k\otimes\left(\primeLambda^T\right)^k
$
is
$
\sum_{k=-\infty}^{\infty}
a_k\left(\primeLambda^T\right)^k,
$
which is
that
of
$
\LocalPhi\left(\mathcal{L}_{\primemu}\left(\primeLambda\right)\right)
$
by (\ref{eqn:EntryComponentExpansion}).
It works the same for $\psi_k$ instead of $\varphi_k$.
\end{proof}

The proposition says that
if we know $\varphi_k$ and $\psi_k$'s from the rank $1$ cases,
we immediately get the result for higher rank cases as well,
just by `\textbf{substituting the matrix $\primeLambda^T$ for the scalar $\primelambda$}'.
See the example:

\begin{ex}\label{ex:HigherRankExample}
Consider the loop with a local system
$
\mathcal{L}_{\primemu}
:=
\mathcal{L}
\left(\left(3,-2,2\right),\primelambda,\primemu\right)
$
of rank $\primemu\in\mathbb{Z}_{\ge1}$.
It consists of the same underlying loop $L:=L\left(3,-2,2\right)$ as in Subsection \ref{sec:LMFIllustration}
(Figure \ref{fig:LocalizedMirrorFunctorExample}),
trivial vector bundle $E_{\primemu}$\ of rank $\primemu$ over the domain $S^1$ of $L$,
and a flat connection $\nabla_{\primemu}$ on $E_{\primemu}$ whose holonomy is $J_{\primemu}\left(\primelambda\right)$
at the point {\small$\color{red}\bigstar$}.
Then,
by the result for $\mathcal{L}_1$ in Subsection \ref{sec:LMFIllustration} and Proposition \ref{prop:HigherRankLMF},
the corresponding matrix factorization is
%
%
%
%
%
$$
\makeatletter\setlength\BA@colsep{2pt}\makeatother
\begin{blockarray}{c @{\hspace{2pt}} c @{\hspace{7pt}} ccc}
  &
  & \color{darkgreen} \small\text{$\left(\left.E\right|_p\right)^*$}
  & \color{darkgreen} \small\text{$\left(\left.E\right|_q\right)^*$}
  & \color{darkgreen} \small\text{$\left(\left.E\right|_r\right)^*$}
  \\[1mm]
  \begin{block}{c @{\hspace{2pt}} c @{\hspace{7pt}} (ccc)}
    & \color{darkgreen} \small\text{$\left(\left.E\right|_s\right)^*$} & z I_{\primemu} & 0 & 0
    \\[1mm]
    \LocalPhi\left(\mathcal{L}_{\primemu}\right) =
    & \color{darkgreen} \small\text{$\left(\left.E\right|_t\right)^*$} & y^2 I_{\primemu} & x I_{\primemu} & -z I_{\primemu}
    \\[1mm]
    & \color{darkgreen} \small\text{$\left(\left.E\right|_u\right)^*$} & -x^2 J_{\primemu}\left(\primelambda\right)^T & 0 & y I_
    \\
  \end{block}
\end{blockarray}
\quad
\begin{blockarray}{c @{\hspace{2pt}} c @{\hspace{7pt}} ccc @{\hspace{6pt}} c}
  &
  & \color{darkgreen} \small\text{$\left(\left.E\right|_s\right)^*$}
  & \color{darkgreen} \small\text{$\left(\left.E\right|_t\right)^*$}
  & \color{darkgreen} \small\text{$\left(\left.E\right|_u\right)^*$}
  &
  \\[1mm]
  \begin{block}{c @{\hspace{2pt}} c @{\hspace{7pt}} (ccc) @{\hspace{6pt}} c}
    & \color{darkgreen} \small\text{$\left(\left.E\right|_p\right)^*$} & xy I_{\primemu} & 0 & 0
    &
    \\[1mm]
    \text{and}\quad\LocalPsi\left(\mathcal{L}_{\primemu}\right) =
    & \color{darkgreen} \small\text{$\left(\left.E\right|_q\right)^*$} & -y^3 I_{\primemu} + zx^2 J_{\primemu}\left(\primelambda\right)^T & yz I_{\primemu} & z^2 I_{\primemu}
    & .
    \\[1mm]
    & \color{darkgreen} \small\text{$\left(\left.E\right|_r\right)^*$} & x^3 J_{\primemu}\left(\primelambda\right)^T & 0 & zx I_{\primemu}
    &
    \\
  \end{block}
\end{blockarray}
$$

\end{ex}

\section{Matrix Factorizations from Loops with a Local System}\label{sec:MfFromLag}

In this section,
we first
define the {canonical form
of immersed loops in $\POP$
and local systems} on them,
parameterized by
\emph{normal loop words} (\S \ref{sec:LoopWords}) and \emph{loop data} (\S \ref{sec:LoopData}),
respectively.
Then we
compute their mirror image under the localized mirror functor,
which provides the {canonical form of matrix factorizations of $xyz$}
(\S \ref{sec:MFfromLoopsWithLocalSystem}).
The matrix factorizations corresponding to periodic loop words are decomposable,
which implies that \emph{non-primitive} loops (with a local system) are decomposable in the Fukaya category
(\S \ref{sec:PeriodicCase}).
Meanwhile, an exceptional case is handled separately (\S \ref{sec:ExceptionalCase}).

\subsection{Loop words and canonical form of immersed loops}\label{sec:LoopWords}
We first recall from \cite{CJKR} the concept of \emph{loop words},
which parameterize free homotopy classes of loops in $\POP$.
They were introduced
in order
to pick a specific representative in each (hyperbolic) free homotopy class,
motivated by the observation that two freely homotopic loops
correspond to homotopically equivalent matrix factorizations
under the localized mirror functor (Theorem \ref{thm:CylinderFreeMFFinite}).

Note that the fundamental group of $\POP$ can be presented as
$\pi_1\left(\POP\right) = \left<\alpha,\beta,\gamma\left|\alpha\beta\gamma=1\right.\right>$
with the based loops $\alpha$, $\beta$ and $\gamma$ in $\POP$ shown in Figure \ref{fig:gen}.
Also recall that there is a one-to-one correspondence between the free homotopy classes of loops in $\POP$ and the conjugacy classes in $\pi_1\left(\POP\right)$. 


          }
\centering
\caption{Generators $\alpha$, $\beta$, $\gamma$ of $\pi_1(\POP)$}
\label{fig:gen}
\end{subfigure}
\begin{subfigure}{0.43\textwidth}
          \adjustbox{height=44mm}{
          %
          }
\centering
\captionsetup{width=1\linewidth}
\caption{Canonical form $L\left(w',\primelambda,\primemu\right)$ of loops a local system}
\label{fig:GeneralLoops}
\end{subfigure}
\centering
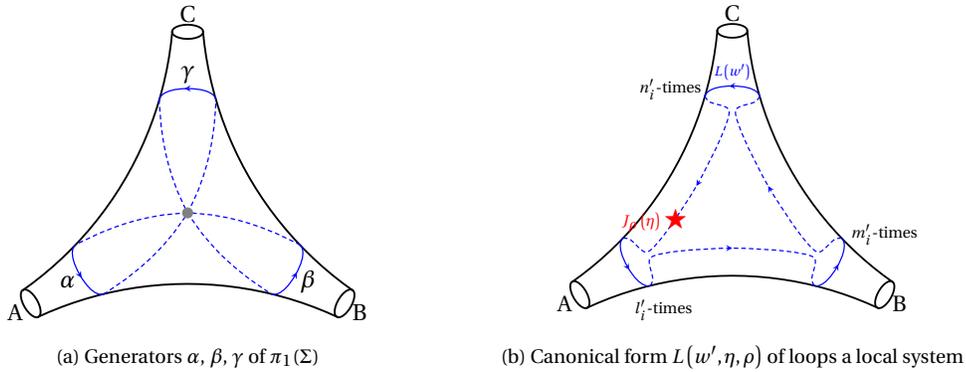
\caption{Fundamental group and loop data}
\end{figure}

\begin{defn}\label{def:LoopWord}
A \textnormal{\textbf{loop word}} of \emph{length} $3\tau$ is
$$
w' = (l_1',m_1',n_1',l_2',m_2',n_2',\dots,l_\tau',m_\tau',n_\tau')\in \mathbb{Z}^{3\tau}
$$
($\tau\in\mathbb{Z}_{\ge1}$).
The associated loop,
denoted as
$$
L\left(w'\right),
$$
is illustrated in Figure \ref{fig:GeneralLoops}.
It visits three holes A, B, and C in turn, winding them around the number of times specified in $w'$.
Namely,
starting from the point
{\small$\color{red}\bigstar$}
marked in the figure,
it winds
hole A $l_1'$-times,
hole B $m_1'$-times,
hole C $n_1'$-times,
hole A $l_2'$-times,
hole B $m_2'$-times,
and so on.
After finally it winds hole C $n_{\tau}'$-times,
it returns to the point {\small$\color{red}\bigstar$} to form a closed loop.
We perturb it if it is necessary to put them together into a transversal set.

Note that its free homotopy class in
$
\left[S^1,\POP\right]
=
\left.\pi_1\left(\POP\right)\right/\sim_{\textnormal{conjugation}}$
is
$$
\left[L\left(w'\right)\right]
=
\left[
\alpha^{l_1'}\beta^{m_1'}\gamma^{n_1'}\alpha^{l_2'}\beta^{m_2'}\gamma^{n_2'}\cdots\alpha^{l_\tau'}\beta^{m_\tau'}\gamma^{n_\tau'}
\right].
$$
Two loop words $w'$ and $\tilde{w}'$ are regarded as \emph{equivalent} if
$
\left[L\left(w'\right)\right]
=
\left[L\left(\tilde{w}'\right)\right].
$
\end{defn}

\begin{ex}
The loop described in Figure \ref{fig:FukObjects} and Figure \ref{fig:LocalizedMirrorFunctorExample} is (a perturbation of) $L(3,-2,2)$. 
\end{ex}

We
denote the $j$-th value of a loop word $w'$ as $w'_j$ so that
\begin{align*}
w'
&=
\left(w'_1, w'_2, w'_3, \dots, w'_{3\tau-1}, w'_{3\tau}\right)
\in\mathbb{Z}^{3\tau}.
\end{align*}
Then any tuple
$
(w'_{k},w'_{k+1},\dots,w'_l)
$
for some distinct $k$, $l\in\mathbb{Z}_{3\tau}$ is called a \emph{subword} in $w'$.
We regard the index $i$ of $l'_i$, $m'_i$ and $n'_i$ to be in $\mathbb{Z}_\tau$ (hence $3i\in\mathbb{Z}_{3\tau}$)
and the index $j$ of $w'_j$ to be in $\mathbb{Z}_{3\tau}$.
Therefore,
for example,
$(w'_{3\tau-1},w'_{3\tau},w'_1)$ is a subword.
We define the \emph{$1$-shift} of a loop word $w'$ to be
$$
w'^{(1)}=\left(l'_2,m'_2,n'_2,\dots,l'_\tau,m'_\tau,n'_\tau,l'_1,m'_1,n'_1\right)
\in\mathbb{Z}^{3\tau}
$$
and \emph{$k$-shift} to be $w'^{(k)}$ which is obtained from $w'$ by applying the $1$-shift $k$-times ($k \in \Z$).

For a
loop word $w'$,
we define its \emph{$N$-concatenation} $\left(w'\right)^{\hash N}$ as the $N$ repetitions of $w'$.
It is called \emph{periodic} if it is $N$-concatenation of another
loop word for some $N\in\mathbb{Z}_{\ge2}$.
For an immersed loop $L:S^1\rightarrow \POP$,
we define its \emph{$N$-concatenation} by the immersed loop
$$
L^{\hash N}:S^1\rightarrow \POP,\quad
e^{2\pi it}\mapsto L\left(e^{2 N \pi it}\right).
$$
A loop $L$ or its free homotopy class $[L]$ are called \emph{non-primitive}
if $L$ is freely homotopic to an $N$-concatenation of another loop for some $N\in\mathbb{Z}_{\ge2}$.
Otherwise, they are called \emph{primitive}.

Note that
if a loop word $w'$ is periodic,
then the associated
loop $L\left(w'\right)$ and its free homotopy class $\left[L\left(w'\right)\right]$
are non-primitive.
But the converse is not true in general as a non-periodic loop word $w'$ can be equivalent to a periodic one.
It will be fixed when we will regard only \emph{normal} loop words (Corollary \ref{cor:PeriodicNonprimitive}).


The following lemma is easy to check.
\begin{lemma}\label{lem:equimove}
The following operations on a loop word $w'$ do not change the equivalence class of $w'$:
\begin{itemize}
\item (inserting 0s) insert the subword $(0,0,0)$ somewhere in $w'$,
\item (removing 0s) remove a subword $(0,0,0)$ in $w'$ if it exists,
\item (adding $1$s around $0$) add $(1,1,1)$ to the subword $(w_{j-1}',0,w_{j+1}')$ in $w'$ where $w_j'=0$, and
\item (subtracting $1$s around $1$) subtract $(1,1,1)$ from the subword $(w_{j-1}',1,w_{j+1}')$ in $w'$ where $w_j'=1$,
\item (shifting) take $k$-shift of $w'$ for some $k\in\mathbb{Z}$.
\end{itemize}
\end{lemma}

The converse statement is also true,
but its proof involves a non-trivial word problem.

 \begin{prop}\cite{CJKR}\label{prop:equihomotopy}
Two loop words $w'$ and $\tilde{w}'$ are equivalent if and only if $\tilde{w}'$ can be obtained from $w'$ by performing the above five operations finitely many times.
\end{prop}

Note that several equivalent loop words can represent the same free homotopy class.
To find a unique representative in each class, we introduce the following \textbf{normal form}  of loop words.
It will play an important role in the conversion formula between loop data and band data.

\begin{defn}\label{defn:normal2}
A loop word $w'$ is said to be \textnormal{\textbf{normal}} if it satisfies the following conditions:
\begin{itemize}
\item any subword of the form $\left(a,1,b\right)$ in $w'$ satisfies $a, b\le0$,
\item any subword of the form $\left(a,0,b\right)$ in $w'$ satisfies $a\le-1$, $b\ge1$ or $a\ge1$, $b\le-1$ or $a, b\ge1$,
\item $w'$ has no subword of the form $(0,-1,-1,\dots,-1,0)$, and
\item $w'$ does not consist only of $-1$, that is, $w'\ne\left(-1,-1,\dots,-1\right)$.
\end{itemize}
\end{defn}

\vspace{30mm}

We say that a loop $L$ or its free homotopy class $[L]$ are
\emph{elliptic} if $L$ is null-homotopic,
\emph{parabolic} if $L$ is freely homotopic to some concatenation of a boundary loop,
and \emph{hyperbolic} otherwise
\footnote{
The terminologies come from hyperbolic geometry.
In fact,
if we assign a hyperbolic metric to the surface,
the elliptic, parabolic, and hyperbolic loops correspond to
concepts already in use in hyperbolic geometry.
We refer readers to Section 9.6 in \cite{Ratcliffe}.
}.
An elliptic loop is obstructed,
and a parabolic loop $L$ can always be deformed so that it doesn't meet the reference $\mathbb{L}$ at all,
which means that the corresponding matrix factorization $\LocalF\left(L\right)$ is null-homotopic.
Therefore, elliptic and parabolic loops will be excluded from our consideration.

According to Definition \ref{def:LoopWord},
each elliptic or parabolic loop is produced by a loop word
equivalent to one of
$\left(l',0,0\right)$,
$\left(0,m',0\right)$
or
$\left(0,0,n'\right)$
for some
$l', m', n'\in\mathbb{Z}$.
Loop words
in those forms are called \emph{non-hyperbolic},
while the others are called \emph{hyperbolic}.
Therefore,
hyperbolic loop words produce hyperbolic loops.
Interestingly, the above normality condition automatically rules out non-hyperbolic loop words.
Moreover,
the normal form up to shifting gives exactly one representative among equivalent hyperbolic loop words.


\begin{prop}\cite{CJKR}\label{prop:normalnormal}
Any normal loop word is hyperbolic.
Conversely, any hyperbolic loop word is equivalent to a unique normal loop word up to shifting.

\end{prop}

This also implies that
two normal loop words $w'$ and $\tilde{w}'$ are equivalent
if and only if
they coincide up to shifting,
that is,
$\tilde{w}'=w'^{(k)}$ for some $k\in\mathbb{Z}$.


\begin{cor}\label{cor:PeriodicNonprimitive}
A normal loop word $w'$ is periodic
if and only if
the associated loop $L\left(w'\right)$ and its free homotopy class
$\left[L\left(w'\right)\right]$ are non-primitive.
\end{cor}


Now we give $\POP$ a hyperbolic metric with three cusps.
It can be achieved by considering the Poincar\'e disk as the universal cover of $\POP$
as shown in Figure \ref{fig:UniversalCover}.
In fact,
such a metric is unique up to isometry
(Theorem 9.8.8 in \cite{Ratcliffe}).
It is well-known in hyperbolic geometry
that there is exactly one (immersed) closed geodesic in each primitive hyperbolic free homotopy class of loops in $\POP$
(Theorem 9.6.4 in \cite{Ratcliffe}).
This provides another 
description of normal loop words.

\begin{prop}\label{prop:GeodesicNormalLoopWordCorrespondence}
There is a one-to-one correspondence
\begin{align*}
\left\{\textnormal{closed geodesics in }\POP \right\}
\ \ &\overset{1:1}{\leftrightarrow}\ \
\left.\left\{\textnormal{non-periodic normal loop words}\right\}\right/\sim_{\textnormal{shifting}}.
\end{align*}
\end{prop}

\begin{figure}[H]
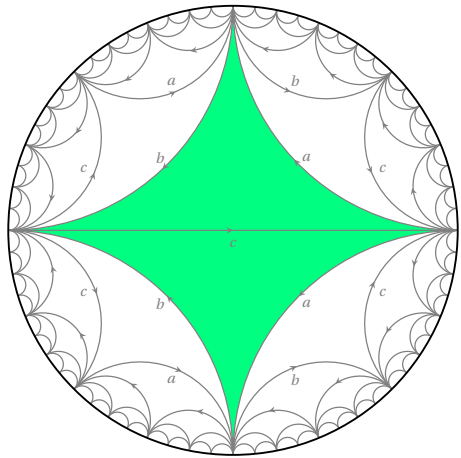

     \centering
          \adjustbox{height=60mm}{

          }
\centering
\captionsetup{width=1\linewidth}
\caption{Fundamental domain of $\POP$ in its universal cover (Poincar\'e disk)}
\label{fig:UniversalCover}
\end{figure}


\begin{spacing}{0.97}

\subsection{Loop data and canonical form of loops with a local system
}\label{sec:LoopData}


\begin{defn}\label{defn:LoopData}
A \textnormal{\textbf{loop datum}} $(w',\primelambda,\primemu)$ consists of the following:
\begin{itemize}
\item (normal loop word) $w' = (l'_1,m'_1,n'_1,l'_2,m'_2,n'_2,\dots,l'_\tau,m'_\tau,n'_\tau)\in \mathbb{Z}^{3\tau}$ for some $\tau \in \mathbb{Z}_{\ge1}$,
\item (holonomy parameter)\ $\primelambda \in \field^\times$,
\item ((geometric) rank) $\primemu \in \mathbb{Z}_{\ge1}$.
\end{itemize}
It represents an object $\left(L,E,\nabla\right)$ of $\Fuk\left(\POP\right)$,
denoted as
$$
\mathcal{L}\left(w',\primelambda,\primemu\right),
$$
that consists of the loop $L:=L\left(w'\right)$
defined in Definition \ref{def:LoopWord},
a trivial $\field$-vector bundle $E$ of rank $\primemu$ over the domain $S^1$ of $L$,
and a flat connection $\nabla$ on $E$
whose holonomy is $J_{\primemu}\left(\primelambda\right)$
at the point {\small$\color{red}\bigstar$} on $L$
(marked in Figure \ref{fig:GeneralLoops}).
We refer to it as the
\textnormal{\textbf{canonical form of
loops with a local system}}
corresponding to the loop datum
$\left(w',\primelambda,\primemu\right)
$.


\end{defn}

%
%
%
%
%
%

\begin{prop}
Let
$\left(L,E,\nabla\right)$ be a loop with a local system
with $L=L\left(w'\right)$ for some normal loop word $w'$.
%
Then there are finitely many pairs
$
\left(\primelambda_1,\primemu_1\right),\dots,\left(\primelambda_d,\primemu_d\right)
\in\field^\times\times\mathbb{Z}_{\ge1}
$
($d\in\mathbb{Z}_{\ge1}$)
such that
$\left(E,\nabla\right)$ is gauge equivalent to
the direct sum
$$
\left(E_1,\nabla_1\right)\oplus\cdots\oplus \left(E_d,\nabla_d\right),
$$
where $\left(L,E_i,\nabla_i\right)$
($i\in\left\{1,\dots,d\right\}$)
is the canonical form $\mathcal{L}\left(w',\primelambda_i,\primemu_i\right)$
corresponding to the loop datum $\left(w',\primelambda_i,\primemu_i\right)$.
Moreover, the choice of the pairs is unique up to the order.

\end{prop}

\begin{proof}
Note that every $\field$-vector bundle over $S^1$ is trivial.
At a fixed point {\small$\color{red}\bigstar$} on $L$, 
the holonomy
$
\hol_{\small\color{red}\bigstar}(E)
$
of $\left(E,\nabla\right)$
along $L$ at {\small$\color{red}\bigstar$}
is represented by some matrix
$\primeLambda\in\operatorname{GL}_{\primemu}\left(\field\right)$.
Then the Jordan canonical form of $\primeLambda$
has Jordan blocks
$J_{\primemu_1}\left(\primelambda_1\right),\dots,J_{\primemu_d}\left(\primelambda_d\right)$
($d\in\mathbb{Z}_{\ge1}$)
for some pairs
$\left(\primelambda_1,\primemu_1\right),\dots,\left(\primelambda_d,\primemu_d\right)\in\field^\times \times\mathbb{Z}_{\ge1}$,
which yields the desired decomposition.
\end{proof}

This shows that any indecomposable local system on a fixed loop
$L=L\left(w'\right)$
is gauge equivalent to a unique canonical form
$\mathcal{L}\left(w',\primelambda,\primemu\right)$.
Combining this with Proposition \ref{prop:GeodesicNormalLoopWordCorrespondence},
we get the following:  

\begin{cor}\label{cor:GeodesicLocalSystemLoopData}
There is a one-to-one correspondence
$$
\left.
\left\{
\textnormal{\textnormal{{closed geodesics in $\POP$ with an indecomposable local system}}}
\right\}
\right/\sim_\textnormal{gauge equivalence}
$$
$$
\hspace{-55mm}
\quad\overset{1:1}{\leftrightarrow}\quad
\left.\left\{\textnormal{non-periodic loop data}
\right\}
\right/
\sim_{\textnormal{shifting}}
\footnote{
We say that a loop datum is \emph{periodic} if its normal loop word is periodic,
and define a \emph{shift} of a loop datum as the shifting of its normal loop word.
}.
$$

\end{cor}


\subsection{Matrix factorizations from canonical form of loops with a local system}
\label{sec:MFfromLoopsWithLocalSystem}

To get a matrix factorization defined over $\field$
from the
loop with a local system
$
\mathcal{L}\left(w',\primelambda,\primemu\right)
$
constructed
in Definition \ref{defn:LoopData},
we first need to ensure that
the underlying loop $L\left(w'\right)$
doesn't bound an immersed cylinder with the reference loop $\mathbb{L}$
(Seidel Lagrangian).
We briefly summarize the discussion in \cite{CJKR} here:

\begin{defn}\label{defn:CylinderFree}
A loop $L$ in $\POP$ is said to be \textnormal{\textbf{cylinder-free}} with
$\mathbb{L}$
if there is no immersion
$j:S^1\times[0,1]\rightarrow\POP$
that satisfies
$j\left(e^{2\pi it},0\right) = L\left(\imath(t)\right)$
and
$j\left(e^{2\pi it},1\right) = \mathbb{L}\left(\jmath(t)\right)$
for some immersions $\imath$, $\jmath:S^1\rightarrow S^1$.
\end{defn}

\begin{thm}\cite{CJKR}\label{thm:CylinderFreeMFFinite}
For an object $\mathcal{L}:=(L,E,\nabla)$ in $\Fuk\left(\POP\right)$
whose underlying loop $L$ is cylinder-free with $\mathbb{L}$,
its mirror matrix factorization
$
\LocalF\left(\mathcal{L}\right)
=
\left(\LocalPhi\left(\mathcal{L}\right),\LocalPsi\left(\mathcal{L}\right)\right)
$
is well-defined over $\field$.
More precisely,
this means that
the moduli spaces
involved in formula \ref{eqn:ComponentOfDeformedDifferential}
are finite.

Moreover,
the homotopy class of $\LocalF\left(\mathcal{L}\right)$ is invariant
under free homotopy of the underlying curve $L$ and
gauge equivalence of the flat vector bundle $(E,\nabla)$
\footnote{
See also \cite[
Proposition 5.4]{CHL-toric}
for its invariance under Hamiltonian isotopy of $L$.
}.
\end{thm}

\begin{prop}\cite{CJKR}\label{prop:CylinderFree}
For a normal loop word $w'$ other than of the form $\left(2,2,2\right)^{\hash \tau}$, the corresponding loop $L\left(w'\right)$ is cylinder-free with $\mathbb{L}$.
\end{prop}

\end{spacing}

\newpage

Now we compute the matrix factorization corresponding to the canonical form
$
\left(L,E,\nabla\right)
:=
\mathcal{L}\left(w',\primelambda,\primemu\right)
$
for a loop datum $\left(w',\primelambda,\primemu\right)$
with $w'\ne(2,2,2)^{\hash \tau}$.
For
a loop word of length $3\tau$
$$
w' = \left(l_1',m_1',n_1',l_2',m_2',n_2',\dots,l_{\tau}',m_{\tau}',n_{\tau}'\right),
$$
the corresponding loop $L=L\left(w'\right)$ has $6\tau$ intersections with $\mathbb{L}$.
We name even-degree angles from $L$ to $\mathbb{L}$ by
$
p_1,q_1,r_1,\dots,p_{\tau},q_{\tau},r_{\tau}\in\hom^0\left(L,\mathbb{L}\right)
$
and odd-degree angles from $L$\ to $\mathbb{L}$ by
$
s_1,t_1,u_1,\dots,s_{\tau},t_{\tau},u_{\tau}\in\hom^1\left(L,\mathbb{L}\right)
$
in the order along the orientation of $L$.

\begin{figure}[H]
$$
          \adjustbox{height=60mm}{

          }
$$
\caption{
Canonical form $\left(L,E,\nabla\right)
=
\mathcal{L}\left(w',\primelambda,\primemu\right)
$
and Seidel Lagrangian
$\mathbb{L}$
}
\end{figure}

\begin{prop}\label{thm:MFFromLag}
For a
loop datum $\left(w',\primelambda,\primemu\right)$
with
$w'\ne(2,2,2)^{\hash \tau}$
\footnote{
For the case $w'=(2,2,2)^{\hash \tau}$,
we still get the same form
even if we use the loop constructed in Definition \ref{def:LoopWord},
which is not cylinder-free with $\mathbb{L}$.
For $\primelambda=-1$,
however,
it fails to be a matrix factor of $xyz$
as its determinant is zero.
For $\primelambda\ne-1$,
it is a valid matrix factor of $xyz$
but the opposite matrix factor
$\LocalPsi\left(\mathcal{L}\left(w',\primelambda,\primemu\right)\right)$
cannot be directly obtained by counting polygons
between $L$ and $\mathbb{L}$,
because it involves some moduli spaces
having infinitely many elements. 
To justify it,
we will need to develop additional explanation
but we won't to do in that way here.
We will rather treat them separately in Subsection \ref{sec:(2,2,2)loop}.
},
the corresponding loop with a local system
$
\left(L,E,\nabla\right)
:=
\mathcal{L}\left(w',\primelambda,\primemu\right)
$
is converted under the localized mirror functor
into
the matrix factor
$
\LocalPhi\left(\mathcal{L}\left(w',\primelambda,\primemu\right)\right)
$
given by
\begin{center}
\adjustbox{scale=0.86}{
$\makeatletter\setlength\BA@colsep{2pt}\makeatother
\begin{blockarray}{c @{\hspace{2pt}} c @{\hspace{7pt}} cccccccc}
  &
  & \color{darkgreen} \small\text{$\left(\left.E\right|_{p_1}\right)^*$}
  & \color{darkgreen} \small\text{$\left(\left.E\right|_{q_1}\right)^*$}
  & \color{darkgreen} \small\text{$\left(\left.E\right|_{r_1}\right)^*$}
  & \color{darkgreen} \small\text{$\left(\left.E\right|_{p_2}\right)^*$}
  & \color{darkgreen} \small\text{$\cdots$}
  & \color{darkgreen} \small\text{$\left(\left.E\right|_{q_\tau}\right)^*$}
  & \color{darkgreen} \small\text{$\left(\left.E\right|_{r_\tau}\right)^*$}
  \\[1mm]
  \begin{block}{c @{\hspace{2pt}} c @{\hspace{7pt}} (ccccccc)c}
    & \color{darkgreen} \small\text{$\left(\left.E\right|_{s_1}\right)^*$}
    & z I_{\primemu} & -y^{m_1'-1} I_{\primemu} & 0 & 0 & \cdots & 0 & -(-x)^{-l_1'}\left(J_{\primemu}\left(\primelambda\right)^T\right)^{-1}
    \\[2mm]
    & \color{darkgreen} \small\text{$\left(\left.E\right|_{t_1}\right)^*$}
    & y^{-m_1'} I_{\primemu} & x I_{\primemu} & -z^{n_1'-1} I_{\primemu} & 0 & \cdots & 0 & 0
    \\[2mm]
    & \color{darkgreen} \small\text{$\left(\left.E\right|_{u_1}\right)^*$}
    & 0 & z^{-n_1'} I_{\primemu} & y I_{\primemu} & -(-x)^{l_2'-1} I_{\primemu} & \cdots & 0 & 0
    \\[2mm]
    & \color{darkgreen} \small\text{$\left(\left.E\right|_{s_2}\right)^*$}
    & 0 & 0 & -(-x)^{-l_2'} I_{\primemu} & z I_{\primemu} & \ddots & \vdots & \vdots
    \\[2mm]
    & \color{darkgreen} \small\text{$\vdots$}
    & \vdots & \vdots & \vdots & \ddots & \ddots & -y^{m_\tau'-1} I_{\primemu} & 0
    \\[2mm]
    & \color{darkgreen} \small\text{$\left(\left.E\right|_{t_\tau}\right)^*$}
    & 0 & 0 & 0 & \cdots & y^{-m_\tau'} I_{\primemu} & x I_{\primemu} & -z^{n_\tau'-1} I_{\primemu}
    \\[2mm]
    & \color{darkgreen} \small\text{$\left(\left.E\right|_{u_\tau}\right)^*$}
    & -(-x)^{l_1'-1} J_{\primemu}\left(\primelambda\right)^T & 0 & 0 & \cdots & 0 & z^{-n_\tau'} I_{\primemu} & y I_{\primemu}
    & \small \text{$3\tau\primemu\times3\tau\primemu$}
    \\
  \end{block}
\end{blockarray},
$
}
\end{center}
where $x^a, y^a, z^a$ are regarded as $0$ if $a<0$.

\end{prop}
\begin{proof}
For $\primemu=1$ case,
the computation is essentially the same as what we did for
$
\LocalPhi\left(
\mathcal{L}
\left(\left(3,-2,2\right),\primelambda,1\right)
\right)
$
in
Subsection \ref{sec:LMFIllustration}.
A rigorous proof can be found in \cite{CJKR}.
For $\primemu\ge2$,
we use
Proposition \ref{prop:HigherRankLMF}
to `substitute the matrix $J_{\primemu}\left(\primelambda\right)^T$ for the scalar $\primelambda$',
as we did in Example \ref{ex:HigherRankExample}.
\end{proof}

Inspired by this observation,
we propose the following definition:

\begin{defn}\label{defn:CanonicalFormMF}
For a loop datum $\left(w',\lambda,\primemu\right)$ with $\left(w',\lambda\right)\ne\left((2,2,2)^{\hash \tau},1\right)$,
consider the matrix
\begin{equation*}
\varphi\left(w',\lambda,\primemu\right):=\adjustbox{scale=1}{
$
\makeatletter\setlength\BA@colsep{2pt}\makeatother
\begin{blockarray}{c@ {\hspace{8pt}} cccccccc}
  \begin{block}{c @{\hspace{8pt}} (ccccccc)c}
    & z I_{\primemu} & -y^{m_1'-1} I_{\primemu} & 0 & 0 & \cdots & 0 & -x^{-l_1'}J_{\primemu}\left(\lambda\right)^{-1}
    \\[2mm]
    & -y^{-m_1'} I_{\primemu} & x I_{\primemu} & -z^{n_1'-1} I_{\primemu} & 0 & \cdots & 0 & 0
    \\[2mm]
    & 0 & -z^{-n_1'} I_{\primemu} & y I_{\primemu} & -x^{l_2'-1} I_{\primemu} & \cdots & 0 & 0
    \\[2mm]
    & 0 & 0 & -x^{-l_2'} I_{\primemu} & z I_{\primemu} & \ddots & \vdots & \vdots
    \\[2mm]
    & \vdots & \vdots & \vdots & \ddots & \ddots & -y^{m_\tau'-1} I_{\primemu} & 0
    \\[2mm]
    & 0 & 0 & 0 & \cdots & -y^{-m_\tau'} I_{\primemu} & x I_{\primemu} & -z^{n_\tau'-1} I_{\primemu}
    \\[2mm]
    & -x^{l_1'-1} J_{\primemu}\left(\lambda\right) & 0 & 0 & \cdots & 0 & -z^{-n_\tau'} I_{\primemu} & y I_{\primemu}
    & \small \text{$3\tau\primemu\times3\tau\primemu$}
    \\
  \end{block}
\end{blockarray}
$
}
\end{equation*}
where $x^a, y^a, z^a$ are regarded as $0$ if $a<0$.
Then it is a matrix factor of $xyz$ with the opposite factor denoted by
$\psi\left(w',\lambda,\primemu\right)$
\footnote{
Note that it is a consequence of Theorem \ref{thm:CylinderFreeMFFinite}, Proposition \ref{prop:CylinderFree}
and Theorem \ref{thm:LagMFCorrespondenceNondegenerate}.
An algebraic proof is given in
the proof of Theorem \ref{thm:MFCMCorrespondence}.
Indeed,
the opposite matrix $\psi\left(w',\lambda,\primemu\right)$
can be explicitly written,
see Corollary 9.6 in \cite{CJKR}.
}.
We refer to the pair
$
\left(\varphi\left(w',\lambda,\primemu\right),\psi\left(w',\lambda,\primemu\right)\right)
$
the
\textnormal{\textbf{canonical form of
matrix factorizations}
of $xyz$}
corresponding to the loop datum $\left(w',\lambda,\primemu\right)$.

%

\end{defn}

We can transform the matrix
$
\LocalPhi\left(\mathcal{L}\left(w',\primelambda,\primemu\right)\right)
$
obtained in Proposition \ref{thm:MFFromLag}
into $\varphi\left(w',\lambda,\primemu\right)$
by some bases change and prove the following theorem:

\begin{thm}\label{thm:LagMFCorrespondenceNondegenerate}
For a
loop datum $\left(w',\primelambda,\primemu\right)$
with
$w'\ne(2,2,2)^{\hash \tau}$,
there is an isomorphism
\begin{align*}
\LocalF\left(\mathcal{L}\left(w',\primelambda,\primemu\right)\right)
\ \ &\cong\ \
\left(\varphi\left(w',\lambda,\primemu\right),\psi\left(w',\lambda,\primemu\right)\right)
\end{align*}
in ${\MF}\left(xyz\right)$,
where
$$
\lambda
:=
(-1)^{
l_1+\cdots+l_{\tau}+\tau
}
\primelambda
$$
for
$
l_i
:=
l_i'
+1
-\mathbb{1}_{n_{i-1}'\ge1}
-\mathbb{1}_{l_i'\ge1}
-\mathbb{1}_{m_i'\ge1}
$
($i\in\mathbb{Z}_{\tau}$)
\footnote{
This is a part of the conversion formula
from loop data to band data
(Definition \ref{def:ConversionFromLooptoBand}).
}
.


\end{thm}

\begin{proof}

Note that exactly one of block components
$-y^{m_1'-1}I_{\primemu}$
and
$y^{-m_1'}I_{\primemu}$
survives in
$
\LocalPhi\left(\mathcal{L}\left(w',\primelambda,\primemu\right)\right),
$
for example,
as $y^a$ is regarded as $0$ if $a<0$.
By changing the sign of some basis elements,
we can replace its block components
$$
y^{-m_1'}I_{\primemu},\ \ \
z^{-n_1'}I_{\primemu},\ \ \
-(-x)^{l_2'-1}I_{\primemu},\ \ \
-(-x)^{-l_2'}I_{\primemu},\ \ \
\dots\ \ \,\ \ \
y^{-m_\tau'}I_{\primemu},\ \ \
z^{-n_\tau'}I_{\primemu},\ \ \
-(-x)^{l_1'-1} J_{\primemu}\left(\primelambda\right)^T,\ \ \
-(-x)^{-l_1'}\left(J_{\primemu}\left(\primelambda\right)^T\right)^{-1}
$$
with
$$
-y^{-m_1'}I_{\primemu},\ \ \
-z^{-n_1'}I_{\primemu},\ \ \
-x^{l_2'-1}I_{\primemu},\ \ \
-x^{-l_2'}I_{\primemu},\ \ \
\dots\ \ \,\ \ \
-y^{-m_\tau'}I_{\primemu},\ \ \
-z^{-n_\tau'}I_{\primemu},\ \ \
-x^{l_1'-1} (-1)^\dagger J_{\primemu}\left(\primelambda\right)^T,\ \ \
-x^{-l_1'} (-1)^\dagger \left(J_{\primemu}\left(\primelambda\right)^T\right)^{-1}
$$
in order,
where $(-1)^\dagger$ is the total sign change given by
$$
\begin{matrix}
(-1)^\dagger
:=
(-1)^{
\left(l_1'-1\right) \mathbb{1}_{l_1'\ge1}
-l_1' \mathbb{1}_{l_1'\le0}
+\mathbb{1}_{m_1'\le0}
+\mathbb{1}_{n_1'\le0}
+\left(l_2'-1\right) \mathbb{1}_{l_2'\ge1}
-l_2' \mathbb{1}_{l_2'\le0}
+\mathbb{1}_{m_2'\le0}
+\cdots
+\mathbb{1}_{m_\tau'\le0}+\mathbb{1}_{n_\tau'\le0}
}
=
(-1)^{
l_1+\cdots+l_\tau+\tau
}.
\end{matrix}
$$

The last two can be again replaced by
$
-x^{l_1'-1}J_{\primemu}\left(\lambda\right)
$
and
$
-x^{-l_1'}J_{\primemu}\left(\lambda\right)^{-1}
$
for $\lambda:=(-1)^\dagger\primelambda$,
respectively,
by some bases change using the fact that
the matrix
$
(-1)^\dagger J_{\primemu}\left(\primelambda\right)^T
$
is similar to
$
J_{\primemu}\left(\lambda\right)
$
from the relation
$$
\begin{psmallmatrix}
0 & \cdots & 0 & 0 & 1
\\
0 & \cdots & 0 & (-1)^\dagger & 0
\\
0 & \cdots & 1 & 0 & 0
\\
\rotatebox{90}{$\scriptstyle\cdots$} & \reflectbox{\rotatebox{135}{$\scriptstyle\cdots$}} & \rotatebox{90}{$\scriptstyle\cdots$} & \rotatebox{90}{$\scriptstyle\cdots$} & \rotatebox{90}{$\scriptstyle\cdots$}
\\
\pm1& \cdots & 0 & 0 & 0
\end{psmallmatrix}
(-1)^\dagger
\begin{psmallmatrix}
\primelambda & 0 & 0 & \cdots & 0
\\
1 & \primelambda & 0 & \cdots & 0
\\
0 & 1 & \primelambda & \cdots & 0
\\
\rotatebox{90}{$\scriptstyle\cdots$} & \rotatebox{90}{$\scriptstyle\cdots$} & \reflectbox{\rotatebox{45}{$\scriptstyle\cdots$}} & \reflectbox{\rotatebox{45}{$\scriptstyle\cdots$}} & \rotatebox{90}{$\scriptstyle\cdots$}
\\
0 & 0 & \cdots & 1 & \primelambda
\end{psmallmatrix}
\begin{psmallmatrix}
0 & 0 & 0 & \cdots &
\pm1
\\
\rotatebox{90}{$\scriptstyle\cdots$} & \rotatebox{90}{$\scriptstyle\cdots$} & \rotatebox{90}{$\scriptstyle\cdots$} & \reflectbox{\rotatebox{135}{$\scriptstyle\cdots$}} & \rotatebox{90}{$\scriptstyle\cdots$}
\\
0 & 0 & 1 & \cdots & 0
\\
0 & (-1)^\dagger & 0 & \cdots & 0
\\
1 & 0 & 0 & \cdots & 0
\end{psmallmatrix}
=
\begin{psmallmatrix}
\lambda & 1 & 0 & \cdots & 0
\\
0 & \lambda & 1 & \cdots & 0
\\
0 & 0 & \lambda & \reflectbox{\rotatebox{45}{$\scriptstyle\cdots$}} & \rotatebox{90}{$\scriptstyle\cdots$}
\\
\rotatebox{90}{$\scriptstyle\cdots$} & \rotatebox{90}{$\scriptstyle\cdots$} & \rotatebox{90}{$\scriptstyle\cdots$} & \reflectbox{\rotatebox{45}{$\scriptstyle\cdots$}} & 1
\\
0 & 0 & 0 & \cdots & \lambda
\end{psmallmatrix}.
\vspace{-5mm}
$$
\end{proof}

To
manipulate and analyze
the canonical form of
matrix factorizations,
it is convenient to introduce notations on some special matrices.

\begin{notation}
Denote the $N\times N$ Jordan block of eigenvalue $0$ and its transpose as
$$
J_N
:=
\begin{psmallmatrix}
0 & 1 & 0 & \cdots & 0
\\
0 & 0 & 1 & \cdots & 0
\\
0 & 0 & 0 & \reflectbox{\rotatebox{45}{$\scriptstyle\cdots$}} & \rotatebox{90}{$\scriptstyle\cdots$}
\\
\rotatebox{90}{$\scriptstyle\cdots$} & \rotatebox{90}{$\scriptstyle\cdots$} & \rotatebox{90}{$\scriptstyle\cdots$} & \reflectbox{\rotatebox{45}{$\scriptstyle\cdots$}} & 1
\\
0 & 0 & 0 & \cdots & 0
\end{psmallmatrix}_{N\times N}
\quad\text{and}\quad
K_N
:=
J_N^T
=
\begin{psmallmatrix}
0 & 0 & 0 & \cdots & 0
\\
1 & 0 & 0 & \cdots & 0
\\
0 & 1 & 0 & \cdots & 0
\\
\rotatebox{90}{$\scriptstyle\cdots$} & \rotatebox{90}{$\scriptstyle\cdots$} & \reflectbox{\rotatebox{45}{$\scriptstyle\cdots$}} & \reflectbox{\rotatebox{45}{$\scriptstyle\cdots$}} & \rotatebox{90}{$\scriptstyle\cdots$}
\\
0 & 0 & \cdots & 1 & 0
\end{psmallmatrix}_{N\times N}.
$$
Note that the $n\times n$ Jordan block of eigenvalue $\lambda\in\field$ is
$J_n\left(\lambda\right)=\lambda I_n + J_n$.
We will also frequently use
$$
R_{N}\left(\lambda\right)
:=
J_N + \lambda K_N^{N-1}
=
\begin{psmallmatrix}
0 & 1 & 0 & \cdots & 0
\\
0 & 0 & 1 & \cdots & 0
\\
0 & 0 & 0 & \reflectbox{\rotatebox{45}{$\scriptstyle\cdots$}} & \rotatebox{90}{$\scriptstyle\cdots$}
\\
\rotatebox{90}{$\scriptstyle\cdots$} & \rotatebox{90}{$\scriptstyle\cdots$} & \rotatebox{90}{$\scriptstyle\cdots$} & \reflectbox{\rotatebox{45}{$\scriptstyle\cdots$}} & 1
\\
\lambda & 0 & 0 & \cdots & 0
\end{psmallmatrix}_{N \times N}
\quad\text{and}\quad
R_N^T\left(\lambda\right)
=
K_N + \lambda J_N^{N-1}
=
\begin{psmallmatrix}
0 & 0 & 0 & \cdots & \lambda
\\
1 & 0 & 0 & \cdots & 0
\\
0 & 1 & 0 & \cdots & 0
\\
\rotatebox{90}{$\scriptstyle\cdots$} & \rotatebox{90}{$\scriptstyle\cdots$} & \reflectbox{\rotatebox{45}{$\scriptstyle\cdots$}} & \reflectbox{\rotatebox{45}{$\scriptstyle\cdots$}} & \rotatebox{90}{$\scriptstyle\cdots$}
\\
0 & 0 & \cdots & 1 & 0
\end{psmallmatrix}_{N \times N }
\footnote{
Note that both are just $\left(\lambda\right)_{1\times1}$
if $N=1$.
},
$$
and their enlargement by replacing  $\lambda$ with $J_{\primemu}\left(\lambda\right)^{\pm1}$
$$
R_N\left(J_{\primemu}\left(\lambda\right)\right)
=
J_N\otimes I_{\primemu} + K_N^{N-1}\otimes J_{\primemu}\left(\lambda\right)
=
\begin{psmallmatrix}
0 & I_{\primemu} & 0 & \cdots & 0
\\
0 & 0 & I_{\primemu} & \cdots & 0
\\
0 & 0 & 0 & \reflectbox{\rotatebox{45}{$\scriptstyle\cdots$}} & \rotatebox{90}{$\scriptstyle\cdots$}
\\
\rotatebox{90}{$\scriptstyle\cdots$} & \rotatebox{90}{$\scriptstyle\cdots$} & \rotatebox{90}{$\scriptstyle\cdots$} & \reflectbox{\rotatebox{45}{$\scriptstyle\cdots$}} & I_{\primemu}
\\
J_{\primemu}\left(\lambda\right) & 0 & 0 & \cdots & 0
\end{psmallmatrix}_{N\primemu\times N\primemu},
$$
$$
R_N^T\left(J_{\primemu}\left(\lambda\right)^{-1}\right)
=
K_N\otimes I_{\primemu} + J_N^{N-1}\otimes J_{\primemu}\left(\lambda\right)^{-1}
=
\begin{psmallmatrix}
0 & 0 & 0 & \cdots & J_{\primemu}\left(\lambda\right)^{-1}
\\
I_{\primemu} & 0 & 0 & \cdots & 0
\\
0 & I_{\primemu} & 0 & \cdots & 0
\\
\rotatebox{90}{$\scriptstyle\cdots$} & \rotatebox{90}{$\scriptstyle\cdots$} & \reflectbox{\rotatebox{45}{$\scriptstyle\cdots$}} & \reflectbox{\rotatebox{45}{$\scriptstyle\cdots$}} & \rotatebox{90}{$\scriptstyle\cdots$}
\\
0 & 0 & \cdots & I_{\primemu} & 0
\end{psmallmatrix}_{N\primemu\times N\primemu}.
$$
For square matrices
$A_1,\dots,A_\tau$,
we denote
the block diagonal matrix made from them as
\begin{equation}\label{eqn:MatrixDecomposition}
\displaystyle
\bigoplus_{i=1}^\tau
A_i
:=
\begin{psmallmatrix}
A_1 & 0 & 0 & \cdots & 0
\\
0 & A_2 & 0 & \cdots & 0
\\
0 & 0 & A_3 & \cdots & 0
\\
\rotatebox{90}{$\scriptstyle\cdots$} & \rotatebox{90}{$\scriptstyle\cdots$} & \rotatebox{90}{$\scriptstyle\cdots$} & \reflectbox{\rotatebox{45}{$\scriptstyle\cdots$}} & \rotatebox{90}{$\scriptstyle\cdots$}
\\
0 & 0 & 0 & \cdots & A_\tau
\end{psmallmatrix}.
\end{equation}

\end{notation}

Using these new notations,
we can write the canonical form 
given in Definition \ref{defn:CanonicalFormMF} for $\primemu=1$ as
$$
\varphi\left(w',\lambda,1\right)
=
\varphi\left(w',0,1\right)
-\lambda x^{l_1'-1} K_{3\tau}^{3\tau-1}
-\lambda^{-1} x^{-l_1'} J_{3\tau}^{3\tau-1},
$$
where
$
\varphi\left(w',0,1\right)
$
is obtained by putting $\lambda=\lambda^{-1}=0$
in the expression of $\varphi\left(w',\lambda,1\right)$.
Then the general expression for arbitrary $\primemu$\ is obtained from it by
`substituting the matrix $J_{\primemu}\left(\lambda\right)$ for the scalar $\lambda$',
that is,
\begin{equation}\label{eqn:CanonicalFormKroneckerProduct}
\varphi\left(w',\lambda,\primemu\right)
=
\varphi\left(w',0,1\right)\otimes I_{\primemu}
-x^{l_1'-1} K_{3\tau}^{3\tau-1}\otimes J_{\primemu}\left(\lambda\right)
-x^{-l_1'} J_{3\tau}^{3\tau-1}\otimes J_{\primemu}\left(\lambda\right)^{-1}.
\end{equation}
It
can be
also
viewed as
a consequence of Proposition
\ref{prop:HigherRankLMF}.

By reordering rows and columns of the canonical form $\varphi\left(w',\lambda,\primemu\right)$,
we obtain an \textbf{alternative canonical form}
\begin{equation}\label{eqn:AlternativeCanonicalForm}
\varphi_{\operatorname{alt}}\left(w',\lambda,\primemu\right):=
\begin{pmatrix}
z I_{\tau\primemu}
& -\displaystyle\bigoplus_{i=1}^\tau y^{m_i'-1} I_{\primemu}
& -\left(\displaystyle\bigoplus_{i=1}^\tau x^{-l_i'} I_{\primemu} \right)
  R_N^T\left(J_{\primemu}\left(\lambda\right)^{-1}\right)
\\
- \displaystyle\bigoplus_{i=1}^\tau y^{-m_i'} I_{\primemu}
& x I_{\tau\primemu}
& -\displaystyle\bigoplus_{i=1}^\tau z^{n_i'-1} I_{\primemu}
\\
-R_N\left(J_{\primemu}\left(\lambda\right)\right)

 \left(\displaystyle\bigoplus_{i=1}^\tau x^{l_i'-1} I_{\primemu}\right)
& - \displaystyle\bigoplus_{i=1}^\tau z^{-n_i'} I_{\primemu}
& y I_{\tau\primemu}
\end{pmatrix}_{3\tau\primemu\times3\tau\primemu}.
\end{equation}
Sometimes it is more convenient to work with this alternative version than with the original one.

\subsection{Periodic case}\label{sec:PeriodicCase}
In this subsection,
we show that objects corresponding to the \textbf{periodic} normal loop words are \textbf{decomposable}.
The core of the decomposition lies in the following linear algebra problem:

\begin{lemma}\label{lem:JordanFormOfMatrices}
The Jordan canonical form of
$R_N\left(J_{\primemu}\left(\lambda\right)\right)$
and $R_N^T\left(J_{\primemu}\left(\lambda\right)^{-1}\right)$
are
$
\displaystyle\bigoplus_{k=0}^{N-1}
J_{\primemu}\left(\lambda_{k}\right)
$
and
$
\displaystyle\bigoplus_{k=0}^{N-1}
J_{\primemu}\left(\lambda_{k}^{-1}\right)
\footnote{
If $\primemu=1$,
they are diagonalizable 
via a transition matrix given by the Vandermonde matrix
$
\left(\lambda_{j}^i\right)_{0\le i,j\le N-1}
$.
},
$
respectively,
where
$\lambda_{0},\dots,\lambda_{N-1}\in\field^\times$ are the $N$-th roots of $\lambda\in\field^\times$
(i.e., distinct solutions of $x^N = \lambda$).

\end{lemma}

\begin{proof}
By the following Lemma \ref{lem:KroneckerProductSwitching}, 
the matrix
$
R_N\left(J_{\primemu}\left(\lambda\right)\right)
=
J_N\otimes I_{\primemu} + K_N^{N-1}\otimes J_{\primemu}\left(\lambda\right)
$
is similar to
$$
I_{\primemu}\otimes J_N + J_{\primemu}\left(\lambda\right) \otimes K_N^{N-1}
=
\begin{psmallmatrix}
J_N + \lambda K_N^{N-1} & K_N^{N-1} & 0 & \cdots & 0
\\
0 & J_N + \lambda K_N^{N-1} & K_N^{N-1} & \cdots & 0
\\
0 & 0 & J_N + \lambda K_N^{N-1} & \reflectbox{\rotatebox{45}{$\scriptstyle\cdots$}} & \rotatebox{90}{$\scriptstyle\cdots$}
\\
\rotatebox{90}{$\scriptstyle\cdots$} & \rotatebox{90}{$\scriptstyle\cdots$} & \rotatebox{90}{$\scriptstyle\cdots$} & \reflectbox{\rotatebox{45}{$\scriptstyle\cdots$}} & K_N^{N-1}
\\
0 & 0 & 0 & \cdots & J_N + \lambda K_N^{N-1}
\end{psmallmatrix}_{\primemu N\times \primemu N}
=
\begin{psmallmatrix}
R_{N}\left(\lambda\right) & K_N^{N-1} & 0 & \cdots & 0
\\
0 & R_N\left(\lambda\right) & K_N^{N-1} & \cdots & 0
\\
0 & 0 & R_N\left(\lambda\right) & \reflectbox{\rotatebox{45}{$\scriptstyle\cdots$}} & \rotatebox{90}{$\scriptstyle\cdots$}
\\
\rotatebox{90}{$\scriptstyle\cdots$} & \rotatebox{90}{$\scriptstyle\cdots$} & \rotatebox{90}{$\scriptstyle\cdots$} & \reflectbox{\rotatebox{45}{$\scriptstyle\cdots$}} & K_N^{N-1}
\\
0 & 0 & 0 & \cdots & R_N\left(\lambda\right)
\end{psmallmatrix}_{\primemu N\times \primemu N},
$$
whose characteristic polynomial
(in $t$) is
$$
\left(\det\left(t I_N - R_N\left(\lambda\right)\right)\right)^{\primemu}
=
\left|
\begin{smallmatrix}
t & -1 & 0 & \cdots & 0
\\
0 & t & -1 & \cdots & 0
\\
0 & 0 & t & \reflectbox{\rotatebox{45}{$\scriptstyle\cdots$}} & \rotatebox{90}{$\scriptstyle\cdots$}
\\
\rotatebox{90}{$\scriptstyle\cdots$} & \rotatebox{90}{$\scriptstyle\cdots$} & \rotatebox{90}{$\scriptstyle\cdots$} & \reflectbox{\rotatebox{45}{$\scriptstyle\cdots$}} & -1
\\
-\lambda & 0 & 0 & \cdots & t
\end{smallmatrix}
\right|_{N \times N}^{\primemu}
=
\left(t^N-\lambda\right)^{\primemu}
=
\left(t-\lambda_{0}\right)^{\primemu}\cdots\left(t-\lambda_{N-1}\right)^{\primemu}.
$$
It is straightforward to check that the eigenspace
for each eigenvalue $t=\lambda_{k}$
($k\in\left\{0,\dots,N-1\right\}$)
has dimension only $1$,
which completes the proof for $R_N\left(J_{\primemu}\left(\lambda\right)\right)$.
The proof for $R_N^T\left(J_{\primemu}\left(\lambda\right)^{-1}\right)$ is similar.
\end{proof}

\begin{lemma}\label{lem:KroneckerProductSwitching}
For two matrices
$
A\in\field^{m_1\times n_1}
$
and
$
B\in\field^{m_2\times n_2},
$
their Kronecker products
$A\otimes B \in \field^{m_1 m_2\times n_1 n_2}$ and $B\otimes A \in \field^{m_2 m_1\times n_2 n_1}$
are \emph{permutation equivalent}.
More specifically,
there is a \emph{perfect shuffle matrix}
$S_{p,q}\in O(pq)$
for each $\left(p,q\right)\in\mathbb{Z}_{\ge1}\times\mathbb{Z}_{\ge1}$
whose entries are $0$ or $1$,
such that
$S_{p,q}^T = S_{p,q}^{-1} = S_{q,p}$
and
$$
S_{m_1,m_2}\left(A\otimes B\right)S_{n_2,n_1}
=
B\otimes A
$$
for any
$
A\in\field^{m_1\times n_1}
$
and
$
B\in\field^{m_2\times n_2}.
$

In particular,
if $A$ and $B$ are square matrices,
$A\otimes B$ and $B\otimes A$ are similar.
\end{lemma}
\begin{proof}
One can get $B\otimes A$ from $A\otimes B$
(and vice versa)
just by reordering rows and columns. 
It is straightforward to check the detail
(see \cite{enwiki:1215962176}).
\end{proof}

\begin{thm}\label{thm:PeriodicDecomposable}
Let $\left(w',\lambda,\primemu\right)$ be a loop datum
with
$\left(w',\lambda\right)\ne\left((2,2,2)^{\hash \tau},1\right)$.
If the normal loop word $w'$ is periodic,
i.e.,
$w' = \left(\tilde{w}'\right)^{\hash N}\in \mathbb{Z}^{3\tau}$
for another normal loop word $\tilde{w}'\in\mathbb{Z}^{3\tilde{\tau}}$
($\tau=N\tilde{\tau}$),
there is
an invertible matrix $V\in\operatorname{GL}_{3\tau\primemu}\left(\field\right)$
such that
$$
 \varphi\left(w',\lambda,\primemu\right)
 =
 V^{-1}
 \begin{psmallmatrix}
 \varphi\left(\tilde{w}',\lambda_{0},\primemu\right) & 0 & \cdots & 0
 \\
 0 & \varphi\left(\tilde{w}',\lambda_{1},\primemu\right) & \cdots & 0
 \\
 \scalebox{.75}{\vdots} & \scalebox{.75}{\vdots} & \scalebox{.75}{$\ddots$} &
 \scalebox{.75}{\vdots}
 \\
 0 & 0 & \cdots & \varphi\left(\tilde{w}',\lambda_{N-1},\primemu\right)
 \end{psmallmatrix}
 V,
 $$
where 
$\lambda_{0},\dots,\lambda_{N-1}\in\field^\times$ are the $N$-th roots of $\lambda$.
This yields a decomposition in $\MF(xyz)$
$$
 \varphi\left(w',\lambda,\primemu\right)
 \cong
 \bigoplus_{k=0}^{N-1} \varphi\left(\tilde{w}',\lambda_{k},\primemu\right).
$$

\end{thm}

\begin{proof}
Recall the
formula (\ref{eqn:CanonicalFormKroneckerProduct})
$$
\varphi\left(w',\lambda,\primemu\right)
=
\varphi\left(w',0,1\right)\otimes I_{\primemu}
-x^{l_1'-1} K_{3\tau}^{3\tau-1}\otimes J_{\primemu}\left(\lambda\right)
-x^{-l_1'} J_{3\tau}^{3\tau-1}\otimes J_{\primemu}\left(\lambda\right)^{-1}.
$$
Here we have
$$
\varphi\left(w',0,1\right)
=
I_N\otimes
\varphi\left(\tilde{w}',0,1\right)
-x_1^{l_1'-1} J_N \otimes K_{3\tilde{\tau}}^{3\tilde{\tau}-1}
-x_1^{-l_1'} K_N \otimes J_{3\tilde{\tau}}^{3\tilde{\tau}-1}
$$
for a periodic word
$
w' = \left(\tilde{w}'\right)^{\hash  N}.
$
Denoting by
$
\tilde{\varphi}
:=
\varphi\left(\tilde{w}',0,1\right),
$
we can rewrite
$
\varphi\left(w',\lambda,\primemu\right)
$
as
$$
\left(
I_N\otimes
\tilde{\varphi}
-x_1^{l_1'-1} J_N \otimes K_{3\tilde{\tau}}^{3\tilde{\tau}-1}
-x_1^{-l_1'} K_N \otimes J_{3\tilde{\tau}}^{3\tilde{\tau}-1}
\right)
\otimes I_{\primemu}
-x^{l_1'-1} K_N^{N-1} \otimes K_{3\tilde{\tau}}^{3\tilde{\tau}-1}\otimes J_{\primemu}\left(\lambda\right)
-x^{-l_1'} J_N^{N-1} \otimes J_{3\tilde{\tau}}^{3\tilde{\tau}-1}\otimes J_{\primemu}\left(\lambda\right)^{-1},
$$
where we also used trivial identities
$
K_{3\tau}^{3\tau-1}
=
K_N^{N-1}\otimes K_{3\tilde{\tau}}^{3\tilde{\tau}-1}
$
and
$
J_{3\tau}^{3\tau-1}
=
J_N^{N-1}\otimes
J_{3\tilde{\tau}}^{3\tilde{\tau}-1}.
$


By Lemma \ref{lem:KroneckerProductSwitching},
we can switch the order in the Kronecker product
$
A\otimes B \otimes C
$
above into
$
A\otimes C \otimes B
$
to get a similar matrix.
That is,
$
\varphi\left(w',\lambda,\primemu\right)$
is similar to
\begin{equation}\label{eqn:SimilarToCanonicalForm}
\begin{matrix}
I_N\otimes I_{\primemu}\otimes
\tilde{\varphi}
-x_1^{l_1'-1}\left(J_N\otimes I_{\primemu} + K_N^{N-1}\otimes J_{\primemu}\left(\lambda\right)\right)\otimes K_{3\tilde{\tau}}^{3\tilde{\tau}-1}
-x_1^{-l_1'}\left(K_N\otimes I_{\primemu} + J_N^{N-1}\otimes J_{\primemu}\left(\lambda\right)^{-1}\right)\otimes J_{3\tilde{\tau}}^{3\tilde{\tau}-1}
\\[2mm]
\hspace{-45mm}
=
I_{N\primemu} \otimes
\tilde{\varphi}
- x_1^{l_1'-1} R_N\left(J_{\primemu}\left(\lambda\right)\right)
 \otimes K_{3\tilde{\tau}}^{3\tilde{\tau}-1}
- x_1^{-l_1'} R_N^T\left(J_{\primemu}\left(\lambda\right)^{-1}\right) \otimes J_{3\tilde{\tau}}^{3\tilde{\tau}-1}
\end{matrix}
\end{equation}
via transition matrices given by
$
I_N \otimes S_{3\tilde{\tau},\primemu}
$
and
$
I_N \otimes S_{\primemu,3\tilde{\tau}}.
$
In other words,
(\ref{eqn:SimilarToCanonicalForm}) is another expression for
$$
\left(
I_N \otimes S_{3\tilde{\tau},\primemu}
\right)
\varphi\left(w',\lambda,\primemu\right)
\left(
I_N \otimes S_{\primemu,3\tilde{\tau}}
\right).
$$

Now we may assume $l_1'\ge1$.
(The other case is handled in the same way.)
Then we can drop the third term in (\ref{eqn:SimilarToCanonicalForm}).
Lemma \ref{lem:JordanFormOfMatrices}
implies that there is an invertible matrix
$P\in\operatorname{GL}_{N\primemu}\left(\field\right)$
satisfying
$$
R_N\left(J_{\primemu}\left(\lambda\right)\right)P
=
P
\left(
\bigoplus_{k=0}^{N-1}
J_{\primemu}\left(\lambda_{k}\right)
\right).
$$
Multiplying
$
P\otimes I_{3\tilde{\tau}}
$
to (\ref{eqn:SimilarToCanonicalForm})
on the right yields
\begin{align*}
\left(
I_N \otimes S_{3\tilde{\tau},\primemu}
\right)
\varphi\left(w',\lambda,\primemu\right)
\left(
I_N \otimes S_{\primemu,3\tilde{\tau}}
\right)
\left(P\otimes I_{3\tilde{\tau}}\right)
&=
P\otimes\tilde{\varphi}
-
x_1^{l_1'-1}
\left(R_N\left(J_{\primemu}\left(\lambda\right)\right)P\right)
\otimes
K_{3\tilde{\tau}}^{3\tilde{\tau}-1}
\\[1mm]
&=
P\otimes\tilde{\varphi}
-
x_1^{l_1'-1}
P
\left(
\bigoplus_{k=0}^{N-1}
J_{\primemu}\left(\lambda_{k}\right)
\right)
\otimes
K_{3\tilde{\tau}}^{3\tilde{\tau}-1}
\\[1mm]
&=
\left(P\otimes I_{3\tilde{\tau}}\right)
\left(
I_{N\primemu}\otimes\tilde{\varphi}
-
x_1^{l_1'-1}
\left(
\bigoplus_{k=0}^{N-1}
J_{\primemu}\left(\lambda_{k}\right)
\right)
\otimes
K_{3\tilde{\tau}}^{3\tilde{\tau}-1}
\right)
\\[1mm]
&=
\left(P\otimes I_{3\tilde{\tau}}\right)
\bigoplus_{k=0}^{N-1}
\left(
I_{\primemu}\otimes\tilde{\varphi}
-
x_1^{l_1'-1}
J_{\primemu}\left(\lambda_{k}\right)
\otimes
K_{3\tilde{\tau}}^{3\tilde{\tau}-1}
\right).
\end{align*}
Lemma \ref{lem:KroneckerProductSwitching} says that each direct summand
$
I_{\primemu}\otimes\tilde{\varphi}
-
x_1^{l_1'-1}
J_{\primemu}\left(\lambda_{k}\right)
\otimes
K_{3\tilde{\tau}}^{3\tilde{\tau}-1}
$
is similar to
$$
\varphi\left(\tilde{w}',\lambda_{k},\primemu\right)
=
\tilde{\varphi}\otimes I_{\primemu}
-
x_1^{l_1'-1} K_{3\tilde{\tau}}^{3\tilde{\tau}-1}\otimes
J_{\primemu}\left(\lambda_{k}\right)
$$
via transition matrices
$
S_{\primemu,3\tilde{\tau}}
$
and
$
S_{3\tilde{\tau},\primemu},
$
and hence their direct sum is similar to
$
\bigoplus_{k=0}^{N-1}
\varphi\left(\tilde{w}',\lambda_{k},\primemu\right)
$
via transition matrices
$
\bigoplus_{k=0}^{N-1} S_{\primemu,3\tilde{\tau}}
=
I_N\otimes S_{\primemu,3\tilde{\tau}}
$
and
$
\bigoplus_{k=0}^{N-1} S_{3\tilde{\tau},\primemu}
=
I_N\otimes S_{3\tilde{\tau},\primemu}.
$
Therefore,
we can write
$$
\left(
I_N\otimes S_{\primemu,3\tilde{\tau}}
\right)
\left(P^{-1}\otimes I_{3\tilde{\tau}}\right)
\left(
I_N \otimes S_{3\tilde{\tau},\primemu}
\right)
\varphi\left(w',\lambda,\primemu\right)
\left(
I_N \otimes S_{\primemu,3\tilde{\tau}}
\right)
\left(P\otimes I_{3\tilde{\tau}}\right)
\left(
I_N\otimes S_{3\tilde{\tau},\primemu}
\right)
=
\bigoplus_{k=0}^{N-1}
\varphi\left(\tilde{w}',\lambda_{k},\primemu\right),
$$
which shows the claim with
$
V:=
\left(
I_N \otimes S_{\primemu,3\tilde{\tau}}
\right)
\left(P\otimes I_{3\tilde{\tau}}\right)
\left(
I_N\otimes S_{3\tilde{\tau},\primemu}
\right)
\in\operatorname{GL}_{3\tau\primemu}\left(\field\right).
$
\end{proof}

\begin{cor}\label{cor:PeriodicDecomposableFuk}
Let $\left(w',\primelambda,\primemu\right)$ be a loop datum
with
$w'\ne(2,2,2)^{\hash \tau}$.
If the normal loop word $w'$ is periodic,
i.e.,
$w' = \left(\tilde{w}'\right)^{\hash N}\in \mathbb{Z}^{3\tau}$
for another normal loop word $\tilde{w}'\in\mathbb{Z}^{3\tilde{\tau}}$
($\tau=N\tilde{\tau}$),
there is a decomposition in $\Fuk\left(\POP\right)$
$$
\mathcal{L}\left(w',\primelambda,\primemu\right)
 \cong
 \bigoplus_{k=0}^{N-1} \mathcal{L}\left(\tilde{w}',\primelambda_{k},\primemu\right),
$$
 where 
$\primelambda_{0},\dots,\primelambda_{N-1}\in\field^\times$ are the $N$-th roots of $\primelambda$.
\end{cor}

\begin{remark}
This shows that
\textnormal{\textbf{non-primitive loops
(with a local system)
are decomposable}} in the Fukaya category,
which is not intuitively obvious.
In following up works,
we will see that this is indeed a general and intrinsic feature of the Fukaya category of hyperbolic Riemann surfaces.
\end{remark}

\subsection{Non-cylinder-free case $w'=(2,2,2)^{\hash \tau}$}\label{sec:(2,2,2)loop}
\label{sec:ExceptionalCase}

Note that a loop $L$ in $\POP$ is not cylinder-free with $\mathbb{L}$
only if
it is freely homotopic to $\mathbb{L}^{\hash \tau}$
for some $\tau\in\mathbb{Z}$.
Only in cases of $w'=\left(1,0,1,0,1,0\right)^{\hash \tau}$ or $\left(2,2,2\right)^{\hash \tau}$ for some $\tau\in\mathbb{Z}_{\ge1}$,
the corresponding loop $L\left(w'\right)$ has the same free homotopy type with $\mathbb{L}^{\hash \tau}$ and $\mathbb{L}^{\hash -\tau}$,
respectively.
Then it has a chance to bound an immersed cylinder with $\mathbb{L}$.
In the former case,
it is not the case for our canonical form given in Definition \ref{def:LoopWord}.
In the latter case,
however,
our canonical form bounds an immersed cylinder with $\mathbb{L}$,
which indeed produces
some infinite moduli spaces
involved in the formula \ref{eqn:ComponentOfDeformedDifferential}
for the matrix factor
$\LocalPsi\left(\mathcal{L}\left(w',\primelambda,\primemu\right)\right)$
\footnote{
Still the matrix factorization can be defined over the Novikov field.
Moreover,
unless $\primelambda=-1$,
we can use the formula
$
1-\primelambda+\primelambda^2-\primelambda^3+\cdots = \frac{1}{1+\primelambda}
$
to evaluate $T=1$ in some infinite series to get a matrix factorization over $\field$,
which is isomorphic to the corresponding (original) canonical form.
But instead of justifying that formula,
we will take a detour using the perturbed loop and showing Proposition \ref{prop:DegenerateOriginalCanonicalForm}.
}.

To prevent this,
we take a very specific \textbf{perturbed version of the loop} for the case
$w'=\left(2,2,2\right)^{\hash \tau}$.
We first define the loop $L:=L\left(2,2,2\right)$ as shown in Figure \ref{fig:Perturbed222}.
Then let $L\left(\left(2,2,2\right)^{\hash \tau}\right)$ be the $\tau$-concatenation
(with {\small$\color{red}\bigstar$} as a starting point)
of it.
For a loop datum
$\left((2,2,2)^{\hash \tau},\primelambda,\primemu\right)$,
we define the loop with a local system
$\mathcal{L}\left((2,2,2)^{\hash \tau},\primelambda,\primemu\right)$
as in Definition \ref{defn:LoopData}
so that it has
holonomy $J_{\primemu}\left(\primelambda\right)$
at {\small$\color{red}\bigstar$},
using the modified underlying loop
$L\left(\left(2,2,2\right)^{\hash \tau}\right)$.

\begin{figure}[h]
\setlength\arraycolsep{0pt}
\centering
$

}
\right)
}_{\LocalPsi(\mathcal{L})}
}
\\[-2mm]
\end{matrix}
\\[14mm]
\end{matrix}
\end{matrix}
$
\captionsetup{width=1\linewidth}
\caption{
Loop with a local system
$
\mathcal{L}:=\mathcal{L}\left((2,2,2),\primelambda,\primemu\right)
$
and the corresponding matrix factorization}
\label{fig:Perturbed222}
\end{figure}


\begin{defn}\label{def:DegenerateCanonicalFormMF}
The \textnormal{\textbf{degenerate canonical form of
matrix factorizations}}
of $xyz$ corresponding to a loop datum
$
\left((2,2,2)^{\hash \tau},\lambda,\primemu\right)
$
is defined as
$$
\begin{matrix}
\hspace{-80mm}
\left(
{\varphi_{\deg}}\left(\left(2,2,2\right)^{\hash \tau}, \lambda, \primemu\right),
{\psi_{\deg}} \left(\left(2,2,2\right)^{\hash \tau}, \lambda, \primemu\right)
\right)
\\[3mm]
:=
\left(
\arraycolsep=0pt\def\arraystretch{1}
\begin{psmallmatrix}
-zx I_{\tau\primemu} & 0 & 0 & 0
\\[2mm]
z I_{{\tau\primemu}} & -y I_{{\tau\primemu}} & 0 & 0
\\[2mm]
0 & x I_{{\tau\primemu}} & -z I_{{\tau\primemu}} & 0
\\[2mm]
- x R_{\tau,\primemu}\left(\lambda\right) & 0 & y I_{{\tau\primemu}} & -xy I_{\tau\primemu}
\end{psmallmatrix}
_{4\tau\primemu \times 4\tau\primemu}
,
\begin{psmallmatrix}
-y I_{\tau\primemu} & 0 & 0 & 0
\\[2mm]
-z I_{\tau\primemu} & -zx I_{{\tau\primemu}} & 0 & 0
\\[2mm]
-x I_{\tau\primemu} & -x^2 I_{{\tau\primemu}} & -xy I_{{\tau\primemu}} & 0
\\[2mm]
- I_{\tau\primemu} + R_{\tau,\primemu}\left(\lambda\right) & -x I_{\tau\primemu} & -y I_{\tau\primemu} & -z I_{\tau\primemu}
\end{psmallmatrix}
_{4\tau\primemu \times 4\tau\primemu}
\right).
\end{matrix}
$$
\end{defn}


\begin{thm}\label{thm:LagMFCorrespondenceDegenerate}
For a
loop datum $\left((2,2,2)^{\hash \tau},\primelambda,\primemu\right)$,
there is an isomorphism
\begin{align*}
\LocalF\left(\mathcal{L}\left((2,2,2)^{\hash \tau},\primelambda,\primemu\right)\right)
\ \ &\cong\ \
\left(
{\varphi_{\deg}}\left(\left(2,2,2\right)^{\hash \tau}, \lambda, \primemu\right),
{\psi_{\deg}} \left(\left(2,2,2\right)^{\hash \tau}, \lambda, \primemu\right)
\right)
\end{align*}
in ${\MF}\left(xyz\right)$,
where
$
\lambda
:=
(-1)^{\tau}
\primelambda.
$


\end{thm}

\begin{proof}
For $\tau=1$ case,
the matrix factorization
$
\LocalF\left(\mathcal{L}\left((2,2,2),\primelambda,\primemu\right)\right)
$
is computed in Figure \ref{fig:Perturbed222},
using the same method as in Proposition \ref{thm:MFFromLag}.
It can be transformed to
$
\left(
{\varphi_{\deg}}\left(\left(2,2,2\right)^{\hash \tau}, \lambda, \primemu\right),
{\psi_{\deg}} \left(\left(2,2,2\right)^{\hash \tau}, \lambda, \primemu\right)
\right)
$
by changing the sign of some basis elements.
The case of $\tau\ge2$ is handled similarly.
\end{proof}

For $\lambda\ne1$,
we can reduce the degenerate canonical form into the original version using the following `\textbf{matrix reduction}':

\begin{lemma}\cite{CJKR}\label{lem:MatrixReducingProcess}
Let $S$ be
the power series ring
$
\field[[x_1,\dots,x_n]]
$
of $n$ variables
and $f\in S$ its nonzero element.
Assume that the pair
$
\left(
\begin{psmallmatrix} C & D \\[1mm] E^T & u \end{psmallmatrix},
\begin{psmallmatrix} F & G \\[1mm] H^T & v \end{psmallmatrix}
\right)
$
is a matrix factorization of $f$ in $S$,
for some matrices $C$, $F\in S^{k\times k}$,
$D$, $E$, $G$, $H\in S^{k\times1}$
and $u$, $v\in S$.
If $u$ or $v$ is a unit in $S$,
there is an isomorphism
$$
\left(
\begin{pmatrix} C & D \\[1mm] E^T & u \end{pmatrix},
\begin{pmatrix} F & G \\[1mm] H^T & v \end{pmatrix}
\right)
\ \cong \
\left(C-Du^{-1} E^T,F\right)
\oplus
\left(1,f\right)
\quad
\text{or}
\quad
\left(C,F-Gv^{-1}H^T\right)
\oplus
\left(f,1\right),
$$
respectively,
in $\MF(f)$.
Therefore,
the pair is isomorphic to the reduced pair
$
\left(C-Du^{-1} E^T,F\right)
$
or
$
\left(C,F-Gv^{-1}H^T\right),
$
respectively,
in the homotopy category $\underline{\MF}(f)$.
%
%
%
%
%

\end{lemma}

\begin{proof}
Assume that $u$ is a unit in $S$,
and
consider the following diagram:
$$
\newcommand{\scriptverteq}{\mathrel{\rotatebox{90}{$\scriptstyle=$}}}
\tikzset{
    labl/.style={anchor=south, rotate=90, inner sep=.4mm}
}
\begin{tikzcd}[arrow style=tikz,>=stealth,row sep=4em,column sep=6em] 
S^k\oplus S
  \arrow[r, "\spmat{ C & D \\[1mm] E^T & u }"]
  \arrow[d, "\cong" labl, swap, "\spmat{ I_k & 0 \\[1mm] u^{-1}E^T & 1 }\hspace{3mm}"]
&
S^k\oplus S
  \arrow[r, "\spmat{F & G \\[1mm] H^T & v}"]
  \arrow[d, "\cong" labl, "\spmat{ I_k & -Du^{-1} \\[1mm] 0 & u^{-1} }"]
&
S^k\oplus S
  \arrow[d, "\cong" labl, "\spmat{ I_k & 0 \\[1mm] u^{-1}E^T & 1 }"]
\\
S^k\oplus S
  \arrow[r, "\spmat{ C-Du^{-1}E^T & 0 \\[1mm] 0 & 1 }"]
&
S^k\oplus S
  \arrow[r, "\hspace{7mm}\spmat{ F & 0 \\[1mm] 0 & f }"]
&
S^k\oplus S
\end{tikzcd}
$$
The commutativity of the diagram is immediate from some matrix calculations using the fact that the original pair
is a matrix factorization of $f$. Also, note that the vertical maps
are all isomorphisms.
This yields an isomorphism
$
\left(
\begin{psmallmatrix} C & D \\[1mm] E^T & u \end{psmallmatrix},
\begin{psmallmatrix} F & G \\[1mm] H^T & v \end{psmallmatrix}
\right)
\cong
\left(C-Du^{-1} E^T,F\right)
\oplus
\left(1,f\right)
$
in $\MF(f)$.

One can construct an explicit homotopy between $\left(1,f\right)$ or $\left(f,1\right)$
and the zero object
$
\begin{tikzcd}[arrow style=tikz,>=stealth, sep=20pt, every arrow/.append style={shift left=0.5}]
   0
     \arrow{r}{}
   &
   0
     \arrow{l}{}
\end{tikzcd}
$
in $\MF(f)$,
which shows that it is a zero object in the homotopy category $\underline{\MF}(f)$.
(It is also a consequence of Proposition \ref{prop:StableHomotopySame}.)
This means that we can drop the direct summand $\left(1,f\right)$ or $\left(f,1\right)$
in $\underline{\MF}(f)$.
\end{proof}

\begin{prop}\label{prop:DegenerateOriginalCanonicalForm}
If $\lambda\ne1$,
the degenerate canonical form
$
\left(
{\varphi_{\deg}}\left(\left(2,2,2\right)^{\hash \tau}, \lambda, \primemu\right),
{\psi_{\deg}} \left(\left(2,2,2\right)^{\hash \tau}, \lambda, \primemu\right)
\right)
$
($\tau,\primemu\in\mathbb{Z}_{\ge1}$)
is isomorphic
to the original version
$
\left(
{\varphi}\left(\left(2,2,2\right)^{\hash \tau}, \lambda, \primemu\right),
{\psi} \left(\left(2,2,2\right)^{\hash \tau}, \lambda, \primemu\right)
\right)
$
in $\underline{\MF}(xyz)$.

\end{prop}

\begin{proof}
In the simplest case $\tau=\primemu=1$,
the reducing process in Lemma \ref{lem:MatrixReducingProcess} along a unit $-1+\lambda$ in $\psi_{\deg}\left((2,2,2),\lambda,1\right)$
yields a commutative diagram,
which proves the claim:
$$
\newcommand{\scriptverteq}{\mathrel{\rotatebox{90}{$\scriptstyle=$}}}
\tikzset{
    labl/.style={anchor=south, rotate=90, inner sep=.5mm}
}
\begin{tikzcd}[arrow style=tikz,>=stealth,row sep=5em,column sep=7em] 
S^3\oplus S
  \arrow[r, "\smat{\varphi_{\deg}\left((2,2,2),\lambda,1\right)}"]
  \arrow[d, "\cong" labl, swap, "\smat{\begin{psmallmatrix} I_3 & -\left(1-\lambda\right)^{-1} \begin{psmallmatrix}y \\ z \\ x \end{psmallmatrix} \\[1mm] 0 & -\left(1-\lambda\right)^{-1} \end{psmallmatrix} \\[3mm] }\hspace{3mm}"]
&
S\oplus S^3
  \arrow[r, "\smat{\psi_{\deg}\left((2,2,2),\lambda,1\right)}"]
  \arrow[d, "\cong" labl, "\smat{\begin{psmallmatrix} 0 & I_3 \\[1mm] 1 & \left(1-\lambda\right)^{-1} \begin{psmallmatrix} x & y & z \end{psmallmatrix} \end{psmallmatrix}\\[6mm]}"]
&
S^3\oplus S
  \arrow[d, "\cong" labl, "\smat{\begin{psmallmatrix} I_3 & -\left(1-\lambda\right)^{-1} \begin{psmallmatrix}y \\ z \\ x \end{psmallmatrix} \\[1mm] 0 & -\left(1-\lambda\right)^{-1} \end{psmallmatrix} \\[3mm] }\hspace{3mm}"]
\\
S^3\oplus S
  \arrow[r, "\spmat{ \varphi\left((2,2,2),\lambda,1\right) & 0 \\[1mm] 0 & xyz }"]
&
S^3\oplus S
  \arrow[r, "\spmat{ \psi\left((2,2,2),\lambda,1\right) & 0 \\[1mm] 0 & 1 }"]
&
S^3\oplus S
\end{tikzcd}
$$

%
%
%
%
%
%
%
%
%
%

In general case,
we `substitute the matrix $R_{\tau,\primemu}\left(\lambda\right)$ for the scalar $\lambda$'
to get a commutative diagram:
$$
\newcommand{\scriptverteq}{\mathrel{\rotatebox{90}{$\scriptstyle=$}}}
\tikzset{
    labl/.style={anchor=south, rotate=90, inner sep=.5mm}
}
\begin{tikzcd}[arrow style=tikz,>=stealth,row sep=7em,column sep=11em] 
\scriptstyle S^{3\tau\primemu}\oplus S^{\tau\primemu}
  \arrow[r, "\smat{\varphi_{\deg}\left((2,2,2)^{\hash \tau},\lambda,\primemu\right)}"]
  \arrow[d, "\cong" labl, swap, "\smat{\begin{psmallmatrix} I_{3\tau\primemu} & -\begin{psmallmatrix}y I_{\tau\primemu} \\ z I_{\tau\primemu} \\ x I_{\tau\primemu} \end{psmallmatrix} \left(I_{\tau\primemu}-R_{\tau,\primemu}\left(\lambda\right)\right)^{-1} \\[1mm] 0 & -\left(I_{\tau\primemu}-R_{\tau,\primemu}\left(\lambda\right)\right)^{-1}  \end{psmallmatrix} \\[3mm] }\hspace{3mm}"]
&
\scriptstyle S^{\tau\primemu}\oplus S^{3\tau\primemu}
  \arrow[r, "\smat{\psi_{\deg}\left((2,2,2)^{\hash \tau},\lambda,\primemu\right)}"]
  \arrow[d, "\cong" labl, "\smat{\begin{psmallmatrix} 0 & I_{3\tau\primemu} \\[1mm] I_{\tau\primemu} & \left(I_{\tau\primemu}-R_{\tau,\primemu}\left(\lambda\right)\right)^{-1} \begin{psmallmatrix} x I_{\tau\primemu} & y I_{\tau\primemu} & z I_{\tau\primemu} \end{psmallmatrix} \end{psmallmatrix}\\[6mm]}"]
&
\scriptstyle S^{3\tau\primemu}\oplus S^{\tau\primemu}
  \arrow[d, "\cong" labl, "\smat{\begin{psmallmatrix} I_{3\tau\primemu} & -\begin{psmallmatrix}y I_{\tau\primemu} \\ z I_{\tau\primemu} \\ x I_{\tau\primemu} \end{psmallmatrix} \left(I_{\tau\primemu}-R_{\tau,\primemu}\left(\lambda\right)\right)^{-1} \\[1mm] 0 & -\left(I_{\tau\primemu}-R_{\tau,\primemu}\left(\lambda\right)\right)^{-1}  \end{psmallmatrix} \\[3mm] }\hspace{3mm}"]
\\
\scriptstyle S^{3\tau\primemu}\oplus S^{\tau\primemu}
  \arrow[r, "\spmat{ \varphi_{\operatorname{alt}}\left((2,2,2)^{\hash \tau},\lambda,\primemu\right) & 0 \\[1mm] 0 & xyz I_{\tau\primemu} }"]
&
\scriptstyle S^{3\tau\primemu}\oplus S^{\tau\primemu}
  \arrow[r, "\spmat{ \psi_{\operatorname{alt}}\left((2,2,2)^{\hash \tau},\lambda,\primemu\right) & 0 \\[1mm] 0 & I_{\tau\primemu} }"]
&
\scriptstyle S^{3\tau\primemu}\oplus S^{\tau\primemu}
\end{tikzcd}
$$
Therefore,
the degenerate canonical form
is isomorphic to the alternative canonical form
(\ref{eqn:AlternativeCanonicalForm})
in $\underline{\MF}(xyz)$,
which is again isomorphic to the original canonical form in
$\MF(xyz)$.
\end{proof}

This shows that even if the normal loop word is
$w'=(2,2,2)^{\hash \tau}$
(and hence we use the perturbed loop for $\mathcal{L}\left((2,2,2)^{\hash \tau},\primelambda,\primemu\right)$),
we can still define and use the original canonical form
of matrix factorizations
$
\left(
{\varphi}\left(\left(2,2,2\right)^{\hash \tau}, \lambda, \primemu\right),
{\psi} \left(\left(2,2,2\right)^{\hash \tau}, \lambda, \primemu\right)
\right)
$
unless $\lambda=1$
(or equivalently $\primelambda=(-1)^{\tau}$).
Therefore, we would not call them all degenerate cases.

\begin{defn}\label{defn:DegenerateLoopData}
A normal loop word is called
\textnormal{\textbf{non-cylinder-free}}
if it is $w'=(2,2,2)^{\hash \tau}$
for some $\tau\in\mathbb{Z}_{\ge1}$,
and a loop datum is called
\textnormal{\textbf{degenerate}}
if it is

(1)
$
\left(w'=(2,2,2)^{\hash \tau},\primelambda=(-1)^{\tau},\primemu\right)
$
and parameterizes
loops with a local system
$
\mathcal{L}\left((2,2,2)^{\hash \tau},(-1)^{\tau},\primemu\right)
$,

(2)
$
\left(w'=(2,2,2)^{\hash \tau},\lambda=1,\primemu\right)
$
and parameterizes
matrix factorizations
$
\left(
{\varphi_{\deg}}\left(\left(2,2,2\right)^{\hash \tau}, 1, \primemu\right),
{\psi_{\deg}} \left(\left(2,2,2\right)^{\hash \tau}, 1, \primemu\right)
\right)
$.
\end{defn}

Now
Theorem \ref{thm:LagMFCorrespondenceNondegenerate},
Theorem \ref{thm:LagMFCorrespondenceDegenerate},
and Proposition \ref{prop:DegenerateOriginalCanonicalForm}
complete the proof of
Theorem \ref{thm:LagMFCorrespondence} in the introduction.

For $\tau\ge2$,
the non-cylinder-free loop word $(2,2,2)^{\hash \tau}$ is \textbf{periodic},
and
the corresponding objects are \textbf{decomposable}
as in the general case
(Theorem \ref{thm:PeriodicDecomposable}):

\begin{prop}\label{prop:PeriodicDecomposableDegenerate}
For a loop datum
$\left((2,2,2)^{\hash \tau},\lambda,\primemu\right)$,
there is a decomposition
$$
{\varphi_{\deg}}\left((2,2,2)^{\hash \tau}, \lambda, \primemu\right)
\cong
\bigoplus_{k=0}^{\tau-1} \varphi_{\deg}\left((2,2,2),\lambda_{\tau,k},\primemu\right)
$$
in $\MF(xyz)$,
where
$\lambda_{\tau,0},\dots,\lambda_{\tau,\tau-1}\in\field$ are the $\tau$-th roots of $\lambda$.

\end{prop}

\begin{proof}
It is proven in the same way as Theorem \ref{thm:PeriodicDecomposable}.
\end{proof}

Combining Proposition \ref{prop:PeriodicDecomposableDegenerate} and Proposition \ref{prop:DegenerateOriginalCanonicalForm},
we get a decomposition of the matrix factorization for periodic degenerate cases
($w'=(2,2,2)^{\hash \tau},\lambda=1$)
as
\begin{align}\label{eqn:PeriodicDegenerateCaseDecompositionMF}
\begin{split}
{\varphi_{\deg}}\left((2,2,2)^{\hash \tau}, 1, \primemu\right)
&\cong
\bigoplus_{k=0}^{\tau-1} \varphi_{\deg}\left((2,2,2), e^{2\pi i \cdot \frac{k}{\tau}},\primemu\right)
\quad\text{in }\MF(xyz)
\\
&\cong
{\varphi_{\deg}}\left((2,2,2),1,\primemu\right)
\oplus
\bigoplus_{k=1}^{\tau-1} \varphi \left((2,2,2), e^{2\pi i \cdot \frac{k}{\tau}}, \primemu \right)
\quad\text{in }\underline{\MF}(xyz).
\end{split}
\end{align}

This also implies the decomposition of the loop with a local system for periodic degenerate cases
($w'=(2,2,2)^{\hash \tau},\primelambda=(-1)^\tau$)
as
\begin{equation}\label{eqn:PeriodicDegenerateCaseDecompositionFuk}
\mathcal{L}\left((2,2,2)^{\hash \tau},(-1)^{\tau},\primemu\right)
\cong
\bigoplus_{k=0}^{\tau-1}
\mathcal{L}\left((2,2,2),-e^{2\pi i \cdot \frac{k}{\tau}},\primemu\right)
\quad\text{in }\Fuk\left(\POP\right),
\end{equation}
extending the result of Corollary \ref{cor:PeriodicDecomposableFuk}
to any periodic normal loop words.

\section{Matrix Factorizations from Maximal Cohen-Macaulay Modules}

In this section,
we first recall
some general concepts on \emph{maximal Cohen-Macaulay modules}
which we will need in the rest of the section,
including
the
\emph{Macaulayfication} and its computation
(\S \ref{sec:MCMMacaulayficationMF}).
Then
we recall Burban-Drozd's classification \cite{BD17} of \emph{band-type} indecomposable maximal Cohen-Macaulay modules
over
$A:=\left.\field[[x,y,z]]\right/(xyz)$.
Their canonical form is provided as a submodule of a free module $A^{\tau\mu}$,
whose generators are given in terms of \emph{band data} (\S \ref{sec:BandData}).
Those band data are in one-to-one correspondence with loop data by the \emph{conversion formula} given in \cite{CJKR}
(\S \ref{sec:ConversionFormula}).
Using this correspondence,
we prove the main theorem in this section
that the canonical form of
matrix factorizations corresponding to loop data is related to
the canonical form of band-type maximal Cohen-Macaulay modules corresponding to band data
under Eisenbud's equivalence
(\S \ref{sec:MFfromCM}).
The \emph{degenerate case} will be treated separately (\S \ref{sec:DegenerateCase}).

\subsection{Maximal Cohen-Macaulay modules and Macaulayfication}
\label{sec:MCMMacaulayficationMF}

Let $\left(A,\mathfrak{m}\right)$ be a Noetherian local ring,
$\mathbb{k}:=\left.A\right/\mathfrak{m}$
its residue field,
and
$
d:=\operatorname{kr.dim}(A)
$
its Krull dimension.

\begin{defn}\label{defn:MCMModules}
A
Noetherian
$A$-module $M$\ is called \textnormal{\textbf{maximal Cohen-Macaulay}}
if
$$
\operatorname{Ext}_A^i\left(\mathbb{k},M\right)=0
\quad\text{for }i\in\left\{0,\dots,d-1\right\}.
$$
We denote by $\CM(A)$ the \textnormal{\textbf{category of maximal Cohen-Macaulay modules over $A$}},
which is a full subcategory of the category $A-\operatorname{mod}$ of
Noetherian (i.e. finitely generated)
$A$-modules.
\end{defn}

We refer readers, for example, to \cite{Yo} for general properties
and representations of maximal Cohen-Macaulay modules,
and \cite{HB93} for (not necessarily maximal) Cohen-Macaulay modules.
However,
note that many authors,
including \cite{Yo},
refer to maximal Cohen-Macaulay modules simply as \emph{Cohen-Macaulay modules}.

\subsubsection{Macaulayfication}

Now let us focus on our specific case
where
$A$ is given by
$
\left.\field[[x,y,z]]\right/(xyz).
$
It is an example of \emph{surface singularities}
(i.e. $d=2$),
and
many technics and representation-theoretic aspects
for them were developed and studied in \cite{BD08,BD17}.
We will
especially
use the following \emph{Macaulayfying} process,
which
naturally
associates any
Noetherian
$A$-module a maximal Cohen-Macaulay $A$-module:

\begin{defn}\cite{BD08}\label{defn:Macaulayfication}
The \textnormal{\textbf{Macaulayfication}} of a
Noetherian
$A$-module $\tilde{M}$
is defined by
$$
\tilde{M}^{\dagger}
:=
\Hom_A\left(\Hom_A\left(\tilde{M},A\right),A\right)
\footnote{
We are using the fact that $A$ is \emph{Gorenstein in codimension one}
and therefore its \emph{canonical module} is isomorphic to $A$.
See \cite{BD08} for general definition when $A$ does not have a such property.
}.
$$
%
It is indeed a maximal Cohen-Macaulay $A$-module
\footnote{
Here the fact that the Krull dimension of $A$ is $2$ is essential.
See Lemma 3.1 in \cite{BD08}.
}.
This defines a functor
$
\dagger:A-\operatorname{mod}\rightarrow\CM(A),
$
which is called the
\emph{Macaulayfication functor}.
\end{defn}

We will use
a combinatorial tool to compute the Macaulayfication of a given $A$-module in practice:

\begin{defn}
Let $\tilde{M}$ be an $A$-submodule of a free module $A^r$.
If there is an element $F\in A^r\setminus\tilde{M}$
such that $xF$, $yF$, $zF\in\tilde{M}$,
we call it a \textnormal{\textbf{Macaulayfying element}} of $\tilde{M}$ in $A^r$.
\end{defn}

\begin{prop}\cite{BD17}
For an $A$-submodule $\tilde{M}$ of a free module $A^r$,
the following hold:

\noindent
(1) $\tilde{M}$ is maximal Cohen-Macaulay if and only if there is no Macaulayfying element of $\tilde{M}$ in $A^r$.
We have $\tilde{M}^\dagger = \tilde{M}$ in this case.

\noindent
(2) $\tilde{M}^\dagger \cong \left<\tilde{M},F\right>_A^\dagger$
holds for any Macaulayfying element $F$ of $\tilde{M}$ in $A^r$.
\end{prop}

\begin{proof}
See Proposition 4.2 in \cite{CJKR} for the proof of (1),
and Lemma 1.5 in \cite{BD17} for (2).
\end{proof}

\subsection{Band data and canonical form of maximal Cohen-Macaulay modules}\label{sec:BandData}

We
recall the concept of the \emph{band data} from \cite{BD17},
which parameterize \emph{band-type} indecomposable maximal Cohen Macaulay modules over
$A=\left.\field[[x,y,z]]\right/(xyz)$.
Here we use a slightly modified version of band words
given in \cite{CJKR}
for our specific singularity $A$,
in order
to match them with the loop data
in the next subsection.


\begin{defn}\label{defn:mbd}
A \textnormal{\textbf{band datum}} $(w,\lambda,\mu)$ consists of the following:
\begin{itemize}
\item (band word) $w = \left(l_1,m_1,n_1,l_2,m_2,n_2,\dots,l_\tau,m_\tau,n_\tau\right)\in \mathbb{Z}^{3\tau}$ for some $\tau \in \mathbb{Z}_{\ge1}$,
\item (eigenvalue) $\lambda\in \field^\times$,
\item ((algebraic) multiplicity) $\mu \in \mathbb{Z}_{\ge1}$.
\end{itemize}
It defines an $A$-module, denoted by
$$
\tilde{M}\left(w,\lambda,\mu\right),
$$
as
an $A$-submodule of $A^{\tau\mu}$ generated by all columns of the $6$ matrices
\begin{center}
$x^2y^2I_{\tau\mu}$,\quad $y^2z^2I_{\tau\mu}$,\quad $z^2x^2I_{\tau\mu}$,\quad
$
\pi_x\left(w,\lambda,\mu\right):=
$
\adjustbox{scale=0.898}{
$
\begin{pmatrix}
x^{l_1^- +2}yI_\mu & zx^{l_2^+ +2}I_\mu & \cdots & 0 \\
0 & x^{l_2^- +2}yI_\mu & \ddots & \vdots \\
\vdots & \vdots & \ddots & zx^{l_\tau^+ +2}I_\mu \\
zx^{l_1^+ +2}J_\mu(\lambda) & 0 & \cdots & x^{l_\tau^- +2}yI_\mu
\end{pmatrix}_{\tau\mu \times \tau\mu},
$
}

$
\pi_y\left(w,\lambda,\mu\right):=
$
\adjustbox{scale=0.898}{
$
\begin{pmatrix}
\left(xy^{m_1^+ +2} + y^{m_1^- +2}z\right)I_\mu & 0 & \cdots & 0 \\
0 & \left(xy^{m_2^+ +2} + y^{m_2^- +2}z\right)I_\mu & \cdots & 0 \\
\vdots & \vdots & \ddots & \vdots \\
0 & 0 & \cdots & \left(xy^{m_\tau^+ +2} + y^{m_\tau^- +2}z\right)I_\mu
\end{pmatrix}_{\tau\mu \times \tau\mu},
$
}

$
\pi_z\left(w,\lambda,\mu\right):=
$
\adjustbox{scale=0.899}{
$
\begin{pmatrix}
\left(yz^{n_1^+ +2}+z^{n_1^- +2}x\right)I_\mu & 0 & \cdots & 0 \\
0 & \left(yz^{n_2^+ +2} + z^{n_2^- +2}x\right)I_\mu & \cdots & 0 \\
\vdots & \vdots & \ddots & \vdots \\
0 & 0 & \cdots & \left(yz^{n_\tau^+ +2} + z^{n_\tau^- +2}x\right)I_\mu
\end{pmatrix}_{\tau\mu \times \tau\mu},
$
}
\end{center}
where
$a^+:=\max\left\{0,a\right\}$
and
$a^-:=\max\left\{0,-a\right\}$
for $a\in\mathbb{Z}$.
\end{defn}

The notion of \emph{shift}, \emph{subword}, \emph{concatenation} and \emph{periodicity} of a band word
can be defined similarly following those of a loop word.
Two band words $w$ and $\tilde{w}$ are considered \emph{equivalent}
if they coincide up to shifting,
that is,
$\tilde{w}=w^{(k)}$ for some $k\in\mathbb{Z}$.

In general,
however,
the $A$-module
$\tilde{M}\left(w,\lambda,\mu\right)$
fails to be maximal Cohen-Macaulay.
So we need to \emph{Macaulayfy} it to get an object in $\CM(A)$ from a band datum.

\begin{defn}\label{defn:CanonicalFormMCM}
Given a band datum $\left(w,\lambda,\mu\right)$,
we define the corresponding maximal Cohen-Macaulay $A$-module
$$
M\left(w,\lambda,\mu\right)
:=
\tilde{M}(w,\lambda,\mu)^\dagger
$$
as the Macaulayfication
of
$\tilde{M}(w,\lambda,\mu)$.
We refer to it as the
\textnormal{\textbf{canonical form of band-type
maximal Cohen-Macaulay modules}}
over $A$
corresponding to the band datum $\left(w,\lambda,\mu\right)$.
\end{defn}

The maximal Cohen-Macaulay module $M\left(w,\lambda,\mu\right)$
constructed from any band datum is
\textbf{locally free on the punctured spectrum of $A$},
that is,
for any $\mathfrak{p}\in\Spec(A)\setminus \left\{\mathfrak{m}\right\}$
the localization $M_{\mathfrak{p}}$ is a free $A_{\mathfrak{p}}$-module,
where $\mathfrak{m}$ is the maximal ideal of $A$.
The converse also holds,
which is the following classification theorem:

\begin{thm}\cite{BD17}\label{thm:BDClassification}
Any indecomposable maximal Cohen-Macaulay module over $A$
that is locally free on the punctured spectrum 
is isomorphic to the canonical form $M\left(w,\lambda,\mu\right)$ for some unique non-periodic band datum $\left(w,\lambda,\mu\right)$ up to shifting of the band word $w$.

\end{thm}

The proof of the theorem actually follows from a highly non-trivial representation-theoretic study of the category $\CM(A)$,
which is applicable to much broader class (i.e. \emph{degenerate cusp}) of \emph{non-isolated} surface singularities.
To be more specific,
they created its equivalent category $\Tri(A)$,
called the \emph{category of triples}
(or \emph{Burban-Drozd's triple category} in some literature),
and a natural equivalence functor
$
\mathbb{F}_{\BD}:\CM(A)\xrightarrow{\simeq} \Tri(A)
\footnote{
See also \cite{BZ20} for another elaboration on this equivalence and its applications.
}.
$
Its object consists of two modules and a `linearized morphism' between them,
and the indecomposable objects can be classified via a \emph{matrix problem} on the linear map.
As a result,
they fall into the following two types: \emph{band-type} and \emph{string-type}.
The canonical form $\Theta\left(w,\lambda,\mu\right)$ of band-type indecomposable objects in $\Tri(A)$ is given in Figure \ref{fig:CanonicalForm2},
which are parameterized by the band data $\left(w,\lambda,\mu\right)$.
Then they determine the corresponding objects $M\left(w,\lambda,\mu\right)$ in $\CM(A)$
under the equivalence,
which we take as presented in Definition \ref{defn:CanonicalFormMCM}.
\vspace{-2mm}
\begin{figure}[H]
\centering
$

$
\vspace{-2mm}
\captionsetup{width=1\linewidth}
\caption{Canonical form of band-type indecomposable objects in $\Tri\left(A\right)$}
\label{fig:CanonicalForm2}
\end{figure}

\vspace{-3mm}

\begin{remark}
One can also consider a direct functor
$
\underline{\MF}(xyz)\rightarrow \underline{\Tri}(A)
$
by composing the Eisenbud's cokernel functor and Burban-Drozd's functor $\mathbb{F}_{\BD}$.
Then we have an alternative way to compute objects in $\underline{\Tri}(A)$
corresponding to the canonical form of objects in $\underline{\MF}(xyz)$.
This gives another proof for Theorem
\ref{thm:MFCMCorrespondence}
without going through the Macaulayfication process.
However, since it involves additional algebro-geometric consideration for the category $\Tri(A)$
(\cite[Chapter 5]{Rho23}),
we defer it to a separate future work.
\end{remark}

\begin{remark}\label{rmk:StringCorrespondence}
Indecomposable maximal Cohen-Macaulay modules over $A$
that are \emph{not} locally free on the punctured spectrum 
were also classified in \cite{BD17}
by using \emph{string data} instead of band data.
On the mirror side,
they correspond to arcs starting and ending at boundaries of $\POP$.
Most statements and proofs in the present paper
apply to these cases without significant modifications,
while one should take care on some technical details when
proving homotopy invariance
of matrix factorizations
to establish a one-to-one correspondence of them with open geodesics.
There we have to consider \emph{wrapped morphisms} between arcs (non-compact Lagrangian submanifolds),
and a discussion on the wrapped Fukaya category of immersed non-exact Lagrangians
should precede this.
Here we present only some correspondence between basic objects as in Figure \ref{fig:HMSGeneratingObject},
which also correspond to $z\cdot xy$, $x\cdot yz$ and $y\cdot zx$ in $\underline{\MF}(xyz)$,
respectively.
\end{remark}
\vspace{-3mm}
\begin{figure}[H]
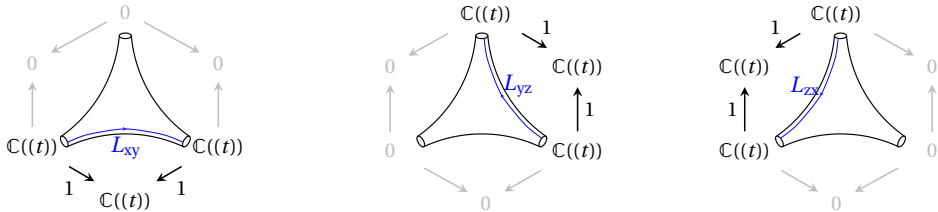

\setlength\arraycolsep{15pt}
$

$
\centering
\captionsetup{width=1\linewidth}
\caption{Mirror symmetry of generating objects}
\label{fig:HMSGeneratingObject}
\end{figure}

\newpage

\subsection{Conversion formula between loop/band data}\label{sec:ConversionFormula}

We define the \emph{conversion formula},
which underlies the correspondence between the canonical form of matrix factorizations of $xyz$
(in loop data)
and
the canonical form of band-type indecomposable maximal Cohen-Macaulay modules over $A$
(in band data).
It was first introduced in \cite{CJKR},
and now we reformulate and extend it for higher rank/multiplicity.

%


\begin{defn}[Conversion from loop data to band data]\label{def:ConversionFromLooptoBand}
A loop datum $(w',\primelambda,\primemu)$ with a normal loop word
$
w' = (w_1',w_2',w_3',w_4',w_5',w_6',\dots,w_{3\tau-2}',w_{3\tau-1}',w_{3\tau}')\in\mathbb{Z}^{3\tau}
$
is converted to a band datum $\left(w,\lambda,\mu\right)$ as follows:
\begin{itemize}
\item
The normal loop word $w'$ is converted to the band word $w\in\mathbb{Z}^{3\tau}$,
defined as
$$
w_j
:=
w_j'+1
-\mathbb{1}_{w_{j-1}'\ge1}
-\mathbb{1}_{w_{j}'\ge1}
-\mathbb{1}_{w_{j+1}'\ge1}
\quad\left(j\in\mathbb{Z}_{3\tau}\right).
$$
\item
The holonomy parameter $\primelambda\in\field^\times$ is converted to
the eigenvalue
$$
\lambda:= (-1)^{l_1 + \cdots + l_\tau + \tau}\primelambda
\quad\text{where $l_i := w_{3i-2}$ for $i\in\mathbb{Z}_\tau$}.
$$
\item
The geometric rank $\primemu$ is converted to the algebraic multiplicity
\footnote{
We do not define it for periodic degenerate cases ($w'=(2,2,2)^{\hash \tau}, \primelambda=(-1)^\tau$
with $\tau\ge2$),
because the corresponding loop with a local system and matrix factorization are
decomposed into $\tau$ pieces
((\ref{eqn:PeriodicDegenerateCaseDecompositionFuk})
and
(\ref{eqn:PeriodicDegenerateCaseDecompositionMF}))
and they are mapped to maximal Cohen-Macaulay modules
having different algebraic multiplicities
(\ref{eqn:PeriodicDegenerateCaseMappedToCM}).
}
$$
\mu
:=
\begin{cases}
\primemu
&
\text{in non-degenerate cases},
\\
\primemu+1
&
\text{in the non-periodic degenerate case } \left(w'=(2,2,2), \primelambda=-1\right).
\end{cases}
$$
\end{itemize}

\end{defn}

We call a band datum \textbf{degenerate} if it is
$\left(w=(0,0,0)^{\hash \tau},\lambda=1,\mu\right)$
for some
$\tau,\mu\in\mathbb{Z}_{\ge1}$.
Note that
it corresponds to a degenerate loop datum
$\left(w'=(2,2,2)^{\hash \tau},\primelambda=(-1)^{\tau},\primemu\right)$
under above the conversion formula above
(while ignoring the relation between $\mu$ and $\primemu$
which is not defined if $\tau\ge2$).


\begin{defn}[Conversion from band data to loop data]\label{def:ConversionFromBandtoLoop}
A band datum $(w,\lambda,\mu)$ with a band word
$
w = \left(w_1,w_2,w_3,w_4,w_5,w_6,\dots,w_{3\tau-2},w_{3\tau-1},w_{3\tau}\right)\in\mathbb{Z}^{3\tau}
$
is converted to a loop datum $\left(w',\primelambda,\primemu\right)$ as follows:
\begin{itemize}
\item
The band word $w$ is converted to the normal loop word $w'\in\mathbb{Z}^{3\tau}$,
defined as
$$
w_j'
:=
w_j-1
+\delta_{j-1}
+\delta_{j}
+\delta_{j+1}
\quad\left(j\in\mathbb{Z}_{3\tau}\right)
$$
where 
$\delta = \delta(w)\in\left\{0,1\right\}^{3\tau}$
is given by
$$
\hspace{20mm}
\delta_j :=
\left\{
\begin{array}{cl}

\setlength\arraycolsep{0pt}
\begin{array}{l}
0 \\ \\ \\
\end{array} 

&

\setlength\arraycolsep{0pt}
\begin{array}{l}
\text{if} \\ \\ \\
\end{array} 

\quad \left\{
                         \begin{array}{l}
                         w_j<0, \text{ or} \\
                         w_j=0 \text{ and at least one of the first non-zero entries adjacent to the} \\
                         \hspace{11mm}\text{string of }$0$\text{s containing $w_j$ (exists and) is negative,}
                         \end{array}
                         \right. \\[6mm]
1 & \text{otherwise}
\end{array}
\right.
\quad\left(j\in\mathbb{Z}_{3\tau}\right).
$$
\item
The eigenvalue $\lambda\in\field^\times$ is converted to
the holonomy parameter
$$
\primelambda:= (-1)^{l_1 + \cdots + l_\tau + \tau}\lambda
\quad\text{where $l_i := w_{3i-2}$ for $i\in\mathbb{Z}_\tau$}.
$$
\item
The algebraic multiplicity $\mu$ is converted to the geometric rank
\footnote{
We do not define it for periodic degenerate cases ($w=(0,0,0)^{\hash \tau}, \lambda=1$ with $\tau\ge2$)
in the same reason as above.
}
$$
\primemu
:=
\begin{cases}
\mu
&
\text{in non-degenerate cases},
\\
\mu-1
&
\text{in the non-periodic degenerate case } \left(w=(0,0,0), \lambda=1\right).
\end{cases}
$$
\end{itemize}
\end{defn}

\begin{prop}\cite{CJKR}\label{prop:ConversionFormulaInverseToEachOther}
(1) The loop word $w'$ converted from a band word $w$ is indeed normal.

\noindent
(2) 
Two conversion formula above
are inverse to each other.
\end{prop}

\vspace{30mm}

\begin{ex}\label{ex:ConversionExample}

The following shows the conversion between a normal loop word $w'$ and a band word $w$:

\newcolumntype{x}[1]{>{\centering\hspace{0pt}}p{#1}} 

\newcommand{\corcmidrule}[1][0.6pt]{\\[\dimexpr-\normalbaselineskip-\belowrulesep-\aboverulesep-#1\relax]} 

\begin{center}
\setlength{\tabcolsep}{0.9pt}
\begin{tabular}{cccx{4.5mm}cx{4.5mm}cx{4.5mm}cx{4.5mm}cx{4.5mm}cx{4.5mm}cx{4.5mm}cx{4.5mm}cx{4.5mm}cx{4.5mm}cx{4.5mm}cx{4.5mm}cx{4.5mm}cx{4.5mm}cx{4.5mm}c}
$w'$ & $=$ & $($ & $8$ & $,$ & $2$ & $,$ & $3$ & $,$ & $-1$ & $,$ & $-1$ & $,$ & $-4$ & $,$ & $-1$ & $,$ & $0$ & $,$ & $5$ & $,$ & $0$ & $,$ & $-2$ & $,$ & $1$ & $,$ & $0$ & $,$ & $4$ & $,$ & $6$ & $)$ \\
\arrayrulecolor{blue}\cmidrule[0.6pt]{4-8}
\corcmidrule\arrayrulecolor{red}\cmidrule[0.6pt]{10-18}
\corcmidrule\arrayrulecolor{blue}\cmidrule[0.6pt]{20-20}
\corcmidrule\arrayrulecolor{red}\cmidrule[0.6pt]{22-24}
\corcmidrule\arrayrulecolor{blue}\cmidrule[0.6pt]{26-26}
\corcmidrule\arrayrulecolor{red}\cmidrule[0.6pt]{28-28}
\corcmidrule\arrayrulecolor{blue}\cmidrule[0.6pt]{30-32} \\[-3mm]

$\mathbb{1}_{w'\ge1}=\delta(w)$ & $=$ & $($ & $1$ & $,$ & $1$ & $,$ & $1$ & $,$ & $0$ & $,$ & $0$ & $,$ & $0$ & $,$ & $0$ & $,$ & $0$ & $,$ & $1$ & $,$ & $0$ & $,$ & $0$ & $,$ & $1$ & $,$ & $0$ & $,$ & $1$ & $,$ & $1$ & $)$ \\[3.8mm]


$w$ & $=$ & $($ & $6$ & $,$ & $0$ & $,$ & $2$ & $,$ & $-1$ & $,$ & $0$ & $,$ & $-3$ & $,$ & $0$ & $,$ & $0$ & $,$ & $5$ & $,$ & $0$ & $,$ & $-2$ & $,$ & $1$ & $,$ & $-1$ & $,$ & $3$ & $,$ & $4$ & $)$ \\
\arrayrulecolor{blue}\cmidrule[0.6pt]{4-8}
\corcmidrule\arrayrulecolor{red}\cmidrule[0.6pt]{10-18}
\corcmidrule\arrayrulecolor{blue}\cmidrule[0.6pt]{20-20}
\corcmidrule\arrayrulecolor{red}\cmidrule[0.6pt]{22-24}
\corcmidrule\arrayrulecolor{blue}\cmidrule[0.6pt]{26-26}
\corcmidrule\arrayrulecolor{red}\cmidrule[0.6pt]{28-28}
\corcmidrule\arrayrulecolor{blue}\cmidrule[0.6pt]{30-32}

\end{tabular}
\end{center}
%
The
holonomy parameter $\primelambda$
and
eigenvalue $\lambda$
in this case are related by
$
\lambda
=
(-1)^{6-1+0+0-1+5}\primelambda=
-\primelambda
$.

\end{ex}

\subsection{Matrix Factorizations from canonical form of maximal Cohen-Macaulay modules}\label{sec:MFfromCM}

Let $S:=\field[[x,y,z]]$
be a power series ring
and $A:=\left.\field[[x,y,z]]\right/(xyz)$
a hypersurface singularity.
For a non-degenerate band datum $\left(w,\lambda,\mu\right)$,
let $w'$ be the converted normal loop word from the band word $w$
and $\primemu:=\mu$.
They define objects in the canonical form as follows:
\begin{itemize}
\item
the matrix factorization
$\left(\varphi\left(w',\lambda,\primemu\right),\psi\left(w',\lambda,\primemu\right)\right)$
corresponding to the loop datum
$\left(w',\lambda,\primemu\right)$,
\item
the maximal Cohen-Macaulay module $M\left(w,\lambda,\mu\right)$
corresponding to the band datum $\left(w,\lambda,\mu\right)$.
\end{itemize}
Now we state our main theorem in this section that they are related
under Eisenbud's equivalence:

\begin{thm}\label{thm:MFCMCorrespondence}
Let $\left(w,\lambda,\mu\right)$ be a non-degenerate band datum,
$w'$ the converted normal loop word from the band word $w$
and $\primemu:=\mu$.
Then
there is an isomorphism in $\underline{\CM}(A)$
$$
\cok\underline{\varphi}\left(w',\lambda,\primemu\right)
\cong
M\left(w,\lambda,\mu\right),
$$
or equivalently,
regarding $M\left(w,\lambda,\mu\right)$
as an $S$-module
$M_S\left(w,\lambda,\mu\right)$,
it has a free resolution of the form
$$
\begin{tikzcd}[arrow style=tikz,>=stealth,row sep=5em,column sep=2em] 
0
  \arrow[r]
&
S^n
  \arrow[rr,"{\varphi\left(w',\lambda,\primemu\right)}"]
&&
S^n
  \arrow[r,"\pi"]
&
M_S\left(w,\lambda,\mu\right)
  \arrow[r,""]
&
0.
\end{tikzcd}
$$

\end{thm}


To prove the theorem, we need to compute an explicit
Macaulayfication of the module in Definition \ref{defn:mbd} and find its resolution.
The case of $\mu=1$ was carried out in \cite{CJKR}. We have developed our setup so that we can extend the construction of $\mu=1$ to the general case
by
`\textbf{substituting the matrix $J_{\mu}\left(\lambda\right)$ for the scalar $\lambda$}'
in all matrices.
(Recall that we did the same operation in the geometric side (Proposition \ref{prop:SubstitutionGeometry}).)
We explain the procedure in detail below:

\subsubsection{Preservation of Exactness}\label{sec:PreservationOfExactness}

Let $R$ be a $\field$-algebra and consider a one-parameter family of matrices in the form
$$
\varphi\left(\lambda\right)
=
\sum_{k=-N}^{N}\varphi_k\lambda^k
\in
R^{m\times n}
$$
for some matrices $\varphi_k\in R^{m\times n}$.
For any $\field$-valued square matrix
$\Lambda\in\field^{\mu\times\mu}$,
we associate a new matrix by
$$
\varphi\left(\Lambda\right)
:=
\sum_{k=-N}^{N}
\varphi_k\otimes \Lambda^k
\in
R^{m\mu\times n\mu}.
$$
We call the pair
$\left(\varphi\left(\lambda\right),\varphi\left(\Lambda\right)\right)$
a \textbf{$\left(\lambda,\Lambda\right)$-substitution pair}.
We have seen many examples:

\begin{itemize}
\item
$
\left(
\varphi\left(w',\lambda,1\right),
\varphi\left(w',\lambda,\primemu\right)
\right)
$
(Definition \ref{defn:CanonicalFormMF})
forms a $\left(\lambda,J_{\primemu}\left(\lambda\right)\right)$-substitution pair.
We denote them as
$$
\varphi\left(w',\lambda\right)
:=
\varphi\left(w',\lambda,1\right)
\quad\text{and}\quad
\varphi\left(w',J_{\primemu}\left(\lambda\right)\right)
:=
\varphi\left(w',\lambda,\primemu\right).
$$
\item
$
\left(
\pi_{\chi}\left(w,\lambda,1\right),
\pi_{\chi}\left(w,\lambda,\mu\right)
\right)
$
(Definition \ref{defn:mbd})
forms a $\left(\lambda,J_{\mu}\left(\lambda\right)\right)$-substitution pair
for each $\chi\in\left\{x,y,z\right\}$,
denoted as
$$
\pi_{\chi}\left(w,\lambda\right)
:=
\pi_{\chi}\left(w,\lambda,1\right)
\quad\text{and}\quad
\pi\left(w,J_{\mu}\left(\lambda\right)\right)
:=
\pi\left(w,\lambda,\mu\right).
$$
\end{itemize}

Such a pair enjoys a nice homological property,
namely,
$\left(\lambda,\Lambda\right)$-substitution
preserves the exactness of a sequence
as in the following proposition.
We thank Kyoungmo Kim
for providing us the idea of the proof.

%

\newcommand{\ringS}{R}
\begin{prop}
\label{Lem:Substitution}

Let
$$
\varphi\left(\lambda\right) = \sum_{k=-N}^{N} \varphi_k \lambda^k \in \ringS^{m \times n}
\quad\text{and}\quad
\psi\left(\lambda\right) = \sum_{k=-N}^{N} \psi_k \lambda^k \in \ringS^{l \times m}
$$
be one-parameter families of matrices for some 
$\varphi_k\in R^{m\times n}$,
$\psi_k\in R^{l\times n}$,
which form a sequence
\begin{equation}\label{eqn:ExactSequence1}
\newcommand{\scriptverteq}{\mathrel{\rotatebox{90}{$\scriptstyle=$}}}
\tikzset{
    labl/.style={anchor=south, rotate=90, inner sep=.4mm}
}
\begin{tikzcd}[arrow style=tikz,>=stealth,row sep=3em,column sep=3em] 
\ringS^n
  \arrow[r,"\varphi\left(\lambda\right)"]
&
\ringS^m
  \arrow[r,"\psi\left(\lambda\right)"]
&
\ringS^l
\end{tikzcd}
\end{equation}
for each $\lambda\in\field$.
Given a square matrix $\Lambda\in\field^{\mu\times\mu}$,
we have
new matrices
$$
\varphi\left(\Lambda\right) = \sum_{k=-N}^{N} \varphi_k \otimes \Lambda^k \in \ringS^{m\mu \times n\mu}
\quad\text{and}\quad
\psi\left(\Lambda\right) = \sum_{k=-N}^{N} \psi_k \otimes \Lambda^k \in \ringS^{l\mu \times m\mu}
$$
so that
$\left(\varphi\left(\lambda\right),\varphi\left(\Lambda\right)\right)$
and
$\left(\psi\left(\lambda\right),\psi\left(\Lambda\right)\right)$
form $\left(\lambda,\Lambda\right)$-substitution pairs,
which also make a sequence
\begin{equation}\label{eqn:ExactSequence2}
\newcommand{\scriptverteq}{\mathrel{\rotatebox{90}{$\scriptstyle=$}}}
\tikzset{
    labl/.style={anchor=south, rotate=90, inner sep=.4mm}
}
\begin{tikzcd}[arrow style=tikz,>=stealth,row sep=3em,column sep=3em] 
\ringS^n\otimes \ringS^{\mu}
  \arrow[r,"\varphi\left(\Lambda\right)"]
&
\ringS^m\otimes \ringS^{\mu}
  \arrow[r,"\psi\left(\Lambda\right)"]
&
\ringS^l\otimes \ringS^{\mu}.
\end{tikzcd}
\end{equation}

Now if the sequence (\ref{eqn:ExactSequence1}) is exact,
more precisely,
if we assume
\begin{enumerate}[label=(\roman*)]
\item
$
\psi\left(\lambda\right)
\varphi\left(\lambda\right)
=
0
$
for any $\lambda\in\field$,
\item
$
\im\varphi\left(\lambda\right)
=
\ker\psi\left(\lambda\right)
$
for any eigenvalue $\lambda$ of $\Lambda$,
\end{enumerate}
then the sequence (\ref{eqn:ExactSequence2}) is also exact,
i.e.,
$\im\varphi \left(\Lambda\right) = \ker \psi \left(\Lambda\right)$.
\end{prop}

\begin{proof}
The assumption (i) is equivalent to each coefficient of $\lambda^k$
in the expansion of $\psi\left(\lambda\right)\varphi\left(\lambda\right)$ being zero
\footnote{
This is true for any infinite (e.g. algebraically closed) field, including our field $\field$.
},
that is,
$\psi\left(\lambda\right)\varphi\left(\lambda\right)$
is zero as a Laurent polynomial in $\lambda$.
This implies $\psi\left(\Lambda\right)\varphi\left(\Lambda\right)=0$.

The converse is not immediate.
First we show that we can replace $\Lambda$ with any similar matrix $J$.
Namely,
let $J\in\field^{\mu\times\mu}$
be a matrix satisfying
$\Lambda=P^{-1}JP$
for some invertible matrix
$P\in\operatorname{GL}_{\mu}\left(\field\right)$.
Then the computation
\begin{align*}
\varphi\left(\Lambda\right)
&=
\sum_{k=-N}^{N}
\varphi_k\otimes\Lambda^k
=
\sum_{k=-N}^{N}
\varphi_k\otimes
\left(
P^{-1}J^k P
\right)
=
\sum_{k=-N}^{N}
\left(I_{m}\otimes P^{-1}\right)
\left(\varphi_k\otimes J^k\right)
\left(I_{n}\otimes P\right)
\\
&=
\left(I_{m}\otimes P^{-1}\right)
\left(
\sum_{k=-N}^{N}
\varphi_k\otimes J^k\right)
\left(I_{n}\otimes P\right)
=
\left(I_{m}\otimes P^{-1}\right)
\varphi\left(J\right)
\left(I_{n}\otimes P\right)
\end{align*}
and
the same computation for $\psi\left(\Lambda\right)$ show
the commutativity of the following diagram:
$$
\newcommand{\scriptverteq}{\mathrel{\rotatebox{90}{$\scriptstyle=$}}}
\tikzset{
    labl/.style={anchor=south, rotate=90, inner sep=.4mm}
}
\begin{tikzcd}[arrow style=tikz,>=stealth,row sep=3em,column sep=3em] 
\ringS^n\otimes \ringS^{\mu}
  \arrow[r,"\varphi\left(\Lambda\right)"]
  \arrow[d,"\cong" labl,swap,"I_n\otimes P\hspace{2mm}"]
&
\ringS^m\otimes \ringS^{\mu}
  \arrow[r,"\psi\left(\Lambda\right)"]
  \arrow[d,"\cong" labl,swap,"I_m\otimes P\hspace{2mm}"]
&
\ringS^l\otimes \ringS^{\mu}
  \arrow[d,"\cong" labl,swap,"I_l\otimes P\hspace{2mm}"]
\\
\ringS^n\otimes \ringS^{\mu}
  \arrow[r,"\varphi\left(J\right)"]
&
\ringS^m\otimes \ringS^{\mu}
  \arrow[r,"\psi\left(J\right)"]
&
\ringS^l\otimes \ringS^{\mu}
\end{tikzcd}
$$
Since the vertical maps are all isomorphisms,
the statement is equivalent to
$\im\varphi \left(J\right) = \ker \psi \left(J\right)$.
This enables us to replace $\Lambda$ with its Jordan canonical form.

Furthermore,
if $J$ is decomposed into block diagonals
$J_1,\dots,J_k$
(in the sense of (\ref{eqn:MatrixDecomposition})),
an analogous
argument 
shows that
the statement holds for $J$ if and only if
it holds for each block $J_i$.
Therefore,
it is enough to prove the statement
for a Jordan block $J_\mu\left(\lambda\right)$,
where $\lambda$ is an eigenvalue of $\Lambda$.

We proceed by an induction on $\mu$.
For $\mu=1$,
the statement is the same as the assumption (ii).
Now let $\mu\ge2$ and assume that the statement is true for
$\Lambda=J_{\mu-1}\left(\lambda\right)$.
We deform $\varphi\left(J_{\mu}\left(\lambda\right)\right)$
into its similar matrix
$
\begin{psmallmatrix}
\varphi\left(J_{\mu-1}\left(\lambda\right)\right) & * \\ 0 & \varphi\left(\lambda\right)
\end{psmallmatrix}
$
by the following computation:
\vspace{0mm}
{\allowdisplaybreaks
\begin{align*}
\varphi\left(J_{\mu}\left(\lambda\right)\right)
&=
\sum_{k=-N}^{N}
\varphi_k
\otimes
J_{\mu}\left(\lambda\right)^k
=
\sum_{k=-N}^{N}
\varphi_k
\otimes
\begin{psmallmatrix}
J_{\mu-1}\left(\lambda\right) & * \\ 0 & \lambda
\end{psmallmatrix}^k
=
\sum_{k=-N}^{N}
\varphi_k
\otimes
\begin{psmallmatrix}
J_{\mu-1}\left(\lambda\right)^k & * \\ 0 & \lambda^k
\end{psmallmatrix}
\\
&=
\sum_{k=-N}^{N}
S_{\mu,m}
\left(
\begin{psmallmatrix}
J_{\mu-1}\left(\lambda\right)^k & * \\ 0 & \lambda^k
\end{psmallmatrix}
\otimes
\varphi_k
\right)
S_{n,\mu}
=
S_{\mu,m}
\left(
\sum_{k=-N}^{N}
\begin{psmallmatrix}
J_{\mu-1}\left(\lambda\right)^k\otimes\varphi_k & * \\ 0 & \lambda^k\varphi_k
\end{psmallmatrix}
\right)
S_{n,\mu}
\\
&=
S_{\mu,m}
\left(
\sum_{k=-N}^{N}
\begin{psmallmatrix}
S_{m,\mu-1}\left(\varphi_k\otimes J_{\mu-1}\left(\lambda\right)^k\right)S_{\mu-1,n} & * \\ 0 & \lambda^k\varphi_k
\end{psmallmatrix}
\right)
S_{n,\mu}
\\
&=
S_{\mu,m}
\left(S_{m,\mu-1}\oplus1\right)
\begin{psmallmatrix}
\varphi\left(J_{\mu-1}\left(\lambda\right)\right) & * \\ 0 & \varphi\left(\lambda\right)
\end{psmallmatrix}
\left(S_{\mu-1,n}\oplus1\right)
S_{n,\mu}.
\end{align*}
}
The same computation also works for
$\psi\left(J_{\mu}\left(\lambda\right)\right)$,
which yields the following commutative diagram:
$$
\newcommand{\scriptverteq}{\mathrel{\rotatebox{90}{$\scriptstyle=$}}}
\tikzset{
    labl/.style={anchor=south, rotate=90, inner sep=.4mm}
}
\begin{tikzcd}[arrow style=tikz,>=stealth,row sep=4em,column sep=6em] 
\scriptstyle
\ringS^n\otimes \left(\ringS^{\mu-1}\oplus \ringS\right)
  \arrow[r,"\varphi\left(J_{\mu}\left(\lambda\right)\right)"]
  \arrow[d,"\cong" labl,swap,"\smat{\left(S_{\mu-1,n}\oplus1\right)S_{n,\mu}\\[5mm]}\hspace{2mm}"]
&
\scriptstyle
\ringS^m\otimes \left(\ringS^{\mu-1}\oplus \ringS\right)
  \arrow[r,"\psi\left(J_{\mu}\left(\lambda\right)\right)"]
  \arrow[d,"\cong" labl,swap,"\smat{\left(S_{\mu-1,m}\oplus1\right)S_{m,\mu}\\[5mm]}\hspace{2mm}"]
&
\scriptstyle
\ringS^l\otimes \left(\ringS^{\mu-1}\oplus \ringS\right)
  \arrow[d,"\cong" labl,swap,"\smat{\left(S_{\mu-1,l}\oplus1\right)S_{l,\mu}\\[5mm]}\hspace{2mm}"]
\\
\scriptstyle
\left(\ringS^n\otimes \ringS^{\mu-1}\right)
\oplus
\left(\ringS^n\otimes \ringS\right)
  \arrow[r,"\spmat{\varphi\left(J_{\mu-1}\left(\lambda\right)\right) & * \\ 0 & \varphi\left(\lambda\right)}"]
&
\scriptstyle
\left(\ringS^m\otimes \ringS^{\mu-1}\right)
\oplus
\left(\ringS^m\otimes \ringS\right)
  \arrow[r,"\spmat{\psi\left(J_{\mu-1}\left(\lambda\right)\right) & * \\ 0 & \psi\left(\lambda\right)}"]
&
\scriptstyle
\left(\ringS^l\otimes \ringS^{\mu-1}\right)
\oplus
\left(\ringS^l\otimes \ringS\right)
\end{tikzcd}
$$
Note that
the vertical maps are natural isomorphisms.
As we know
$\psi\left(J_{\mu}\left(\lambda\right)\right)\varphi\left(J_{\mu}\left(\lambda\right)\right)=0$
from our first discussion,
the composition in the bottom row
also vanishes.
The exactness of the bottom row easily follows from
the induction hypothesis and the assumption (ii),
also implying that the top row is exact.
\end{proof}


\subsubsection{Proof of the theorem}

Now recall that we have
$
\tilde{M}\left(w,\lambda,\mu\right) = \im\tilde{\pi}\left(w,\lambda,\mu\right) \subset A^{\tau\mu},
$
where
$$
\tilde{\pi}\left(w,\lambda,\mu\right) :=
\left(
\arraycolsep=2pt\def\arraystretch{1}
\begin{array}{c|c|c|c|c|c}
x^2 y^2 I_{\tau\mu} & y^2 z^2 I_{\tau \mu} & z^2 x^2 I_{\tau \mu} & \pi_x\left(w,J_\mu\left(\lambda\right)\right) & \pi_y\left(w,J_\mu\left(\lambda\right)\right) & \pi_z\left(w,J_\mu\left(\lambda\right)\right)
\end{array}
\right)_{\tau\mu\times 6\tau\mu}
$$
is an $A$-valued matrix, or an $A$-module map $A^{6\tau\mu}\rightarrow A^{\tau\mu}$.
Because it does not directly become maximal Cohen-Macaulay,
we will find its Macaulayfying elements in $A^{\tau\mu}$ to Macaulayfy it.
Here we briefly recall our previous discussion in \cite{CJKR} for $\mu=1$ case:

For a fixed band word $w$,
consider
a one-parameter family of elements
in $A^{\tau}$
of the form
\begin{equation}\label{eqn:OneParameterMacaulayfyingElements}
F\left(\lambda\right)
=
F_{-}\lambda^{-1} + F_0 + F_+ \lambda
\quad\text{where }
F_{-}, F_0, F_+\in A^{\tau}
\text{ and }
\lambda\in\field^\times
\end{equation}
(see (9.8) in \cite{CJKR})
that satisfies
\begin{equation}\label{eqn:MacaulayfyingEquation}
\chi F\left(\lambda\right)
=
\tilde{\pi}\left(w,\lambda,1\right)
a_{\chi}\left(\lambda\right)
\quad\text{for each }
\chi\in\left\{x,y,z\right\}
\text{ and }
\lambda\in\field^\times
\end{equation}
for some
$
a_{\chi}\left(\lambda\right)
=
a_{\chi,-}\lambda^{-1}
+
a_{\chi,0}
+
a_{\chi,1}\lambda
$
where
$
a_{\chi,-},
a_{\chi,0},
a_{\chi,+}
\in
A^{6\tau\mu}$
(see (9.10) in \cite{CJKR}).
Then we have
$
\chi F\left(\lambda\right)
\in
\tilde{M}\left(w,\lambda,1\right)
$
for each $\chi\in\left\{x,y,z\right\}$
and $\lambda\in\field^\times$,
which means that $F\left(\lambda\right)$ is
a Macaulayfying element of
$\tilde{M}\left(w,\lambda,1\right)$
in
$A^{\tau}$
for each
$\lambda\in\field^\times$.


\begin{thm}[Theorem 9.1 in \cite{CJKR}]\label{thm:MFCMCorrespondenceMultOne}
Let $\left(w,\lambda,1\right)$ be a non-degenerate band datum
and $w'$ the converted normal loop word from the band word $w$.
Then
there are
Macaulayfying elements
$
F_1\left(\lambda\right),\dots,F_{\xi}\left(\lambda\right)
$
of $\tilde{M}\left(w,\lambda,1\right)$
in $A^\tau$ of the form (\ref{eqn:OneParameterMacaulayfyingElements}),
realizing
the Macaulayfication of $\tilde{M}\left(w,\lambda,1\right)$ as
$$
M\left(w,\lambda,1\right)
=
\tilde{M}\left(w,\lambda,1\right)^\dagger
=
\left<\tilde{M}\left(w,\lambda,1\right),F_1\left(\lambda\right),\dots,F_\xi\left(\lambda\right)\right>
\footnote{
In degenerate case $(w=(0,0,0)^{\hash \tau},\lambda=1)$,
we find only one Macaulayfying element
(which doesn't belong to a one-parameter family in $\lambda$) to acheive the Macaulayfication,
because
there are no Macaulayfying elements
for $(w=(0,0,0)^{\hash \tau},\lambda\ne1)$.
}.
$$

Moreover,
denoting the right side as $\im\pi\left(w,\lambda,1\right)$
for some matrix $\pi\left(w,\lambda,1\right)\in A^{\tau\times 3\tau}$,
it fits into the free resolution
of $M_S\left(w,\lambda,1\right)$,
which is $M\left(w,\lambda,1\right)$ viewed as an $S$-module:
\begin{equation}\label{eqn:FreeResolutionMultOne}
\begin{tikzcd}[arrow style=tikz,>=stealth,row sep=5em,column sep=2em] 
0
  \arrow[r]
&
S^n
  \arrow[rr,"{\varphi\left(w',\lambda,1\right)}"]
&&
S^n
  \arrow[rr,"{\pi\left(w,\lambda,1\right)}"]
&&
M_S\left(w,\lambda,1\right)
  \arrow[r,""]
&
0.
\end{tikzcd}
\end{equation}

\end{thm}

Now we
`\textbf{substitute $J_{\mu}\left(\lambda\right)$ for $\lambda$}'
in all matrices to accomplish the same result for $\mu\ge2$:

%
%
%

\begin{proof}[Proof of Theorem \ref{thm:MFCMCorrespondence}]
For each Macaulayfying element
$
F\left(\lambda\right)
=
F_- \lambda^{-1} + F_0 + F_+ \lambda
\in A^\tau
$
of $\tilde{M}\left(w,\lambda,1\right)$ in $A^\tau$
given in Theorem \ref{thm:MFCMCorrespondenceMultOne},
we associate a matrix
$
F\left(J_\mu\left(\lambda\right)\right)
:=
F_-\otimes J_{\mu}\left(\lambda\right)^{-1}
+
F_0\otimes I_{\mu}
+
F_+\otimes J_{\mu}\left(\lambda\right)
\in
A^{\tau \mu \times \mu}
$
to get a
$\left(\lambda,J_{\mu}\left(\lambda\right)\right)$-substitution pair
$\left(F\left(\lambda\right),F\left(J_\mu\left(\lambda\right)\right)\right)$.
Then relation
(\ref{eqn:MacaulayfyingEquation})
implies
$$
\chi F\left(J_\mu\left(\lambda\right)\right) = \tilde{\pi}\left(w,\lambda,\mu\right) a_\chi\left(J_\mu\left(\lambda\right)\right)
\quad\text{for each}\ \ 
\chi\in\left\{x,y,z\right\},
$$
which shows that each column of $F\left(J_\mu\left(\lambda\right)\right)$ is a Macaulayfying element of $\tilde{M}\left(w,\lambda,\mu\right)$ in $A^{\tau\mu}$.
Therefore,
$$
M\left(w,\lambda,\mu\right)
=\tilde{M}\left(w,\lambda,\mu\right)^\dagger
=\left<\tilde{M}\left(w,\lambda,\mu\right), F_{1}\left(J_\mu\left(\lambda\right)\right),\dots,F_{\xi}\left(J_\mu\left(\lambda\right)\right)\right>^\dagger.
$$
Theorem \ref{thm:MFCMCorrespondenceMultOne} also implies that
$$
\left<\tilde{M}\left(w,\lambda,\mu\right), F_{1}\left(J_\mu\left(\lambda\right)\right),\dots,F_{\xi}\left(J_\mu\left(\lambda\right)\right)\right>
= \im\pi\left(w,J_\mu\left(\lambda\right)\right)
\subset A^{\tau\mu}.
$$

On the other hand, applying Proposition \ref{Lem:Substitution} to the above resolution (\ref{eqn:FreeResolutionMultOne}),
we have a free resolution
\begin{equation*}
\begin{tikzcd}[arrow style=tikz,>=stealth,row sep=5em,column sep=2.5em] 
  0 \arrow[r] & S^{3\tau\mu} \arrow[rr, "\varphi{(w',J_\mu\left(\lambda\right))}"] & & S^{3\tau\mu} \arrow[rr, "\pi{(w,J_\mu\left(\lambda\right))}"] & & \im\pi\left(w,J_\mu\left(\lambda\right)\right)_S \arrow[r] & 0
\end{tikzcd}
\end{equation*}
of $\im\pi\left(w,J_\mu\left(\lambda\right)\right)$ as an $S$-module. Finally, we know that $\varphi\left(w',J_\mu\left(\lambda\right)\right)$\ is a matrix factor of $xyz$, namely,
$$
\varphi\left(w',J_\mu\left(\lambda\right)\right) \psi\left(w',J_\mu\left(\lambda\right)\right)
= xyz I_{3\tau\mu},
$$
which implies that
$
\im\pi\left(w,J_\mu\left(\lambda\right)\right)
$
is already maximal Cohen-Macaulay and hence equals $M\left(w,\lambda,\mu\right)$.  
\end{proof}

\subsection{Degenerate case}\label{sec:DegenerateCase}

Under the conversion formula,
the degenerate band data
(without multiplicity)
$
\big(w=(0,0,0)^{\hash \tau},
\lambda=1\big)
$
and the degenerate loop data
(without rank)
$
\big(w'=(2,2,2)^{\hash \tau},
\lambda=1\big)
$
are converted to each other.
Recall that we have the degenerate canonical form of matrix factorizations 
(Definition \ref{def:DegenerateCanonicalFormMF})
$$
\left(
\varphi_{\deg}\left((2,2,2)^{\hash \tau},1,\primemu\right),
\psi_{\deg}\left((2,2,2)^{\hash \tau},1,\primemu\right)\right).
$$

For $\tau\ge2$,
the normal loop word $(2,2,2)^{\hash \tau}$ is periodic and the matrix factorization is decomposed into $\tau$ pieces,
each of which corresponds to the non-periodic normal loop word $(2,2,2)$
(Proposition \ref{prop:PeriodicDecomposableDegenerate}).
Among them,
only one piece
still has eigenvalue $\lambda=1$.
We compare it with the maximal Cohen-Macaulay module
corresponding to the non-periodic degenerate band data
$
\left(w=(0,0,0),\lambda=1,\mu\right):
$

\begin{itemize}
\item
the matrix factorization
$\left(\varphi_{\deg}\left((2,2,2),1,\primemu\right),\psi_{\deg}\left((2,2,2),1,\primemu\right)\right)$,
\item
the maximal Cohen-Macaulay module $M\left((0,0,0),1,\mu\right)$.
\end{itemize}
It turns out that they correspond to each other under the relation
$
\primemu=\mu-1:
$

\begin{thm}\label{thm:MFCMCorrespondenceDegenerate}
For a geometric rank $\primemu\in\mathbb{Z}_{\ge1}$ and algebraic multiplicity $\mu\in\mathbb{Z}_{\ge2}$ with $\primemu=\mu-1$,
we have
$$
\cok\underline{\varphi_{\deg}}\left((2,2,2),1,\primemu\right)
\cong
M\left((0,0,0),1,\mu\right)
$$
in $\underline{\CM}(A)$.
For $\mu=1$,
the right side
$
M\left((0,0,0),1,1\right)
\cong
A
\footnote{
See Remark 9.5 in \cite{BD17}.
It also follows from our discussion below.
}
$
is a zero object in $\underline{\CM}(A)$.
\end{thm}

Theorem \ref{thm:MFCMCorrespondence} and Theorem \ref{thm:MFCMCorrespondenceDegenerate}
are combined to give Theorem \ref{thm:MFCMCorrespondenceIntro} in the introduction.
To prove Theorem \ref{thm:MFCMCorrespondenceDegenerate},
it is convenient to introduce a \textbf{reduced form}
of our matrix factorization:
We define

\hspace{10mm}
$
\left(
\widetilde{\varphi_{\deg}}\left(\left(2,2,2\right), 1, {\primemu}\right),
\widetilde{\psi_{\deg}} \left(\left(2,2,2\right), 1, {\primemu}\right)
\right)
$
$$
\hspace{10mm}
:=
\left(
\begin{psmallmatrix}
-zx\mathbf{e}_1^T & 0 & 0 & 0
\\[2mm]
z I_{\primemu} & -y I_{\primemu} & 0 & 0
\\[2mm]
0 & x I_{\primemu} & -z I_{\primemu} & 0
\\[2mm]
- x J_{\primemu} \left(1\right) & 0 & y I_{\primemu} & -xy \mathbf{e}_{\primemu}
\end{psmallmatrix}
_{\scriptscriptstyle (3\primemu+1)\times(3\primemu+1)}
,
\begin{psmallmatrix}
-y \mathbf{e}_1 & -xy J_{\primemu}(0)^T & -y^2 J_{\primemu}(0)^T & -yz J_{\primemu}(0)^T
\\[2mm]
-z \mathbf{e}_1 & -zx J_{\primemu}(1)^T & -yz J_{\primemu}(0)^T & -z^2 J_{\primemu}(0)^T
\\[2mm]
-x \mathbf{e}_1 & -x^2 J_{\primemu}(1)^T & -xy J_{\primemu}(1)^T & -zx J_{\primemu}(0)^T
\\[2mm]
0 & -x \mathbf{e}_{\primemu}^T & -y \mathbf{e}_{\primemu}^T & -z \mathbf{e}_{\primemu}^T
\end{psmallmatrix}
_{\scriptscriptstyle (3\primemu+1)\times(3\primemu+1)}
\right)
$$
for $\primemu\in\mathbb{Z}_{\ge1}$,
and
$$
\left(
\widetilde{\varphi_{\deg}}\left(\left(2,2,2\right), 1,0\right),
\widetilde{\psi_{\deg}} \left(\left(2,2,2\right), 1,0\right)
\right)
:= \left(-xyz,-1\right)
$$
for $\primemu=0$.

\begin{lemma}\label{lem:DegenerateReduction}
For $\primemu\in\mathbb{Z}_{\ge1}$,
the degenerate canonical form
$
\big(\varphi_{\deg}\left((2,2,2),1,\primemu\right),
\psi_{\deg}\left((2,2,2),1,\primemu\right)\big)
$
is isomorphic to its reduced form
$
\left(
\widetilde{\varphi_{\deg}}\left(\left(2,2,2\right), 1, {\primemu}\right),
\widetilde{\psi_{\deg}} \left(\left(2,2,2\right), 1, {\primemu}\right)
\right)
$
in $\underline{\MF}(xyz)$.

\end{lemma}

\begin{proof}
The submatrix $-I_{\tau\primemu} + R_{\tau,\primemu}\left(\lambda\right) = -I_\primemu + J_\primemu\left(1\right)$ of $\psi_{\deg}\left((2,2,2),1,\primemu\right)$ has $\left(\primemu-1\right)$ units.
So we can use Lemma \ref{lem:MatrixReducingProcess} $\left(\primemu-1\right)$-times
to reduce it to $(3\primemu+1)\times(3\primemu+1)$ size,
the computation is straightforward.
\end{proof}

\begin{proof}[Proof of Theorem \ref{thm:MFCMCorrespondenceDegenerate}]
Let $\mu\in\mathbb{Z}_{\ge1}$.
We will find a free resolution
\begin{equation}\label{eqn:FreeResolutionDegenerate}
\begin{tikzcd}[arrow style=tikz,>=stealth,row sep=5em,column sep=2em] 
0
  \arrow[r]
&
S^{3\mu-2}
  \arrow[rrrr,"{\widetilde{\varphi_{\deg}}\left(\left(2,2,2\right), 1, {\mu-1}\right)}"]
&&&&
S^{3\mu-2}
  \arrow[r,"\pi"]
&
M_S\left((0,0,0),1,\mu\right)
  \arrow[r,""]
&
0.
\end{tikzcd}
\end{equation}
Then Eisenbud's equivalence and Lemma \ref{lem:DegenerateReduction} complete the proof.

Note that
$$
\tilde{M}\left((0,0,0),\lambda,\mu\right)
=
\im\tilde{\pi}\left((0,0,0),\lambda,\mu\right)
=
\im
\left(
\arraycolsep=2pt\def\arraystretch{1}
\begin{array}{c|c|c}
zx^{2} J_\mu(\lambda) + x^{2}y I_\mu & \left(xy^{2} + y^{2}z\right) I_\mu & \left(yz^{2}+z^{2}x\right) I_\mu
\end{array}
\right)_A
\subset
A^\mu
$$
for any $\lambda\in\mathbb{C}^\times$ and $\mu\in\mathbb{Z}_{\ge1}$.
It is easy to check that the sequence
$$
\begin{tikzcd}[arrow style=tikz,>=stealth,column sep = 20pt] 
S^{4} \arrow[rrrr,
"\spmat{
z & -y & 0 & 0
\\[2mm]
0 & x & -z & 0
\\[2mm]
- \lambda x & 0 & y & -xy
}"
] & & & & S^{3} \arrow[rrrrrr,
"\scriptstyle\left(
\lambda zx^{2} + x^{2}y \left| xy^{2} + y^{2}z \right| yz^{2}+z^{2}x
\right)
"
] & & & & & &
\tilde{M}\left((0,0,0),\lambda,1\right)_S \arrow[r] & 0
\end{tikzcd}
$$
is exact for any $\lambda\in\mathbb{C}^\times$.
Then by Proposition \ref{Lem:Substitution},
$$
\begin{tikzcd}[arrow style=tikz,>=stealth,column sep = 20pt] 
S^{4\mu} \arrow[rrrrrr,
"\spmat{
z I_\mu & -y I_\mu & 0 & 0
\\[2mm]
0 & x I_\mu & -z I_\mu & 0
\\[2mm]
- x J_\mu(1) & 0 & y I_\mu & -xy I_\mu
}"
] & & & &  & & S^{3\mu} \arrow[rrrrrrrrr,
"\scriptstyle\left(
zx^{2} J_\mu(1) + x^{2}y I_\mu \left| \left(xy^{2} + y^{2}z\right) I_\mu \right| \left(yz^{2}+z^{2}x\right) I_\mu
\right)
"
] & & & & & & & & &
\tilde{M}\left((0,0,0),1,\mu\right)_S \arrow[r] & 0
\end{tikzcd}
$$
is also exact. Using
$$
\spmat{
z I_\mu & -y I_\mu & 0
\\[2mm]
0 & x I_\mu & -z I_\mu
\\[2mm]
- x J_\mu(1) & 0 & y I_\mu
}
\spmat{y I_\mu
\\[2mm]
z I_\mu
\\[2mm]
x I_\mu
}
=
\spmat{0
\\[2mm]
0
\\[2mm]
-xy J_\mu(0)
},
$$
we can reduce it to the following,
which is still exact:
$$
\begin{tikzcd}[arrow style=tikz,>=stealth,column sep = 20pt] 
S^{3\mu+1} \arrow[rrrrrr,
"\spmat{
z I_\mu & -y I_\mu & 0 & 0
\\[2mm]
0 & x I_\mu & -z I_\mu & 0
\\[2mm]
- x J_\mu(1) & 0 & y I_\mu & -xy \mathbf{e}_\mu
}"
] & & & &  & & S^{3\mu} \arrow[rrrrrrrrr,
"\scriptstyle\left(
zx^{2} J_\mu(1) + x^{2}y I_\mu \left| \left(xy^{2} + y^{2}z\right) I_\mu \right| \left(yz^{2}+z^{2}x\right) I_\mu
\right)
"
] & & & & & & & & &
\tilde{M}\left((0,0,0),1,\mu\right)_S \arrow[r] & 0.
\end{tikzcd}
$$

There is a Macaulayfying element
$
F
:=
(xy+yz+zx) \mathbf{e}_1
:=
\left(xy+yz+zx,0,\dots,0\right)\in A^{\tau}
$
of
$
\tilde{M}\left((0,0,0),1,\mu\right)
$
in $A^\tau$.
So we enlarge it as
$
M_0 \left((0,0,0),1,\mu\right)
:=
\left<
F,
\tilde{M}\left((0,0,0),1,\mu\right)\right>_A \subset A^\tau,
$
then
$$
\begin{tikzcd}[arrow style=tikz,>=stealth,column sep = 20pt] 
S^{3\mu+4} \arrow[rrrrrrrr,
"\spmat{
x & y & z & 0 & 0 & 0 & 0
\\[2mm]
-\mathbf{e}_1 & 0 & 0 & z I_\mu & -y I_\mu & 0 & 0
\\[2mm]
0 & -\mathbf{e}_1 & 0 & 0 & x I_\mu & -z I_\mu & 0
\\[2mm]
0 & 0 & -\mathbf{e}_1 & - x J_\mu \left(1\right) & 0 & y I_\mu & -xy \mathbf{e}_\mu
}"
] & & & & & & & & S^{3\mu+1}
&&&&&&&&&&&&
\end{tikzcd}
$$
$$
\begin{tikzcd}[arrow style=tikz,>=stealth,column sep = 20pt] 
&&&&&
\arrow[rrrrrrrrrrr,
"\scriptstyle\left(
(xy+yz+zx)\mathbf{e}_1 \left | zx^{2} J_\mu(1) + x^{2}y I_\mu \left| \left(xy^{2} + y^{2}z\right) I_\mu \right| \left(yz^{2}+z^{2}x\right.\right) I_\mu
\right)
"
] & & & & & & & & & & &
M_0\left((0,0,0),1,\mu\right)_S \arrow[r] & 0
\end{tikzcd}
$$
is exact.
(One can check it using Lemma 9.7 in \cite{CJKR}.)

We can reduce matrices along three unit entries of the first matrix.
(See Lemma Lemma 9.8 in \cite{CJKR}.)
As a result,
for $\mu=1$,
we get a free resolution
$$
\begin{tikzcd}[arrow style=tikz,>=stealth,column sep = 20pt] 
0 \arrow[r] & S \arrow[rr, "-xyz"] && S \arrow[rrr, "xy+yz+zx"] &&& M_0\left((0,0,0),1,1\right)_S \arrow[r] & 0
\end{tikzcd}
$$
and for $\mu\ge2$,
we have
$$
\begin{tikzcd}[arrow style=tikz,>=stealth,column sep = 20pt] 
0 \arrow[r] &
S^{3\mu-2} \arrow[rrrrrrr,
"\spmat{
-zx\mathbf{e}_1^T & 0 & 0 & 0
\\[2mm]
z I_{\mu-1} & -y I_{\mu-1} & 0 & 0
\\[2mm]
0 & x I_{\mu-1} & -z I_{\mu-1} & 0
\\[2mm]
- x J_{\mu-1} \left(1\right) & 0 & y I_{\mu-1} & -xy \mathbf{e}_{\mu-1}
}"
] & & & & & &  & S^{3\mu-2}
&&&&&&&&&&&&
\end{tikzcd}
$$
$$
\begin{tikzcd}[arrow style=tikz,>=stealth,column sep = 20pt] 
&&&&&
\arrow[rrrrrrrrrrrr,
"
\scriptstyle
\spmat{
xy+yz+zx
&
zx^2 \mathbf{e}_1^T
&
0
&
0
\\[2mm]
0
&
zx^2 J_{\mu-1}(1) + x^2 y I_{\mu-1}
&
\left(xy^2 + y^2 z\right) I_{\mu-1}
&
\left(yz^2 + z^2 x\right) I_{\mu-1}
}
"
] & & & & & & & & & & & &
M_0\left((0,0,0),1,\mu\right)_S \arrow[r] & 0.
\end{tikzcd}
$$
Note that the left matrix is $\widetilde{\varphi_{\deg}}\left((2,2,2),1,\mu-1\right)$.
As it is a matrix factor of $xyz$, the induced map $S^{3\mu-2}\rightarrow S^{3\mu-2}$ is injective
and $M_0\left((0,0,0),1,\mu\right)$ is maximal Cohen-Macaulay,
implying that
it is the same as
the Macaulayfication
$
M\left((0,0,0),1,\mu\right)
$
$
=
$
$
\tilde{M}\left((0,0,0),1,\mu\right)^\dagger.
$
So we achieved the desired free resolution (\ref{eqn:FreeResolutionDegenerate}) of
$
M\left((0,0,0),1,\mu\right)_S
$
for any $\mu\ge1$.
\end{proof}

\begin{spacing}{0.979}
We finish this section with
a remark on the \textbf{periodic cases}:
We showed that the matrix factorizations corresponding to periodic loop data are decomposable
(Theorem \ref{thm:PeriodicDecomposable} for cylinder-free case
and Proposition \ref{prop:PeriodicDecomposableDegenerate} for non-cylinder-free case).
In non-degenerate cases,
they are mapped to maximal Cohen-Macaulay modules
corresponding to periodic band data
 (Theorem \ref{thm:MFCMCorrespondence}).
It yields the
decomposition
\begin{equation}\label{eqn:CMDecomposition}
M\left(w,\lambda,\mu\right)
\cong
\bigoplus_{k=0}^{N-1} M \left(\tilde{w}, \lambda_{k},\mu\right)
\end{equation}
in $\underline{\CM}(A)$,
where
$
\left(w,\lambda,\mu\right)
$
is a non-degenerate band datum
with a periodic band word
$w=\tilde{w}^{\hash N}$ for another band word $\tilde{w}$,
and $\lambda_{0},\dots,\lambda_{N-1}\in\field^\times$
are the $N$-th roots of $\lambda$.
In fact,
an investigation in the category $\Tri(A)$
proves that the decomposition (\ref{eqn:CMDecomposition})
is still valid in $\CM(A)$,
even for
non-degenerate
band data. 

Now
something tricky happens
in \textbf{periodic degenerate cases}:
For a
band datum
$
\left(w=(0,0,0)^{\hash \tau},\lambda=1,\mu\right),
$
the corresponding maximal Cohen-Macaulay module is decomposed as
$$
M\left(\left(0,0,0\right)^{\hash \tau},1,\mu\right)
\cong
\bigoplus_{k=0}^{\tau-1} M \left(\left(0,0,0\right), e^{2\pi i \cdot \frac{k}{\tau}},\mu\right).
$$
Note that only the first direct summand is still degenerate,
and the rank of its converted loop datum is shifted only in that piece.
Namely,
the corresponding matrix factorization and loop with a local system is
\begin{equation}\label{eqn:PeriodicDegenerateCaseMappedToMFFuk}
{\varphi_{\deg}}\left((2,2,2),1,\mu-1\right)
\oplus
\bigoplus_{k=1}^{\tau-1} \varphi \left((2,2,2), e^{2\pi i \cdot \frac{k}{\tau}}, \mu \right)
\quad\text{and}\quad
\mathcal{L}\left((2,2,2),-1,\mu-1\right)
\oplus
\bigoplus_{k=1}^{\tau-1}
\mathcal{L}\left((2,2,2),-e^{2\pi i \cdot \frac{k}{\tau}},\mu\right).
\end{equation}
Therefore,
the object in $\underline{\CM}(A)$ corresponding to a periodic degenerate band datum is mapped to
objects in $\underline{\MF}(xyz)$ or $\Fuk\left(\POP\right)$
that are decomposed into pieces having different geometric ranks.

Conversely,
for a loop datum
$
\left(w'=(2,2,2)^{\hash \tau},\primelambda=-1,\primemu\right)
$
or
$
\left(w'=(2,2,2)^{\hash \tau},\lambda=1,\primemu\right),
$
we observed in (\ref{eqn:PeriodicDegenerateCaseDecompositionFuk}) and (\ref{eqn:PeriodicDegenerateCaseDecompositionMF}) the decomposition of the corresponding loop with a local system
and matrix factorization.
Now we know that they correspond to
the decomposition of maximal Cohen-Macaulay module
\begin{equation}\label{eqn:PeriodicDegenerateCaseMappedToCM}
M\left((0,0,0),1,\primemu+1\right)
\oplus
\bigoplus_{k=1}^{\tau-1}
M\left((0,0,0),e^{2\pi i \cdot\frac{k}{\tau}},\primemu\right),
\end{equation}
where only the first direct summand has
a shifted
multiplicity.
\end{spacing}


\section{Applications}

\subsection{Flip of loops and dual of modules}

Generally speaking,
a symplectomorphism
(diffeomorphism preserving the symplectic form)
between symplectic manifolds
induces an equivalence on their Fukaya categories.
There are some obvious symmetries in our pair-of-pants surface $\POP$,
each of which induces a corresponding auto-equivalence on
$\Fuk\left(\POP\right)$,
and hence on $\underline{\MF}(xyz)$ and $\underline{\CM}(A)$.

In this subsection,
we take a look at the $\mathbb{Z}_2$-symmetry given by flipping $\POP$ back-and-forth,
which is
described in
Example \ref{ex:FlipDual}.
It is given by an orientation-reversing diffeomorphism
$\imath:\POP\rightarrow\POP$,
which is an anti-symplectomorphism
($\imath^*\omega=-\omega$).
Such a map defines a natural \emph{contravariant $\AI$-functor}
$
\imath:
\Fuk\left(\POP\right)
\rightarrow
\Fuk\left(\POP\right)
$
(\S \ref{sec:ContravariantAIFunctor}).
We also define the \emph{transpose functor} in $\MF_{\AI}(f)$
as a contravariant $\AI$-functor (\S \ref{sec:TransposeFunctor}),
and show that,
in our situation,
two $\AI$-functors are related under the localized mirror functor
(\S \ref{sec:FlipTranspose}).
It is also related with the \emph{duality functor}
$\Hom_A\left(-,A\right)$
in $\underline{\CM}(A)$
under Eisenbud's equivalence
(\S \ref{sec:FlipDual}).
We also give a description of these operations in terms of loop/band data
(\S \ref{sec:DualFlipCanonicalForms}).

\subsubsection{Anti-symplectomorphism and contravariant $\AI$-functor on Fukaya categories}
\label{sec:ContravariantAIFunctor}
We first recall some general algebraic notions
following \cite{S08}:

\begin{defn}
Given a $\mathbb{Z}_2$-graded $\AI$-category $\mathcal{A}$,
the \textnormal{\textbf{opposite $\AI$-category}} $\mathcal{A}^{\operatorname{op}}$
consists of the same class of objects
$
\operatorname{Ob}\left(\mathcal{A}^{\operatorname{op}}\right)
:=
\operatorname{Ob}\left(\mathcal{A}\right),
$
switched morphism spaces
$
\hom_{\mathcal{A}^{\operatorname{op}}}^\bullet\left(\mathcal{L}_0,\mathcal{L}_1\right)
:=
\hom_{\mathcal{A}}^\bullet\left(\mathcal{L}_1,\mathcal{L}_0\right)
$
($\bullet\in\mathbb{Z}_2$),
and $\AI$-operations $\left\{\operatorm_k^{\operatorname{op}}\right\}_{k\ge1}$ defined as
$$
\operatorm_k^{\operatorname{op}}\left(g_1,\dots,g_k\right)
:=
(-1)^{\left|g_1\right|+\cdots+\left|g_k\right|-k}
\operatorm_k\left(g_k,\dots,g_1\right)
$$
for
$
g_i\in
\hom_{\mathcal{A}^{\operatorname{op}}}^\bullet\left(\mathcal{L}_{i-1},\mathcal{L}_i\right)
=
\hom_{\mathcal{A}}^\bullet\left(\mathcal{L}_{i},\mathcal{L}_{i-1}\right)
$
($i\in\left\{1,\dots,k\right\}$, $\bullet\in\mathbb{Z}_2$).
\end{defn}

A straightforward calculation shows that $\mathcal{A}^{\operatorname{op}}$
is indeed an
$\AI$-category.

\begin{defn}\label{defn:AinftyContravariantFunctor}

A \textnormal{\textbf{contravariant $\AI$-functor}} $\mathcal{G}:=\left\{\mathcal{G}_k\right\}_{k\ge0}$ between two $\AI$-categories $\mathcal{A}$ and $\mathcal{B}$
is an $\AI$-functor from $\mathcal{A}^{\operatorname{op}}$ to $\mathcal{B}$.
Equivalently, it can be defined by giving a mapping
$$
\mathcal{G}_0:
\operatorname{Ob}\left(\mathcal{A}\right)
\rightarrow
\operatorname{Ob}\left(\mathcal{B}\right)
$$
and $\mathbb{k}$-linear maps $\left(k\ge1\right)$
$$\mathcal{G}_{k}
:
\hom_{\mathcal{A}}\left(\mathcal{L}_1,\mathcal{L}_0\right)
\otimes \cdots \otimes
\hom_{\mathcal{A}}\left(\mathcal{L}_{k},\mathcal{L}_{k-1}\right)
\rightarrow
\hom_{\mathcal{B}}\left(\mathcal{G}_{0}\left(\mathcal{L}_0\right),\mathcal{G}_{0}\left(\mathcal{L}_k\right)\right)$$
of degree $1-k$,
satisfying \emph{$\AI$-relations}
\begin{align}\label{eqn:AI-relationsContravariantFunctor}
\begin{split}
&\sum_{1\le k \le n}\sum_{1\le i_1<\cdots<i_k= n} \operatorm_{k}^{\mathcal{B}} \left(\mathcal{G}_{i_{1}} \left( g_{1}, \dots, g_{i_{1}} \right), \dots, \mathcal{G}_{i_{k}} \left( g_{i_{k-1} + 1}, \dots, g_{n} \right) \right)\\
&\hspace{10mm}
=\sum_{0\le i<j\le n} (-1)^{\left|g_{1}\right| + \cdots + \left|g_{j}\right| - j} \mathcal{G}_{n-j+i+1} \left (g_{1}, \dots, g_{i}, \operatorm_{j-i}^{\mathcal{A}} \left( g_{j}, \ldots, g_{i+1} \right), g_{j+1}, \dots, g_{n} \right)
\end{split}
\end{align}
for any fixed $n\in\mathbb{Z}_{\ge1}$
and morphisms
$
g_i\in
\hom_{\mathcal{A}}^\bullet\left(\mathcal{L}_{i},\mathcal{L}_{i-1}\right)
$($i\in\left\{1,\dots,n\right\}$, $\bullet\in\mathbb{Z}_2$).

It induces an ordinary contravariant functor
$
H^0\left(\mathcal{G}\right):H^0\left(\mathcal{A}\right)\rightarrow H^0\left(\mathcal{B}\right),
$
whose mapping on objects is $\mathcal{G}_0$
and action on morphisms is given by
$
\left[g\right]\mapsto\left[\mathcal{G}_1\left(g\right)\right]$.

\end{defn}


Now
let $\left(\Sigma,\omega\right)$ and $\left(\Sigma',\omega'\right)$ be $2$-dimensional symplectic manifolds
(possibly with boundary) of finite type
(as in \S \ref{sec:CptFukSurface})
and
$\imath:\Sigma\rightarrow\Sigma'$ an
anti-symplectomorphism
($\imath^*\omega'=-\omega$).
We will define a contravariant $\AI$-functor
$\imath:=\left\{\imath_k\right\}_{k\ge0}:\Fuk\left(\Sigma\right)\rightarrow\Fuk\left(\Sigma'\right)$ as follows:

Any object
$\mathcal{L}:=\left(L,E,\nabla\right)$
of $\Fuk\left(\Sigma\right)$
consists of
a loop
$L:S^1\rightarrow \Sigma$,
a finite-rank $\field$-vector bundle $E$ over $S^1$,
and a flat connection $\nabla$ on $E$.
We define its image under the functor $\imath$ as the triple
$$
\imath_0\left(\mathcal{L}\right):=\left(\imath\left(L\right),E^*,\nabla^*\right),
$$
where $\imath\left(L\right):=\imath\circ L:S^1\rightarrow\Sigma'$ is the image of $L$ under $\imath$,
$E^*$ is the dual vector bundle of $E$ over $S^1$,
and $\nabla^*$ is the dual connection of $\nabla$.

For two objects 
$\mathcal{L}_i:=\left(L_i,E_i,\nabla_i\right)$
($i\in\left\{0,1\right\}$),
note that there are bijections
$$
\chi^\bullet\left(L_1,L_0\right)
\xleftrightarrow{1:1}
\chi^\bullet\left(\imath\left(L_0\right),\imath\left(L_1\right)\right),
\quad
q\leftrightarrow\imath\left(q\right)
\quad
\left(\bullet\in\mathbb{Z}_2\right)
$$
as shown in Figure \ref{fig:MorphismsInversion}
for $\bullet=0$ case.
Therefore, we have
\begin{align*}
\hom^\bullet\left(\imath\left(\mathcal{L}_0\right),\imath\left(\mathcal{L}_1\right)\right)
&=
\bigoplus_{q'\in \chi^\bullet\left(\imath\left(L_0\right),\imath\left(L_1\right)\right)}
\Hom_{\field}\left(\left.E_0^*\right|_{q'},\left.E_1^*\right|_{q'}\right)
\\
&=
\bigoplus_{q\in \chi^\bullet\left(L_1,L_0\right)}
\Hom_{\field}\left(\left(\left.E_0\right|_{q}\right)^*,\left(\left.E_1\right|_{q}\right)^*\right)
\quad
\left(\bullet\in\mathbb{Z}_2\right).
\end{align*}

We define
$
\imath_1
:
\hom\left(\mathcal{L}_1,\mathcal{L}_0\right)
\rightarrow
\hom\left(\imath\left(\mathcal{L}_0\right),\imath\left(\mathcal{L}_1\right)\right)
$
by
\begin{equation}\label{eqn:MorphismSending}
g\in \Hom_{\field}\left(\left.E_1\right|_q,\left.E_0\right|_q\right)
\mapsto
(-1)^{\left|g\right|}
g^*\in\Hom_{\field}\left(\left(\left.E_0\right|_q\right)^*,\left(\left.E_1\right|_q\right)^*\right)
\end{equation}
for a base morphism $g$ over $q\in\chi^\bullet\left(L_1,L_0\right)$
($\bullet\in\mathbb{Z}_2$),
and then linearly extend it to any morphisms.
Higher components $\imath_{k\ge2}$
are defined to be zero.

\begin{figure}[H]
   \begin{minipage}{0.49\textwidth}
     \centering
$
\setlength\arraycolsep{2pt}

$
\captionsetup{width=1\linewidth}
\caption{$q\in\chi^0\left(L_1,L_0\right)$, $\imath(q)\in\chi^0\left(\imath\left(L_0\right),\imath\left(L_1\right)\right)$}
\label{fig:MorphismsInversion}
   \end{minipage}\hfill
   \begin{minipage}{0.51\textwidth}
     \centering
$
\setlength\arraycolsep{2pt}
%
\end{matrix}
$
\captionsetup{width=1\linewidth}
\caption{A polygon and its image under $\imath$}
\label{fig:PolygonInversion}
   \end{minipage}\hfill
\end{figure}


\begin{prop}
The functor
$\imath:\Fuk\left(\Sigma\right)\rightarrow\Fuk\left(\Sigma'\right)$
defined above is indeed a contravariant $\AI$-functor.


\end{prop}

\begin{proof}

As $\imath_{k\ge2}=0$,
the required $\AI$-relations (\ref{eqn:AI-relationsContravariantFunctor})
simplify to
\begin{equation}\label{eqn:SimplifiedAIRelations}
\operatorm_n\left(\imath_1\left(g_1\right),\dots,\imath_1\left(g_n\right)\right)
=
(-1)^{
\left|g_1\right|+\cdots+\left|g_n\right|-n
}
\imath_1\left(\operatorm_n\left(g_n,\dots,g_1\right)\right)
\end{equation}
for
$
n\in\mathbb{Z}_{\ge1}
$
and
$
g_i\in
\hom_{\Fuk\left(\Sigma\right)}^\bullet\left(\mathcal{L}_{i},\mathcal{L}_{i-1}\right)
$
($i\in\left\{1,\dots,n\right\}$, $\bullet\in\mathbb{Z}_2$).

For
$q_i\in \chi^\bullet\left(L_{i},L_{i-1}\right)$
and
$g_i\in\Hom_{\field}\left(\left.E_{i}\right|_{q_i},\left.E_{i-1}\right|_{q_i}\right)$
($i\in\left\{1,\dots,n\right\}$, $\bullet\in\mathbb{Z}_2$),
the left side is
\begin{equation}\label{eqn:AIOperationDual}
\operatorm_n\left((-1)^{\left|g_1\right|}g_1^*,\dots,(-1)^{\left|g_n\right|}g_n^*\right)
=
\sum_{r'\in\chi\left(\imath\left(L_0\right),\imath\left(L_n\right)\right)}
\sum_{u'\in\mathcal{M}\left(\imath\left(q_1\right),\dots,\imath\left(q_n\right),\overline{r'}\right)}
(-1)^{\left|q_1\right|+\cdots+\left|q_n\right|}
\sign\left(u'\right)
\hol_{r'}\left(\partial u'\right)
\left(g_1^*,\dots,g_n^*\right).
\end{equation}

Note that there is a bijection between angles
$$
\chi\left(L_n,L_0\right)
\xleftrightarrow{1:1}
\chi\left(\imath\left(L_0\right),\imath\left(L_n\right)\right),
\quad
r\leftrightarrow\imath\left(r\right)
$$
and
between immersed polygons
$$
\mathcal{M}\left(q_n,\dots,q_1,\overline{r}\right)
\xleftrightarrow{1:1}
\mathcal{M}\left(\imath\left(q_1\right),\dots,\imath\left(q_n\right),\imath\left(\overline{r}\right)\right),
\quad
u\leftrightarrow\imath\left(u\right):=\imath\circ u
$$
as shown in Figure \ref{fig:PolygonInversion},
for each $r\in\chi\left(L_n,L_0\right)$.
Therefore,
the right side of (\ref{eqn:AIOperationDual}) is replaced by
\begin{equation}\label{eqn:AIOperationDual2}
\sum_{r\in\chi\left(L_n,L_0\right)}
\sum_{u\in\mathcal{M}\left(q_n,\dots,q_1,\overline{r}\right)}
(-1)^{\left|q_1\right|+\cdots+\left|q_n\right|}
\sign\left(\imath\left(u\right)\right)
\hol_{\imath(r)}\left(\partial \left(\imath(u)\right)\right)
\left(g_1^*,\dots,g_n^*\right).
\end{equation}

For each pair of $u$ and $\imath\left(u\right)$ in those sets,
from the sign rule (\ref{eqn:sign(u)}),
we have
{\allowdisplaybreaks
\begin{align}\label{eqn:SignDifference}
\begin{split}
\sign(u)\sign\left(\imath\left(u\right)\right)
&=
(-1)^\wedge
\left(
\sum_{i=1}^n
\left|q_i\right|
\mathbb{1}_{\operatorname{o}\left(L_{i-1}\right)\ne\operatorname{o}\left(\partial u\right)}
+
\left|r\right|
\mathbb{1}_{\operatorname{o}\left(L_0\right)\ne\operatorname{o}\left(\partial u\right)}
+
\sum_{i=1}^n
\left|\imath\left(q_i\right)\right|
\mathbb{1}_{\operatorname{o}\left(\imath\left(L_{i}\right)\right)\ne\operatorname{o}\left(\partial u\right)}
+
\left|\imath(r)\right|
\mathbb{1}_{\operatorname{o}\left(\imath\left(L_n\right)\right)\ne\operatorname{o}\left(\partial u\right)}
\right)
\\
&=
(-1)^\wedge
\left(
\sum_{i=1}^n
\left|q_i\right|
\left(
\mathbb{1}_{\operatorname{o}\left(L_{i-1}\right)\ne\operatorname{o}\left(\partial u\right)}
-
\mathbb{1}_{\operatorname{o}\left(L_{i}\right)\ne\operatorname{o}\left(\partial u\right)}
+1
\right)
+
\left|r\right|
\left(
\mathbb{1}_{\operatorname{o}\left(L_0\right)\ne\operatorname{o}\left(\partial u\right)}
-
\mathbb{1}_{\operatorname{o}\left(L_n\right)\ne\operatorname{o}\left(\partial u\right)}
+1
\right)
\right),
\end{split}
\end{align}
}
where we used the fact that
$
\left|q_i\right|=\left|\imath\left(q_i\right)\right|,
$
$
\left|r\right|=\left|\imath\left(r\right)\right|
$
($i\in\left\{1,\dots,n\right\}$)
and
that the orientation
of $L_i$ coincides with
that of $\partial u$
if and only if
the orientation
of $\imath\left(L_i\right)$
differs from that of $\partial\left(\imath\left(u\right)\right)$.
Note also that
$$
\left(
\left|q_i\right|=1
\quad\Leftrightarrow\quad
\mathbb{1}_{\operatorname{o}\left(L_{i-1}\right)\ne\operatorname{o}\left(\partial u\right)}
=
\mathbb{1}_{\operatorname{o}\left(L_{i}\right)\ne\operatorname{o}\left(\partial u\right)}
\right)
\quad\text{and}\quad
\left(
\left|r\right|=1
\quad\Leftrightarrow\quad
\mathbb{1}_{\operatorname{o}\left(L_{0}\right)\ne\operatorname{o}\left(\partial u\right)}
\ne
\mathbb{1}_{\operatorname{o}\left(L_{n}\right)\ne\operatorname{o}\left(\partial u\right)}
\right),
$$
which reduce (\ref{eqn:SignDifference}) to
$
\displaystyle(-1)^
{
\sum_{i=1}^n\left|q_i\right|
}.
$

On the other hand,
$
\hol_{\imath(r)}\left(\partial \left(\imath(u)\right)\right)
\left(g_1^*,\dots,g_n^*\right)
$
is
by definition given as
$$
P\left(\left(\partial \left(\imath(u)\right)\right)_0\right)
\circ
g_n^*
\circ
P\left(\left(\partial \left(\imath(u)\right)\right)_1\right)
\circ
g_{n-1}^*
\circ
\cdots
\circ
g_2^*
\circ
P\left(\left(\partial \left(\imath(u)\right)\right)_{n-1}\right)
\circ
g_1^*
\circ
P\left(\left(\partial \left(\imath(u)\right)\right)_n\right),
$$
where
each
$
P\left(\left(\partial\left(\imath\left(u\right)\right)\right)_i\right)
$
is the parallel transport with respect to $\nabla_i^*$
from
$\left.E_i^*\right|_{\imath\left(q_i\right)}$
to
$\left.E_i^*\right|_{\imath\left(q_{i+1}\right)}$,
which is the dual
$
P\left(\left(\partial u\right)_i\right)^*
$
of the parallel transport
$
P\left(\left(\partial u\right)_i\right)
$
with respect to $\nabla_i$ from $\left.E_i\right|_{q_{i+1}}$ to $\left.E_i\right|_{q_i}$.
Therefore,
we can replace the total composition by
$
\left(\hol_r\left(\partial u\right)\left(g_n,\dots,g_1\right)\right)^*.
$

Summing up,
we can rewrite
(\ref{eqn:AIOperationDual2})
as
$$
\sum_{r\in\chi\left(L_n,L_0\right)}
\sum_{u\in\mathcal{M}\left(q_n,\dots,q_1,\overline{r}\right)}
\sign(u)
\left(\hol_r\left(\partial u\right)\left(g_n,\dots,g_1\right)\right)^*,
$$
whose degree is
$
\left|r\right|=2-n + \left|q_1\right|+\cdots+\left|q_n\right|.
$
Therefore,
it is
the same as
$
(-1)^{\left|g_1\right|+\cdots+\left|g_n\right|-n}
\imath_1\left(\operatorm_n\left(g_n,\dots,g_1\right)\right)
$,
or the right side of (\ref{eqn:SimplifiedAIRelations}).
\end{proof}

\subsubsection{
Transpose functor on $\MF_{\AI}(f)$}
\label{sec:TransposeFunctor}
Let $S$ be the power series ring $\field[[x_1,\dots,x_m]]$ of $m$ variables,
and $f\in S$ its nonzero element.
Taking transpose of matrix factorizations of $f$
gives rise to
a
contravariant $\AI$-functor
$
-\operatorname{Tr}
:
\MF_{\AI}(f)\rightarrow \MF_{\AI}(f),
$
called the \textbf{(minus) transpose functor}
\footnote{
We put the minus sign here to match with the \emph{flip functor} in the next subsection.
However,
the functor with a minus sign and one without a minus sign are
\emph{($\AI$-)quasi-isomorphic} to each other.
},
which we now define:

For an object
$
\begin{tikzcd}[arrow style=tikz,>=stealth, sep=20pt, every arrow/.append style={shift left=0.5}]
   P^0
     \arrow{r}{\varphi}
   &
   P^1
     \arrow{l}{\psi}
\end{tikzcd}
$
with free $S$-modules $P^0$, $P^1$,
we define its
image
as
$
\begin{tikzcd}[arrow style=tikz,>=stealth, sep=20pt, every arrow/.append style={shift left=0.5}]
   \left(P^1\right)^*
     \arrow{r}{-\varphi^*}
   &
   \left(P^0\right)^*,
     \arrow{l}{-\psi^*}
\end{tikzcd}
$
where $P^*$ denotes the $S$-dual $\Hom_S\left(P,S\right)$ of an $S$-module  $P$
and $\varphi^*$, $\psi^*$ denote the natural pull-back maps.
It defines the functor $-\operatorname{Tr}:=\left\{-\operatorname{Tr}_k\right\}_{k\ge0}$
on the object level as
$$
-\operatorname{Tr}_0
:
\operatorname{Ob}\left(\MF_{\AI}(f)\right)
\rightarrow
\operatorname{Ob}\left(\MF_{\AI}(f)\right),
\quad
\left(\varphi,\psi\right)\rightarrow\left(-\varphi^*,-\psi^*\right).
$$

Given two matrix factorizations
$
\begin{tikzcd}[arrow style=tikz,>=stealth, sep=20pt, every arrow/.append style={shift left=0.5}]
   P_0^0
     \arrow{r}{\varphi_0}
   &
   P_0^1
     \arrow{l}{\psi_0}
\end{tikzcd}
$
and
$
\begin{tikzcd}[arrow style=tikz,>=stealth, sep=20pt, every arrow/.append style={shift left=0.5}]
   P_1^0
     \arrow{r}{\varphi_1}
   &
   P_1^1
     \arrow{l}{\psi_1}
\end{tikzcd}
$
of $f$
and
an {even-degree} morphism
$
\left(\alpha:P_0^0\rightarrow P_1^0,\
\beta:P_0^1\rightarrow P_1^1\right)
$
(resp. an odd-degree morphism
$
\big(\gamma:P_0^0\rightarrow P_1^1,\
\delta:P_0^1\rightarrow P_1^0\big)
$)
(see diagrams in (\ref{eqn:MFMorphismDiagram})),
we take dual of the
maps
to define its image under $-\operatorname{Tr}_1$:
$$
-\operatorname{Tr}_1
:
\hom\left(\left(\varphi_0,\psi_0\right),\left(\varphi_1,\psi_1\right)\right)
\rightarrow
\hom\left(\left(\varphi_1^*,\psi_1^*\right),\left(\varphi_0^*,\psi_0^*\right)\right),
\quad
\left(\alpha,\beta\right)
\mapsto
\left(\beta^*,\alpha^*\right)
\quad
\left(
\text{resp. }
\left(\gamma,\delta\right)
\mapsto
\left(\delta^*,\gamma^*\right)
\right).
$$
The higher components $\operatorname{Tr}_{k\ge2}$ are defined to be zero.
It is straightforward to check that $\operatorname{Tr}$ is
a contravariant $\AI$-functor.



\subsubsection{Flip of loops and transpose of matrix factorizations}
\label{sec:FlipTranspose}

Coming back to our specific situation,
the anti-symplectomorphism
$
\imath:\POP\rightarrow\POP
$
described in Example \ref{ex:FlipDual}
and the discussion so far yield the diagram:
\begin{equation}\label{eqn:FlipDualDiagram}
\begin{tikzcd}[arrow style=tikz,>=stealth,row sep=3em,column sep=3em] 
&
\Fuk\left(\POP\right)
  \arrow[r,"\LocalF"]
  \arrow[d,swap,"\textnormal{flip}"]
  \arrow[d,"\imath"]
&
\MF_{\AI}(xyz)
  \arrow[d,swap,"\textnormal{transpose}"]
  \arrow[d,"-\operatorname{Tr}"]
&
\\
&
\Fuk\left(\POP\right)
  \arrow[r,"\LocalF"]
&
\MF_{\AI}(xyz)
&
\end{tikzcd}
\end{equation}
In this subsection,
we will show that the diagram commutes,
in the sense that two $\AI$-functors
$
\LocalF
\circ\imath
$
and
$
-\operatorname{Tr}\circ\LocalF
$
are \emph{quasi-isomorphic}
\footnote{
Two $\AI$-functors are \emph{quasi-isomorphic} to each other
if there are $\AI$-natural transformations
between them satisfying some homotopy conditions.
Any $\AI$-natural transformation induces an (ordinary) natural transformation
between the induced ordinary functors.
See \cite{S08} for the details.
}
to each other.

It is based on the fact that our reference object $\mathbb{L}$
(Seidel Lagrangian)
is invariant under the flipping functor $\imath$.
Namely,
Figure \ref{fig:SeidelInversion}
shows that 
$\imath\left(\mathbb{L}\right)$
consists of the same underlying loop with $\mathbb{L}$
and its local system
is gauge equivalent to
that of $\mathbb{L}$.
Moreover,
one can easily check
(as we did in Proposition \ref{prop:WeakBoundingCochain})
that
$$
-\imath(b)
=
-x\imath(X)-y\imath(Y)-z\imath(Z)
\in\hom^1\left(\imath\left(\mathbb{L}\right),\imath\left(\mathbb{L}\right)\right)
$$
is a weak bounding cochain
with the disk potential $W^{\imath\left(\mathbb{L}\right)}=xyz$.
Therefore,
the pair
$
\left(\imath\left(\mathbb{L}\right),-\imath(b)\right)
$
defines a localized mirror functor
$
\mathcal{F}^{\imath\left(\mathbb{L}\right)}
:
\Fuk\left(\POP\right)
\rightarrow
\MF_{\AI}(xyz).
$

\begin{figure}[H]
   \begin{minipage}{0.5\textwidth}
     \centering
$
\setlength\arraycolsep{2pt}

$
\captionsetup{width=1\linewidth}
\caption{Seidel Lagrangian and its image under $\imath$\hspace*{-4mm}}
\label{fig:SeidelInversion}
   \end{minipage}\hfill
   \begin{minipage}{0.5\textwidth}
     \centering
$
\setlength\arraycolsep{2pt}
%
\end{matrix}
$
\captionsetup{width=1\linewidth}
\caption{A deformed strip and its image under $\imath$\hspace*{-2mm}}
\label{fig:StripInversion}
   \end{minipage}\hfill
\end{figure}
Now we have two localized mirror functors
$\LocalF$ and $\mathcal{F}^{\imath\left(\mathbb{L}\right)}$,
but there is a trivial isomorphism
between
$\left(\mathbb{L},b\right)$
and
$\left(\imath\left(\mathbb{L}\right),-\imath(b)\right)$,
and
it has been already
proven (with much greater generality)
in \cite{CHL-glue}
that such isomorphic weak bounding cochains
induce quasi-isomorphic
$\AI$-functors:


\begin{prop}\cite[Theorem 4.7.(2)]{CHL-glue}
Two localized mirror functors
$\LocalF$
and
$\mathcal{F}^{\imath\left(\mathbb{L}\right)}$
are quasi-isomorphic.
\end{prop}

To show that the diagram (\ref{eqn:FlipDualDiagram}) commutes,
therefore,
it is enough to check
the following alternative:

\begin{prop}\label{prop:CommutativityReversingSwitching}
Two
$\AI$-functors
$
\mathcal{F}^{\imath\left(\mathbb{L}\right)}\circ\imath
$
and
$
-\operatorname{Tr}\circ\LocalF
$
are
the same.
\end{prop}
\begin{proof}
Recall from \S \ref{sec:LMFComputationSubsection} that
$
\LocalF\left(\mathcal{L}\right)
=
\left(
\LocalPhi\left(\mathcal{L}\right),
\LocalPsi\left(\mathcal{L}\right)
\right)
$
is given as maps
$$
 \begin{tikzcd}[arrow style=tikz,>=stealth, sep=65pt, every arrow/.append style={shift left}]
   \hom^0\left(\mathcal{L},\mathbb{L}\right)
   =
   \displaystyle\bigoplus_{p\in\chi^0\left(L,\mathbb{L}\right)}\left(\left.E\right|_p\right)^*
     \arrow{r}{\LocalPhi\left(\mathcal{L}\right)=\operatorm_1^{0,b}} 
   &
   \displaystyle\bigoplus_{s\in\chi^1\left(L,\mathbb{L}\right)}\left(\left.E\right|_s\right)^*
   =
   \hom^1\left(\mathcal{L},\mathbb{L}\right),
     \arrow{l}{\LocalPsi\left(\mathcal{L}\right)=\operatorm_1^{0,b}} 
 \end{tikzcd}
$$
whose
$
\left(\left(\left.E\right|_{s}\right)^*,\left(\left.E\right|_{p}\right)^*\right)
$-component
$
\operatorm_1^{0,b}
:
\left(\left.E\right|_{p}\right)^*
\rightarrow
\left(\left.E\right|_{s}\right)^*
$
for each
$p$, $s\in\chi\left(L,\mathbb{L}\right)$
is 
\begin{equation*}
\sum_{
\begin{matrix}
\\[-6mm]
\scriptscriptstyle
\left(x_1,X_1\right),\dots,\left(x_i,X_i\right)
\\[-1mm]
\scriptscriptstyle\in\left\{(x,X),(y,Y),(z,Z)\right\}
\end{matrix}
}
x_1\cdots x_i\
\sum_{u\in\mathcal{M}\left(p,X_1,\dots,X_i,\overline{s}\right)}
(-1)^
{
(i+1)\mathbb{1}_{\operatorname{o}\left(\mathbb{L}\right)\ne\operatorname{o}\left(\partial u\right)}
+
\#\left(\partial u \cap {\small\color{gray}\bigstar}_{\mathbb{L}}\right)
}
P\left(\left(\partial u\right)_0\right)^*,
\end{equation*}
where
$
P\left(\left(\partial u\right)_0\right)
\in
\Hom_{\field}\left(\left.E\right|_s,\left.E\right|_p\right)
$
is the parallel transport from $\left.E\right|_s$ to $\left.E\right|_p$
along the side of $u$ lying in $L$.

Taking its dual yields the map
$
-\operatorname{Tr}\left(\LocalF\left(\mathcal{L}\right)\right)
=
\left(
-\LocalPhi\left(\mathcal{L}\right)^*,
-\LocalPsi\left(\mathcal{L}\right)^*
\right)
$
given by
$$
 \begin{tikzcd}[arrow style=tikz,>=stealth, sep=65pt, every arrow/.append style={shift left}]
   \left(\hom^1\left(\mathcal{L},\mathbb{L}\right)\right)^*
   =
   \displaystyle\bigoplus_{s\in\chi^1\left(L,\mathbb{L}\right)}\left.E\right|_s
     \arrow{r}{-\LocalPhi\left(\mathcal{L}\right)^*} 
   &
   \displaystyle\bigoplus_{p\in\chi^0\left(L,\mathbb{L}\right)}\left.E\right|_p
   =
   \left(\hom^0\left(\mathcal{L},\mathbb{L}\right)\right)^*,
     \arrow{l}{-\LocalPsi\left(\mathcal{L}\right)^*} 
 \end{tikzcd}
$$
whose
$
\left(\left.E\right|_{p},\left.E\right|_{s}\right)
$-component
$
\left.E\right|_{s}
\rightarrow
\left.E\right|_{p}
$
for each
$p$, $s\in\chi\left(L,\mathbb{L}\right)$
is
\begin{equation*}
-
\sum_{
\begin{matrix}
\\[-6mm]
\scriptscriptstyle
\left(x_1,X_1\right),\dots,\left(x_i,X_i\right)
\\[-1mm]
\scriptscriptstyle\in\left\{(x,X),(y,Y),(z,Z)\right\}
\end{matrix}
}
x_1\cdots x_i\
\sum_{u\in\mathcal{M}\left(p,X_1,\dots,X_i,\overline{s}\right)}
(-1)^
{
(i+1)\mathbb{1}_{\operatorname{o}\left(\mathbb{L}\right)\ne\operatorname{o}\left(\partial u\right)}
+
\#\left(\partial u \cap {\small\color{gray}\bigstar}_{\mathbb{L}}\right)
}
P\left(\left(\partial u\right)_0\right).
\end{equation*}

On the other hand,
the opposite side
$
\mathcal{F}^{\imath\left(\mathbb{L}\right)}\left(\imath\left(\mathcal{L}\right)\right)
=
\left(
\Phi^{\imath\left(\mathbb{L}\right)}\left(\imath\left(\mathcal{L}\right)\right),
\Psi^{\imath\left(\mathbb{L}\right)}\left(\imath\left(\mathcal{L}\right)\right)
\right)
$
is given by
\begin{equation}\label{eqn:DualMaps}
 \begin{tikzcd}[arrow style=tikz,>=stealth, sep=85pt, every arrow/.append style={shift left}]
   \hom^0\left(\imath\left(\mathcal{L}\right),\imath\left(\mathbb{L}\right)\right)
   =
   \displaystyle\bigoplus_{p'\in\chi^0\left(\imath\left(L\right),\imath\left(\mathbb{L}\right)\right)}\left(\left.E^*\right|_{p'}\right)^*
     \arrow{r}{\Phi^{\imath\left(\mathbb{L}\right)}\left(\imath\left(\mathcal{L}\right)\right)=\operatorm_1^{0,-\imath(b)}} 
   &
   \displaystyle\bigoplus_{s'\in\chi^1\left(\imath\left(L\right),\imath\left(\mathbb{L}\right)\right)}\left(\left.E^*\right|_{s'}\right)^*
   =
   \hom^1\left(\imath\left(\mathcal{L}\right),\imath\left(\mathbb{L}\right)\right),
     \arrow{l}{\Psi^{\imath\left(\mathbb{L}\right)}\left(\imath\left(\mathcal{L}\right)\right)=\operatorm_1^{0,-\imath(b)}} 
 \end{tikzcd}
\end{equation}
whose
$
\left(\left(\left.E^*\right|_{s'}\right)^*,\left(\left.E^*\right|_{p'}\right)^*\right)
$-component
$
\operatorm_1^{0,-\imath(b)}
:
\left(\left.E^*\right|_{p'}\right)^*
\rightarrow
\left(\left.E^*\right|_{s'}\right)^*
$
for each
$p'$, $s'\in\chi\left(\imath\left(L\right),\imath\left(\mathbb{L}\right)\right)$
is
%
\begin{equation}\label{eqn:DualComponent}
\sum_{
\begin{matrix}
\\[-6mm]
\scriptscriptstyle
\left(x_1',X_1'\right),\dots,\left(x_i',X_i'\right)
\\[-1mm]
\scriptscriptstyle\in\left\{(-x,\imath(X)),(-y,\imath(Y)),(-z,\imath(Z))\right\}
\end{matrix}
}
x_1'\cdots x_i'\
\sum_{u'\in\mathcal{M}\left(p',X_1',\dots,X_i',\overline{s'}\right)}
(-1)^
{
(i+1)\mathbb{1}_{\operatorname{o}\left(\imath\left(\mathbb{L}\right)\right)\ne\operatorname{o}\left(\partial u'\right)}
+
\#\left(\partial u' \cap \imath\left({\small\color{gray}\bigstar}_{\mathbb{L}}\right)\right)
}
P\left(\left(\partial u'\right)_0\right)^*.
\end{equation}

Under the bijection between angles
$$
\chi^{\bullet}\left(L,\mathbb{L}\right)
\xleftrightarrow{1:1}
\chi^{\bullet+1}\left(\imath\left(L\right),\imath\left(\mathbb{L}\right)\right),
\quad
p\leftrightarrow\imath\left(\overline{p}\right)
\quad
\left(\bullet\in\mathbb{Z}_2\right),
$$
we can put
$p'=\imath\left(\overline{s}\right)$
and
$s'=\imath\left(\overline{p}\right)$
for some $p$, $s\in\chi\left(L,\mathbb{L}\right)$.
Each $X_i'$ is $\imath\left(X_i\right)$ for some $X_i\in\left\{X,Y,Z\right\}$.
There is also
a bijection
between deformed strips
$$
\mathcal{M}\left(p,X_i,\dots,X_1,\overline{s}\right)
\xleftrightarrow{1:1}
\mathcal{M}\left(\imath\left(\overline{s}\right),\imath\left(X_1\right),\dots,\imath\left(X_i\right),\imath\left(p\right)\right),
\quad
u\leftrightarrow\imath\left(u\right):=\imath\circ u
$$
as shown in Figure \ref{fig:StripInversion}.
Also using the identifications
$
\left(\left.E^*\right|_{\imath\left(\overline{s}\right)}\right)^*
=
\left.E\right|_s
$
and
$
\left(\left.E^*\right|_{\imath\left(\overline{p}\right)}\right)^*
=
\left.E\right|_p,
$
we can rewrite
(\ref{eqn:DualMaps})
and
(\ref{eqn:DualComponent})
as
$$
 \begin{tikzcd}[arrow style=tikz,>=stealth, sep=85pt, every arrow/.append style={shift left}]
   \displaystyle\bigoplus_{s\in\chi^1\left(L,\mathbb{L}\right)}\left.E\right|_s
     \arrow{r}{\Phi^{\imath\left(\mathbb{L}\right)}\left(\imath\left(\mathcal{L}\right)\right)=\operatorm_1^{0,-\imath(b)}} 
   &
   \displaystyle\bigoplus_{p\in\chi^0\left(L,\mathbb{L}\right)}\left.E\right|_p
     \arrow{l}{\Phi^{\imath\left(\mathbb{L}\right)}\left(\imath\left(\mathcal{L}\right)\right)=\operatorm_1^{0,-\imath(b)}} 
 \end{tikzcd}
$$
where
the
$
\left(\left.E\right|_{p},\left.E\right|_{s}\right)
$-component
$
\left.E\right|_{s}
\rightarrow
\left.E\right|_{p}
$
is
\begin{equation}\label{eqn:DualComponents2}
\sum_{
\begin{matrix}
\\[-6mm]
\scriptscriptstyle
\left(x_1,X_1\right),\dots,\left(x_i,X_i\right)
\\[-1mm]
\scriptscriptstyle\in\left\{(x,X),(y,Y),(z,Z)\right\}
\end{matrix}
}
\left(-x_1\right)\cdots \left(-x_i\right)\
\sum_{u\in\mathcal{M}\left(p,X_i,\dots,X_1,\overline{s}\right)}
(-1)^
{
(i+1)\mathbb{1}_{\operatorname{o}\left(\imath\left(\mathbb{L}\right)\right)\ne\operatorname{o}\left(\partial \left(\imath\left(u\right)\right)\right)}
+
\#\left(\partial \left(\imath\left(u\right)\right) \cap \imath\left({\small\color{gray}\bigstar}_{\mathbb{L}}\right)\right)
}
P\left(\left(\partial \left(\imath\left(u\right)\right)\right)_0\right)^*.
\end{equation}

The obvious relations
$$
\mathbb{1}_{\operatorname{o}\left(\mathbb{L}\right)\ne\operatorname{o}\left(\partial u\right)}
+
\mathbb{1}_{\operatorname{o}\left(\imath\left(\mathbb{L}\right)\right)\ne\operatorname{o}\left(\partial \left(\imath\left(u\right)\right)\right)}
=
1
\quad\text{and}\quad
\#\left(\partial u \cap {\small\color{gray}\bigstar}_{\mathbb{L}}\right)
=
\#\left(\partial \left(\imath\left(u\right)\right) \cap \imath\left({\small\color{gray}\bigstar}_{\mathbb{L}}\right)\right)
$$
and the fact that
$
P\left(\left(\partial \left(\imath\left(u\right)\right)\right)_0\right)
\in
\Hom_{\field}\left(\left.E^*\right|_p,\left.E^*\right|_s\right)
$
is the dual of
$
P\left(\left(\partial u\right)_0\right)
\in
\Hom_{\field}\left(\left.E\right|_s,\left.E\right|_p\right)
$
replace (\ref{eqn:DualComponents2})
again into
$$
-
\sum_{
\begin{matrix}
\\[-6mm]
\scriptscriptstyle
\left(x_1,X_1\right),\dots,\left(x_i,X_i\right)
\\[-1mm]
\scriptscriptstyle\in\left\{(x,X),(y,Y),(z,Z)\right\}
\end{matrix}
}
x_1\cdots x_i\
\sum_{u\in\mathcal{M}\left(p,X_i,\dots,X_1,\overline{s}\right)}
(-1)^
{
(i+1)\mathbb{1}_{\operatorname{o}\left(\mathbb{L}\right)\ne\operatorname{o}\left(\partial u\right)}
+
\#\left(\partial u \cap {\small\color{gray}\bigstar}_{\mathbb{L}}\right)
}
P\left(\left(\partial u\right)_0\right).
$$

This shows that two functors
are the same on the object level.
It is also straightforward to check that they coincide on the morphism level.
\end{proof}

\newpage

\subsubsection{Flip of loops and dual of modules}
\label{sec:FlipDual}

The commutativity of diagram (\ref{eqn:FlipDualDiagram})
induces the commutativity
of the left square in the following diagram
of ordinary categories and functors:
\begin{equation}\label{eqn:FlipDualDiagramOrdinary}
\begin{tikzcd}[arrow style=tikz,>=stealth,row sep=3em,column sep=3em] 
&
H^0\Fuk\left(\POP\right)
  \arrow[r,"\LocalF"]
  \arrow[d,swap,"\textnormal{flip}"]
  \arrow[d,"\imath"]
&
\underline{\MF}(xyz)
  \arrow[r,"\smat{\cok\\ \simeq}"]
  \arrow[d,swap,"\textnormal{transpose}"]
  \arrow[d,"-\operatorname{Tr}"]
&
\underline{\CM}(A)
  \arrow[d,swap,"\textnormal{dual}"]
  \arrow[d,"{\Hom_A\left(-,A\right)}"]
&
\\
&
H^0\Fuk\left(\POP\right)
  \arrow[r,"\LocalF"]
&
\underline{\MF}(xyz)
  \arrow[r,"\smat{\cok\\ \simeq}"]
&
\underline{\CM}(A)
&
\end{tikzcd}
\end{equation}

We check the commutativity of the right square in the general setting:

\begin{prop}\label{prop:CommutativityTransposeDual}
Let $S$ be
the power series ring
$
\field[[x_1,\dots,x_m]]
$
of $m$ variables,
$f\in S$ its nonzero element
and $A:=\left.S\right/(f)$
the quotient ring.
Then the following diagram
is commutative,
that is,
two compositions of functors are naturally isomorphic to each other:
\begin{equation*}
\begin{tikzcd}[arrow style=tikz,>=stealth,row sep=3em,column sep=3em] 
&
\underline{\MF}(f)
  \arrow[r,"\smat{\cok\\ \simeq}"]
  \arrow[d,swap,"\textnormal{transpose}"]
  \arrow[d,"-\operatorname{Tr}"]
&
\underline{\CM}(A)
  \arrow[d,swap,"\textnormal{dual}"]
  \arrow[d,"{\Hom_A\left(-,A\right)}"]
&
\\
&
\underline{\MF}(f)
  \arrow[r,"\smat{\cok\\ \simeq}"]
&
\underline{\CM}(A)
&
\end{tikzcd}
\end{equation*}

\end{prop}

\begin{proof}

Recall that under Eisenbud's equivalence (Theorem \ref{thm:EisenbudTheorem}),
a matrix factorization
$
\begin{tikzcd}[arrow style=tikz,>=stealth, sep=20pt, every arrow/.append style={shift left=0.5}]
   P^0
     \arrow{r}{\varphi}
   &
   P^1
     \arrow{l}{\psi}
\end{tikzcd}
$
of $f$
corresponds to a maximal Cohen-Macaulay $A$-module
$M:=\cok\underline{\varphi}$,
which admits a $2$-periodic free resolution
given by
$$
\begin{tikzcd}[arrow style=tikz,>=stealth,row sep=5em,column sep=2em] 
\cdots
  \arrow[r]
&
P^0\otimes_S A
  \arrow[r,"\underline{\varphi}"]
&
P^1\otimes_S A
  \arrow[r,"\underline{\psi}"]
&
P^0\otimes_S A
  \arrow[r,"\underline{\varphi}"]
&
P^1\otimes_S A
  \arrow[r,""]
&
M
  \arrow[r]
&
0.
\end{tikzcd}
$$
Taking $\Hom_A\left(-,A\right)$
yields
a $2$-periodic free resolution of
$\Hom_A\left(M,A\right)$
as
$$
\begin{tikzcd}[arrow style=tikz,>=stealth,row sep=5em,column sep=2em] 
0
  \arrow[r]
&
\Hom_A\left(M,A\right)
  \arrow[r,""]
&
\Hom_S\left(P^1,A\right)
  \arrow[r,"\underline{\varphi^*}"]
&
\Hom_A\left(P^0,A\right)
  \arrow[r,"\underline{\psi^*}"]
&
\Hom_A\left(P^1,A\right)
  \arrow[r,""]
&
\cdots,
\end{tikzcd}
$$
where
$
\underline{\varphi^*}
=
\varphi^*\otimes \id_A
:
\left(P^1\right)^*\otimes_S A\rightarrow \left(P^0\right)^*\otimes_S A
$
is induced from
$\varphi^*:\left(P^1\right)^*\rightarrow \left(P^0\right)^*$.
It gives natural isomorphisms
$$
\Hom_A\left(M,A\right)
\cong
\ker\underline{\varphi^*}
=
\im\underline{\psi^*}
\cong
\left.\Hom_A\left(P^0,A\right)\right/\ker\underline{\psi^*}
=
\left.\Hom_A\left(P^0,A\right)\right/\im\underline{\varphi^*}
=
\cok\underline{\varphi^*}
=
\cok\left(\underline{-\varphi^*}\right).
$$
\end{proof}

\subsubsection{Correspondence of canonical forms}
\label{sec:DualFlipCanonicalForms}

Flipping a loop with a local system,
taking transpose of a matrix factorization,
and
taking dual of a maximal Cohen-Macaulay modules
can all be given explicitly
in terms of loop/band data.

\begin{itemize}
\item
In $\Fuk\left(\POP\right)$,
let
$
\left(L,E,\nabla\right)
:=
\mathcal{L}\left(w',\primelambda,\primemu\right)
$
be the loop with a local system
corresponding to a loop datum
$
\left(w',\primelambda,\primemu\right).
$
Its flip is given by
$\left(\imath\left(L\right),E^*,\nabla^*\right)$.
Recall that the free homotopy class of the loop $L\left(w'\right)$
corresponding to the given normal loop word
$
w'=\left(
l_1',m_1',n_1',
\dots,
l_{\tau}',m_{\tau}',n_{\tau}'
\right)\in\mathbb{Z}^{3\tau}
$
is given by
$$
\left[L\left(w'\right)\right]
=
\left[
\alpha^{l_1'}\beta^{m_1'}\gamma^{n_1'}
\cdots
\alpha^{l_\tau'}\beta^{m_\tau'}\gamma^{n_\tau'}
\right],
$$
where $\alpha$, $\beta$, $\gamma$ are generators of $\pi_1\left(\POP\right)$ (Figure \ref{fig:gen}).
It is easy to see that
the free homotopy class of the flipped loop $\imath\left(L\left(w'\right)\right)$
is
$$
\left[\imath\left(L\left(w'\right)\right)\right]
=
\left[
\alpha^{1-l_1'}\beta^{1-m_1'}\gamma^{1-n_1'}
\cdots
\alpha^{1-l_\tau'}\beta^{1-m_\tau'}\gamma^{1-n_\tau'}
\right].
$$
Defining
a new normal loop word $1-w'$ as the normal form of the loop word
\footnote{
It is not normal a priori in general,
but Proposition \ref{prop:normalnormal} ensures that
one can deform it to the unique normal loop word by performing five operations in Lemma \ref{lem:equimove}
finitely many times.
}
$$
\left(
1-l_1',1-m_1',1-n_1',
\dots,
1-l_{\tau}',1-m_{\tau}',1-n_{\tau}'
\right)
\in\mathbb{Z}^{3\tau},
$$two loops $\imath\left(L\left(w'\right)\right)$
and $L\left(1-w'\right)$ are freely homotopic to each other.

The holonomy of $\left(E,\nabla\right)$ at some point 
is represented by a matrix $J_\primemu\left(\primelambda\right)$
(up to conjugacy).
It changes to $\left(J_{\primemu}\left(\primelambda\right)^T\right)^{-1}$
(which is similar to $J_{\primemu}\left(\primelambda\right)^{-1}$
by discussion in the proof of Theorem \ref{thm:LagMFCorrespondenceNondegenerate})
for the dual local system $\left(E^*,\nabla^*\right)$.
This is because
the parallel transport from
$\left.E^*\right|_p$ to $\left.E^*\right|_q$ is
$\left(P\left(L_{p\rightarrow q}\right)^*\right)^{-1}$,
where
$P\left(L_{p\rightarrow q}\right)$
is the parallel transport in $\left(E,\nabla\right)$
from $\left.E\right|_p$ to $\left.E\right|_q$.

Thus,
the flipped loop with a local system
$\left(\imath\left(L\right),E^*,\nabla^*\right)$
is isomorphic to
the canonical form
$\mathcal{L}\left(1-w',\primelambda^{-1},\primemu\right)$,
that is,
they give isomorphic matrix factorizations in
$
\underline{\MF}(xyz)
$
by Theorem \ref{thm:CylinderFreeMFFinite}.

\item
In $\underline{\MF}(xyz)$ (or $\MF(xyz)$),
consider the canonical form
$\varphi\left(w',\lambda,\primemu\right)$
corresponding to a loop datum $\left(w',\lambda,\primemu\right)$.
Its transpose still remains in the canonical form
up to equivalence of loop words and
bases change
$
J_{\primemu}\left(\lambda^{-1}\right)^T
\sim
J_{\primemu}\left(\lambda^{-1}\right)
\sim
J_{\primemu}\left(\lambda\right)^{-1},
$
which results in the canonical form $\varphi\left(1-w',\lambda^{-1},\primemu\right)$.

\item
In $\underline{\CM}(A)$ (or $\CM(A)$),
the canonical form given in terms of Definition \ref{defn:CanonicalFormMCM}
does not directly show the relation with taking dual.
However,
in $\Tri(A)$,
one can handle it in an algebraic way and see that
the dual of the canonical form $M\left(w,\lambda,\mu\right)$ corresponding to a band datum
$\left(w,\lambda,\mu\right)$
is isomorphic to
$M\left(-w,\lambda^{-1},\mu\right)$,
where $-w$ is the band word given by multiplying $-1$ to every entry of $w$.

%

\end{itemize}
The above discussions summarize to the following mappings (up to isomorphism) under the diagram (\ref{eqn:FlipDualDiagramOrdinary}),
while two rows are consistent with our main correspondence (\ref{eqn:MainCorrespondenceDiagram}):
\begin{equation*}\label{eqn:FlipDualDiagramOrdinaryCorrespondence}
\begin{tikzcd}[arrow style=tikz,>=stealth,row sep=3em,column sep=3em] 
&
\mathcal{L}\left(w',\primelambda,\primemu\right)
  \arrow[r,mapsto,"\LocalF"]
  \arrow[d,swap,leftrightarrow,"\textnormal{flip}"]
  \arrow[d,leftrightarrow,"\imath"]
&
\varphi_{\operatorname{(deg)}}\left(w',\lambda,\primemu\right)
  \arrow[r,mapsto,"\smat{\cok\\ \simeq}"]
  \arrow[d,swap,leftrightarrow,"\textnormal{transpose}"]
  \arrow[d,leftrightarrow,"-\operatorname{Tr}"]
&
M\left(w,\lambda,\mu\right)
  \arrow[d,swap,leftrightarrow,"\textnormal{dual}"]
  \arrow[d,leftrightarrow,"{\Hom_A\left(-,A\right)}"]
&
\\
&
\mathcal{L}\left(1-w',\primelambda^{-1},\primemu\right)
  \arrow[r,mapsto,"\LocalF"]
&
\varphi_{\operatorname{(deg)}}\left(1-w',\lambda^{-1},\primemu\right)
  \arrow[r,mapsto,"\smat{\cok\\ \simeq}"]
&
M\left(-w,\lambda^{-1},\mu\right)
&
\end{tikzcd}
\end{equation*}




\begin{ex}\label{ex:FlipDual}

The following shows the correspondence of
loops with a local system
$
\mathcal{L}\left((3,-2,2),\primelambda,1\right)
\leftrightarrow
\mathcal{L}\left((-2,3,-1),\primelambda^{-1},1\right),
$
matrix factorizations
$
\varphi\left((3,-2,2),\lambda,1\right)
\leftrightarrow
\varphi\left((-2,3,-1),\lambda^{-1},1\right),
$
and
maximal Cohen-Macaulay modules
$
M\left((2,-3,1),\lambda,1\right)
\leftrightarrow
M\left((-2,3,-1),\lambda^{-1},1\right)
$
(in $\Tri(A)$)
up to isomorphism,
where $\lambda=-\primelambda$.
\end{ex}
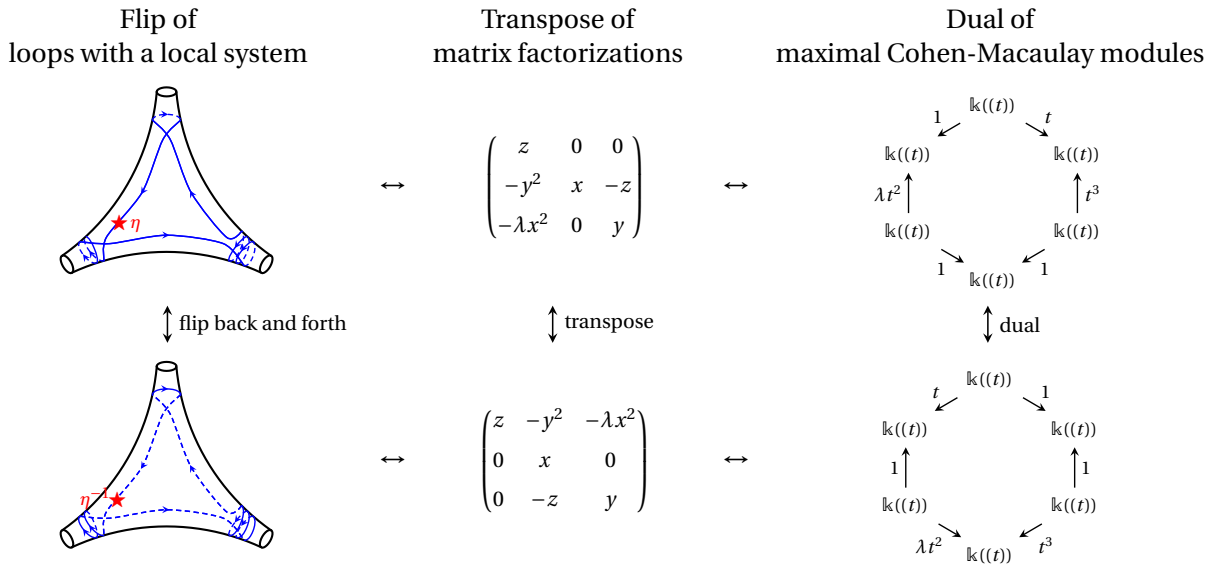
\vspace{-3mm}
\begin{figure}[H]
\hspace{-12mm}
$
\setlength\arraycolsep{5pt}

$
  }
&\leftrightarrow&
  \adjustbox{scale=0.9}{
  \begin{tikzcd}[ampersand replacement=\&, arrow style=tikz,>=stealth, sep=6pt] 
    \&
    {\scriptstyle \mathbb{k}((t))}^{}
      \arrow[ld, swap, "\color{black}1"]
      \arrow[rd, "\color{black}t"]
    \&
    \\[0mm]
    {\scriptstyle \mathbb{k}((t))}^{}
    \&
    \&
    {\scriptstyle \mathbb{k}((t))}^{}
    \\[4mm]
    {\scriptstyle \mathbb{k}((t))}^{}
      \arrow[u, "\color{black} \lambda t^2"]
      \arrow[rd, swap, "\color{black}1"]
    \&
    \&
    {\scriptstyle \mathbb{k}((t))}^{}
      \arrow[u, swap, "\color{black}t^3"]
      \arrow[ld, "\color{black}1"]
    \\[0mm]
    \&
    {\scriptstyle \mathbb{k}((t))}^{}
    \&
  \end{tikzcd}
  }
\\[-2mm]
\hspace{25mm}
\begin{tikzcd}[arrow style=tikz,>=stealth]
    \arrow[d, leftrightarrow, "\text{ flip back and forth}"]
    \\[-1mm]
    \phantom{.}
\end{tikzcd}
& &
\hspace{9mm}
\begin{tikzcd}[arrow style=tikz,>=stealth]
    \arrow[d, leftrightarrow, "\text{ transpose}"]
    \\[-1mm]
    \phantom{.}
\end{tikzcd}
& &
\hspace{5mm}
\begin{tikzcd}[arrow style=tikz,>=stealth]
    \arrow[d, leftrightarrow, "\text{ dual}"]
    \\[-1mm]
    \phantom{.}
\end{tikzcd}
\\[-1mm]

$
  }
&\leftrightarrow&
  \adjustbox{scale=0.9}{
  \begin{tikzcd}[ampersand replacement=\&, arrow style=tikz,>=stealth, sep=6pt] 
    \&
    {\scriptstyle \mathbb{k}((t))}^{}
      \arrow[ld, swap, "\color{black}t"]
      \arrow[rd, "\color{black}1"]
    \&
    \\[0mm]
    {\scriptstyle \mathbb{k}((t))}^{}
    \&
    \&
    {\scriptstyle \mathbb{k}((t))}^{}
    \\[4mm]
    {\scriptstyle \mathbb{k}((t))}^{}
      \arrow[u, "\color{black}1"]
      \arrow[rd, swap, "\color{black}\lambda t^2"]
    \&
    \&
    {\scriptstyle \mathbb{k}((t))}^{}
      \arrow[u, swap, "\color{black}1"]
      \arrow[ld, "\color{black}t^3"]
    \\[0mm]
    \&
    {\scriptstyle \mathbb{k}((t))}^{}
    \&
  \end{tikzcd}
  }
\end{matrix}
$
\label{fig:FlipDual}
\end{figure}


\subsection{Reverse of loops and shift of modules}
Orientation-reversing of loops induces another auto-equivalence on $\Fuk\left(\Sigma\right)$
(\S \ref{sec:ReversingLoops}).
We also define the \emph{switching functor} in $\MF_{\AI}(f)$
(\S \ref{sec:MFShifting}),
and show that,
in our situation,
two $\AI$-functors are related under the localized mirror functor
(\S \ref{sec:ReversingShift}).
They boil down to shift functors of triangulated categories,
and therefore,
are also related with the shift functor of $\underline{\CM}(A)$
(\S \ref{sec:ReversingShiftModule}).
We give an algorithm to compute them in terms of loop/band data
(\S \ref{sec:ReversingShiftCanonicalForms}).

\subsubsection{Orientation-reversing of loops and auto-equivalence on the Fukaya category}
\label{sec:ReversingLoops}

Let $\left(\Sigma,\omega\right)$ be a $2$-dimensional symplectic manifolds
(possibly with boundary) of finite type
(as in \S \ref{sec:CptFukSurface}).
There is an obvious symmetry of objects in $\Fuk\left(\Sigma\right)$,
namely,
we can reverse the orientation of the underlying loop of every object.
It results in an auto-equivalence on $\Fuk\left(\Sigma\right)$
given by a (covariant) $\AI$-functor
$
\jmath
:=
\left\{\jmath_k\right\}_{k\ge0}
:
\Fuk\left(\Sigma\right)\rightarrow\Fuk\left(\Sigma\right)
$
we define
now:

Any object
$\mathcal{L}:=\left(L,E,\nabla\right)$
of $\Fuk\left(\Sigma\right)$
consists of
a loop
$L:S^1\rightarrow \Sigma$,
a finite-rank $\field$-vector bundle $E\rightarrow S^1$, 
and a flat connection $\nabla$ on $E$.
We define its image under the functor $\jmath$ as the triple
$$
\jmath_0\left(\mathcal{L}\right)
:=
\left(
\jmath\left(L\right):=
L\circ\kappa,
\kappa^* E,
\kappa^*\nabla
\right),
$$
where
$
\kappa:S^1\rightarrow S^1,
$
$
e^{2\pi it}\mapsto e^{-2\pi it}
$
denotes the orientation-reversing map.

Note that reversing the orientation of loops doesn't affect their (self-)intersections.
Therefore,
for two objects 
$\mathcal{L}_i:=\left(L_i,E_i,\nabla_i\right)$
($i\in\left\{0,1\right\}$),
the sets
$
\chi^\bullet\left(L_0,L_1\right)
$
and
$
\chi^\bullet\left(\jmath\left(L_0\right),\jmath\left(L_1\right)\right)
$
$
\left(\bullet\in\mathbb{Z}_2\right)
$
are identified,
as shown in Figure \ref{fig:MorphismsReverse}.
The fibers $\left.E_i\right|_p$
and
$\left.\kappa^* E_i\right|_p$
over the preimages (in $S^1$) of the point $p\in\Sigma$
under $L_i$ and $\jmath(L_i)$, respectively,
are also naturally identified.
So there is also a natural identification between
$$
\bigoplus_{p\in\chi^{\bullet}\left(L_0,L_1\right)}
\Hom_{\field}
\left(
\left.E_0\right|_p,
\left.E_1\right|_p
\right)
\quad\text{and}\quad
\bigoplus_{p\in\chi^{\bullet}\left(\jmath\left(L_0\right),\jmath\left(L_1\right)\right)}
\Hom_{\field}
\left(
\left.\kappa^* E_0\right|_p,
\left.\kappa^* E_1\right|_p
\right),
$$

We define
$
\jmath_1
:
\hom\left(\mathcal{L}_0,\mathcal{L}_1\right)
\rightarrow
\hom\left(\jmath\left(\mathcal{L}_0\right),\jmath\left(\mathcal{L}_1\right)\right)
$
by
\begin{equation}\label{eqn:MorphismSending}
f\in \Hom_{\field}\left(\left.E_0\right|_p,\left.E_1\right|_p\right)
\mapsto
(-1)^{\left|f\right|}
f\in\Hom_{\field}\left(\kappa^*\left.E_0\right|_p,\kappa^*\left.E_1\right|_p\right)
\end{equation}
for a base morphism $f$ over $p\in\chi^\bullet\left(L_0,L_1\right)$
($\bullet\in\mathbb{Z}_2$),
and then linearly extend it to any morphisms.
Higher components $\jmath_{k\ge2}$
are defined to be zero.


\begin{figure}[H]
   \begin{minipage}{0.49\textwidth}
     \centering
$
\setlength\arraycolsep{2pt}

$
\captionsetup{width=1\linewidth}
\caption{$p$, $\overline{p}\in\chi\left(L_0,L_1\right)=\chi\left(\jmath\left(\mathcal{L}_0\right),\jmath\left(\mathcal{L}_1\right)\right)$}
\label{fig:MorphismsReverse}
   \end{minipage}\hfill
   \begin{minipage}{0.51\textwidth}
     \centering
$
\setlength\arraycolsep{2pt}
%
\end{matrix}
$
\captionsetup{width=1\linewidth}
\caption{A polygon $u\in\mathcal{M}\left(p_1,\dots,p_k,\overline{q}\right)$}
\label{fig:PolygonInversion}
   \end{minipage}\hfill
\end{figure}


\begin{prop}
The functor
$\jmath:\Fuk\left(\Sigma\right)\rightarrow\Fuk\left(\Sigma\right)$
defined above is indeed a covariant $\AI$-functor.


\end{prop}

\begin{proof}

As $\jmath_{k\ge2}=0$,
the required $\AI$-relations (\ref{eqn:AI-relationsFunctor})
simplify to
\begin{equation}\label{eqn:SimplifiedAIRelationsCovariant}
\operatorm_n\left(\jmath_1\left(f_1\right),\dots,\jmath_1\left(f_n\right)\right)
=
\jmath_1\left(\operatorm_n\left(f_1,\dots,f_n\right)\right)
\end{equation}
for
$
n\in\mathbb{Z}_{\ge1}
$
and
$
f_i\in
\hom^\bullet\left(\mathcal{L}_{i-1},\mathcal{L}_{i}\right)
$
($i\in\left\{1,\dots,n\right\}$, $\bullet\in\mathbb{Z}_2$).
For 
$p_i\in \chi^\bullet\left(L_{i-1},L_{i}\right)$
and
$f_i\in\Hom_{\field}\left(\left.E_{i-1}\right|_{p_i},\left.E_{i}\right|_{p_i}\right)$
($i\in\left\{1,\dots,n\right\}$, $\bullet\in\mathbb{Z}_2$),
recalling the definition of $\operatorm_n$ in (\ref{eqn:OperatormDefinition}),
both the left and right sides are contributed by
the same polygons
$
u
\in
\mathcal{M}\left(p_1,\dots,p_k,\overline{q}\right)
$
for
$
q\in
\chi\left(L_0,L_k\right).
$
But for each polygon $u$,
its bounding loops have different orientations in both sides.
It plays a role only when we compute
$
\sign(u),
$
whose quotient
in both sides is given by
$
(-1)^{
\sum_{i=1}^k \left|p_i\right| + \left|q\right|
}
$
following the sign rule (\ref{eqn:sign(u)}).
It cancels all the signs that occur when taking $\jmath_1$,
which confirms (\ref{eqn:SimplifiedAIRelationsCovariant}).
%
%
%
%
%
\end{proof}

\subsubsection{
Switching functor on $\MF_{\AI}(f)$}
\label{sec:MFShifting}
Let $S$ be the power series ring $\field[[x_1,\dots,x_m]]$ of $m$ variables,
and $f\in S$ its nonzero element.
Switching two factors in
matrix factorizations of $f$
gives rise to
a
covariant $\AI$-functor
$
[1]
:
\MF_{\AI}(f)\rightarrow \MF_{\AI}(f),
$
called the \textbf{switching functor}
\footnote{
It boils down to the shift functor of the triangulated category $\underline{\MF}(f)$.
},
which we now define:

For an object
$
\begin{tikzcd}[arrow style=tikz,>=stealth, sep=20pt, every arrow/.append style={shift left=0.5}]
   P^0
     \arrow{r}{\varphi}
   &
   P^1
     \arrow{l}{\psi}
\end{tikzcd}
$
with free $S$-modules $P^0$, $P^1$,
we define its
image
as
$
\begin{tikzcd}[arrow style=tikz,>=stealth, sep=20pt, every arrow/.append style={shift left=0.5}]
   P^1
     \arrow{r}{\psi}
   &
   P^0.
     \arrow{l}{\varphi}
\end{tikzcd}
$
It defines the functor
$[1]:=\left\{[1]_k\right\}_{k\ge0}$
on the object level as
$$
[1]_0
:
\operatorname{Ob}\left(\MF_{\AI}(f)\right)
\rightarrow
\operatorname{Ob}\left(\MF_{\AI}(f)\right),
\quad
\left(\varphi,\psi\right)\rightarrow\left(\psi,\varphi\right).
$$

Given two matrix factorizations
$
\begin{tikzcd}[arrow style=tikz,>=stealth, sep=20pt, every arrow/.append style={shift left=0.5}]
   P_1^0
     \arrow{r}{\varphi1}
   &
   P_1^1
     \arrow{l}{\psi_1}
\end{tikzcd}
$
and
$
\begin{tikzcd}[arrow style=tikz,>=stealth, sep=20pt, every arrow/.append style={shift left=0.5}]
   P_0^0
     \arrow{r}{\varphi_0}
   &
   P_0^1
     \arrow{l}{\psi_0}
\end{tikzcd}
$
of $f$
and
an {even-degree} morphism
$
\left(\alpha:P_1^0\rightarrow P_0^0,\
\beta:P_1^1\rightarrow P_0^1\right)
$
(resp. an odd-degree morphism
$
\big(\gamma:P_1^0\rightarrow P_0^1,\
\delta:P_1^1\rightarrow P_0^0\big)
$),
we define its image under $[1]_1$ as
$$
[1]_1
:
\hom\left(\left(\varphi_1,\psi_1\right),\left(\varphi_0,\psi_0\right)\right)
\rightarrow
\hom\left(\left(\psi_1,\varphi_1\right),\left(\psi_0,\varphi_0\right)\right),
\quad
\left(\alpha,\beta\right)
\mapsto
\left(\beta,\alpha\right)
\quad
\left(
\text{resp. }
\left(\gamma,\delta\right)
\mapsto
\left(\delta,\gamma\right)
\right).
$$
The higher components $[1]_{k\ge2}$ are defined to be zero.
Then it is straightforward to check that $[1]$ is
a covariant $\AI$-functor.

\subsubsection{Orientation-reversing of loops and switching of matrix factorizations}
\label{sec:ReversingShift}

Now in our situation, we have the diagram of $\AI$-categories and functors: 
\begin{equation}\label{eqn:ReversingShiftingDiagram}
\begin{tikzcd}[arrow style=tikz,>=stealth,row sep=2.71em,column sep=3em] 
&
\Fuk\left(\POP\right)
  \arrow[r,"\LocalF"]
  \arrow[d,swap,"\smat{\textnormal{orientation}\\ \text{-reversing}}"]
  \arrow[d,"\jmath"]
&
\MF_{\AI}(xyz)
  \arrow[d,swap,"\smat{\textnormal{switching}\\ \textnormal{two factors}}"]
  \arrow[d,"{[1]}"]
&
\\
&
\Fuk\left(\POP\right)
  \arrow[r,"\LocalF"]
&
\MF_{\AI}(xyz)
&
\end{tikzcd}
\end{equation}
In this subsection,
we will show that the diagram commutes,
more precisely:

\begin{prop}\label{prop:CommutativityFlipTranspose}

Two $\AI$-functors
$
\LocalF
\circ\jmath
$
and
$
[1]\circ\LocalF
$
are the same.

\end{prop}

\begin{proof}

Recall from \S \ref{sec:LMFComputationSubsection} that
$
\LocalF\left(\mathcal{L}\right)
=
\left(
\LocalPhi\left(\mathcal{L}\right),
\LocalPsi\left(\mathcal{L}\right)
\right)
$
is given as maps
\begin{equation}\label{eqn:MFRecall}
 \begin{tikzcd}[arrow style=tikz,>=stealth, sep=65pt, every arrow/.append style={shift left}]
   \hom^0\left(\mathcal{L},\mathbb{L}\right)
   =
   \displaystyle\bigoplus_{p\in\chi^0\left(L,\mathbb{L}\right)}\left(\left.E\right|_p\right)^*
     \arrow{r}{\LocalPhi\left(\mathcal{L}\right)=\operatorm_1^{0,b}} 
   &
   \displaystyle\bigoplus_{s\in\chi^1\left(L,\mathbb{L}\right)}\left(\left.E\right|_s\right)^*
   =
   \hom^1\left(\mathcal{L},\mathbb{L}\right),
     \arrow{l}{\LocalPsi\left(\mathcal{L}\right)=\operatorm_1^{0,b}} 
 \end{tikzcd}
\end{equation}
whose
$
\left(\left(\left.E\right|_{s}\right)^*,\left(\left.E\right|_{p}\right)^*\right)
$-component
$
\operatorm_1^{0,b}
:
\left(\left.E\right|_{p}\right)^*
\rightarrow
\left(\left.E\right|_{s}\right)^*
$
for each
$p$, $s\in\chi\left(L,\mathbb{L}\right)$
is 
\begin{equation*}
\sum_{
\begin{matrix}
\\[-6mm]
\scriptscriptstyle
\left(x_1,X_1\right),\dots,\left(x_i,X_i\right)
\\[-1mm]
\scriptscriptstyle\in\left\{(x,X),(y,Y),(z,Z)\right\}
\end{matrix}
}
x_1\cdots x_i\
\sum_{u\in\mathcal{M}\left(p,X_1,\dots,X_i,\overline{s}\right)}
(-1)^
{
(i+1)\mathbb{1}_{\operatorname{o}\left(\mathbb{L}\right)\ne\operatorname{o}\left(\partial u\right)}
+
\#\left(\partial u \cap {\small\color{gray}\bigstar}_{\mathbb{L}}\right)
}
P\left(\left(\partial u\right)_0\right)^*,
\end{equation*}
where
$
P\left(\left(\partial u\right)_0\right)
\in
\Hom_{\field}\left(\left.E\right|_s,\left.E\right|_p\right)
$
is the parallel transport from $\left.E\right|_s$ to $\left.E\right|_p$
along the side of $u$ lying in $L$.

Switching positions of the left and right sides in
(\ref{eqn:MFRecall}),
$
\LocalF\left(\mathcal{L}\right)[1]
$
is given by maps
\begin{equation*}
 \begin{tikzcd}[arrow style=tikz,>=stealth, sep=65pt, every arrow/.append style={shift left}]
   \hom^1\left(\mathcal{L},\mathbb{L}\right)
   =
     \arrow{r}{\LocalPsi\left(\mathcal{L}\right)=\operatorm_1^{0,b}} 
   \displaystyle\bigoplus_{s\in\chi^1\left(L,\mathbb{L}\right)}\left(\left.E\right|_s\right)^*
   &
   \displaystyle\bigoplus_{p\in\chi^0\left(L,\mathbb{L}\right)}\left(\left.E\right|_p\right)^*
   =
   \hom^0\left(\mathcal{L},\mathbb{L}\right),
     \arrow{l}{\LocalPhi\left(\mathcal{L}\right)=\operatorm_1^{0,b}} 
 \end{tikzcd}
\end{equation*}
whose
$
\left(\left(\left.E\right|_{p}\right)^*,\left(\left.E\right|_{s}\right)^*\right)
$-component
$
\operatorm_1^{0,b}
:
\left(\left.E\right|_{s}\right)^*
\rightarrow
\left(\left.E\right|_{p}\right)^*
$
for each
$s$, $p\in\chi\left(L,\mathbb{L}\right)$
is 
\begin{equation*}
\sum_{
\begin{matrix}
\\[-6mm]
\scriptscriptstyle
\left(x_1,X_1\right),\dots,\left(x_i,X_i\right)
\\[-1mm]
\scriptscriptstyle\in\left\{(x,X),(y,Y),(z,Z)\right\}
\end{matrix}
}
x_1\cdots x_i\
\sum_{u\in\mathcal{M}\left(s,X_1,\dots,X_i,\overline{p}\right)}
(-1)^
{
(i+1)\mathbb{1}_{\operatorname{o}\left(\mathbb{L}\right)\ne\operatorname{o}\left(\partial u\right)}
+
\#\left(\partial u \cap {\small\color{gray}\bigstar}_{\mathbb{L}}\right)
}
P\left(\left(\partial u\right)_0\right)^*,
\end{equation*} 
where
$
P\left(\left(\partial u\right)_0\right)
\in
\Hom_{\field}\left(\left.E\right|_p,\left.E\right|_s\right)
$
is the parallel transport from $\left.E\right|_p$ to $\left.E\right|_s$
along the side of $u$ lying in $L$.

On the other hand,
the opposite side
$
\LocalF\left(\jmath\left(\mathcal{L}\right)\right)
=
\left(
\LocalPhi\left(\jmath\left(\mathcal{L}\right)\right),
\LocalPsi\left(\jmath\left(\mathcal{L}\right)\right)
\right)
$
is given by
\begin{equation}\label{eqn:DualMaps}
 \begin{tikzcd}[arrow style=tikz,>=stealth, sep=85pt, every arrow/.append style={shift left}]
   \hom^0\left(\jmath\left(\mathcal{L}\right),\mathbb{L}\right)
   =
   \displaystyle\bigoplus_{p'\in\chi^0\left(\jmath\left(L\right),\mathbb{L}\right)}\left(\left.\kappa^*E\right|_{p'}\right)^*
     \arrow{r}{\LocalPhi\left(\jmath\left(\mathcal{L}\right)\right)=\operatorm_1^{0,b}} 
   &
   \displaystyle\bigoplus_{s'\in\chi^1\left(\jmath\left(L\right),\mathbb{L}\right)}\left(\left.\kappa^*E\right|_{s'}\right)^*
   =
   \hom^1\left(\jmath\left(\mathcal{L}\right),\mathbb{L}\right),
     \arrow{l}{\LocalPsi\left(\jmath\left(\mathcal{L}\right)\right)=\operatorm_1^{0,b}} 
 \end{tikzcd}
\end{equation}
whose
$
\left(\left(\left.E^*\right|_{s'}\right)^*,\left(\left.E^*\right|_{p'}\right)^*\right)
$-component
$
\operatorm_1^{0,b}
:
\left(\left.\kappa^* E\right|_{p'}\right)^*
\rightarrow
\left(\left.\kappa^* E\right|_{s'}\right)^*
$
for each
$p'$, $s'\in\chi\left(\jmath\left(L\right),\mathbb{L}\right)$
is
\begin{equation}\label{eqn:ReversingComponents}
\sum_{
\begin{matrix}
\\[-6mm]
\scriptscriptstyle
\left(x_1,X_1\right),\dots,\left(x_i,X_i\right)
\\[-1mm]
\scriptscriptstyle\in\left\{(x,X),(y,Y),(z,Z)\right\}
\end{matrix}
}
x_1\cdots x_i\
\sum_{u\in\mathcal{M}\left(p',X_1,\dots,X_i,\overline{s'}\right)}
(-1)^
{
(i+1)\mathbb{1}_{\operatorname{o}\left(\mathbb{L}\right)\ne\operatorname{o}\left(\partial u\right)}
+
\#\left(\partial u \cap {\small\color{gray}\bigstar}_{\mathbb{L}}\right)
}
P\left(\left(\partial u\right)_0\right)^*,
\end{equation}
where
$
P\left(\left(\partial u\right)_0\right)
\in
\Hom_{\field}\left(\left.\kappa^* E\right|_{s'},\left.\kappa^* E\right|_{p'}\right)
$
is the parallel transport from $\left.\kappa^* E\right|_{s'}$ to $\left.\kappa^* E\right|_{p'}$
along the side of $u$ lying in $L$.


Note that reversing the orientation of $L$ doesn't affect its intersection with $\mathbb{L}$,
but their        degree changes.
Therefore,
we can replace
$p'\in\chi^0\left(\jmath\left(L\right),\mathbb{L}\right)$
and
$s'\in\chi^1\left(\jmath\left(L\right),\mathbb{L}\right)$
with
$s\in \chi^1\left(L,\mathbb{L}\right)$
and
$p\in \chi^0\left(L,\mathbb{L}\right)$,
respectively.
The fibers
$\left.\kappa^* E\right|_{p'}$
and
$\left.\kappa^* E\right|_{s'}$
are also identified with
$\left.E\right|_s$
and
$\left.E\right|_p$,
respectively.
So we can rewrite
(\ref{eqn:DualMaps})
and
(\ref{eqn:ReversingComponents})
as
\begin{equation*}
 \begin{tikzcd}[arrow style=tikz,>=stealth, sep=85pt, every arrow/.append style={shift left}]
   \displaystyle\bigoplus_{s\in \chi^1\left(L,\mathbb{L}\right)}\left(\left.E\right|_s\right)^*
     \arrow{r}{\LocalPhi\left(\jmath\left(\mathcal{L}\right)\right)=\operatorm_1^{0,b}} 
   &
   \displaystyle\bigoplus_{p\in \chi^0\left(L,\mathbb{L}\right)}\left(\left.E\right|_p\right)^*
     \arrow{l}{\LocalPsi\left(\jmath\left(\mathcal{L}\right)\right)=\operatorm_1^{0,b}} 
 \end{tikzcd}
\end{equation*}
whose
$
\left(\left(\left.E\right|_p\right)^*,\left(\left.E\right|_s\right)^*\right)
$-component
$
\operatorm_1^{0,b}
:
\left(\left.E\right|_s\right)^*
\rightarrow
\left(\left.E\right|_p\right)^*
$
is
\begin{equation*}
\sum_{
\begin{matrix}
\\[-6mm]
\scriptscriptstyle
\left(x_1,X_1\right),\dots,\left(x_i,X_i\right)
\\[-1mm]
\scriptscriptstyle\in\left\{(x,X),(y,Y),(z,Z)\right\}
\end{matrix}
}
x_1\cdots x_i\
\sum_{u\in\mathcal{M}\left(s,X_1,\dots,X_i,\overline{p}\right)}
(-1)^
{
(i+1)\mathbb{1}_{\operatorname{o}\left(\mathbb{L}\right)\ne\operatorname{o}\left(\partial u\right)}
+
\#\left(\partial u \cap {\small\color{gray}\bigstar}_{\mathbb{L}}\right)
}
P\left(\left(\partial u\right)_0\right)^*,
\end{equation*}
where
$
P\left(\left(\partial u\right)_0\right)
\in
\Hom_{\field}\left(\left.E\right|_p,\left.E\right|_s\right)
$
is the parallel transport from $\left.E\right|_p$ to $\left.E\right|_s$
along the side of $u$ lying in $L$.

Notice that it is the same expression with 
$
\LocalF\left(\mathcal{L}\right)[1],
$
especially because $\sign(u)$ is not relevant to the orientation of $L$.
This shows that two functors are the same on the object level.
It is also straightforward to check that they coincide on the morphism level.
\end{proof}

\subsubsection{Orientation-reversing of loops and shift of modules}
\label{sec:ReversingShiftModule}


The commutativity of diagram (\ref{eqn:ReversingShiftingDiagram}) induces
the commutativity of the left square in the following diagram of ordinary categories and functors:
\begin{equation}\label{eqn:ReversingTranslateDiagramOrdinary}
\begin{tikzcd}[arrow style=tikz,>=stealth,row sep=3em,column sep=3em] 
&
H^0\Fuk\left(\POP\right)
  \arrow[r,"\LocalF"]
  \arrow[d,swap,"\smat{\textnormal{orientation}\\ \textnormal{reversing}}"]
  \arrow[d,"\jmath"]
&
\underline{\MF}(xyz)
  \arrow[r,"\smat{\cok\\ \simeq}"]
  \arrow[d,swap,"\smat{\textnormal{switching}\\ \textnormal{two factors}}"]
  \arrow[d,"{\left[1\right]}"]
&
\underline{\CM}(A)
  \arrow[d,swap,"\textnormal{shift}"]
  \arrow[d,"{\left[1\right]}"]
&
\\
&
H^0\Fuk\left(\POP\right)
  \arrow[r,"\LocalF"]
&
\underline{\MF}(xyz)
  \arrow[r,"\smat{\cok\\ \simeq}"]
&
\underline{\CM}(A)
&
\end{tikzcd}
\end{equation}

We explain the commutativity of the right square in the general setting:

\begin{defn}\cite{Buc87}
Let $\left(A,\mathfrak{m}\right)$ be a Noetherian local ring
and $M\in\underline{\CM}(A)$ be a maximal Cohen-Macaulay module over $A$.
%
%
%
Choose an injection
$i:M\rightarrow Q$
of $M$ into a finitely generated projective $A$-module $Q$
such that its cokernel is still maximal Cohen-Macaulay.
We define
the \textnormal{\textbf{shift}}
\footnote{
It was called \textbf{translate} in \cite{Buc87},
and \textbf{AR translation} in \cite{Yo}. 
}
of $M$ as
$$
M[1]:=\cok(i),
$$
which is uniquely determined as an object of $\underline{\CM}(A)$ up to isomorphism.

\end{defn}

In the case of hypersurface singularities,
it can be explicitly given in terms of matrix factorizations
under Eisenbud's equivalence:

\begin{prop}\label{prop:CommutativitySwitchingShift}
Let $S$ be
the power series ring
$
\field[[x_1,\dots,x_m]]
$
of $m$ variables,
$f\in S$ its nonzero element
and $A:=\left.S\right/(f)$
the quotient ring.
Then the following diagram
is commutative,
that is,
two compositions of functors are naturally isomorphic to each other:
\begin{equation*}
\begin{tikzcd}[arrow style=tikz,>=stealth,row sep=3em,column sep=3em] 
&
\underline{\MF}(f)
  \arrow[r,"\smat{\cok\\ \simeq}"]
  \arrow[d,swap,"\smat{\textnormal{switching}\\ \textnormal{two factors}}"]
  \arrow[d,"{\left[1\right]}"]&
\underline{\CM}(A)
  \arrow[d,swap,"\textnormal{shift}"]
  \arrow[d,"{\left[1\right]}"]
&
\\
&
\underline{\MF}(f)
  \arrow[r,"\smat{\cok\\ \simeq}"]
&
\underline{\CM}(A)
&
\end{tikzcd}
\end{equation*}

\end{prop}

\begin{proof}
Recall that under Eisenbud's equivalence (Theorem \ref{thm:EisenbudTheorem}),
a matrix factorization
$
\begin{tikzcd}[arrow style=tikz,>=stealth, sep=20pt, every arrow/.append style={shift left=0.5}]
   P^0
     \arrow{r}{\varphi}
   &
   P^1
     \arrow{l}{\psi}
\end{tikzcd}
$
of $f$
corresponds to a maximal Cohen-Macaulay $A$-module
$M:=\cok\underline{\varphi}$,
which admits a $2$-periodic free resolution
given by
$$
\begin{tikzcd}[arrow style=tikz,>=stealth,row sep=5em,column sep=2em] 
\cdots
  \arrow[r]
&
P^0\otimes_S A
  \arrow[r,"\underline{\varphi}"]
&
P^1\otimes_S A
  \arrow[r,"\underline{\psi}"]
&
P^0\otimes_S A
  \arrow[r,"\underline{\varphi}"]
&
P^1\otimes_S A
  \arrow[r,""]
&
M
  \arrow[r]
&
0.
\end{tikzcd}
$$
From this we have natural isomorphisms
$$
M
=
\cok\underline{\varphi}
=
\left.\left(P^1\otimes_S A\right)\right/\im\underline{\varphi}
=
\left.\left(P^1\otimes_S A\right)\right/\ker\underline{\psi}
\cong
\im\underline{\psi}
=
\ker\underline{\varphi},
$$
and hence there is a natural embedding of $M$
into a finitely generated free $A$-module
$P^0\otimes_S A$:
$$
i
:
M\cong\ker\underline{\varphi}
\rightarrow
P^0\otimes_S\ A.
$$
Taking its cokernel gives
$$
M[1]
=
\cok(i)
=
\left.\left(P^0\otimes_S A\right)\right/\ker\underline{\varphi}
\cong
\left.\left(P^0\otimes_S A\right)\right/\im\underline{\psi}
=
\cok\underline{\psi},
$$
which is also maximal Cohen-Macaulay,
being the image of the switched matrix factorization $\left(\psi,\varphi\right)$ under Eisenbud's equivalence.
\end{proof}

\subsubsection{Correspondence of canonical forms}
\label{sec:ReversingShiftCanonicalForms}

Orientation-reversing of a loop with a local system,
switching two factors of a matrix factorization,
and
the shift of a maximal Cohen-Macaulay modules
can all be given explicitly
in terms of loop/band data.

\begin{itemize}
\item
In $\Fuk\left(\POP\right)$,
let
$
\left(L,E,\nabla\right)
:=
\mathcal{L}\left(w',\primelambda,\primemu\right)
$
be the loop with a local system
corresponding to a loop datum
$
\left(w',\primelambda,\primemu\right).
$
Its orientation-reverse is given by
$\left(\jmath\left(L\right),\kappa^*E,\kappa^*\nabla\right)$.
Recall that the free homotopy class of the loop $L\left(w'\right)$
corresponding to the given normal loop word
$
w'=\left(
l_1',m_1',n_1',
\dots,
l_{\tau}',m_{\tau}',n_{\tau}'
\right)\in\mathbb{Z}^{3\tau}
$
is given by
$$
\left[L\left(w'\right)\right]
=
\left[
\alpha^{l_1'}\beta^{m_1'}\gamma^{n_1'}
\cdots
\alpha^{l_\tau'}\beta^{m_\tau'}\gamma^{n_\tau'}
\right],
$$
where $\alpha$, $\beta$, $\gamma$ are generators of $\pi_1\left(\POP\right)$ (Figure \ref{fig:gen}).
The free homotopy class of the orientation-reversed loop $\jmath\left(L\left(w'\right)\right)$
is
\begin{align*}
\left[\jmath\left(L\left(w'\right)\right)\right]
&=
\left[
\gamma^{-n_\tau'}\beta^{-m_\tau'}\alpha^{-l_\tau'}\cdots
\gamma^{-n_1'}\beta^{-m_1'}\alpha^{-l_1'}
\right].
\\
&=
\left[
\alpha^{-l_1'}\beta^0\gamma^{-n_\tau'}
\alpha^0\beta^{-m_\tau'}\gamma^0
\alpha^{-l_\tau'}\beta^0
\cdots
\gamma^{-n_1'}
\alpha^0\beta^{-m_1'}\gamma^0
\right].
\end{align*}
Define
a new normal loop word $w'[1]$, the \textbf{reverse} of $w'$,
as the normal form of the loop word
\footnote{
It is not normal a priori in general,
but Proposition \ref{prop:normalnormal} ensures that
one can deform it to the unique normal loop word by performing five operations in Lemma \ref{lem:equimove}
finitely many times.
}
\begin{equation}\label{eqn:ReversingLoopWord}
\left(-l_1',0,-n_\tau',0,-m_\tau',0,-l_\tau',0,\dots,-n_1',0,-m_1',0\right)
\in\mathbb{Z}^{6\tau},
\end{equation}
then two loops $\jmath\left(L\left(w'\right)\right)$
and $L\left(w'[1]\right)$ are freely homotopic to each other.

The holonomy of $\left(E,\nabla\right)$ at some point 
is represented by a matrix $J_\primemu\left(\primelambda\right)$
(up to conjugacy).
It changes to $J_{\primemu}\left(\primelambda\right)^{-1}$
for the pull-back local system $\left(\kappa^*,\kappa^*\nabla\right)$.

Thus,
the orientation-reversed loop with a local system
$\left(\jmath\left(L\right),\kappa^*E,\kappa^*\nabla\right)$
is isomorphic to
the canonical form
$\mathcal{L}\left(w'[1],\primelambda^{-1},\primemu\right)$,
that is,
they give isomorphic matrix factorizations in
$
\underline{\MF}(xyz)
$
by Theorem \ref{thm:CylinderFreeMFFinite}.

\item
In $\underline{\MF}(xyz)$,
it is just easy to switch two factors of
canonical form
$
\left(
\varphi\left(w',\lambda,\primemu\right),
\psi\left(w',\lambda,\primemu\right)
\right)
$
corresponding to a loop datum $\left(w',\lambda,\primemu\right)$,
but then $\psi\left(w',\lambda,\primemu\right)$
no longer appears in the canonical form.
Even worse,
it is never obvious how to change it into the canonical form
$\varphi\left(\tilde{w}',\tilde{\lambda},\tilde{\primemu}\right)$
for some another loop datum $\left(\tilde{w}',\tilde{\lambda},\tilde{\primemu}\right)$.
Example
\ref{ex:ReversingTranslate}
shows that the length $3\tau$ of the word $w'$ can be also changed.
(In general,
we have
$\frac{1}{2}\tau\le\tilde{\tau}\le2\tau$,
where $3\tilde{\tau}$ is the length of $\tilde{w}'$.)

\item
In $\underline{\CM}(A)$ or $\underline{\Tri}(A)$,
there is no easy way or formula to compute the shift of modules.

%

\end{itemize}

The above discussions say that so far the only way to compute the canonical form of
the shift of maximal Cohen-Macaulay modules
is to make a detour to use the geometric operation in the Fukaya category.
We summarize the procedure as follows:

\begin{prop}[Shift algorithm]\label{prop:ShiftAlgorithm}
Let
$
M\left(w,\lambda,\mu\right)
\in\underline{\CM}(A)
$
be the maximal Cohen-Macaulay module over $A$
corresponding to a band datum $\left(w,\lambda,\mu\right)$.
Its shift is given by
$$
M\left(w,\lambda,\mu\right)[1]
=
M\left(w[1],\pm\lambda^{-1},\mu\right),
$$
where the band datum
$
\left(w[1],\pm \lambda^{-1},\mu\right)
$
is computed in the following manner:

\begin{enumerate}
\item
Convert the band datum $\left(w,\lambda,\mu\right)$ into a loop datum $\left(w',\primelambda,\primemu\right)$,
following Definition \ref{def:ConversionFromBandtoLoop}.
\item
Compute the reverse $w'[1]$ of the
loop word $w'$ 
(i.e. find the normal form of
the loop word
(\ref{eqn:ReversingLoopWord})).
\item
Convert the loop datum $\left(w'[1],\primelambda^{-1},\primemu\right)$ again into a band datum
$
\left(w[1],\pm \lambda^{-1},\mu\right)
$,
following Definition \ref{def:ConversionFromLooptoBand}.

\end{enumerate}

\end{prop}

We have the following mappings (up to isomorphism)
under the diagram (\ref{eqn:ReversingTranslateDiagramOrdinary}),
while two rows are consistent with our main correspondence (\ref{eqn:MainCorrespondenceDiagram}):
\begin{equation*}\label{eqn:FlipDualDiagramOrdinaryCorrespondence}
\begin{tikzcd}[arrow style=tikz,>=stealth,row sep=3em,column sep=3em] 
&
\mathcal{L}\left(w',\primelambda,\primemu\right)
  \arrow[r,mapsto,"\LocalF"]
  \arrow[d,swap,leftrightarrow,"\smat{\textnormal{orientation}\\ \textnormal{reversing}}"]
  \arrow[d,leftrightarrow,"\jmath"]
&
\varphi_{\operatorname{(deg)}}\left(w',\lambda,\primemu\right)
  \arrow[r,mapsto,"\smat{\cok\\ \simeq}"]
  \arrow[d,swap,leftrightarrow,"\smat{\textnormal{switching}\\ \textnormal{two factors}}"]
  \arrow[d,leftrightarrow,"{[1]}"]
&
M\left(w,\lambda,\mu\right)
  \arrow[d,swap,leftrightarrow,"\textnormal{shift}"]
  \arrow[d,leftrightarrow,"{[1]}"]
&
\\
&
\mathcal{L}\left(w'[1],\primelambda^{-1},\primemu\right)
  \arrow[r,mapsto,"\LocalF"]
&
\varphi_{\operatorname{(deg)}}\left(w'[1],\pm\lambda^{-1},\primemu\right)
  \arrow[r,mapsto,"\smat{\cok\\ \simeq}"]
&
M\left(w[1],\pm\lambda^{-1},\mu\right)
&
\end{tikzcd}
\end{equation*}


\begin{ex}\label{ex:AlgorithmExample}

The following shows the computation of $w[1]$ from the band word $w$ in Example \ref{ex:ConversionExample}:

\newcolumntype{x}[1]{>{\centering\hspace{0pt}}p{#1}} 

\newcommand{\corcmidrule}[1][0.6pt]{\\[\dimexpr-\normalbaselineskip-\belowrulesep-\aboverulesep-#1\relax]} 

\begin{center}
\setlength{\tabcolsep}{-0.3pt}
\begin{tabular}
{cccx{4.5mm}cx{4.5mm}cx{4.5mm}
cx{4.5mm}cx{4.5mm}cx{4.5mm}
cx{4.5mm}cx{4.5mm}cx{4.5mm}
cx{4.5mm}cx{4.5mm}cx{4.5mm}
cx{4.5mm}cx{4.5mm}cx{4.5mm}
cx{4.5mm}cx{4.5mm}cx{4.5mm}
cx{4.5mm}cx{4.5mm}cx{4.5mm}
cx{4.5mm}cx{4.5mm}cx{4.5mm}
cx{4.5mm}cx{4.5mm}cx{4.5mm}
cx{4.5mm}cx{4.5mm}cx{4.5mm}c}
$w$ & $=$ & $($ & $6$ & $,$ & $0$ & $,$ & $2$ & $,$ & $-1$ & $,$ & $0$ & $,$ & $-3$ & $,$ & $0$ & $,$ & $0$ & $,$ & $5$ & $,$ & $0$ & $,$ & $-2$ & $,$ & $1$ & $,$ & $-1$ & $,$ & $3$ & $,$ & $4$ & $)$ \\
\arrayrulecolor{blue}\cmidrule[0.6pt]{4-8}
\corcmidrule\arrayrulecolor{red}\cmidrule[0.6pt]{10-18}
\corcmidrule\arrayrulecolor{blue}\cmidrule[0.6pt]{20-20}
\corcmidrule\arrayrulecolor{red}\cmidrule[0.6pt]{22-24}
\corcmidrule\arrayrulecolor{blue}\cmidrule[0.6pt]{26-26}
\corcmidrule\arrayrulecolor{red}\cmidrule[0.6pt]{28-28}
\corcmidrule\arrayrulecolor{blue}\cmidrule[0.6pt]{30-32}

\\[-3mm]

$w'$ & $=$ & $($ & $8$ & $,$ & $2$ & $,$ & $3$ & $,$ & $-1$ & $,$ & $-1$ & $,$ & $-4$ & $,$ & $-1$ & $,$ & $0$ & $,$ & $5$ & $,$ & $0$ & $,$ & $-2$ & $,$ & $1$ & $,$ & $0$ & $,$ & $4$ & $,$ & $6$ & $)$ \\

\\[-1mm]

$w'[1]$ & $=$ & $($
 & $-8$ & $,$ & $0$ & $,$ & $-6$ & $,$ & $0$ & $,$ & $-4$ & $,$ & $0$ & $,$ & $0$ & $,$ & $0$ & $,$ & $-1$ & $,$ & $0$ & $,$ & $2$ & $,$ & $0$ & $,$ & $0$ & $,$ & $0$ & $,$ & $-5$
 & $,$ & $0$ & $,$ & $0$ & $,$ & $0$ & $,$ & $1$ & $,$ & $0$ & $,$ & $4$ & $,$ & $0$ & $,$ & $1$ & $,$ & $0$ & $,$ & $1$ & $,$ & $0$ & $,$ & $-3$ & $,$ & $0$ & $,$ & $-2$ & $,$ & $0$ & $)$
\\
\end{tabular}
\end{center}
%

\end{ex}


\begin{ex}\label{ex:ReversingTranslate}

The following shows the correspondence of
loops with a local system
$
\mathcal{L}\left((3,-2,2),\primelambda,1\right)
\leftrightarrow
\mathcal{L}\left((-2,1,-1,0,2,0),\primelambda^{-1},1\right),
$
matrix factorizations
$
\varphi\left((3,-2,2),\lambda,1\right)
\leftrightarrow
\varphi\left((-2,1,-1,0,2,0),-\lambda^{-1},1\right),
$
and
maximal Cohen-Macaulay modules
$
M\left((2,-3,1),\lambda,1\right)
\leftrightarrow
M\left((-2,1,-1,0,2,0),-\lambda^{-1},1\right)
$
(in $\Tri(A)$)
up to isomorphism,
where $\lambda=-\primelambda$.
(First two in the bottom row are not presented in the canonical forms.)
\end{ex}
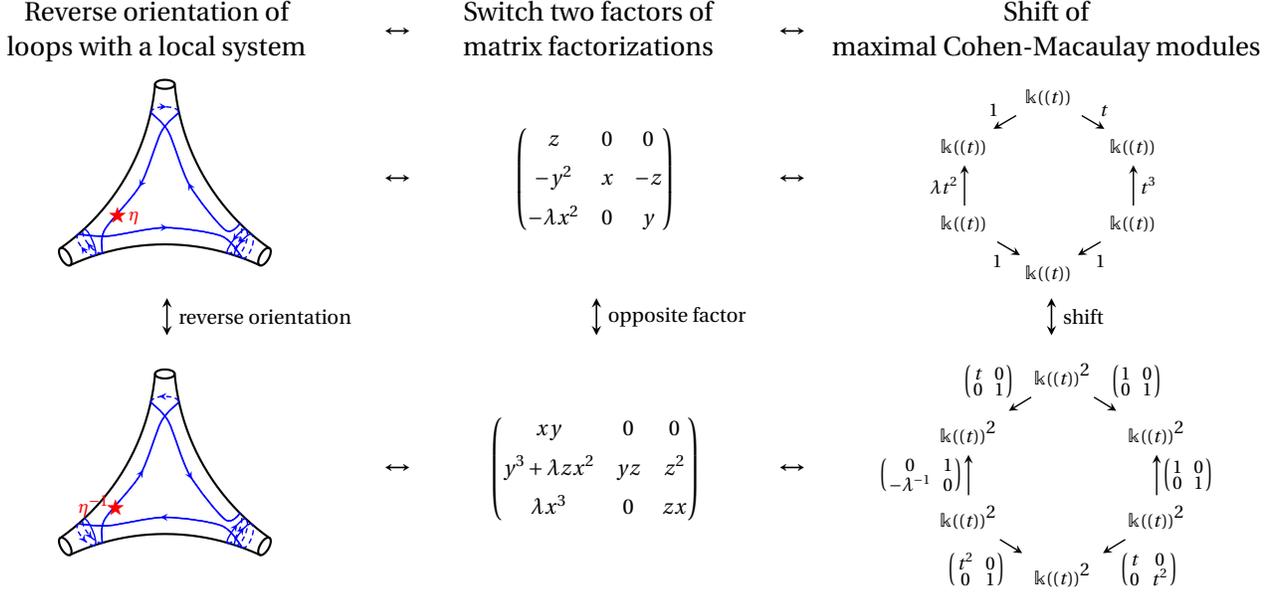
\vspace{-3mm}
\begin{figure}[H]
\hspace{-12mm}
$
\setlength\arraycolsep{5pt}

$
  }
&\leftrightarrow&
  \adjustbox{scale=0.9}{
  \begin{tikzcd}[ampersand replacement=\&, arrow style=tikz,>=stealth, sep=6pt] 
    \&
    {\scriptstyle \mathbb{k}((t))}^{}
      \arrow[ld, swap, "\color{black}1"]
      \arrow[rd, "\color{black}t"]
    \&
    \\[0mm]
    {\scriptstyle \mathbb{k}((t))}^{}
    \&
    \&
    {\scriptstyle \mathbb{k}((t))}^{}
    \\[4mm]
    {\scriptstyle \mathbb{k}((t))}^{}
      \arrow[u, "\color{black} \lambda t^2"]
      \arrow[rd, swap, "\color{black}1"]
    \&
    \&
    {\scriptstyle \mathbb{k}((t))}^{}
      \arrow[u, swap, "\color{black}t^3"]
      \arrow[ld, "\color{black}1"]
    \\[0mm]
    \&
    {\scriptstyle \mathbb{k}((t))}^{}
    \&
  \end{tikzcd}
  }
\\[-2mm]
\hspace{26mm}
\begin{tikzcd}[arrow style=tikz,>=stealth]
    \arrow[d, leftrightarrow, "\text{ reverse orientation}"]
    \\[-1mm]
    \phantom{.}
\end{tikzcd}
& &
\hspace{19mm}
\begin{tikzcd}[arrow style=tikz,>=stealth]
    \arrow[d, leftrightarrow, "\text{ opposite factor}"]
    \\[-1mm]
    \phantom{.}
\end{tikzcd}
& &
\hspace{7mm}
\begin{tikzcd}[arrow style=tikz,>=stealth]
    \arrow[d, leftrightarrow, "\text{ shift}"]
    \\[-1mm]
    \phantom{.}
\end{tikzcd}
\\[-1mm]

$
  }
&\leftrightarrow&
  \adjustbox{scale=0.9}{
  \begin{tikzcd}[ampersand replacement=\&, arrow style=tikz,>=stealth, sep=6pt] 
    \&
    {\scriptstyle \mathbb{k}((t))}^{2}
      \arrow[ld, swap, "\color{black}\spmat{t&~~0 \\ 0&~~1}"]
      \arrow[rd, "\color{black}\spmat{1&~~0 \\ 0&~~1}"]
    \&
    \\[0mm]
    {\scriptstyle \mathbb{k}((t))}^{2}
    \&
    \&
    {\scriptstyle \mathbb{k}((t))}^{2}
    \\[4mm]
    {\scriptstyle \mathbb{k}((t))}^{2}
      \arrow[u, "\color{black}\spmat{0&~~1 \\ -\lambda^{-1}&~~0}"]
      \arrow[rd, swap, "\color{black}\spmat{t^2&~~0 \\ 0&~~1}"]
    \&
    \&
    {\scriptstyle \mathbb{k}((t))}^{2}
      \arrow[u, swap, "\color{black}\spmat{1&~~0 \\ 0&~~1}"]
      \arrow[ld, "\color{black}\spmat{t&~~0 \\ 0&~~t^2}"]
    \\[0mm]
    \&
    {\scriptstyle \mathbb{k}((t))}^{2}
    \&
  \end{tikzcd}
  }
\end{matrix}
$
\label{ReverseTranslation}
\end{figure}


\subsection{Higher rank/multiplicity and twisted complexes}

In this subsection,
we first recall the concept of
\emph{twisted complexes} in an $\AI$-category
and related notions,
based on \cite{S08,Boc21}
(\S \ref{sec:TwistedComplexes}).
Then we associate any twisted complex in $\MF_{\AI}(f)$
an equivalent matrix factorization
(\S \ref{sec:MFTwistedComplexes}).
We derive a formula for extending the localized mirror functor to twisted complexes
(\S \ref{sec:LMFTwistedComplexes}).
Finally,
we show that our canonical objects in $\underline{\MF}(xyz)$
as well as $\Fuk\left(\POP\right)$ of higher rank
are isomorphic to twisted complexes of lower rank objects
(\S \ref{sec:CanonicalFormMFTwistedComplexes},
\S \ref{sec:CanonicalFormFukTwistedComplexes}).

\subsubsection{Twisted complexes and twisted completion}
\label{sec:TwistedComplexes}

%
%
%
%
%
%
%

\begin{defn}\label{defn:TwistedComplex}
Let $\mathcal{A}$ be a $\mathbb{Z}_2$-graded $\AI$-category over a field $\mathbb{k}$.
An \textnormal{\textbf{abstract twisted complex}}
\footnote{
The terminology is intended to distinguish it from the \emph{rigid twisted complex} in $\MF_{\AI}(f)$ below (\ref{eqn:RigidTwistedComplex}).
}
in $\mathcal{A}$
is a pair $\left(\mathcal{L},\delta\right)$,
which consists of
\begin{itemize}
\item
a {\emph{direct sum of shifted objects}},
which is a formal expression of the form
$$
\mathcal{L}
:=
\bigoplus_{i=1}^N \mathcal{L}_i\left[k_i\right]
$$
for some $N\in\mathbb{Z}_{\ge1}$,
$\mathcal{L}_i \in\operatorname{Ob}\left(\mathcal{A}\right)$
and
$
k_i\in\mathbb{Z}_2
$
($i\in\left\{1,\dots,N\right\}$),
\item
a collection of morphisms
$$
\delta
:=
\big(\delta_{ij}
\in
\hom_{\mathcal{A}}^{1}
\left(\mathcal{L}_i,\mathcal{L}_j\right)\left[-k_i+k_j\right]
\footnote{
For a $\mathbb{Z}_2$-graded vector space $V:=V^0\oplus V^1$
and $k\in\mathbb{Z}_2$,
we denote by $V[k]$ its $k$-shift,
i.e.,
$V[k]^{\bullet}
=
V^{\bullet+k}$
for $\bullet\in\mathbb{Z}_2$.
}
\big)_{1\le i<j \le N}
$$
satisfying
the
\emph{Maurer-Cartan equation}
\begin{equation}\label{eqn:TwistedComplexMaurerCartanEquation}
\operatorm_0^\delta\left(\mathcal{L}\right)
:=
\sum_{n=1}^{\infty}
\operatorm_n
\left(
\delta,\dots,\delta
\right)
\footnote{
It is a finite sum because $\operatorm_n\left(\delta,\dots,\delta\right)$
vanishes for $n\ge N$.
}
=
0,
\end{equation}
where
$
\operatorm_n
\big(
\delta,\dots,\delta
\big)
$
is an element of
$
\displaystyle
\bigoplus_{1\le i,j\le N}
\hom_{\mathcal{A}}^{2}\left(\mathcal{L}_i,\mathcal{L}_j\right)\left[-k_i+k_j\right]
$
with components given by
\begin{equation}\label{eqn:MaurerCartanComponents}
\left(
\operatorm_{n}
\left(\delta,\dots,\delta\right)\right)_{ij}
:=
\sum_{
\smat{1\le n \le j-i\\ i<i_1<\cdots<i_{n-1}<j}
}
\operatorm_n
\left(
\delta_{ii_1},
\delta_{i_1 i_2},
\dots,
\delta_{i_{n-1}j}
\right)
\in\hom_{\mathcal{A}}^{2}\left(\mathcal{L}_i,\mathcal{L}_j\right)\left[-k_i+k_j\right].
\end{equation}

\end{itemize}

\end{defn}

\begin{defn}\label{defn:TwistedCompletion}
Given a $\mathbb{Z}_2$-graded $\AI$-category $\mathcal{A}$,
its \textnormal{\textbf{twisted completion}} 
$\operatorname{Tw}\mathcal{A}$
is a $\mathbb{Z}_2$-graded $\AI$-category defined as follows:
\begin{itemize}
\item
Its objects are the abstract twisted complexes in $\mathcal{A}$,
\item
The morphism space between
$
\left(\mathcal{L}_0=\bigoplus_{i=1}^{N_0}\mathcal{L}_{0i}\left[k_{0i}\right],\delta_0\right)
$
and
$
\left(\mathcal{L}_1=\bigoplus_{i=1}^{N_1}\mathcal{L}_{1i}\left[k_{1i}\right],\delta_1\right)
$
is
$$
\hom_{\operatorname{Tw}\mathcal{A}}^{\bullet}
\left(
\bigoplus_{i=1}^{N_0} \mathcal{L}_{0i} \left[k_{0i}\right],\
\bigoplus_{j=1}^{N_1} \mathcal{L}_{1j} \left[k_{1j}\right]
\right)
:=
\bigoplus_{i=1}^{N_0}
\bigoplus_{j=1}^{N_1}
\hom_{\mathcal{A}}^{\bullet}\left(\mathcal{L}_{0i},\mathcal{L}_{1j}\right)\left[-k_{0i}+k_{1j}\right]
\quad
\left(\bullet\in\mathbb{Z}_2\right),
$$
\item
The $\AI$-operations
$
\left\{\operatorm_k^{\operatorname{Tw}\mathcal{A}}\right\}_{k\ge1}
$
are defined as
$$
\hspace{10mm}
\operatorm_k^{\operatorname{Tw}\mathcal{A}}
:
\hom_{\operatorname{Tw}\mathcal{A}}\left(\left(\mathcal{L}_0,\delta_0\right),\left(\mathcal{L}_1,\delta_1\right)\right)
\otimes
\cdots
\otimes
\hom_{\operatorname{Tw}\mathcal{A}}\left(\left(\mathcal{L}_{k-1},\delta_{k-1}\right),\left(\mathcal{L}_k,\delta_k\right)\right)
\rightarrow
\hom_{\operatorname{Tw}\mathcal{A}}\left(\left(\mathcal{L}_0,\delta_0\right),\left(\mathcal{L}_k,\delta_k\right)\right),
$$
$$
\hspace{10mm}
\displaystyle
\left(f_1,\dots,f_k\right)
\mapsto
\operatorm_k^{\delta_0,\dots,\delta_k}\left(f_1,\dots,f_k\right)
:=
\sum_{
m_0,\dots,m_k\ge0
}
\operatorm_{k+m_0+\cdots+m_k}^{\mathcal{A}}
\big(
\underbrace{
\delta_0,\dots,\delta_0
}_{m_0
},
f_1,
\underbrace{
\delta_1,\dots,\delta_1
}_{m_1
},
f_2,
\dots,
f_k,
\underbrace{
\delta_k,\dots,\delta_k
}_{m_k
}
\big)
\footnote{
Each component is defined in the same manner as in (\ref{eqn:MaurerCartanComponents}).
}.
$$
for
$
f_i
\in
\hom_{\operatorname{Tw}\mathcal{A}}
\left(
\left(\mathcal{L}_{i-1},\delta_{i-1}\right),
\left(\mathcal{L}_i,\delta_i\right)
\right)
$
($i\in\left\{1,\dots,k\right\}$).
\end{itemize}
It
is a \emph{triangulated $\AI$-category},
and
its cohomological category
$
H^0\left(\operatorname{Tw}\mathcal{A}\right)
$
becomes
a triangulated
category
in the classical sense
(\cite[\S I.3]{S08}).
\end{defn}

Note that there is a natural embedding
$
\mathcal{A}
\hookrightarrow
\operatorname{Tw}\mathcal{A}
$
of $\AI$-categories
sending each object $\mathcal{L}\in\operatorname{Ob}(\mathcal{A})$
to the trivial abstract twisted complex
$\left(\mathcal{L}[0],0\right)$.

Taking the twisted completion of $\AI$-categories is \emph{functorial} in the following sense:

\begin{prop}\label{prop:TwistedCompletionInducedFunctor}
An $\AI$-functor 
$\mathcal{F}:=\left\{\mathcal{F}_k\right\}_{k\ge0}:\mathcal{A}\rightarrow\mathcal{B}$
between $\AI$-categories
induces
an $\AI$-functor
$
\operatorname{Tw}\mathcal{F}:=\left\{\left(\operatorname{Tw}\mathcal{F}\right)_k\right\}_{k\ge0}
:
\operatorname{Tw}\mathcal{A}
\rightarrow
\operatorname{Tw}\mathcal{B}
$
between their twisted completions
defined as follows:
%
\begin{itemize}
\item
An abstract twisted complex
$
\left(
\bigoplus_{i=1}^N \mathcal{L}_i[k_i],\delta\right)
$
in $\mathcal{A}$ is mapped to
the abstract twisted complex
in $\mathcal{B}$ given by
$$
\left(
\bigoplus_{i=1}^N
\mathcal{F}_0\left(\mathcal{L}_i\right)[k_i],\
\sum_{n=1}^{\infty}\mathcal{F}_n\left(\delta,\dots,\delta\right)
\right),
$$
where
$
\mathcal{F}_n
\big(
\delta,\dots,\delta
\big)
$
is an element of
$
\displaystyle
\bigoplus_{1\le i,j\le N}
\hom_{\mathcal{B}}^{1}\left(\mathcal{F}_0\left(\mathcal{L}_i\right),\mathcal{F}_0\left(\mathcal{L}_j\right)\right)\left[-k_i+k_j\right]
$
with
components
\begin{equation*}
\hspace{12mm}
\left(
\mathcal{F}_{n}
\left(\delta,\dots,\delta\right)\right)_{ij}
:=
\sum_{
\smat{1\le n \le j-i\\ i<i_1<\cdots<i_{n-1}<j}
}
\mathcal{F}_n
\left(
\delta_{ii_1},
\delta_{i_1 i_2},
\dots,
\delta_{i_{n-1}j}
\right)
\in\hom_{\mathcal{B}}^{1}\left(\mathcal{F}_0\left(\mathcal{L}_i\right),\mathcal{F}_0\left(\mathcal{L}_j\right)\right)\left[-k_i+k_j\right].
\end{equation*}

\item
Higher components
$
\left\{\left(\operatorname{Tw}\mathcal{F}\right)_k\right\}_{k\ge1}
$
are given by
$$
\hspace{-10mm}
\left(\operatorname{Tw}\mathcal{F}\right)_k
:
\hom_{\operatorname{Tw}\mathcal{A}}\left(\left(\mathcal{L}_0,\delta_0\right),\left(\mathcal{L}_1,\delta_1\right)\right)
\otimes
\cdots
\otimes
\hom_{\operatorname{Tw}\mathcal{A}}\left(\left(\mathcal{L}_{k-1},\delta_{k-1}\right),\left(\mathcal{L}_k,\delta_k\right)\right)
$$
$$
\hspace{87mm}
\rightarrow
\hom_{\operatorname{Tw}\mathcal{B}}\left(\left(\operatorname{Tw}\mathcal{F}\right)_0\left(\left(\mathcal{L}_0,\delta_0\right)\right),\left(\operatorname{Tw}\mathcal{F}\right)_0\left(\left(\mathcal{L}_k,\delta_k\right)\right)\right),
$$
$$
\hspace{30mm}
\displaystyle
\left(f_1,\dots,f_k\right)
\mapsto
\sum_{
m_0,\dots,m_k\ge0
}
\mathcal{F}_{k+m_0+\cdots+m_k}
\big(
\underbrace{
\delta_0,\dots,\delta_0
}_{m_0
},
f_1,
\underbrace{
\delta_1,\dots,\delta_1
}_{m_1
},
f_2,
\dots,
f_k,
\underbrace{
\delta_k,\dots,\delta_k
}_{m_k
}
\big).
$$

\end{itemize}
\end{prop}

The induced $\AI$-functor
$
\operatorname{Tw}\mathcal{F}
:
\operatorname{Tw}\mathcal{A}
\rightarrow
\operatorname{Tw}\mathcal{B}
$
also boils down to an exact functor
$
H^0\left(\operatorname{Tw}\mathcal{F}\right)
:
H^0\left(\operatorname{Tw}\mathcal{A}\right)
\rightarrow
H^0\left(\operatorname{Tw}\mathcal{B}\right)
$
between classical triangulated categories.

\subsubsection{Twisted complexes of matrix factorizations}
\label{sec:MFTwistedComplexes}

Let $S$ be the power series ring $\field[[x_1,\dots,x_m]]$ of $m$ variables,
and $f\in S$ its nonzero element.
We will demonstrate
that
the $\AI$-category of matrix factorizations
$\MF_{\AI}(f)$
has an intrinsic notion of twisted complexes,
by constructing them
as actual objects in it.

In
\S \ref{sec:MFShifting},
we already defined the shift functor
$
[1]
:
\MF_{\AI}(f)\rightarrow \MF_{\AI}(f),
$
which simply switches the position of two matrices in a given matrix factorization.
More precisely,
given
a matrix factorization
$
\begin{tikzcd}[arrow style=tikz,>=stealth, sep=20pt, every arrow/.append style={shift left=0.5}]
   P^0
     \arrow{r}{\varphi}
   &
   P^1,
     \arrow{l}{\psi}
\end{tikzcd}
$
we define
its \textbf{$k$-shift}
($k\in\mathbb{Z}_2$)
as
$
\begin{tikzcd}[arrow style=tikz,>=stealth, sep=20pt, every arrow/.append style={shift left=0.5}]
   P[k]^0
     \arrow{r}{\varphi[k]}
   &
   P[k]^1,
     \arrow{l}{\psi[k[}
\end{tikzcd}
$
where
$P[k]^i:=P^{i+k}$ and
$$
\left(\varphi,\psi\right)[k]
:=
\left(\varphi[k],\psi[k]\right)
:=
\begin{cases}
\left(\varphi,\psi\right)
&\text{if }k=0,
\\
\left(\psi,\varphi\right)
&\text{if }k=1.
\end{cases}
$$

Now we
associate
any abstract twisted complex given in Definition \ref{defn:TwistedComplex}
an object in $\MF_{\AI}(f)$.
Suppose that 
we have
finitely many
shifted matrix factorizations
$
\begin{tikzcd}[arrow style=tikz,>=stealth, sep=25pt, every arrow/.append style={shift left=0.5}]
   P_i[k_i]^0
     \arrow{r}{\varphi_i[k_i]}
   &
   P_i[k_i]^1
     \arrow{l}{\psi_i[k_i]}
\end{tikzcd}
$
($N\in\mathbb{Z}_{\ge1}$, $i\in\left\{1,\dots,N\right\}$)
and
morphisms between them
$$
\left(\gamma_{ij},\delta_{ij}\right)
\in
\hom^1\left(\left(\varphi_j,\psi_j\right),\left(\varphi_i,\psi_i\right)\right)
[-k_j+k_i]
=
\Hom_S\left(P_j[k_j]^0,P_i[k_i]^1\right)
\times
\Hom_S\left(P_j[k_j]^1,P_i[k_i]^0\right)
$$
for $1\le i< j \le N$,
which form the following
left diagram
(not necessarily commutative):
\begin{equation}\label{eqn:TwMorphismDiagrams}
\setlength\arraycolsep{4mm}
\begin{matrix}
\begin{tikzcd}[arrow style=tikz,>=stealth,row sep=2.5em,column sep=2.5em] 
&
P_{j}[k_j]^0
  \arrow[r,"{\varphi_{j}[k_j]}"]
  \arrow[dl,swap,"\smat{\gamma_{ij}\\[-1mm]}"]
&
P_{j}[k_j]^1
  \arrow[r,"{\psi_{j}[k_j]}"]
  \arrow[dl,swap,"\smat{\delta_{ij}\\[-2mm]}"]
&
P_{j}[k_j]^0
  \arrow[dl,swap,"\smat{\gamma_{ij}\\[-2mm]}"]
\\
P_i\left[k_i\right]^1
  \arrow[r,"{\psi_i[k_i]}"]
&
P_i\left[k_i\right]^0
  \arrow[r,"{\varphi_i[k_i]}"]
&
P_i[k_i]^1
&
\end{tikzcd}
&\begin{tikzcd}[arrow style=tikz,>=stealth,row sep=2.5em,column sep=3em] 
&
P^0
  \arrow[r,"\varphi"]
  \arrow[dl,swap,"\smat{\gamma\\[-1mm]}"]
&
P^1
  \arrow[r,"\psi"]
  \arrow[dl,swap,"\smat{\delta\\[-2mm]}"]
&
P^0
  \arrow[dl,swap,"\smat{\gamma\\[-2mm]}"]
\\
P^1
  \arrow[r,"\psi"]
&
P^0
  \arrow[r,"\varphi"]
&
P^1
&
\end{tikzcd}
\end{matrix}
\end{equation}

We can arrange them into the block matrix form as follows:
\begin{align*}
\left(\varphi,\psi\right)
&:=
\left(
\bigoplus_{i=1}^N \varphi_i[k_i],\
\bigoplus_{i=1}^N \psi_i[k_i]
\right)
\\[2mm]
&\phantom{:}=
\adjustbox{scale=0.86}{
$
\left(
\begin{blockarray}{c @{\hspace{2pt}} c @{\hspace{7pt}} ccccc}
  &
  & \color{darkgreen} \small\text{$P_1\left[k_1\right]^0$}
  & \color{darkgreen} \small\text{$P_2\left[k_2\right]^0$}
  & \color{darkgreen} \small\text{$\cdots$}
  & \color{darkgreen} \small\text{$P_N\left[k_N\right]^0$}
  \\[1mm]
  \begin{block}{c @{\hspace{2pt}} c @{\hspace{7pt}} (ccccc)}
    & \color{darkgreen} \small\text{$P_1\left[k_1\right]^1$}
    & \varphi_1\left[k_1\right] & 0 & \cdots & 0
    \\[2mm]
    & \color{darkgreen} \small\text{$P_2\left[k_2\right]^1$}
    & 0 & \varphi_2\left[k_2\right] & \cdots & 0
    \\[0mm]
    & \color{darkgreen} \small\text{$\vdots$}
    & \vdots & \vdots & \ddots & \vdots
    \\[2mm]
    & \color{darkgreen} \small\text{$P_N\left[k_N\right]^1$}
    & 0 & 0 & \cdots & \varphi_N\left[k_N\right]
    \\
  \end{block}
\end{blockarray}
\ \ ,
\begin{blockarray}{c @{\hspace{2pt}} c @{\hspace{7pt}} ccccc}
  &
  & \color{darkgreen} \small\text{$P_1\left[k_1\right]^1$}
  & \color{darkgreen} \small\text{$P_2\left[k_2\right]^1$}
  & \color{darkgreen} \small\text{$\cdots$}
  & \color{darkgreen} \small\text{$P_N\left[k_N\right]^1$}
  \\[1mm]
  \begin{block}{c @{\hspace{2pt}} c @{\hspace{7pt}} (ccccc)}
    & \color{darkgreen} \small\text{$P_1\left[k_1\right]^0$}
    & \psi_1\left[k_1\right] & 0 & \cdots & 0
    \\[2mm]
    & \color{darkgreen} \small\text{$P_2\left[k_2\right]^0$}
    & 0 & \psi_2\left[k_2\right] & \cdots & 0
    \\[0mm]
    & \color{darkgreen} \small\text{$\vdots$}
    & \vdots & \vdots & \ddots & \vdots
    \\[2mm]
    & \color{darkgreen} \small\text{$P_N\left[k_N\right]^0$}
    & 0 & 0 & \cdots & \psi_N\left[k_N\right]
    \\
  \end{block}
\end{blockarray}
\
\right),
$
}
\end{align*}
\begin{align*}
\left(\gamma,\delta\right)
&:=
\left(\left(\gamma_{ij},\delta_{ij}\right)
\in
\hom^1\left(\left(\varphi_j,\psi_j\right),\left(\varphi_i,\psi_i\right)\right)
[-k_j+k_i]
\right)
_{1\le i < j \le N}
\\[2mm]
&\phantom{:}=
\adjustbox{scale=0.86}{
$
\left(
\begin{blockarray}{c @{\hspace{2pt}} c @{\hspace{7pt}} ccccc}
  &
  & \color{darkgreen} \small\text{$P_1\left[k_1\right]^0$}
  & \color{darkgreen} \small\text{$P_2\left[k_2\right]^0$}
  & \color{darkgreen} \small\text{$\cdots$}
  & \color{darkgreen} \small\text{$P_N\left[k_N\right]^0$}
  \\[1mm]
  \begin{block}{c @{\hspace{2pt}} c @{\hspace{7pt}} (ccccc)}
    & \color{darkgreen} \small\text{$P_1\left[k_1\right]^1$}
    & 0 & \gamma_{12} & \cdots & \gamma_{1N}
    \\[0mm]
    & \color{darkgreen} \small\text{$P_2\left[k_2\right]^1$}
    & 0 & 0 & \ddots & \vdots
    \\[0mm]
    & \color{darkgreen} \small\text{$\vdots$}
    & \vdots & \vdots & \ddots & \gamma_{(N-1)N}
    \\[2mm]
    & \color{darkgreen} \small\text{$P_N\left[k_N\right]^1$}
    & 0 & 0 & \cdots & 0
    \\
  \end{block}
\end{blockarray}
\ \ ,
\begin{blockarray}{c @{\hspace{2pt}} c @{\hspace{7pt}} ccccc}
  &
  & \color{darkgreen} \small\text{$P_1\left[k_1\right]^1$}
  & \color{darkgreen} \small\text{$P_2\left[k_2\right]^1$}
  & \color{darkgreen} \small\text{$\cdots$}
  & \color{darkgreen} \small\text{$P_N\left[k_N\right]^1$}
  \\[1mm]
  \begin{block}{c @{\hspace{2pt}} c @{\hspace{7pt}} (ccccc)}
    & \color{darkgreen} \small\text{$P_1\left[k_1\right]^0$}
    & 0 & \delta_{12} & \cdots & \delta_{1N}
    \\[0mm]
    & \color{darkgreen} \small\text{$P_2\left[k_2\right]^0$}
    & 0 & 0 & \ddots & \vdots
    \\[0mm]
    & \color{darkgreen} \small\text{$\vdots$}
    & \vdots & \vdots & \ddots & \delta_{(N-1)N}
    \\[2mm]
    & \color{darkgreen} \small\text{$P_N\left[k_N\right]^0$}
    & 0 & 0 & \cdots & 0
    \\
  \end{block}
\end{blockarray}
\
\right).
$
}
\end{align*}
Observe that $\left(\varphi,\psi\right)$ forms a new matrix factorization
$
\begin{tikzcd}[arrow style=tikz,>=stealth, sep=20pt, every arrow/.append style={shift left=0.5}]
   P^0
     \arrow{r}{
     \varphi
     }
   &
   P^1,
     \arrow{l}{
     \psi
     }
\end{tikzcd}
$
where $P^\bullet:=\bigoplus_{i=1}^N P_i[k_i]^\bullet$
($\bullet\in\mathbb{Z}_2$),
and
$\left(\gamma,\delta\right)$
defines a morphism
$
\left(\gamma,\delta\right)
\in
\hom^1\left(\left(\varphi,\psi\right),\left(\varphi,\psi\right)\right),
$
which form the
right diagram in
(\ref{eqn:TwMorphismDiagrams})
(not necessarily commutative).

The $\AI$-operations between them are computed as
$$
\operatorm_1\left(\left(\gamma,\delta\right)\right)
=
\left(
\psi\gamma+\delta\varphi,\
\varphi\delta+\gamma\psi
\right),
\quad
\operatorm_2
\left(\left(\gamma,\delta\right),\left(\gamma,\delta\right)\right)
=
\left(
\delta\gamma,\
\gamma \delta
\right)
\quad\text{and}\quad
\operatorm_{k\ge3}
\left(
\left(\delta,\gamma\right),
\dots,
\left(\delta,\gamma\right)
\right)
=
0,
$$
following
(\ref{eqn:MFAIOperations}).
It is easy to check that
their components are also compatible with the abstract definition in
(\ref{eqn:MaurerCartanComponents}).
Therefore,
the Maurer-Cartan equation
(\ref{eqn:TwistedComplexMaurerCartanEquation})
is phrased as
\begin{equation}\label{eqn:MFMaurerCartanEquation}
0
=
\operatorm_1
\left(\left(\gamma,\delta\right)\right)
+
\operatorm_2
\left(\left(\gamma,\delta\right),\left(\gamma,\delta\right)\right)
=
\left(
\psi\gamma+\delta\varphi+\delta\gamma,
\varphi\delta+\gamma\psi+\gamma\delta
\right),
\end{equation}
and an abstract twisted complex in $\MF_{\AI}(f)$
is equivalent to
a pair
$
\left(\left(\varphi,\psi\right),\left(\gamma,\delta\right)\right)
$
satisfying
(\ref{eqn:MFMaurerCartanEquation})
\footnote{
The pair is also an example of a \emph{bounding cochain}.
Compare it with the \emph{weak Maurer-Cartan equation} (\ref{eqn:wu}).
}.

We define the \textbf{rigid twisted complex} in $\MF_{\AI}(f)$ associated to such a pair
as
$
\begin{tikzcd}[arrow style=tikz,>=stealth, sep=25pt, every arrow/.append style={shift left=0.5}]
   P^0
     \arrow{r}{
     \varphi+\gamma
     }
   &
   P^1.
     \arrow{l}{
     \psi+\delta
     }
\end{tikzcd}
$
More precisely,
it is given as
\begin{align}\label{eqn:RigidTwistedComplex}
\begin{split}
\operatorname{Tw}\left(\left(\varphi,\psi\right),\left(\gamma,\delta\right)\right)
&:=
\left(\varphi+\gamma,\psi+\delta\right)
\\[2mm]
&
\hspace{-6mm}
\phantom{:}=
\adjustbox{scale=0.86}{
$
\left(
\begin{blockarray}{c @{\hspace{2pt}} c @{\hspace{7pt}} ccccc}
  &
  & \color{darkgreen} \small\text{$P_1\left[k_1\right]^0$}
  & \color{darkgreen} \small\text{$P_2\left[k_2\right]^0$}
  & \color{darkgreen} \small\text{$\cdots$}
  & \color{darkgreen} \small\text{$P_N\left[k_N\right]^0$}
  \\[1mm]
  \begin{block}{c @{\hspace{2pt}} c @{\hspace{7pt}} (ccccc)}
    & \color{darkgreen} \small\text{$P_1\left[k_1\right]^1$}
    & \varphi_1\left[k_1\right] & \gamma_{12} & \cdots & \gamma_{1N}
    \\[0mm]
    & \color{darkgreen} \small\text{$P_2\left[k_2\right]^1$}
    & 0 & \varphi_2\left[k_2\right] & \ddots & \vdots
    \\[0mm]
    & \color{darkgreen} \small\text{$\vdots$}
    & \vdots & \vdots & \ddots & \gamma_{(N-1)N}
    \\[2mm]
    & \color{darkgreen} \small\text{$P_N\left[k_N\right]^1$}
    & 0 & 0 & \cdots & \varphi_N\left[k_N\right]
    \\
  \end{block}
\end{blockarray}
\ \ ,
\begin{blockarray}{c @{\hspace{2pt}} c @{\hspace{7pt}} ccccc}
  &
  & \color{darkgreen} \small\text{$P_1\left[k_1\right]^1$}
  & \color{darkgreen} \small\text{$P_2\left[k_2\right]^1$}
  & \color{darkgreen} \small\text{$\cdots$}
  & \color{darkgreen} \small\text{$P_N\left[k_N\right]^1$}
  \\[1mm]
  \begin{block}{c @{\hspace{2pt}} c @{\hspace{7pt}} (ccccc)}
    & \color{darkgreen} \small\text{$P_1\left[k_1\right]^0$}
    & \psi_1\left[k_1\right] & \delta_{12} & \cdots & \delta_{1N}
    \\[0mm]
    & \color{darkgreen} \small\text{$P_2\left[k_2\right]^0$}
    & 0 & \psi_2\left[k_2\right] & \ddots & \vdots
    \\[0mm]
    & \color{darkgreen} \small\text{$\vdots$}
    & \vdots & \vdots & \ddots & \delta_{(N-1)N}
    \\[2mm]
    & \color{darkgreen} \small\text{$P_N\left[k_N\right]^0$}
    & 0 & 0 & \cdots & \psi_N\left[k_N\right]
    \\
  \end{block}
\end{blockarray}
\
\right).
$
}
\end{split}
\end{align}
It is indeed a matrix factorization of $f$,
as a direct consequence of the equation (\ref{eqn:MFMaurerCartanEquation}).

\begin{prop}
The embedding
$
i
:
\MF_{\AI}(f)
\hookrightarrow
\operatorname{Tw}\MF_{\AI}(f)
$
of $\AI$-categories is a quasi-equivalence.

\end{prop}

\begin{proof}
The natural embedding $i$ sends
each matrix factorization
$
\left(\varphi,\psi\right)
$
to the abstract twisted complex
$
\left(\left(\varphi,\psi\right)[0],0\right).
$
The morphism space between such two complexes
$
\left(\left(\varphi_0,\psi_0\right)[0],0\right)
$
and
$
\left(\left(\varphi_1,\psi_1\right)[0],0\right)
$
is
identified with
the original hom space
$
\hom^{\mathbb{Z}_2}\left(\left(\varphi_1,\psi_1\right),\left(\varphi_0,\psi_0\right)\right).
$
On the morphism level,
the first component $i_1$ is the identity map on that space,
and the higher components $i_{k\ge2}$ are defined to be zero.

We define its quasi-inverse $\pi:\operatorname{Tw}\MF_{\AI}(f)\rightarrow \MF_{\AI}(f)$
by sending each abstract twisted complex
$$
\left(
\left(\varphi,\psi\right)
:=
\left(
\bigoplus_{i=1}^N \varphi_i[k_i],\
\bigoplus_{i=1}^N \psi_i[k_i]
\right),\
\left(\gamma,\delta\right)
:=
\left(\left(\gamma_{ij},\delta_{ij}\right)
\in
\hom^1\left(\left(\varphi_j,\psi_j\right),\left(\varphi_i,\psi_i\right)\right)
[-k_j+k_i]
\right)
_{1\le i < j \le N}
\right)
$$
to the rigid twisted complex
$
\operatorname{Tw}\left(\left(\varphi,\psi\right),\left(\gamma,\delta\right)\right).
$
Given two abstract twisted complexes
$
\left(\left(\varphi_0,\psi_0\right),\left(\gamma_0,\delta_0\right)\right)
$
and
$
\left(\left(\varphi_1,\psi_1\right),\left(\gamma_1,\delta_1\right)\right),
$
their hom space
\begin{align*}
\hom_{\operatorname{Tw}\MF_{\AI}(f)}^{\bullet}
\left(
\left(
\left(\varphi_0,\psi_0\right),
\left(\gamma_0,\delta_0\right)
\right),
\left(
\left(\varphi_1,\psi_1\right),
\left(\gamma_1,\delta_1\right)
\right)
\right)
&=
\bigoplus_{i=1}^{N_0}
\bigoplus_{j=1}^{N_1}
\hom
^\bullet
\left(\left(\varphi_{1j},\psi_{1j}\right),\left(\varphi_{0i},\psi_{0i}\right)\right)\left[-k_{0i}+k_{1j}\right]
\\[-1mm]
&
\hspace{-43mm}
=
\bigoplus_{i=1}^{N_0}
\bigoplus_{j=1}^{N_1}
\left(
\Hom_S\left(P_{1j}[k_{1j}]^0,P_{0i}[k_{0i}]^\bullet\right)
\times
\Hom_S\left(P_{1j}[k_{1j}]^1,P_{0i}[k_{0i}]^{1+\bullet}\right)
\right)
\quad
\left(\bullet\in\mathbb{Z}_2\right)
\end{align*}
is naturally identified with the hom space between rigid twisted complexes
\begin{align*}
\hom_{\MF_{\AI}(f)}^{\bullet}
\left(
\operatorname{Tw}
\left(\left(\varphi_0,\psi_0\right)
\left(\gamma_0,\delta_0\right)
\right),
\operatorname{Tw}
\left(
\left(\varphi_1,\psi_1\right),
\left(\gamma_1,\delta_1\right)
\right)
\right)
&
\\
&
\hspace{-58mm}
=
\Hom_S
\left(
\bigoplus_{i=1}^{N_1}
P_{1j}[k_{1j}]^0,\
\bigoplus_{j=1}^{N_0}
P_{0i}[k_{0i}]^\bullet
\right)
\times
\Hom_S
\left(
\bigoplus_{i=1}^{N_1}
P_{1j}[k_{1j}]^1,\
\bigoplus_{j=1}^{N_0}
P_{0i}[k_{0i}]^{1+\bullet}
\right)
\quad
\left(\bullet\in\mathbb{Z}_2\right).
\end{align*}

Therefore,
we can define
the first component
$\pi_1$ of the $\AI$-functor $\pi$
as the identity map on that space,
and the higher components $\pi_{k\ge2}$ as zero.
It is straightforward to check that $i$ and $\pi$ are indeed quasi-inverse to each other.
\end{proof}

\subsubsection{Twisted complexes under localized mirror functor}
\label{sec:LMFTwistedComplexes}
Combining above discussions,
we derive a formula for extending the localized mirror functor to twisted complexes.

\begin{prop}\label{prop:LMFTwistedComplex}

The image of
an abstract twisted complex
$
\left(
\mathcal{L}
:=
\bigoplus_{i=1}^N
\mathcal{L}_i
,\delta\right)
$
in $W\Fuk\left(\POP\right)$
under the induced localized mirror functor
$
\operatorname{Tw}\LocalF
:
\operatorname{Tw}W\Fuk\left(\POP\right)
\rightarrow
\MF_{\AI}(xyz)
$
is
the rigid twisted complex
$$
 \begin{tikzcd}[arrow style=tikz,>=stealth, sep=75pt, every arrow/.append style={shift left}]
   \displaystyle
   \bigoplus_{i=1}^N
   \hom^0\left(\mathcal{L}_i,\mathbb{L}\right)
     \arrow{r}{\LocalPhi\left(\left(\mathcal{L},\delta\right)\right)=\operatorm_1^{\delta,b}} 
   &
   \displaystyle
   \bigoplus_{j=1}^N
   \hom^1\left(\mathcal{L}_j,\mathbb{L}\right),
     \arrow{l}{\LocalPsi\left(\left(\mathcal{L},\delta\right)\right)=\operatorm_1^{\delta,b}} 
 \end{tikzcd}
$$
with components given by
\begin{equation}\label{eqn:RigidTwistedComplexComponents}
\left(
\operatorm_1^{\delta,b}
\right)_{ij}
:=
\sum_{
\smat{0\le n \le j-i\\ i<i_1<\cdots<i_{n-1}<j}
}
\sum_{k=0}^{\infty}
\operatorm_{n+1+k}
\big(\delta_{ii_1},\delta_{i_1 i_2},\dots,\delta_{i_{n-1} j},
-,
\underbrace{b,\dots,b}_{k}
\big)
\footnote{
It is just $\operatorm_{1+k}\left(-,b,\dots,b\right)$
if $n=0$.
}
:
\hom^\bullet\left(\mathcal{L}_j,\mathbb{L}\right)
\rightarrow
\hom^{\bullet+1}\left(\mathcal{L}_i,\mathbb{L}\right)
\end{equation}
for $1\le i \le j \le N$ and $\bullet\in\mathbb{Z}_2$.

In particular,
its diagonal components are given by
the original mirror
$\left(\LocalPhi\left(\mathcal{L}_i\right),\LocalPsi\left(\mathcal{L}_i\right)\right)$
(\ref{eqn:LMFImage})
of $\mathcal{L}_i$,
and its strictly upper triangular components
are
determined by inserting $(j-i)$-times of the twisting $\delta$.
\end{prop}

\begin{proof}

Recall from Proposition \ref{prop:TwistedCompletionInducedFunctor} that
the induced functor
$
\operatorname{Tw}\LocalF
:
\operatorname{Tw}W\Fuk\left(\POP\right)
\rightarrow
\operatorname{Tw}\MF_{\AI}(xyz)
$
maps the given abstract twisted complex to
the abstract twisted complex in $\MF_{\AI}(xyz)$ given by
\begin{equation}\label{eqn:AbstractTwistedComplexMF}
\left(
\bigoplus_{i=1}^N
\LocalF_0\left(\mathcal{L}_i\right),\
\sum_{n=1}^{\infty}\LocalF_n\left(\delta,\dots,\delta\right)
\right),
\end{equation}
where
$
\LocalF_n\left(\delta,\dots,\delta\right)$
is
an element of
$
\displaystyle
\bigoplus_{1\le i,j\le N}
\hom_{\MF_{\AI}(xyz)}^{1}\left(\LocalF_0\left(\mathcal{L}_i\right),\LocalF_0\left(\mathcal{L}_j\right)\right)
$
with
components
\begin{equation*}
\left(
\LocalF_{n}
\left(\delta,\dots,\delta\right)\right)_{ij}
:=
\sum_{
\smat{1\le n \le j-i\\ i<i_1<\cdots<i_{n-1}<j}
}
\LocalF_n
\left(
\delta_{ii_1},
\delta_{i_1 i_2},
\dots,
\delta_{i_{n-1}j}
\right)
\in\hom_{\MF_{\AI}(xyz)}^{1}\left(\LocalF_0\left(\mathcal{L}_i\right),\LocalF_0\left(\mathcal{L}_j\right)\right).
\end{equation*}

The definition of $\LocalF$ in Theorem \ref{thm:lmf}
identifies
those with
\begin{align*}
\LocalF_0\left(\mathcal{L}_i\right)
&=
\operatorm_1^{0,b}
:
\hom^\bullet\left(\mathcal{L}_i,\mathbb{L}\right)
\rightarrow
\hom^{\bullet+1}\left(\mathcal{L}_i,\mathbb{L}\right)
\quad\text{and}
\\
\LocalF_n
\left(
\delta_{ii_1},
\delta_{i_1 i_2},
\dots,
\delta_{i_{n-1}j}
\right)
&=
\operatorm_{n+1}^{0,\dots,0,b}
\left(
\delta_{ii_1},
\delta_{i_1 i_2},
\dots,
\delta_{i_{n-1}j},
-
\right)
:
\hom^\bullet\left(\mathcal{L}_j,\mathbb{L}\right)
\rightarrow
\hom^{\bullet+1}\left(\mathcal{L}_i,\mathbb{L}\right).
\end{align*}
Finally,
under (\ref{eqn:RigidTwistedComplex}),
the rigid twist complex
in $\MF_{\AI}(xyz)$
corresponding to
(\ref{eqn:AbstractTwistedComplexMF})
is given by
\begin{align*}
\bigoplus_{i=1}^N
\LocalF_0\left(\mathcal{L}_i\right)
+\sum_{n=1}^{\infty}\LocalF_n\left(\delta,\dots,\delta\right)
&=
\operatorm_1^{0,b}(-)
+
\sum_{n=1}^{\infty}
\operatorm_{n+1}^{0,\dots,0,b}
\big(
\underbrace{\delta,\dots,\delta}_{n},
-
\big)
=
\sum_{n=0}^{\infty}
\operatorm_{n+1}^{0,\dots,0,b}
\big(
\underbrace{\delta,\dots,\delta}_{n},
-
\big),
\end{align*}
which is the map
$
\displaystyle
\operatorm_1^{\delta,b}
:
\bigoplus_{i=1}^N
\hom^{\bullet}\left(\mathcal{L}_i,\mathbb{L}\right)
\rightarrow
\bigoplus_{i=1}^N
\hom^{\bullet+1}\left(\mathcal{L}_i,\mathbb{L}\right)
$
with the same components as given in (\ref{eqn:RigidTwistedComplexComponents}).
\end{proof}

\begin{remark}

One can also directly define the localized mirror functor
$\LocalF:\operatorname{Tw}W\Fuk\left(\POP\right)\rightarrow\MF_{\AI}(xyz)$
based on the above formula,
not passing through Proposition \ref{prop:TwistedCompletionInducedFunctor}.
For instance,
the identity
$\left(\operatorm_1^{\delta,b}\right)^2 = W^{\mathbb{L}} \cdot \id_{\hom\left(\mathcal{L},\mathbb{L}\right)}$
for any twisted complex
(as well as \emph{bounding cochain})
$
\left(
\mathcal{L}
:=
\bigoplus_{i=1}^N
\mathcal{L}_i
,\delta\right)
$
in $W\Fuk\left(\POP\right)$
follows from the same line of proof as in Lemma \ref{lem:MFIdentity}.
Then one can extend the definition given in Theorem \ref{thm:lmf}
by using $\operatorm_1^{\delta,b}$
and
$\operatorm_{k+1}^{\delta,\dots,\delta,b}$
instead of $\operatorm_1^{0,b}$
and
$\operatorm_{k+1}^{0,\dots,0,b}$,
respectively.

\end{remark}

\subsubsection{Canonical form of matrix factorizations viewed as twisted complexes}
\label{sec:CanonicalFormMFTwistedComplexes}

Now we show that our canonical form of matrix factorizations of $xyz$
defined in Definition \ref{defn:CanonicalFormMF}
with
a higher rank is expressed as a
twisted complex consisting of
several copies of the corresponding object with rank $1$.

\begin{prop}\label{prop:MFCanonicalFormTwistedComplex}
The canonical form
$\left(\varphi\left(w',\lambda,\primemu\right),\psi\left(w',\lambda,\primemu\right)\right)$
of matrix factorizations of $xyz$ corresponding to a non-degenerate loop datum
is quasi-isomorphic
in $\MF_{\AI}(xyz)$
(and isomorphic in $\MF(xyz)$)
to the rigid twisted complex
$$
\operatorname{Tw}
\left(
\left(
\varphi\left(w',\lambda,1\right)
^{\oplus \primemu},\
\psi\left(w',\lambda,1\right)
^{\oplus \primemu}
\right),
\left(\gamma,\delta\right)
\right)
$$
for some morphisms
$$
\left(\gamma,\delta\right)
:=
\left(\left(\gamma_{ij},\delta_{ij}\right)
\in
\hom^1\left(\left(\varphi\left(w',\lambda,1\right),\psi\left(w',\lambda,1\right)\right),\left(\varphi\left(w',\lambda,1\right),\psi\left(w',\lambda,1\right)\right)\right)
\right)
_{1\le i < j \le \primemu}.
$$
\end{prop}

\begin{proof}
Recall from (\ref{eqn:CanonicalFormKroneckerProduct}) that the canonical form of matrix factorizations of $xyz$
corresponding to a non-degenerate loop datum
$\left(w',\lambda,\primemu\right)$
is written as
\begin{align*}
\varphi\left(w',\lambda,\primemu\right)
&=
\varphi\left(w',0,1\right)\otimes I_{\primemu}
-x^{l_1'-1} K_{3\tau}^{3\tau-1}\otimes J_{\primemu}\left(\lambda\right)
-x^{-l_1'} J_{3\tau}^{3\tau-1}\otimes J_{\primemu}\left(\lambda\right)^{-1}
\\
&=
\varphi\left(w',\lambda,1\right)\otimes I_{\primemu}
-x^{l_1'-1} K_{3\tau}^{3\tau-1}\otimes \left(J_{\primemu}\left(\lambda\right)-\lambda I_{\primemu}\right)
-x^{-l_1'} J_{3\tau}^{3\tau-1}\otimes \left(J_{\primemu}\left(\lambda\right)^{-1}-\lambda^{-1} I_{\primemu}\right)
\\
&=
\varphi\left(w',\lambda,1\right)\otimes I_{\primemu}
-x^{l_1'-1} K_{3\tau}^{3\tau-1}\otimes J_{\primemu}
-x^{-l_1'} J_{3\tau}^{3\tau-1}\otimes \left(-\lambda^{-2}J_{\primemu}+\lambda^{-3}J_{\primemu}^2-\cdots-\left(-\lambda\right)^{-\primemu}J_{\primemu}^{\primemu-1}\right).
\end{align*}
Using Lemma \ref{lem:KroneckerProductSwitching},
we know that it is similar to
$$
I_{\primemu}
\otimes
\varphi\left(w',\lambda,1\right)
-
x^{l_1'-1}
J_{\primemu}
\otimes
K_{3\tau}^{3\tau-1}
+
x^{-l_1'}
\left(\lambda^{-2}J_{\primemu}-\lambda^{-3}J_{\primemu}^2+\cdots+\left(-\lambda\right)^{-\primemu}J_{\primemu}^{\primemu-1}\right)
\otimes
J_{3\tau}^{3\tau-1},
$$
which is expressed as
\begin{equation}\label{eqn:MFHigherTwistedComplex}
\begin{psmallmatrix}
\varphi\left(w',\lambda,1\right) & -x^{l_1'-1} K_{3\tau}^{3\tau-1} & 0 & \cdots & 0
\\
0 & \varphi\left(w',\lambda,1\right) & -x^{l_1'-1} K_{3\tau}^{3\tau-1} & \cdots & 0
\\
0 & 0 & \varphi\left(w',\lambda,1\right) & \reflectbox{\rotatebox{45}{$\scriptstyle\cdots$}} & \rotatebox{90}{$\scriptstyle\cdots$}
\\
\rotatebox{90}{$\scriptstyle\cdots$} & \rotatebox{90}{$\scriptstyle\cdots$} & \rotatebox{90}{$\scriptstyle\cdots$} & \reflectbox{\rotatebox{45}{$\scriptstyle\cdots$}} & -x^{l_1'-1} K_{3\tau}^{3\tau-1}
\\
0 & 0 & 0 & \cdots & \varphi\left(w',\lambda,1\right)
\end{psmallmatrix}_{\primemu 3\tau\times \primemu 3\tau}
\text{and}\quad
\begin{psmallmatrix}
\varphi\left(w',\lambda,1\right) & \lambda^{-2}x^{-l_1'} J_{3\tau}^{3\tau-1} & -\lambda^{-3}x^{-l_1'} J_{3\tau}^{3\tau-1} & \cdots & \left(-\lambda\right)^{-\primemu}x^{-l_1'} J_{3\tau}^{3\tau-1}
\\
0 & \varphi\left(w',\lambda,1\right) & \lambda^{-2}x^{-l_1'} J_{3\tau}^{3\tau-1} & \cdots & \left(-\lambda\right)^{-\primemu+1}x^{-l_1'} J_{3\tau}^{3\tau-1}
\\
0 & 0 & \varphi\left(w',\lambda,1\right) & \reflectbox{\rotatebox{45}{$\scriptstyle\cdots$}} & \rotatebox{90}{$\scriptstyle\cdots$}
\\
\rotatebox{90}{$\scriptstyle\cdots$} & \rotatebox{90}{$\scriptstyle\cdots$} & \rotatebox{90}{$\scriptstyle\cdots$} & \reflectbox{\rotatebox{45}{$\scriptstyle\cdots$}} & \lambda^{-2}x^{-l_1'} J_{3\tau}^{3\tau-1}
\\
0 & 0 & 0 & \cdots & \varphi\left(w',\lambda,1\right)
\end{psmallmatrix}_{\primemu 3\tau\times \primemu 3\tau},
\end{equation}
in the case of $l_1'\ge1$ and $l_1'\le0$,
respectively.
This observation together with the definition (\ref{eqn:RigidTwistedComplex})
of rigid twisted complexes prove the proposition.
\end{proof}

\subsubsection{Canonical form of loops with a local system viewed as twisted complexes}
\label{sec:CanonicalFormFukTwistedComplexes}

In this subsection,
we will realize (\ref{eqn:MFHigherTwistedComplex}) as the image of
an abstract twisted complex
in $\Fuk\left(\POP\right)$
under the induced localized mirror functor
$
\operatorname{Tw}\LocalF
:
\operatorname{Tw}W\Fuk\left(\POP\right)
\rightarrow
\MF_{\AI}(xyz)
$,
as a direct consequence of
Proposition
\ref{prop:LMFTwistedComplex}.

We first take
the loop with a rank $1$ local system
$
\mathcal{L}:=\left(L,E,\nabla\right):=\mathcal{L}\left(w',\primelambda,1\right),
$
where $\left(w',\primelambda,1\right)$
is the non-degenerate loop datum
corresponding to $\left(w',\lambda,1\right)$
under Theorem \ref{thm:LagMFCorrespondenceNondegenerate}.
Its underlying loop
$L=L\left(w'\right)$
has a marked point $o_L\in\chi^1\left(L,L\right)$
that we assume is located nearby the point {\small$\color{red}\bigstar$}
as in Figure \ref{fig:HamiltonianPerturbation}.
We denote by
$
o:=\left.\id\right|_{o_L}
\in
\Hom_{\field}\left(\left.E\right|_{o_L},\left.E\right|_{o_L}\right)
\subseteq
\hom^1\left(\mathcal{L},\mathcal{L}\right).
$

Consider the abstract twisted complex
\begin{equation}\label{eqn:FukAbstractTwistedComplex}
\left(
\mathcal{L}\left(w',\primelambda,1\right)
^{\oplus \primemu},\
\delta
=
\begin{psmallmatrix}
0 & o_{12} & 0 & \cdots & 0
\\
0 & 0 & o_{23} & \cdots & 0
\\
0 & 0 & 0 & \reflectbox{\rotatebox{45}{$\scriptstyle\cdots$}} & \rotatebox{90}{$\scriptstyle\cdots$}
\\
\rotatebox{90}{$\scriptstyle\cdots$} & \rotatebox{90}{$\scriptstyle\cdots$} & \rotatebox{90}{$\scriptstyle\cdots$} & \reflectbox{\rotatebox{45}{$\scriptstyle\cdots$}} & o_{(N-1)N}
\\
0 & 0 & 0 & \cdots & 0
\end{psmallmatrix}
\right)
\end{equation}
that consists of
the direct sum of $\primemu$-copies of
$\mathcal{L}_i:=\mathcal{L}=\mathcal{L}\left(w',\primelambda,1\right)$
($i\in\left\{1,\dots,\primemu\right\}$)
and
a collection of morphisms
$$
\delta
:=
\big(\delta_{ij}
\in
\hom^{1}
\left(\mathcal{L}_i,\mathcal{L}_j\right)
\big)_{1\le i<j \le \primemu}
$$
where
$\delta_{ij}$
is nontrivial
only for $j=i+1$,
in which case
it
is
$
o_{i(i+1)}
:=
o
\in
\hom^1\left(\mathcal{L}_i,\mathcal{L}_{i+1}\right)
=
\hom^1\left(\mathcal{L},\mathcal{L}\right).
$

\begin{prop}\label{prop:LMFTwistedComplexComputation}
The mirror image of the abstract twisted complex (\ref{eqn:FukAbstractTwistedComplex})
in $\Fuk\left(\POP\right)$
is
quasi-isomorphic
in $\MF_{\AI}(xyz)$
(and isomorphic in $\MF(xyz)$)
to the rigid twisted complex
$$
J_{\primemu}\left(\primelambda\right)^{-1}\otimes\varphi_{-1}
+
I_{\primemu}\otimes\varphi_0
+
J_{\primemu}\left(\primelambda\right)\otimes\varphi_1
=
\begin{psmallmatrix}
\LocalPhi\left(\mathcal{L}\left(w',\primelambda,1\right)\right) & \varphi_1-\primelambda^{-2}\varphi_{-1} & \primelambda^{-3}\varphi_{-1} & \cdots & -(-\primelambda)^{-\primemu}\varphi_{-1}
\\
0 & \LocalPhi\left(\mathcal{L}\left(w',\primelambda,1\right)\right) & \varphi_1-\primelambda^{-2}\varphi_{-1} & \cdots & -(-\primelambda)^{-\primemu+1}\varphi_{-1}
\\
0 & 0 & \LocalPhi\left(\mathcal{L}\left(w',\primelambda,1\right)\right) & \reflectbox{\rotatebox{45}{$\scriptstyle\cdots$}} & \rotatebox{90}{$\scriptstyle\cdots$}
\\
\rotatebox{90}{$\scriptstyle\cdots$} & \rotatebox{90}{$\scriptstyle\cdots$} & \rotatebox{90}{$\scriptstyle\cdots$} & \reflectbox{\rotatebox{45}{$\scriptstyle\cdots$}} & \varphi_1-\primelambda^{-2}\varphi_{-1}
\\
0 & 0 & 0 & \cdots & \LocalPhi\left(\mathcal{L}\left(w',\primelambda,1\right)\right)
\end{psmallmatrix}_{\primemu 3\tau\times \primemu 3\tau}
$$
where
$
\LocalPhi\left(\mathcal{L}\left(w',\primelambda,1\right)\right)
=
\primelambda^{-1}\varphi_{-1}
+\varphi_0
+
\primelambda\varphi_1
$
for $\varphi_{-1}$, $\varphi_0$, $\varphi_1\in\field[[x,y,z]]^{3\tau\times3\tau}$
(from Proposition \ref{thm:MFFromLag}).

In particular,
it is also isomorphic to the
rigid twisted complex
(\ref{eqn:MFHigherTwistedComplex})
given in Proposition \ref{prop:MFCanonicalFormTwistedComplex}.
\end{prop}

\begin{proof}
According to
Proposition
\ref{prop:LMFTwistedComplex},
the mirror image is given in the form
$$
 \begin{tikzcd}[arrow style=tikz,>=stealth, sep=85pt, every arrow/.append style={shift left}]
   \displaystyle
   \hom^0\left(\mathcal{L},\mathbb{L}\right)
   ^{\oplus \primemu}
     \arrow{r}{\LocalPhi\left(\left(\mathcal{L}^{\oplus \primemu},\delta\right)\right)=\operatorm_1^{\delta,b}} 
   &
   \displaystyle
   \hom^1\left(\mathcal{L},\mathbb{L}\right)
   ^{\oplus \primemu}.
     \arrow{l}{\LocalPsi\left(\left(\mathcal{L}^{\oplus \primemu},\delta\right)\right)=\operatorm_1^{\delta,b}} 
 \end{tikzcd}
$$
Each diagonal component of $\LocalPhi\left(\left(\mathcal{L}^{\oplus \primemu},\delta\right)\right)$
is
$\LocalPhi\left(\mathcal{L}\right)$,
and strictly upper triangular components are
\begin{equation}\label{eqn:FukMFTwistedComponents}
\left(
\operatorm_1^{\delta,b}
\right)_{ij}
:=
\sum_{k=0}^{\infty}
\operatorm_{j-i+1+k}
\big(o_{i(i+1)},
\dots,o_{(j-1)j},
-,
\underbrace{b,\dots,b}_{k}
\big)
:
\hom^\bullet\left(\mathcal{L}_j,\mathbb{L}\right)
\rightarrow
\hom^{\bullet+1}\left(\mathcal{L}_i,\mathbb{L}\right)
\end{equation}
for $1\le i < j \le N$ and $\bullet\in\mathbb{Z}_2$.
Note that
it depends only on the difference $j-i$,
and hence we assume $i=1$
without loss of generality.
Substituting
$
f=\left.f\right|_p
\in
\hom^0\left(\mathcal{L}_j,\mathbb{L}\right)
$
for some
$p\in\chi^0\left(L_j,\mathbb{L}\right)$
into it becomes
$$
\sum_{
\begin{matrix}
\scriptscriptstyle
\left(x_1,X_1\right),\dots,\left(x_k,X_k\right)
\\
\scriptscriptstyle
\in\left\{(x,X),(y,Y),(z,Z)\right\}
\end{matrix}
}x_1\dots x_k\
\sum_{s\in\chi^1\left(L_1,\mathbb{L}\right)}
\sum_{u\in\mathcal{M}\left(o_{12},o_{23}\dots,o_{(j-1)j},p,X_1,\dots,X_k,\overline{s}\right)}
\sign(u)
\operatorname{hol}_s\left(\partial u\right)
\left(o_{12},
o_{23},
\dots,o_{(j-1)j},f,X_1,\dots,X_k\right),
$$
following the same procedure
as in \S \ref{sec:LMFComputationSubsection}.

We will associate each element $u$ in the moduli space
\begin{equation}\label{eqn:PerturbedModuliSpace}
\mathcal{M}\left(o_{12},o_{23},\dots,o_{(j-1)j},p,X_1,\dots,X_k,\overline{s}\right)
\end{equation}
with an element $u'$ in
$\mathcal{M}\left(p,X_1,\dots,X_k,\overline{s}\right)$,
by considering
recursive perturbations of $\primemu$-copies of $L_i=L=L\left(w'\right)$
($i\in\left\{1,\dots,\primemu\right\}$)
as described in Figure \ref{fig:HamiltonianPerturbation}.
\begin{figure}[H]
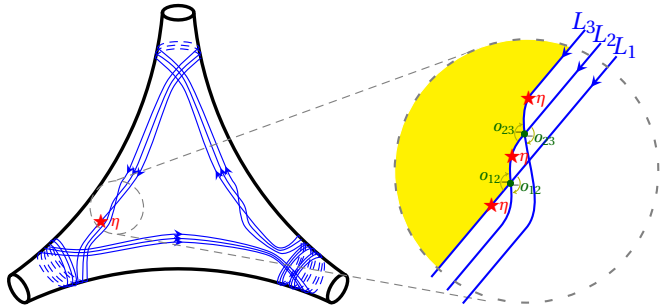

\centering
          \adjustbox{height=40mm}{

          }
\captionsetup{width=1\linewidth}
\caption{
Recursive perturbations of a loop $L$
}
\label{fig:HamiltonianPerturbation}
\end{figure}

\noindent
\emph{Case i) The boundary orientation of $u$ coincides with the orientations of $L_i$'s.}
Such a polygon $u$
cannot have angles at more than one of 
$
o_{12},o_{23},\dots,o_{(j-1)j}
$
in a consecutive manner
because of their arrangement as in Figure \ref{fig:HamiltonianPerturbation}.
Hence
it
exists only when $j=2$,
and it has an angle at $o_{12}$.
Ignoring $o_{12}$ and identifying $L_1$ and $L_2$,
$u$ has an obvious counterpart $u'$ in
$\mathcal{M}\left(p,X_1,\dots,X_k,\overline{s}\right)$
that passes through the point {\small$\color{red}\bigstar$}
and has the same boundary orientation with $L_i$'s.
In other words,
such polygons $u$  are identified with
the polygons $u'$ in
$\mathcal{M}\left(p,X_1,\dots,X_k,\overline{s}\right)$
such that the holonomy of $\partial u'$ contains one $\primelambda$-factor.

Under the correspondence,
it is easily checked that
$\sign(u)$ and $\sign(u')$ are the same,
and the holonomy of $\partial u$ and $\partial u'$
contain zero and one $\primelambda$-factor,
respectively,
because they pass through the point {\small$\color{red}\bigstar$}
$0$ and $1$ times,
respectively.
As those polygons $u'$ contribute to the factor $\primelambda\varphi_1$
in the decomposition of $\LocalPhi\left(\mathcal{L}\left(w',\primelambda,1\right)\right)$,
their corresponding polygons $u$ contribute $\varphi_1$ to the component (\ref{eqn:FukMFTwistedComponents})
for $j=i+1$.

\noindent
\emph{Case ii) The boundary orientation of $u$ differs from the orientations of $L_i$'s.}
Such a polygon $u$ now have angles at $o_{12},o_{23},\dots,o_{(j-1)j}$ consecutively
as drawn in yellow in Figure \ref{fig:HamiltonianPerturbation}.
Ignoring those angles and all perturbations,
$u$ has an obvious counterpart $u'$ in
$\mathcal{M}\left(p,X_1,\dots,X_k,\overline{s}\right)$
that passes through the point {\small$\color{red}\bigstar$}
and has the opposite boundary orientation with $L_i$'s.
In other words,
such polygons $u$  are identified with
the polygons $u'$ in
$\mathcal{M}\left(p,X_1,\dots,X_k,\overline{s}\right)$
such that the holonomy of $\partial u'$ contains one $\primelambda^{-1}$-factor.

Under the correspondence,
the quotient of $\sign(u)$ and $\sign(u')$ is given by
$(-1)^{j-i}$,
because $u$ have $j-i$ more angles of odd degree that $u'$,
where its boundary orientation is different from loops.
The holonomy of $\partial u$ and $\partial u'$ contain one
$\primelambda^{-(j-i+1)}$ and $\primelambda^{-1}$ factor,
respectively,
because they pass through the point {\small$\color{red}\bigstar$}
with the opposite orientation $j-i+1$ and $1$ times,
respectively.
As those polygons $u'$ contribute to the factor $\primelambda^{-1}\varphi_{-1}$
in the decomposition of $\LocalPhi\left(\mathcal{L}\left(w',\primelambda,1\right)\right)$,
their corresponding polygons $u$ contribute $-(-\primelambda)^{-(j-i+1)}\varphi_{-1}$ to the component (\ref{eqn:FukMFTwistedComponents}).

To summarize,
the polygons in Case i) and Case ii) contribute to
$J_{\primemu}\left(\primelambda\right)\otimes\varphi_1$
and
$J_{\primemu}\left(\primelambda\right)^{-1}\otimes\varphi_{-1}$
part of the given matrix factorization,
respectively
(while the $I_{\primemu}\otimes\varphi_0$ part is not relevant to the twisting $o_{i(i+1)}$'s).

The second statement follows from bases change as done in Theorem \ref{thm:LagMFCorrespondenceNondegenerate}.
(One can also use (\ref{eqn:CanonicalFormKroneckerProduct}).)
\end{proof}


\begin{remark}
Proposition \ref{prop:LMFTwistedComplexComputation}
also reveals that the canonical form
$\mathcal{L}\left(w',\primelambda,\primemu\right)$
of loops with a local system
corresponding to non-degenerate loop data $\left(w',\primelambda,\primemu\right)$
is quasi-isomorphic in $\operatorname{Tw}\Fuk\left(\POP\right)$
to the abstract twisted complex (\ref{eqn:FukAbstractTwistedComplex}),
which is made of $\primemu$-copies of $\mathcal{L}\left(w',\primelambda,1\right)$
and odd-degree morphism $o$'s between them.
In fact,
it can be also derived from purely Fukaya-categorical discussions,
not appealing to homological mirror symmetry.
It has been shown in \cite[Theorem 5.8]{B17} that
every higher rank local system over a loop can be realized as an abstract twisted complex of
rank $1$ local systems,
using de Rham version of Fukaya category.

But our specific realization of recursive perturbations as in Figure \ref{fig:HamiltonianPerturbation}
still suggests how we should perturb underlying loops
when we want to work with the perturbation method instead of de Rham version.
Especially when inputs of an $\AI$-operation involve multiple
$e_L$'s and $o_L$'s,
its definition becomes more tricky and unsymmetric in some sense,
which was not fully explained in Remark \ref{rmk:AIoperationwithInfinitesimalGenerators}.
Nevertheless,
some systematic recursive perturbations can be made
so that the $\AI$-relations remain valid.

This also gives an example where the mapping cone of two loops at their intersection is
not quasi-isomorphic to their surgery at that point,
contrary to the usual situation which has been explained
in many places in the literature including \cite[Lemma 5.4]{Ab08}, \cite[Theorem 4.1]{OPS18}
and \cite[Theorem 6.68]{Boc21}.
This happens because two perturbed loops cannot satisfy the minimality condition,
as also remarked in \cite[Lemma 2.25]{AS21}.
We hope our explicit construction of twisted complexes involving $o_L$'s
can be extended to realizing mapping cones of arbitrary morphisms
in Fukaya category as geometric objects.
We will come back to these points in another future work.

\end{remark}


\newpage

\appendix

\section{Relevant Categories}

\subsection{$\AI$-category}

Let us first recall the definition (and convention) of an \emph{$\AI$-category}
over a field $\mathbb{k}$ and related concepts.

\begin{defn}\label{defn:AinftyCategory}
A $\mathbb{Z}_2$-graded \textnormal{\textbf{$\AI$-category}} $\mathcal{A}$ over $\mathbb{k}$ consists of
a class of objects $\operatorname{Ob}\left(\mathcal{A}\right)$,
a $\mathbb{Z}_2$-graded $\mathbb{k}$-vector space
$
\hom(\mathcal{L}_0,\mathcal{L}_1)
=
\hom^0\left(\mathcal{L}^0,\mathcal{L}^1\right)
\oplus
\hom^1\left(\mathcal{L}^0,\mathcal{L}^1\right)
$
for $\mathcal{L}_0$, $\mathcal{L}_1 \in \operatorname{Ob}\left(\mathcal{A}\right)$,
and \emph{$\AI$-operations} $\left\{ \operatorm_{k} \right\}_{k \geq 1}$
given by $\mathbb{k}$-linear maps
$$
\operatorm_{k}
:
\hom\left(\mathcal{L}_0,\mathcal{L}_1\right)
\otimes \cdots \otimes
\hom\left(\mathcal{L}_{k-1},\mathcal{L}_k\right)
\rightarrow
\hom\left(\mathcal{L}_0,\mathcal{L}_k\right)
$$
of degree $2-k$
\footnote{
Here it is just $k$ as we are using the $\mathbb{Z}_2$-grading.
}
satisfying \emph{$\AI$-relations}
\begin{equation}\label{eqn:AI-relation}
\sum_{0\le i<j\le n} (-1)^{\left|f_{1}\right| + \cdots + \left|f_{i}\right| - i} \operatorm_{n-j+i+1} \left (f_{1},\ldots,f_{i}, \operatorm_{j-i} \left( f_{i+1}, \ldots, f_{j} \right),f_{j+1}, \ldots, f_{n} \right)=0
\end{equation}
for any fixed $n\in\mathbb{Z}_{\ge1}$
and morphisms
$
f_i\in\hom^\bullet\left(\mathcal{L}_{i-1},\mathcal{L}_{i}\right)
$
($i\in\left\{1,\dots,n\right\}$, $\bullet\in\mathbb{Z}_2$).


It is called (strictly) \textnormal{\textbf{unital}} if each object $\mathcal{L}$ has a \emph{unit}
$
\id_{\mathcal{L}}
\in
\hom\left(\mathcal{L},\mathcal{L}\right)
$
satisfying
\begin{equation}\label{eqn:Unit}
\operatorm_2\left(\id_{\mathcal{L}},f\right)
=
f,
\quad
\operatorm_2\left(g,\id_{\mathcal{L}}\right)
=
(-1)^{\left|g\right|}
g
\quad\text{and}\quad
\operatorm_k\left(\dots,\id_{\mathcal{L}},\dots\right) = 0
\quad
\text{if }k\ne2
\end{equation}
for any $\mathcal{L}'\in\operatorname{Ob}\left(\mathcal{A}\right)$, $f\in\hom\left(\mathcal{L},\mathcal{L}'\right)$ and $g\in\hom^\bullet\left(\mathcal{L}',\mathcal{L}\right)$ ($\bullet\in\mathbb{Z}_2$).

\end{defn}

\begin{defn}\label{defn:AinftyFunctor}

An \textnormal{\textbf{$\AI$-functor}} $\mathcal{F}:=\left\{\mathcal{F}_k\right\}_{k\ge0}$ between two $\AI$-categories $\mathcal{A}$ and $\mathcal{B}$ consists of
a mapping
$$\mathcal{F}_{0} : \operatorname{Ob}(\mathcal{A}) \to \operatorname{Ob}(\mathcal{B})$$
and $\mathbb{k}$-linear maps ($k\ge1$)
$$
\mathcal{F}_{k}
:
\hom_{\mathcal{A}}\left(\mathcal{L}_0,\mathcal{L}_1\right)
\otimes \cdots \otimes
\hom_{\mathcal{A}}\left(\mathcal{L}_{k-1},\mathcal{L}_k\right)
\rightarrow
\hom_{\mathcal{B}}\left(\mathcal{F}_{0}\left(\mathcal{L}_0\right),\mathcal{F}_{0}\left(\mathcal{L}_k\right)\right)
$$
of degree $1-k$,
satisfying \emph{$\AI$-relations}
\begin{align}\label{eqn:AI-relationsFunctor}
\begin{split}
&\sum_{1\le k \le n}\sum_{1\le i_1<\cdots<i_k= n} \operatorm_{k}^{\mathcal{B}} \left(\mathcal{F}_{i_{1}} \left( f_{1}, \dots, f_{i_{1}} \right), \dots, \mathcal{F}_{n-i_{k-1}} \left( f_{i_{k-1} + 1}, \dots, f_{n} \right) \right)\\
&\hspace{10mm}
=\sum_{0\le i<j\le n} (-1)^{\left|f_{1}\right| + \cdots + \left|f_{i}\right| - i} \mathcal{F}_{n-j+i+1} \left (f_{1}, \dots, f_{i}, \operatorm_{j-i}^{\mathcal{A}} \left( f_{i+1}, \ldots, f_{j} \right), f_{j+1}, \dots, f_{n} \right)
\end{split}
\end{align}
for any fixed $n\in\mathbb{Z}_{\ge1}$
and morphisms
$
f_i\in\hom_{\mathcal{A}}^{\bullet}\left(\mathcal{L}_{i-1},\mathcal{L}_{i}\right)
$
($i\in\left\{1,\dots,n\right\}$, $\bullet\in\mathbb{Z}_2$).

The \textnormal{\textbf{composition}}
$
\mathcal{G}\circ\mathcal{F}
:=
\left\{
\left(
\mathcal{G}\circ\mathcal{F}
\right)_k
\right\}_{k\ge0}
$
of two $\AI$-functors
$
\mathcal{F}:\mathcal{A}\rightarrow\mathcal{B}
$
and
$
\mathcal{G}:\mathcal{B}\rightarrow\mathcal{C}
$
is given by
$$
\left(\mathcal{G}\circ\mathcal{F}\right)_n
\left(f_1,\dots,f_n\right)
:=
\sum_{1\le k \le n}\sum_{1\le i_1<\cdots<i_k= n} \mathcal{G}_k \left(\mathcal{F}_{i_{1}} \left( f_{1}, \dots, f_{i_{1}} \right), \dots, \mathcal{F}_{n-i_{k-1}} \left( f_{i_{k-1} + 1}, \dots, f_{n} \right) \right).
$$

An $\AI$-functor $\mathcal{F}:\mathcal{A}\rightarrow\mathcal{B}$ is called \textnormal{\textbf{unital}} if $\mathcal{A}$, $\mathcal{B}$ are unital and
$$
\mathcal{F}_1\left(\id_{\mathcal{L}}\right)
=
\id_{\mathcal{F}_0\left(\mathcal{L}\right)}
\quad\text{and}\quad
\mathcal{F}_k\left(\dots,\id_{\mathcal{L}},\dots\right) = 0
\quad\text{if }k\ge2
$$
for any $\mathcal{L}\in\operatorname{Ob}\left(\mathcal{A}\right)$.

\end{defn}

The $\mathbb{Z}_2$-graded $\mathbb{k}$-vector space
$
\hom\left(\mathcal{L}_0,\mathcal{L}_1\right)
=
\hom^0\left(\mathcal{L}_0,\mathcal{L}_1\right)
\oplus
\hom^1\left(\mathcal{L}_0,\mathcal{L}_1\right)
$
has a degree $1$ map
$
\operatorm_1:
\hom\left(\mathcal{L}_0,\mathcal{L}_1\right)
\rightarrow
\hom\left(\mathcal{L}_0,\mathcal{L}_1\right)
$
satisfying $\operatorm_1^2=0$
by the $\AI$-relation for $n=1$.
Therefore,
it becomes a cochain complex equipped with a \textbf{differential} $\operatorm_1$.
Taking its cohomology yields an ordinary category:

\begin{defn}\label{defn:CohomologicalCategory}

For a unital $\AI$-category $\mathcal{A}$,
its \textnormal{\textbf{cohomological category}} $H^0\left(\mathcal{A}\right)$ is an ordinary category
whose objects are the same as $\mathcal{A}$ and the morphism space between two objects $\mathcal{L}_0$, $\mathcal{L}_1$ is given by
$$
\Hom_{H^0\left(\mathcal{A}\right)}\left(\mathcal{L}_0,\mathcal{L}_1\right)
:=
H^0\left(\hom\left(\mathcal{L}_0,\mathcal{L}_1\right), \operatorm_1\right).
$$

A unital $\AI$-functor $\mathcal{F}:\mathcal{A}\rightarrow\mathcal{B}$
induces an ordinary functor
$
H^0\left(\mathcal{F}\right):H^0\left(\mathcal{A}\right)\rightarrow H^0\left(\mathcal{B}\right),
$
whose mapping on objects is $\mathcal{F}_0$
and action on morphisms is given by
$
\left[f\right]\mapsto\left[\mathcal{F}_1\left(f\right)\right].
$

\end{defn}

\subsection{Compact Fukaya category of a surface}\label{sec:CptFukSurface}

We establish our geometric setup of the Fukaya category.
We refer 
to
\cite{FOOO,S08,AJ}
for
its
general definitions and properties,
\cite{Ab08,Se}
for Fukaya category of surfaces,
and
\cite{auroux2014beginner,B17,konstantinov2017higher}
for higher rank vector bundles in Fukaya category.

\subsubsection{Objects}


Let $\left(\Sigma,\omega\right)$ be a $2$-dimensional symplectic manifold (possibly with boundary)
of finite type.
That is, $\Sigma$
is a connected oriented smooth surface (with boundary) of finite type
and $\omega$ is an area form on it.
Then any smooth curve
$
L$
in $\Sigma$
automatically satisfies the Lagrangian conditions
$
\left(
\left.\omega\right|_L=0,\
\dim L = \frac{1}{2}\dim \Sigma
\right)
$
and hence is a Lagrangian submanifold in $\Sigma$.

Consider an immersed
oriented smooth loop
$
L:S^1\rightarrow \Sigma\setminus \partial \Sigma
$
having only transversal self-intersections.
We assign two distinct \textbf{marked points}
$e_L, o_L$
\footnote{
Equivalently,
we can assign a Morse function $f_L:S^1\rightarrow \mathbb{R}$ on the domain of $L$
which has a minimum at $e_L$ and a maximum at $o_L$
so that they are all critical points of $f_L$.
Graph of its differential $df_L$ induces a $C^0$-small Hamiltonian perturbation $\phi_H\left(L\right)$ in a neighborhood of $L$
so that they make transversal intersections at $e_L$ and $o_L$.
}
on the image of $L$ away from its self-intersections,
and call the triple
$
\left(L,e_L,o_L\right)
$
a \textbf{marked loop} in $\Sigma$.
When there is no need to specify marked points,
we will call a marked loop just a \textbf{loop},
and denote it shortly as $L$.

\begin{defn}\label{defn:RegularSet}
A set $\mathcal{O}$ consisting of some marked loops in $\Sigma$ is called \emph{\textbf{transversal}} if it satisfies the following:
\begin{itemize}
\item
Any two distinct loops in $\mathcal{O}$ meet transversally.
\item
There are no triple intersections among loops in $\mathcal{O}$.
\item
Marked points of each loop in $\mathcal{O}$ do not lie on any intersection of itself or any other loop in $\mathcal{O}$.
\end{itemize}

\end{defn}

\begin{defn}
(1)
A loop $L:S^1\rightarrow \Sigma$ is called \emph{\textbf{obstructed}}
if it bounds an immersed disk or `fish-tale'.
This means that there is an immersion $i:D^2\rightarrow\Sigma$
which satisfies $i\left(e^{2\pi it}\right)=L\left( \imath (t) \right)$
for some immersion $\imath:[0,1]\rightarrow S^1$.
Otherwise, $L$ is called \emph{\textbf{unobstructed}}.
A transversal set $\mathcal{O}$ of marked loops in $\Sigma$ is called \emph{\textbf{unobstructed}}
if all of its elements are unobstructed
\footnote{
We need this condition for the $\AI$-relations (without $\operatorm_0$-terms) to hold.
See \cite{Ab08}.
}.

\noindent
(2)
A transversal set $\mathcal{O}$ of marked loops in $\Sigma$ is called \emph{\textbf{full}}
if it contains at least one element in each primitive free homotopy class other than the null-homotopic one.

\end{defn}


Given a full unobstructed set $\mathcal{O}$ of marked loops in $\Sigma$,
we define a $\mathbb{Z}_2$-graded \textbf{compact Fukaya category}
$
\Fuk\left(\Sigma\right)
=
\Fuk_{\field}^{\mathbb{Z}_2}\left(\Sigma;\mathcal{O}\right)
$
with respect to $\mathcal{O}$
over
$\field$.
\begin{defn}
An \textnormal{\textbf{object}} of
$
\Fuk\left(\Sigma\right)
$
is given by a triple
$
\mathcal{L}
=
\left(L,E,\nabla\right),
$
which consists of the following:
\begin{itemize}
\item
a marked loop $L:S^1\rightarrow \Sigma$ in $\mathcal{O}$,
\item
a $\field$-vector bundle $E$ of finite rank $\primemu$ over the domain of $L$,
and
\item
a flat
\footnote{
In fact,
every connection is flat in this case because $\dim S^1 = 1$,
but we still stick to the terminology to emphasize that it defines a local system.
}
connection $\nabla$ on $E$.
\end{itemize}
\end{defn}


\begin{figure}[H]
          \adjustbox{height=37.5mm}{

          }
\centering
\captionsetup{width=1\linewidth}
\caption{Objects in $\Fuk\left(\POP\right)$ where $\POP=S^2\setminus
\left\{3\text{ points}\right\}
$}
\label{fig:FukObjects}
\end{figure}




For a computational purpose,
we want to
(not globally)
trivialize $E$
so that the parallel transport between two points is trivial
away from a special point on $L$.
For that,
we choose a point {\small$\color{red}\bigstar$} on (the domain of) $L$
avoiding any intersection points (also with other loops)
and denote by
$$
\hol_{\small\color{red}\bigstar}(E)
\in
\Hom\left(\left.E\right|_{\small\color{red}\bigstar},\left.E\right|_{\small\color{red}\bigstar}\right)
$$
the holonomy of $(E,\nabla)$ along $L$ at {\small$\color{red}\bigstar$}
\footnote{
It is the only invariant of flat bundles under gauge equivalence.
Two gauge equivalent flat bundles $\left(E_1,\nabla_1\right)$ and $\left(E_2,\nabla_2\right)$
define quasi-isomorphic objects in the Fukaya category.
See, for example,
Proposition 4.9 in \cite{B17}.
}.
We choose an identification 
$
\left.E\right|_{\small\color{red}\bigstar}
\cong
\field^{\primemu}
$
where $\primemu$ is the rank of $E$,
then
$
\hol_{\small\color{red}\bigstar}(E)
$
is represented by some matrix
$
\primeLambda
\in
\operatorname{GL}_{\primemu}\left(\field\right).
$
(We simply say that the object $(L,E,\nabla)$ has a holonomy $\primeLambda$
(at {\small$\color{red}\bigstar$}).)

\begin{prop}\label{rmk:ChooseTrivialization}\label{rmk:HolonomyComputing}
There is a trivialization
$$
\left.E\right|_{S^1\setminus{\small\color{red}\bigstar}}
\cong
\left(S^1\setminus{\small\color{red}\bigstar}\right)
\times
\field^{\primemu}
$$
of $E$ over $S^1\setminus{\small\color{red}\bigstar}$
satisfying the following:
For any two points $p$ and $q$ on (the domain of) $L$ other than {\small$\color{red}\bigstar$},
the parallel transport
$$
P\left(L_{p\rightarrow q}\right)\in \Hom\left(\left.E\right|_p,\left.E\right|_q\right)
$$
from $\left.E\right|_p$ to $\left.E\right|_q$ along $L$
(in the shortest way from $p$ to $q$ following the orientation of $L$)
is represented by

(1) $\primeLambda$ if there is {\small$\color{red}\bigstar$} in the way from $p$ to $q$, and

(2) $I_{\primemu}$
(the identity matrix) otherwise,

\noindent
with respect to the induced identifications
$\left.E\right|_p\cong\field^{\primemu}$
and
$\left.E\right|_q\cong\field^{\primemu}$.
\end{prop}
\begin{proof}
Note that the linear isomorphisms
$
P\left(L_{{\small\color{red}\bigstar}\rightarrow p}\right)
:
E_{\small\color{red}\bigstar}
\xrightarrow{\cong}
E_p
$
for each $p\in S^1\setminus{\small\color{red}\bigstar}$
yield
the bundle isomorphism
$$
\begin{matrix}
\left(S^1\setminus{\small\color{red}\bigstar}\right)
\times
E_{\small\color{red}\bigstar}
\xrightarrow{\cong}
\left.E\right|_{S^1\setminus{\small\color{red}\bigstar}}.
\\[2mm]
\hspace{20.5mm}
\left(p,w\right)
\mapsto
P\left(L_{{\small\color{red}\bigstar}\rightarrow p}\right)
\left(w\right)
\end{matrix}
$$

In case (1),
$$
P\left(L_{p\rightarrow q}\right)
=
P\left(L_{{\small\color{red}\bigstar}\rightarrow p}\right)^{-1}
\circ
\hol_{\small\color{red}\bigstar}(E)
\circ
P\left(L_{{\small\color{red}\bigstar}\rightarrow q}\right)
:
E_p\xrightarrow{\cong}
E_{\small\color{red}\bigstar}
\rightarrow
E_{\small\color{red}\bigstar}
\xrightarrow{\cong}
E_q
$$
coincides with the map
$
\hol_{\small\color{red}\bigstar}(E)
:
E_{\small\color{red}\bigstar}
\rightarrow
E_{\small\color{red}\bigstar}
$
under the identifications of $E_p$ and $E_q$ with $E_{\small\color{red}\bigstar}$.

In case (2),
$$
P\left(L_{p\rightarrow q}\right)
=
P\left(L_{{\small\color{red}\bigstar}\rightarrow p}\right)^{-1}
\circ
P\left(L_{{\small\color{red}\bigstar}\rightarrow q}\right)
:
E_p\xrightarrow{\cong}
E_{\small\color{red}\bigstar}
\xrightarrow{\cong}
E_q
$$
is just the identity on
$
E_{\small\color{red}\bigstar}
$
under the same identifications.
\end{proof}


\subsubsection{Morphisms}

Given two objects
$
\mathcal{L}_0
=
\left(L_0,E_0,\nabla_0\right)
$
and
$
\mathcal{L}_1
=
\left(L_1,E_1,\nabla_1\right)
$
in $\Fuk\left(\Sigma\right)$,
roughly speaking,
their \textbf{morphism space}
$
\hom\left(\mathcal{L}_0,\mathcal{L}_1\right)
$
is defined as
a direct sum of $\mathbb{Z}_2$-graded $\field$-vector spaces
attached to each (self-)intersection of underlying curves $L_0$ and $L_1$.
We will explain the attached vector spaces below
by dividing it into two cases:
\label{page:TwoCasesOfHom}


\textbf{In the case $L_0 \ne L_1$},
as they are compact,
they have finitely many intersection points.
We define the set
$$
\chi\left(L_0,L_1\right)
:=
L_0 \cap L_1,
$$
which is divided into an \emph{even}-part and an \emph{odd}-part according to orientations of two curves at each element.
An element $p\in \chi\left(L_0,L_1\right)$ has an \emph{even}-degree if
the orientation of $T_p L_1 \oplus T_p L_0$ agrees with that of $T_p \Sigma$,
and an \emph{odd}-degree otherwise.
Both situations are compared in
Figure \ref{fig:Morphisms}.
We use $+$ or $-$ signs to indicate that $p$ is even or odd, respectively, writing $\left|p\right|=0$ or $1$.
We denote by
$
\chi^0\left(L_0,L_1\right)
$
(the even-part) and
$
\chi^1\left(L_0,L_1\right)
$
(the odd-part)
the subsets of
$
\chi\left(L_0,L_1\right)
$
consisting of even-degree and odd-degree elements,
respectively.

Each intersection of $L_0$ and $L_1$ contributes one element to each of $\chi\left(L_0,L_1\right)$ and $\chi\left(L_1,L_0\right)$.
We distinguish them by denoting one by $p$ and the other by $\overline{p}$.
Note that
$\left|p\right| + \left|\overline{p}\right| = 1$ always holds.
It is convenient to
view $p\in\chi\left(L_0,L_1\right)$ as a pair of clockwise opposite angles from $L_0$ to $L_1$,
and $\overline{p}\in\chi\left(L_1,L_0\right)$ as the angles from $L_1$ to $L_0$,
as described in Figure \ref{fig:Morphisms}.

\begin{figure}[H]
   \begin{minipage}{0.65\textwidth}
     \centering
\begin{subfigure}{0.45\textwidth}
\adjustbox{height=40mm}{

          }
\centering
\caption{
$p\in\chi^0\left(L_0,L_1\right)$,
$\overline{p}\in\chi^1\left(L_0,L_1\right)$
}
\end{subfigure}
\begin{subfigure}{0.45\textwidth}
\adjustbox{height=40mm}{
%
          }
\centering
\caption{
$p\in\chi^1\left(L_0,L_1\right)$,
$\overline{p}\in\chi^0\left(L_0,L_1\right)$
}
\end{subfigure}
\centering
\captionsetup{width=1\linewidth}
\caption{Angles $p\in \chi\left(L_0,L_1\right)$, $\overline{p}\in \chi\left(L_1,L_0\right)$ when $L_0\ne L_1$}
\label{fig:Morphisms}   \end{minipage}\hfill
   \begin{minipage}{0.35\textwidth}
     \centering
\captionsetup[subfigure]{labelformat=empty}
\begin{subfigure}{1\textwidth}
\adjustbox{height=40mm}{
%
          }
\centering
\caption{
$p\in\chi^0\left(L,L\right)$,
$\overline{p}\in\chi^1\left(L,L\right)$
}
\label{fig:SelfIntersections}
\end{subfigure}
\centering
\captionsetup{width=1\linewidth}
\caption{Angles $p$, $\overline{p}$
$\in \chi\left(L,L\right)$
}
\label{fig:SelfMorphisms}
   \end{minipage}
\end{figure}


\textbf{In the case $L:=L_0=L_1$},
it has finitely many self-intersections.
In a small neighborhood of such a self-intersection $p$,
there are two pieces
$\tilde{L}_a$, $\tilde{L}_b$ of $L$ meeting at $p$.
We may assume that the orientation of
$T_p \tilde{L}_b \oplus T_p \tilde{L}_a$
agrees with that of $T_p\Sigma$.
As in Figure \ref{fig:SelfMorphisms},
the pair of clockwise opposite convex angles from $\tilde{L}_a$ to $\tilde{L}_b$
is denoted by $p$
and has an even-degree.
Its adjacent pair of clockwise opposite convex angles from $\tilde{L}_b$ to $\tilde{L}_a$
is denoted by $\overline{p}$
and has an odd-degree.
Two (pair-of-)angles $p$ and $\overline{p}$ will have different meaning 
when we count polygons involving them.

We consider the marked points $e_L$ and $o_L$ to have even and odd degrees, respectively.
We define the set
$$
\chi\left(L,L\right)
:=
\left\{
e_L, o_L
\right\}
\cup
\left\{
p,\overline{p}
\left|
p:
\text{a self-intersection point of $L$}
\right.
\right\},
$$
which is divided into two subsets
$\chi^0\left(L,L\right)$
and
$\chi^1\left(L,L\right)$
as before.

\textbf{In both cases},
we define the $A_\infty$-morphism spaces as
\begin{equation}\label{eqn:DefinitionOfHom}
\hom^{\bullet}\left(\mathcal{L}_0,\mathcal{L}_1\right)
:=
\bigoplus_{p\in \chi^{\bullet}\left(L_0,L_1\right)}
\Hom_{\field}\left(\left.E_0\right|_p,\left.E_1\right|_{p}\right)
\quad
\left(\bullet\in\mathbb{Z}_2\right)
\end{equation}
where
$\left.E_0\right|_p$ and $\left.E_1\right|_{p}$
are the fibers of $E_0$ and $E_1$ over
the preimages (in $S^1$) of the point $p\in\Sigma$ under $L_0$ and $L_1$
(or under different branches $\tilde{L}_a$ and $\tilde{L}_b$ of $L$ in the case $L:=L_0 = L_1$),
respectively.
They yield a $\mathbb{Z}_2$-graded morphism space
$$
\hom\left(\mathcal{L}_0,\mathcal{L}_1\right)
:=
\hom^0\left(\mathcal{L}_0,\mathcal{L}_1\right)
\oplus
\hom^1\left(\mathcal{L}_0,\mathcal{L}_1\right).
$$
An element $f\in\Hom_{\field}\left(\left.E_0\right|_p,\left.E_1\right|_p\right)$
for each $p\in\chi\left(L_0,L_1\right)$
is called a \emph{base morphism} over $p$,
and we denote it by $\left.f\right|_p$
to specify that fact.

\begin{remark}\label{rmk:TrivializingHom}
According to Proposition \ref{rmk:ChooseTrivialization},
we have identifications
$\left.E_0\right|_{p} \cong \field^{\primemu_0}$
and
$\left.E_1\right|_{p} \cong \field^{\primemu_1}$,
which yield
$
\Hom_{\field}\left(\left.E_0\right|_p,\left.E_1\right|_{p}\right)
\cong
\field^{\primemu_1 \times \primemu_0}.
$
Denoting by $p_{ab}$ the generator corresponding to the $\primemu_1\times\primemu_0$ matrix
whose the only nonzero entry is $1$ in the $(a,b)$-th position,
we have
$$
\hom^\bullet\left(\mathcal{L}_0,\mathcal{L}_1\right)
\cong
\bigoplus_{p\in \chi^\bullet\left(L_0,L_1\right)}
\bigoplus_{\smat{1\le a \le \primemu_1 \\ 1\le b \le \primemu_0}}
\Span\left\{p_{ab}\right\}
\quad
(\bullet\in\mathbb{Z}_2).
$$
In the simplest case $\primemu_0=\primemu_1=1$,
we just write
$$
\hom^\bullet\left(\mathcal{L}_0,\mathcal{L}_1\right)
\cong
\bigoplus_{p\in \chi^\bullet\left(L_0,L_1\right)}
\Span\left\{p\right\}
=
\Span
\left(
\chi^\bullet\left(L_0,L_1\right)
\right)
\quad
(\bullet\in\mathbb{Z}_2).
$$

\end{remark}

\subsubsection{$\AI$-operations}

The \textbf{$A_{\infty}$-operations} $\left\{\operatorm_k\right\}_{k\ge1}$
on base morphisms between objects
$
\mathcal{L}_i
=
\left(L_i,E_i,\nabla_i\right)
$
($i\in\left\{0,\dots,k\right\}$)
are defined as
\begin{align}\label{eqn:OperatormDefinition}
\begin{split}
\operatorm_{k}
:
\hom\left(\mathcal{L}_0,\mathcal{L}_1\right)
\otimes \cdots \otimes
\hom\left(\mathcal{L}_{k-1},\mathcal{L}_k\right)
&\rightarrow
\hom\left(\mathcal{L}_0,\mathcal{L}_k\right)
\\
\left(\left.f_1\right|_{p_1},\dots,\left.f_k\right|_{p_k}\right)
&\mapsto
\sum_{q\in
\chi\left(L_0,L_k\right)
}
\sum_{u\in\mathcal{M}\left(p_1,\dots,p_k,\overline{q}\right)}
\sign\left(u\right)
\hol_q\left(\partial u\right) \left(f_1,\dots,f_k\right)
\end{split}
\end{align}
for any
$
p_i\in\chi\left(L_{i-1},L_i\right)
$
other than $e_L$ or $o_L$
\footnote{
If $e_L$ or $o_L$'s are involved as inputs, the definition of $\AI$-operation becomes much complicated.
See Remark \ref{rmk:AIoperationwithInfinitesimalGenerators} for some cases.
},
such that
$p_{i+1}\ne \overline{p_i}$,
and
$
f_i=\left.f_i\right|_{p_i}
\in
\Hom\left(\left.E_{i-1}\right|_{p_i},\left.E_i\right|_{p_i}\right).
$
Then it is linearly extended to other morphisms.
We will explain the meaning of each component one by one below.

First,
the \textbf{moduli space}
$
\mathcal{M}\left(p_1,\dots,p_k,\overline{q}\right)
$
is the set of immersed $(k+1)$-gons bounded by
$
L_0, L_1, \dots, L_k
$
whose angles consist of $p_1, \dots, p_k, \overline{q}$ in counter-clockwise order.
To be precise,
it means a continuous map
$u:D^2\rightarrow\Sigma$ 
together with $k+1$ points
$
z_1,\dots,z_k,z_0
\in\partial D^2
$
(in counterclockwise order)
such that
the segment of $\partial D^2$ between $z_i$ and $z_{i+1}$ is mapped to $L_i$
($i\in\mathbb{Z}_{k+1}$),
the image of $u$ has a convex corner $p_i$ at
$u\left(z_i\right)=p_i$
($p_{k+1}:=\overline{q}$),
and
$u$ is an orientation-preserving immersion on
$D^2\setminus\left\{z_1,\dots,z_k,z_0\right\}$.
We consider such maps up to automorphisms of the domain $D^2$,
that is,
$\left(\left\{z_1,\dots,z_k,z_0\right\},u\right)$ and $\left(\left\{z_1',\dots,z_k',z_0'\right\},u'\right)$ define the same element in the moduli space
if and only if there is a homeomorphism $\phi:D^2\rightarrow D^2$ such that
$\phi\left(z_i\right) = z_i'$ ($i\in\mathbb{Z}_{k+1}$),
$u$ is a diffeomorphism on $D^2\setminus\left\{z_1,\dots,z_k,z_0\right\}$,
and
$u'=u\circ \phi$.

\begin{figure}[H]
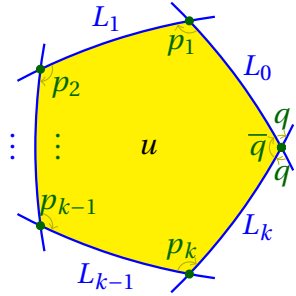

\centering
\adjustbox{height=40mm}{

          }
\captionsetup{width=1\linewidth}
\caption{An element $u$ of $\mathcal{M}\left(p_1,\dots,p_k,\overline{q}\right)$
}
\label{fig:m_kOperator}
\end{figure}




%
%
%

Second, to determine the \textbf{sign} of $u$,
we follow the sign rule established and illustrated in \cite{S08,Se}:
Consider the boundary orientation on $\partial u$ as usual,
that is,
it is given in such a way that $u$ lies on the left along it.
The orientation of $L_0$ is irrelevant.
For $1\le i\le k$,
whenever the orientation of $L_k$ does not match the orientation of $\partial u$,
$(-1)^{\left|p_i\right|}$ is contributed to $\sign\left(u\right)$,
which changes the sign only when the angle $p_i$ from $L_{i-1}$ to $L_i$ has odd-degree.
In addition,
if the orientation of $L_k$ differs from that of $\partial u$,
the sign $(-1)^{\left|q\right|}$ of the output angle $q$ from $L_0$ to $L_k$ is also contributed.
Summing up,
we have
\begin{align}\label{eqn:sign(u)}
\begin{split}
\sign\left(u\right)
&:=
\prod_{
       \smat{
             1\le i \le k,
             \\
             \text{orientation of $L_i$}
             \\
             \text{$\ne$ orientation of $\partial u$}
             }
      }
(-1)^{\left|p_i\right|}
\times
\prod_{
       \smat{
             \text{if orientation of $L_k$}
             \\
             \text{$\ne$ orientation of $\partial u$}
             }
      }
(-1)^{\left|q\right|}
\\
&=
(-1)^\wedge
\left(
\sum_{i=1}^k
\left|p_i\right|
\mathbb{1}_{\operatorname{o}\left(L_i\right)\ne\operatorname{o}\left(\partial u\right)}
+
\left|q\right|
\mathbb{1}_{\operatorname{o}\left(L_k\right)\ne\operatorname{o}\left(\partial u\right)}
\right),
\end{split}
\end{align}
where $\operatorname{o}(L)$ denotes the orientation of a curve $L$
and
$
\mathbb{1}_{\text{statement}}
$
is $1$ if the statement is true and $0$ otherwise.



Third,
the \textbf{holonomy operation} of $\partial u$ at $q$ is defined as
\begin{multline*}
\hol_q\left(\partial u\right)
:
\Hom_{\field}\left(\left.E_0\right|_{p_1},\left.E_1\right|_{p_1}\right)
\otimes \cdots \otimes
\Hom_{\field}\left(\left.E_{k-1}\right|_{p_k},\left.E_{k}\right|_{p_k}\right)
\rightarrow
\Hom_{\field}\left(\left.E_0\right|_{p_0},\left.E_k\right|_{p_0}\right)
\\
\left(f_1,\dots,f_k\right)
\mapsto
P\left(\left(\partial u\right)
_k
\right)
\ \circ \
f_k
\ \circ \
P\left(\left(\partial u\right)
_{k-1}
\right)
\ \circ \
f_{k-1}
\ \circ \
\cdots
\ \circ \
f_2
\ \circ \
P\left(\left(\partial u\right)
_1
\right)
\ \circ \
f_1
\ \circ \
P\left(\left(\partial u\right)
_0
\right)
\end{multline*}
where
$$
P\left(\left(\partial u\right)
_i
\right)
\in
\Hom_{\field}\left(\left.E_i\right|_{
p_i
},\left.E_i\right|_{p_{i+1}}\right)
$$
is the parallel transport
with respect to $\nabla_i$
from $\left.E_i\right|_{
p_i
}$ to $\left.E_i\right|_{p_{i+1}}$
along the side of $u$
(following the boundary orientation)
lying in $L_i$.


%
%
%
%

\begin{remark}\label{rmk:AIoperationwithInfinitesimalGenerators}

When $o_{L_i}$ or $e_{L_i}$ are involved
in
the moduli space
$\mathcal{M}\left(p_1,\dots,p_k,p_{k+1}\right)$,
an appropriate perturbation scheme can be introduced.
Here we follow \cite{Se} and explain
some special cases where we have only one of them in input or output of the $\AI$-operations,
which will be enough for our purpose
\footnote{
But see also Figure \ref{fig:HamiltonianPerturbation}
for the case where many $o_{L_i}$'s are involved.
}.

\noindent
(1)
If $\mathcal{L}_{i-1} = \mathcal{L}_i$
$\left(i\in\left\{1,\dots,k\right\}\right)$
(and hence $L_{i-1}=L_i$ is a loop)
and an input $p_i$ is
$
o_{L_i}
$
from $L_{i-1}=L_i$ to itself,
then elements of
$
\mathcal{M}\left(p_1,\dots,p_{i-1},o_{L_i},p_{i+1},\dots,p_k,\overline{q}\right)
$
are polygons
which have convex corners at $p_1,\dots,p_{i-1}$,
pass through the point $o_{L_i}$,
and then again have convex corners at
$p_{i+1},\dots,p_k,\overline{q}$ in counterclockwise order.

\noindent
(2)
If $\mathcal{L}_{i-1} = \mathcal{L}_i$
$\left(i\in\left\{1,\dots,k\right\}\right)$
and $p_i$ is $e_{L_{i-1}}=e_{L_i}$,
the set
$
\mathcal{M}\left(p_1,\dots,e_{L_i},\dots,p_k,\overline{q}\right)
$
is empty unless $k=2$,
which yields
\begin{equation}\label{eqn:IdentityEquation1}
\operatorm_k\left(\dots,\left.\id\right|_{e_{L_i}},\dots\right)
=
0
\quad
\text{if }k\ne2.
\end{equation}
If $k=2$,
given two objects
$
\mathcal{L}
=
\left(L,E,\nabla\right)
$
and
$
\mathcal{L}'
=
\left(L',E',\nabla'\right),
$
we regard a segment
$
\left(

\right)
$
of $L$ from $e_L$ to any point $p\in \chi\left(L,L'\right)$
(which doesn't pass through $o_L$)
as an `infinitesimal triangle'
whose angles are $e_L$, $p$ and $\overline{p}$.
It provides the unique element of
$
\mathcal{M}\left(e_L,p,\overline{p}\right)
$
and
$
\mathcal{M}\left(\overline{p},e_L,p\right),
$
which yields
\begin{equation}\label{eqn:IdentityEquation2}
\operatorm_2\left(\left.\id\right|_{e_L},\left.f\right|_p\right)
=
\left.f\right|_p
\quad\text{and}\quad
\operatorm_2\left(\left.g\right|_{\overline{p}},\left.\id\right|_{e_L}\right)
=
(-1)^{\left|\overline{p}\right|}
\left.g\right|_{\overline{p}}
\end{equation}
for any
$
\left.f\right|_p
\in\hom\left(\mathcal{L},\mathcal{L}'\right)
$
and
$
\left.g\right|_{\overline{p}}
\in\hom\left(\mathcal{L}',\mathcal{L}\right).
$
Equations (\ref{eqn:IdentityEquation1}) and (\ref{eqn:IdentityEquation2})
imply that 
$\id_{\mathcal{L}} := \left.\id\right|_{e_L}$ is a unit of $\mathcal{L}$.

\noindent
(3)
If
$\mathcal{L}_0 = \mathcal{L}_k =: \mathcal{L}$
and the output $q$ is $e_L$,
elements of
$
\mathcal{M}\left(p_1,\dots,p_k,\overline{e_L}\right)
$
are polygons which have convex corners at
$
p_1,\dots,p_k
$
and then pass through the point $e_L$
in counterclockwise order.

\end{remark}

\begin{remark}
In fact,
the definition given in (\ref{eqn:OperatormDefinition}) involves a crucial problem.
That is,
sometimes there are infinitely many elements in the moduli space $\mathcal{M}\left(p_1,\dots,p_k,\overline{q}\right)$.
Therefore, a priori, we must work over the Novikov field
$$
\Lambda := \left\{ \left. \sum_{i=0}^\infty a_i T^{\lambda_i} \right| a_i \in \field,~\lambda_i \in \mathbb{R},~\lim_{i \to \infty} \lambda_i = \infty \right\}
$$
instead of the base field $\field$
(which makes the hom spaces $\Lambda$-vector spaces)
and define the $\AI$-operations as
\begin{align*}\label{eqn:OperatormDefinitionFiltrated}
\begin{split}
\operatorm_{k}
:
\hom\left(\mathcal{L}_0,\mathcal{L}_1\right)
\otimes \cdots \otimes
\hom\left(\mathcal{L}_{k-1},\mathcal{L}_k\right)
&\rightarrow
\hom\left(\mathcal{L}_0,\mathcal{L}_k\right)
\\
\left(\left.f_1\right|_{p_1},\dots,\left.f_k\right|_{p_k}\right)
&\mapsto
\sum_{q\in
\chi\left(L_0,L_k\right)
}
\sum_{u\in\mathcal{M}\left(p_1,\dots,p_k,\overline{q}\right)}
\sign\left(u\right)
T^{\omega\left(u\right)}
\hol_q\left(\partial u\right) \left(f_1,\dots,f_k\right)
\end{split}
\end{align*}
instead of (\ref{eqn:OperatormDefinition}).
But we demonstrated in \cite{CJKR} that we can evaluate $T=1$
when we compute the matrix factorizations corresponding to cylinder-free loops
under the localized mirror functor
(See Definition \ref{defn:CylinderFree} and Theorem \ref{thm:CylinderFreeMFFinite}.).
In this paper, therefore, we still work over $\field$ and use the definition (\ref{eqn:OperatormDefinition})
when computing the mirror images of cylinder-free loops.
%
%
%
\end{remark}

\subsection{Categories of matrix factorizations}
\label{sec:MFCategories}

Let $S$ be
the formal power series ring
$
\field[[x_1,\dots,x_m]]
$
of $m$ variables, 
and $f\in S$ its nonzero element.
\begin{defn}\label{defn:MatrixFactorization}
A \textnormal{\textbf{matrix factorization}} of $f$ (in $S$) is a pair $\left(\varphi,\psi\right)$
of $S$-module homomorphisms
$
\begin{tikzcd}[arrow style=tikz,>=stealth, sep=20pt, every arrow/.append style={shift left=0.5}]
   P^0
     \arrow{r}{\varphi}
   &
   P^1
     \arrow{l}{\psi}
\end{tikzcd}
$
between two
finite-rank
free $S$-modules $P^0$ and $P^1$
that satisfy
$$
\psi\varphi = f\cdot \id_{P^0}
\quad\text{and}\quad
\varphi\psi = f\cdot \id_{P^1}
\footnote{
These conditions imply that $P^0$ and $P^1$ have the same rank.
Moreover,
$\psi$ is completely determined by $\varphi$ and vice versa.
}.
$$
\end{defn}


There are several versions of
the \emph{categories of matrix factorizations} of $f$.
They all have matrix factorizations of $f$ as objects,
but morphism spaces are different,
which we will now explain:


Given two matrix factorizations
$
\begin{tikzcd}[arrow style=tikz,>=stealth, sep=20pt, every arrow/.append style={shift left=0.5}]
   P_0^0
     \arrow{r}{\varphi_0}
   &
   P_0^1
     \arrow{l}{\psi_0}
\end{tikzcd}
$
and
$
\begin{tikzcd}[arrow style=tikz,>=stealth, sep=20pt, every arrow/.append style={shift left=0.5}]
   P_1^0
     \arrow{r}{\varphi_1}
   &
   P_1^1
     \arrow{l}{\psi_1}
\end{tikzcd}
$
of $f$,
an \emph{even-degree} or \emph{odd-degree} \textbf{morphism} between them
is
a pair
of $S$-module maps
$
\left(\alpha:P_0^0\rightarrow P_1^0,\
\beta:P_0^1\rightarrow P_1^1\right)
$
or
$
\big(\gamma:P_0^0\rightarrow P_1^1,\
\delta:P_0^1\rightarrow P_1^0\big)
$,
respectively.
They consist the \emph{even-part} and \emph{odd-part} of the morphism space defined as
$$
\hom^\bullet\left(\left(\varphi_0,\psi_0\right),\left(\varphi_1,\psi_1\right)\right)
:=
\Hom_S\left(P_0^0,P_1^\bullet\right)
\times
\Hom_S\left(P_0^1,P_1^{1+\bullet}\right)
\quad
(\bullet\in\mathbb{Z}_2),
$$
and form the following diagram but are \emph{not} required to commute it:
\begin{equation}\label{eqn:MFMorphismDiagram}
\setlength\arraycolsep{10mm}
\begin{matrix}
\begin{tikzcd}[arrow style=tikz,>=stealth,row sep=2em,column sep=2em] 
P_0^0
  \arrow[r,"\varphi_0"]
  \arrow[d,swap,"\alpha"]
&
P_0^1
  \arrow[r,"\psi_0"]
  \arrow[d,swap,"\beta"]
&
P_0^0
  \arrow[d,swap,"\alpha"]
\\
P_1^0
  \arrow[r,"\varphi_1"]
&
P_1^1
  \arrow[r,"\psi_1"]
&
P_1^0
\end{tikzcd}
&
\begin{tikzcd}[arrow style=tikz,>=stealth,row sep=2em,column sep=2em] 
&
P_0^0
  \arrow[r,"\varphi_0"]
  \arrow[dl,swap,"\smat{\gamma\\[-1mm]}"]
&
P_0^1
  \arrow[r,"\psi_0"]
  \arrow[dl,swap,"\smat{\delta\\[-2mm]}"]
&
P_0^0
  \arrow[dl,swap,"\smat{\gamma\\[-2mm]}"]
\\
P_1^1
  \arrow[r,"\psi_1"]
&
P_1^0
  \arrow[r,"\varphi_1"]
&
P_1^1
&
\end{tikzcd}
\end{matrix}
\end{equation}


The \textbf{$\mathbb{Z}_2$-graded morphism space}
is a $\mathbb{Z}_2$-graded $\field$-vector space
defined as
$$
\hom^{\mathbb{Z}_2}\left(\left(\varphi_0,\psi_0\right),\left(\varphi_1,\psi_1\right)\right)
:=
\hom^0\left(\left(\varphi_0,\psi_0\right),\left(\varphi_1,\psi_1\right)\right)
\oplus
\hom^1\left(\left(\varphi_0,\psi_0\right),\left(\varphi_1,\psi_1\right)\right).
$$
We define a $\field$-linear map $d$ of degree $1$ on it as
$$\hom^0\left(\left(\varphi_0,\psi_0\right),\left(\varphi_1,\psi_1\right)\right)
\xrightarrow{d}
\hom^1\left(\left(\varphi_0,\psi_0\right),\left(\varphi_1,\psi_1\right)\right)
\quad\text{and}\quad
\hom^1\left(\left(\varphi_0,\psi_0\right),\left(\varphi_1,\psi_1\right)\right)
\xrightarrow{d}
\hom^0\left(\left(\varphi_0,\psi_0\right),\left(\varphi_1,\psi_1\right)\right).
$$
$$
\hspace{13mm}
\left(\alpha,\beta\right)
\mapsto
\left(\varphi_1\circ\alpha-\beta\circ\varphi_0,
\
\psi_1\circ\beta-\alpha\circ\psi_0\right)
\hspace{19mm}
\left(\gamma,\delta\right)
\mapsto
\left(\psi_1\circ\gamma+\delta\circ\varphi_0,
\
\varphi_1\circ\delta+\gamma\circ\psi_0\right)
$$
It satisfies $d^2 = 0$,
making
the morphism space
a $2$-periodic cochain complex
with a \textbf{differential} $d$.

\begin{defn}\label{def:CategoriesOfMatrixFactorizations}
We define four different \textnormal{\textbf{categories of matrix factorizations}} of $f$.
Their objects are matrix factorizations of $f$.
The \textnormal{\textbf{morphism space}} between two matrix factorizations
$\left(\varphi_0,\psi_0\right)$
and $\left(\varphi_1,\psi_1\right)$ 
will be defined,
while
the composition of morphisms and the identity morphisms are given in the obvious way.

\noindent
(1)
In \textnormal{\textbf{differential $\mathbb{Z}_2$-graded category
$\MF_{\operatorname{dg}}(f)$}},
it is given by
$
\left(\hom^{\mathbb{Z}_2}\left(\left(\varphi_0,\psi_0\right),\left(\varphi_1,\psi_1\right)\right),d\right),
$
which is a $\mathbb{Z}_2$-graded $\field$-vector space
equipped with the differential $d$.

\noindent
(2)
In \textnormal{\textbf{(ordinary) category}}
$\MF(f)$,
it is given by
$
Z^0\left(\hom^{\mathbb{Z}_2}\left(\left(\varphi_0,\psi_0\right),\left(\varphi_1,\psi_1\right)\right),d\right),
$
which consists of even-degree morphisms
$
\left(\alpha,\beta\right)
\in
\hom^0\left(\left(\varphi_0,\psi_0\right),\left(\varphi_1,\psi_1\right)\right)
$
that commute the left diagram in (\ref{eqn:MFMorphismDiagram}).

\noindent
(3)
In \textnormal{\textbf{homotopy category}}
$\underline{\MF}(f)$,
it is given by
$
H^0\left(\hom^{\mathbb{Z}_2}\left(\left(\varphi_0,\psi_0\right),\left(\varphi_1,\psi_1\right)\right),d\right).
$
This category is
the same as the
\textnormal{\textbf{stable category}}
$\underline{\MF}(f):=\left.\MF(f)\right/\left\{(1,f),(f,1)\right\}$
(Definition \ref{def:StableCategories}).

\noindent
(4)
In ($\mathbb{Z}_2$-graded) \textnormal{\textbf{$\AI$-category}}
$\MF_{\AI}(f)$,
it is given by
the $\mathbb{Z}_2$-graded $\field$-vector space
$
\hom^{\mathbb{Z}_2}\left(\left(\varphi_1,\psi_1\right),\left(\varphi_0,\psi_0\right)\right).
$
(Note that it is reversed.)
The $\AI$-operations
$\left\{\operatorm_k\right\}_{k\ge1}$ are defined as:
\begin{equation}\label{eqn:MFAIOperations}
\operatorm_1:=d,
\quad
\operatorm_2\left(\left(\alpha_0,\beta_0\right),\left(\alpha_1,\beta_1\right)\right)
:=
(-1)^{\bullet-1}
\left(\alpha_0,\beta_0\right)
\circ
\left(\alpha_1,\beta_1\right),
\quad\text{and}\quad
\operatorm_{k\ge3}:=0
\end{equation}
for
$
\left(\alpha_0,\beta_0\right)
\in
\hom^\bullet\left(\left(\varphi_1,\psi_1\right),\left(\varphi_0,\psi_0\right)\right)
$
($\bullet\in\mathbb{Z}_2$)
and
$
\left(\alpha_1,\beta_1\right)
\in
\hom^{\mathbb{Z}_2}\left(\left(\varphi_2,\psi_2\right),\left(\varphi_1,\psi_1\right)\right).
$

\end{defn}

%
%
%
%
%
%
%
%
%
%
%
%
%
%

\subsection{Stable categories and Eisenbud's equivalence}
\label{sec:EisenbudEquivalence}
\footnote{
All discussions in this section
(except for Proposition \ref{prop:StableHomotopySame})
can be found in Chapter 7 in \cite{Yo}.
}
Let $S$ be
the power series ring
$
\field[[x_1,\dots,x_m]]
$
of $m$ variables,
$f\in S$ its nonzero element
and $A:=\left.S\right/(f)$
the quotient ring.
In this case,
we can relate the category $\CM(A)$ of maximal Cohen-Macaulay modules over $A$
(Definition \ref{defn:MCMModules})
and
the category $\MF(f)$ of matrix factorizations of $f$
(Definition \ref{def:CategoriesOfMatrixFactorizations}(2))
in a \emph{stable} sense.
We need some preparation:

\begin{defn}\label{def:QuotientCategory}
Let $\mathcal{A}$ be a category whose $\Hom$-sets are abelian groups,
and $P$ a set of some objects in $\mathcal{A}$.
For any two objects $M$, $N$ in $\mathcal{A}$,
we denote by $I\left(M,N\right)$ the subgroup of $\Hom_{\mathcal{A}}(M,N)$
generated by all morphisms from $M$ to $N$ that factor through a direct sum of objects in $P$.
We define the \emph{quotient} of $\mathcal{A}$ by $P$ as the category $\left.\mathcal{A}\right/P$
whose objects are the same as $\mathcal{A}$, and the morphism spaces are defined as
$$
\Hom_{\left.\mathcal{A}\right/P}\left(M,N\right)
:=
\left.
\Hom_{\mathcal{A}}\left(M,N\right)
\right/
I\left(M,N\right).
$$
Note that any objects in $P$ are zero objects in the category $\left.\mathcal{A}\right/P$,
and it is the largest quotient of $\mathcal{A}$ with this property.
\end{defn}


Definition \ref{def:QuotientCategory}
enables us
to define the \emph{stable categories} of maximal Cohen-Macaulay modules over $A$
and
matrix factorizations of $f$
\footnote{
In general,
one can take the quotient of a \emph{Frobenius category} by its projective objects
to make it a triangulated category,
called the \emph{stable category}.
See \S 3.3 in \cite{Krause}.
}:

\begin{defn}\label{def:StableCategories}
(1) The \textnormal{\textbf{stable category of maximal Cohen-Macaulay modules}} over $A$
is defined as
$$
\underline{\CM}(A)
:=
\left.\CM(A)\right/\left\{A\right\}.
$$

\noindent
(2) The \textnormal{\textbf{stable category of matrix factorizations}} of $f$
is defined as
$$
\underline{\MF}(f)
:=
\left.\MF(f)\right/\left\{(1,f),(f,1)\right\}.
$$
\end{defn}

The following
justifies our use of the same notation as the homotopy category $\underline{\MF}(f)$
in Definition \ref{def:CategoriesOfMatrixFactorizations}:

\begin{prop}\label{prop:StableHomotopySame}
The stable category and the homotopy category of matrix factorizations of $f$ are the same.
\end{prop}
\begin{proof}
A more general statement and proof for the identification of
the stable category and the homotopy category of cochain complexes
in an additive category
can be found in \S 4.1 in \cite{Krause}.
Here we provide an explicit proof for our setting:

We actually show that their Hom-sets are the same.
For that,
it is enough to show
$$
I\left(\left(\varphi,\psi\right),\left(\varphi',\psi'\right)\right)
=
B^0\left(\hom^{\mathbb{Z}_2}\left(\left(\varphi,\psi\right),\left(\varphi',\psi'\right)\right),d\right)
$$
for any matrix factorizations
$
\begin{tikzcd}[arrow style=tikz,>=stealth, sep=20pt, every arrow/.append style={shift left=0.5}]
   P^0
     \arrow{r}{\varphi}
   &
   P^1
     \arrow{l}{\psi}
\end{tikzcd}
$
and
$
\begin{tikzcd}[arrow style=tikz,>=stealth, sep=20pt, every arrow/.append style={shift left=0.5}]
   P'^0
     \arrow{r}{\varphi'}
   &
   P'^1
     \arrow{l}{\psi'}
\end{tikzcd}
$
of $f$.

First let
$
\left(
\alpha:P^0\rightarrow P'^0,\
\beta:P^1\rightarrow P'^1
\right)
$
be a generator of
$
I\left(\left(\varphi,\psi\right),\left(\varphi',\psi'\right)\right),
$
i.e.,
it factors through a direct sum of $(1,f)$'s and $(f,1)$'s,
which is written as
$
\begin{tikzcd}[arrow style=tikz,>=stealth, sep=35pt, every arrow/.append style={shift left=0.5}]
   S^a\oplus S^b
     \arrow{r}{\spmat{I_a & 0 \\ 0 & fI_b}}
   &
   S^a\oplus S^b
     \arrow{l}{\spmat{fI_a & 0 \\ 0 & I_b}}
\end{tikzcd}
$
for some $a,b\in\mathbb{Z}_{\ge0}$
without loss of generality.
That is,
$\left(\alpha,\beta\right)$
is a composition
$
\left(
\alpha''\circ\alpha',\
\beta''\circ\beta'
\right)
$
of some morphisms
that are described in (and commute the squares in) the following diagram:
$$
\begin{tikzcd}[arrow style=tikz,>=stealth,row sep=3em,column sep=4em] 
P^0
  \arrow[r,"\varphi"]
  \arrow[d,swap,"\alpha':=\spmat{\alpha_1'\\ \alpha_2'}"]
&
P^1
  \arrow[r,"\psi"]
  \arrow[ddl,color=pink,swap,"\smat{\alpha_1''\circ\beta_1'\hspace{4mm}\\[-13mm] }"]
  \arrow[d,swap,"\beta':=\spmat{\beta_1'\\ \beta_2'}"]
&
P^0
  \arrow[d,swap,"\alpha':=\spmat{\alpha_1'\\ \alpha_2'}"]
  \arrow[ddl,color=pink,swap,"\smat{\beta_2''\circ\alpha_2'\hspace{4mm}\\[-13mm] }"]
\\
S^a\oplus S^b
  \arrow[r,"\spmat{I_a & 0 \\ 0 & fI_b}"]
  \arrow[d,swap,"\alpha'':=\spmat{\alpha_1''&\alpha_2''}"]
&
S^a\oplus S^b
  \arrow[r,"\spmat{fI_a & 0 \\ 0 & I_b}"]
  \arrow[d,swap,"\beta'':=\spmat{\beta_1''&\beta_2''}"]
&
S^a\oplus S^b
  \arrow[d,swap,"\alpha'':=\spmat{\alpha_1''&\alpha_2''}"]
\\
P'^0
  \arrow[r,"\varphi'"]
&
P'^1
  \arrow[r,"\psi'"]
&
P'^0
\end{tikzcd}
$$
Now
we know that $\left(\alpha,\beta\right)$ is \emph{null-homotopic}
by checking that it is
the differential of
an odd-degree morphism
$
\left(
\beta_2''\circ\alpha_2':P^0 \rightarrow P'^1,\
\alpha_1''\circ\beta_1':P^1 \rightarrow P'^0
\right).
$
This shows
$
\left(\alpha,\beta\right)
\in
B^0\left(\hom^{\mathbb{Z}_2}\left(\left(\varphi,\psi\right),\left(\varphi',\psi'\right)\right),d\right).
$

Conversely,
an element of 
$
B^0\left(\hom^{\mathbb{Z}_2}\left(\left(\varphi,\psi\right),\left(\varphi',\psi'\right)\right),d\right)
$
is
by definition
given by the differential of
an odd-degree morphism
$
\left(
\gamma:P^0\rightarrow P'^1,\
\delta:P^1\rightarrow P'_0
\right).
$
As the differential is additive,
we may assume $\delta=0$ without loss of generality.
Now the differential of 
$\left(\gamma,0\right)$
is
$
\left(
\psi'\circ\gamma,
\gamma\circ\psi
\right),
$
and the following commutative diagram shows that it factors through the direct sum
of
$
\begin{tikzcd}[arrow style=tikz,>=stealth, sep=25pt, every arrow/.append style={shift left=0.5}]
   P^0
     \arrow{r}{\id}
   &
   P^0
     \arrow{l}{f\cdot\id}
\end{tikzcd}
$
and
$
\begin{tikzcd}[arrow style=tikz,>=stealth, sep=25pt, every arrow/.append style={shift left=0.5}]
   P'^1
     \arrow{r}{f\cdot\id}
   &
   P'^1
     \arrow{l}{\id}
\end{tikzcd}
$,
being an element of
$
I\left(\left(\varphi,\psi\right),\left(\varphi',\psi'\right)\right):
$
$$
\begin{tikzcd}[arrow style=tikz,>=stealth,row sep=3em,column sep=4em] 
P^0
  \arrow[r,"\varphi"]
  \arrow[d,swap,"\spmat{f\cdot\id\\[1mm] \gamma}"]
&
P^1
  \arrow[r,"\psi"]
  \arrow[ddl,color=pink,swap,"\smat{0\hspace{4mm}\\[-13mm] }"]
  \arrow[d,swap,"\spmat{\psi\\[1mm] \gamma\circ\psi}"]
&
P^0
  \arrow[d,swap,"\spmat{f\cdot\id\\[1mm] \gamma}"]
  \arrow[ddl,color=pink,swap,"\smat{\gamma\hspace{4mm}\\[-13mm] }"]
\\
P^0\oplus P'^1
  \arrow[r,"\spmat{\id & 0 \\ 0 & f\cdot\id}"]
  \arrow[d,swap,"\spmat{0 & \psi'}"]
&
P^0\oplus P'^1
  \arrow[r,"\spmat{f\cdot\id & 0 \\ 0 & \id}"]
  \arrow[d,swap,"\spmat{0 & \id}"]
&
P^0\oplus P'^1
  \arrow[d,swap,"\spmat{0 & \psi'}"]
\\
P'^0
  \arrow[r,"\varphi'"]
&
P'^1
  \arrow[r,"\psi'"]
&
P'^0
\end{tikzcd}
\vspace{-6mm}
$$
\end{proof}

Now we can state \emph{Eisenbud's equivalence} between two stable categories:




\begin{thm}[Eisenbud's matrix factorization theorem \cite{E80,Yo}]\label{thm:EisenbudTheorem}

A matrix factorization
$
\begin{tikzcd}[arrow style=tikz,>=stealth, sep=20pt, every arrow/.append style={shift left=0.5}]
   P^0
     \arrow{r}{\varphi}
   &
   P^1
     \arrow{l}{\psi}
\end{tikzcd}
$
of $f$
defines a $2$-periodic acyclic chain complex of $A$-modules
$$
\begin{tikzcd}[arrow style=tikz,>=stealth,row sep=5em,column sep=2em] 
\cdots
  \arrow[r]
&
P^0 \otimes_S A
  \arrow[r,"\underline{\varphi}"]
&
P^1 \otimes_S A
  \arrow[r,"\underline{\psi}"]
&
P^0 \otimes_S A
  \arrow[r,"\underline{\varphi}"]
&
P^1 \otimes_S A
  \arrow[r,"\underline{\psi}"]
&
\cdots,
\end{tikzcd}
$$
where
$\underline{\varphi}:=\varphi\otimes \id_A$
and
$\underline{\psi}:=\psi \otimes\id_A$.
Taking the cokernel of $\underline{\varphi}$ yields a maximal Cohen-Macaulay $A$-module
$M:=\cok\underline{\varphi}$
\footnote{
Note that $M$ admits $2$-periodic free resolution
$
\cdots
\rightarrow
P^0 \otimes_S A
\xrightarrow{\underline{\varphi}}
P^1 \otimes_S A
\xrightarrow{\underline{\psi}}
P^0 \otimes_S A
\xrightarrow{\underline{\varphi}}
P^1 \otimes_S A
\xrightarrow{}
M=\cok\underline{\varphi}
\rightarrow
0.
$
}.
It defines a functor
$
\cok:\MF(f)\rightarrow\CM(A),
$
which also induces a functor between stable categories
$$
\cok:\underline{\MF}(f)\rightarrow\underline{\CM}(A).
$$

Conversely,
a maximal Cohen-Macaulay $A$-module $M$,
regarded as an $S$-module $M_S$,
admits a (not unique) free resolution
$$
\begin{tikzcd}[arrow style=tikz,>=stealth,row sep=5em,column sep=2em] 
0
  \arrow[r]
&
S^n
  \arrow[r,"{\varphi}"]
&
S^n
  \arrow[r,""]
&
M_S
  \arrow[r,""]
&
0.
\end{tikzcd}
$$
It determines
another map $\psi:S^n\rightarrow S^n$
such that
$\varphi\psi=\psi\varphi=f\cdot\id_{S^n}$,
yielding a matrix factorization $\left(\varphi,\psi\right)$ of $f$.
This process defines a quasi-inverse to the above induced functor,
giving an equivalence of stable categories
$
\underline{\MF}(f)
$
and
$
\underline{\CM}(A)
$.

\end{thm}

\bibliographystyle{amsalpha}
\bibliography{IndecomposableXYZHigherRank}

\end{document}